\documentclass[11pt,reqno]{amsart}
\usepackage[margin=3cm]{geometry}
\usepackage{color}               
\usepackage{faktor}
\usepackage{tikz}
\usetikzlibrary{matrix}
\usetikzlibrary{arrows}
\usetikzlibrary{matrix}

\usetikzlibrary{calc}
\usepackage{tikz-cd}
\tikzcdset{scale cd/.style={every label/.append style={scale=#1},
		cells={nodes={scale=#1}}}}
\usepackage{spverbatim}
\usepackage{float}
\usepackage[pdftex,hyperfootnotes=false,colorlinks=false,hypertexnames=false]{hyperref}

\usepackage{amsmath,amssymb,anyfontsize,stmaryrd,mathtools, thmtools,tikz,txfonts,nccmath}
\usepackage[shortlabels]{enumitem}
\usepackage[oldstyle]{libertine}%
\usepackage[T1]{fontenc}
\usepackage[final]{microtype}
\usepackage[utf8]{inputenc}
\usepackage{dutchcal}
\usepackage{csquotes}
\makeatletter
\renewcommand*\libertine@figurestyle{LF}
\makeatother
\usepackage[libertine,libaltvw,liby]{newtxmath}
\makeatletter
\renewcommand*\libertine@figurestyle{OsF}
\makeatother
\usepackage[backend=biber,style=alphabetic,maxalphanames=5,giveninits,sorting=nyt,%
maxbibnames=99]{biblatex}
\usepackage{todonotes}
\reversemarginpar

\usepackage{tikz-cd}
\usepackage[noabbrev,capitalise]{cleveref}
\usepackage{braket}

\theoremstyle{plain}
    \newtheorem{theorem}{Theorem}[section]
    \newtheorem{construction/theorem}[theorem]{Construction/Theorem}
    \newtheorem{corollary}[theorem]{Corollary}
    \newtheorem{lemma}[theorem]{Lemma}
    \newtheorem{proposition}[theorem]{Proposition}
    
\theoremstyle{definition}
    \newtheorem{remark}[theorem]{Remark}
    
    \newtheorem{example}[theorem]{Example}
    \newtheorem{definition}[theorem]{Definition}
    \newtheorem{construction}[theorem]{Construction}

\DeclareMathOperator{\gl}{gl}
\DeclareMathOperator{\val}{\mathrm{val}}
\DeclareMathOperator{\mult}{\mathrm{mult}}
\DeclareMathOperator{\rk}{\mathrm{rank}}
\newtheorem{proofsketch}[theorem]{Proof Sketch}

\newcommand{\comment}[1]{}

\DeclareMathOperator{\Aut}{Aut}

\title{Tropical spin Hurwitz numbers}
\author[L.~L.~Cobigo]{Lou-Jean Leila Cobigo}
\address{L.~L.~Cobigo: Eberhard Karls Universität Tübingen, Fachbereich Mathematik, Auf der Morgenstelle 10, 72076 Tübingen, Germany}
\email{lou-jean-leila.cobigo-bihavan@uni-tuebingen.de}
\author[M.~A.~Hahn]{Marvin Anas Hahn}
\address{M.~A.~Hahn: School of Mathematics, 17 Westland Row, Trinity College Dublin, Dublin 2, Ireland}
\email{hahnma@maths.tcd.ie}

\begin{document}
\begin{abstract}
    We introduce a tropical framework for the study of spin Hurwitz numbers and their completed cycles analogues. We use degeneration techniques to obtain an interpretation of spin Hurwitz numbers via tropical covers. We employ the tropical methods to obtain new closed formulae for spin Hurwitz numbers in special cases. Moreover, we employ Fock space techniques to tropicalise completed cycles spin Hurwitz numbers. Using this, we prove that spin Hurwitz numbers admit a piecewise polynomial structure akin to their classical counterparts. We also derive wall--crossing formulae.
\end{abstract}
\maketitle

\section{Introduction}
	We propose a tropical theory for spin Hurwitz numbers and their completed cycles analogues: Our goal is to provide combinatorial models for objects relevant to the spin invariants, which have so far been studied primarily through geometry, representation theory, and operator formalism.  We use these models for explicit computations and to analyze structural properties. 
	Although classical Hurwitz theory provides important guidance, the spin setting presents genuinely new difficulties and requires separate approaches to the two classes of invariants.
\paragraph{\textbf{Classical theory: type A}}
	Hurwitz numbers enumerate covers of a fixed Riemann surface with specified ramification behavior. Since their emergence in the late nineteenth century, they have attracted sustained interest through their connections to Gromov-Witten theory, integrable systems, and mirror symmetry. Although geometric in origin, their reformulations in other mathematical settings reveal structures that are difficult to see geometrically and provide tools for both structural analysis and explicit computation (\cite{zbMATH07272609}, \cite{zbMATH05049020}, \cite{zbMATH05824576}, \cite{zbMATH06806783}, \cite{zbMATH06460578}).
\begin{figure}[H]
\centering
\tikzset{every picture/.style={line width=0.75pt}} 
\scalebox{0.75}{          
	
	\begin{tikzpicture}[x=0.75pt,y=0.75pt,yscale=-1,xscale=1]

		\draw [color={rgb, 255:red, 0; green, 0; blue, 0 }  ,draw opacity=1 ]   (92.53,227.85) -- (117.95,190.79) ;
		\draw [shift={(119.08,189.14)}, rotate = 124.44] [color={rgb, 255:red, 0; green, 0; blue, 0 }  ,draw opacity=1 ][line width=0.75]    (10.93,-3.29) .. controls (6.95,-1.4) and (3.31,-0.3) .. (0,0) .. controls (3.31,0.3) and (6.95,1.4) .. (10.93,3.29)   ;
		\draw [shift={(91.4,229.5)}, rotate = 304.44] [color={rgb, 255:red, 0; green, 0; blue, 0 }  ,draw opacity=1 ][line width=0.75]    (10.93,-3.29) .. controls (6.95,-1.4) and (3.31,-0.3) .. (0,0) .. controls (3.31,0.3) and (6.95,1.4) .. (10.93,3.29)   ;
		
		\draw [color={rgb, 255:red, 0; green, 0; blue, 0 }  ,draw opacity=1 ]   (200.83,193.47) -- (226.79,225.26) ;
		\draw [shift={(228.06,226.81)}, rotate = 230.76] [color={rgb, 255:red, 0; green, 0; blue, 0 }  ,draw opacity=1 ][line width=0.75]    (10.93,-3.29) .. controls (6.95,-1.4) and (3.31,-0.3) .. (0,0) .. controls (3.31,0.3) and (6.95,1.4) .. (10.93,3.29)   ;
		\draw [shift={(199.56,191.93)}, rotate = 50.76] [color={rgb, 255:red, 0; green, 0; blue, 0 }  ,draw opacity=1 ][line width=0.75]    (10.93,-3.29) .. controls (6.95,-1.4) and (3.31,-0.3) .. (0,0) .. controls (3.31,0.3) and (6.95,1.4) .. (10.93,3.29)   ;
		
		\draw [color={rgb, 255:red, 0; green, 0; blue, 0 }  ,draw opacity=1 ]   (158.9,141.39) -- (159.69,100.54) ;
		\draw [shift={(159.73,98.54)}, rotate = 91.1] [color={rgb, 255:red, 0; green, 0; blue, 0 }  ,draw opacity=1 ][line width=0.75]    (10.93,-3.29) .. controls (6.95,-1.4) and (3.31,-0.3) .. (0,0) .. controls (3.31,0.3) and (6.95,1.4) .. (10.93,3.29)   ;
		\draw [shift={(158.86,143.39)}, rotate = 271.1] [color={rgb, 255:red, 0; green, 0; blue, 0 }  ,draw opacity=1 ][line width=0.75]    (10.93,-3.29) .. controls (6.95,-1.4) and (3.31,-0.3) .. (0,0) .. controls (3.31,0.3) and (6.95,1.4) .. (10.93,3.29)   ;
		
		\draw [color={rgb, 255:red, 0; green, 0; blue, 0 }  ,draw opacity=1 ]   (58.67,222.39) -- (131.92,95.78) ;
		\draw [shift={(132.92,94.05)}, rotate = 120.05] [color={rgb, 255:red, 0; green, 0; blue, 0 }  ,draw opacity=1 ][line width=0.75]    (10.93,-3.29) .. controls (6.95,-1.4) and (3.31,-0.3) .. (0,0) .. controls (3.31,0.3) and (6.95,1.4) .. (10.93,3.29)   ;
		\draw [shift={(57.67,224.12)}, rotate = 300.05] [color={rgb, 255:red, 0; green, 0; blue, 0 }  ,draw opacity=1 ][line width=0.75]    (10.93,-3.29) .. controls (6.95,-1.4) and (3.31,-0.3) .. (0,0) .. controls (3.31,0.3) and (6.95,1.4) .. (10.93,3.29)   ;
		
		\draw [color={rgb, 255:red, 0; green, 0; blue, 0 }  ,draw opacity=1 ]   (265.97,227.78) -- (189.29,97.57) ;
		\draw [shift={(188.27,95.85)}, rotate = 59.51] [color={rgb, 255:red, 0; green, 0; blue, 0 }  ,draw opacity=1 ][line width=0.75]    (10.93,-3.29) .. controls (6.95,-1.4) and (3.31,-0.3) .. (0,0) .. controls (3.31,0.3) and (6.95,1.4) .. (10.93,3.29)   ;
		\draw [shift={(266.98,229.5)}, rotate = 239.51] [color={rgb, 255:red, 0; green, 0; blue, 0 }  ,draw opacity=1 ][line width=0.75]    (10.93,-3.29) .. controls (6.95,-1.4) and (3.31,-0.3) .. (0,0) .. controls (3.31,0.3) and (6.95,1.4) .. (10.93,3.29)   ;
		
		\draw [color={rgb, 255:red, 0; green, 0; blue, 0 }  ,draw opacity=1 ]   (189.73,250.1) -- (129.73,249.26) ;
		\draw [shift={(127.73,249.24)}, rotate = 0.8] [color={rgb, 255:red, 0; green, 0; blue, 0 }  ,draw opacity=1 ][line width=0.75]    (10.93,-3.29) .. controls (6.95,-1.4) and (3.31,-0.3) .. (0,0) .. controls (3.31,0.3) and (6.95,1.4) .. (10.93,3.29)   ;
		\draw [shift={(191.73,250.13)}, rotate = 180.8] [color={rgb, 255:red, 0; green, 0; blue, 0 }  ,draw opacity=1 ][line width=0.75]    (10.93,-3.29) .. controls (6.95,-1.4) and (3.31,-0.3) .. (0,0) .. controls (3.31,0.3) and (6.95,1.4) .. (10.93,3.29)   ;
		
		\draw   (115.33,157.45) .. controls (115.33,154.44) and (117.78,152) .. (120.79,152) -- (203.88,152) .. controls (206.89,152) and (209.33,154.44) .. (209.33,157.45) -- (209.33,173.81) .. controls (209.33,176.83) and (206.89,179.27) .. (203.88,179.27) -- (120.79,179.27) .. controls (117.78,179.27) and (115.33,176.83) .. (115.33,173.81) -- cycle ;
		
		\draw [color={rgb, 255:red, 0; green, 0; blue, 0 }  ,draw opacity=1 ]   (399.53,229.85) -- (424.95,192.79) ;
		\draw [shift={(426.08,191.14)}, rotate = 124.44] [color={rgb, 255:red, 0; green, 0; blue, 0 }  ,draw opacity=1 ][line width=0.75]    (10.93,-3.29) .. controls (6.95,-1.4) and (3.31,-0.3) .. (0,0) .. controls (3.31,0.3) and (6.95,1.4) .. (10.93,3.29)   ;
		\draw [shift={(398.4,231.5)}, rotate = 304.44] [color={rgb, 255:red, 0; green, 0; blue, 0 }  ,draw opacity=1 ][line width=0.75]    (10.93,-3.29) .. controls (6.95,-1.4) and (3.31,-0.3) .. (0,0) .. controls (3.31,0.3) and (6.95,1.4) .. (10.93,3.29)   ;
		
		\draw [color={rgb, 255:red, 0; green, 0; blue, 0 }  ,draw opacity=1 ]   (507.83,195.47) -- (533.79,227.26) ;
		\draw [shift={(535.06,228.81)}, rotate = 230.76] [color={rgb, 255:red, 0; green, 0; blue, 0 }  ,draw opacity=1 ][line width=0.75]    (10.93,-3.29) .. controls (6.95,-1.4) and (3.31,-0.3) .. (0,0) .. controls (3.31,0.3) and (6.95,1.4) .. (10.93,3.29)   ;
		\draw [shift={(506.56,193.93)}, rotate = 50.76] [color={rgb, 255:red, 0; green, 0; blue, 0 }  ,draw opacity=1 ][line width=0.75]    (10.93,-3.29) .. controls (6.95,-1.4) and (3.31,-0.3) .. (0,0) .. controls (3.31,0.3) and (6.95,1.4) .. (10.93,3.29)   ;
		
		\draw [color={rgb, 255:red, 0; green, 0; blue, 0 }  ,draw opacity=1 ]   (465.9,143.39) -- (466.69,102.54) ;
		\draw [shift={(466.73,100.54)}, rotate = 91.1] [color={rgb, 255:red, 0; green, 0; blue, 0 }  ,draw opacity=1 ][line width=0.75]    (10.93,-3.29) .. controls (6.95,-1.4) and (3.31,-0.3) .. (0,0) .. controls (3.31,0.3) and (6.95,1.4) .. (10.93,3.29)   ;
		\draw [shift={(465.86,145.39)}, rotate = 271.1] [color={rgb, 255:red, 0; green, 0; blue, 0 }  ,draw opacity=1 ][line width=0.75]    (10.93,-3.29) .. controls (6.95,-1.4) and (3.31,-0.3) .. (0,0) .. controls (3.31,0.3) and (6.95,1.4) .. (10.93,3.29)   ;
		
		\draw [color={rgb, 255:red, 0; green, 0; blue, 0 }  ,draw opacity=1 ]   (365.67,224.39) -- (438.92,97.78) ;
		\draw [shift={(439.92,96.05)}, rotate = 120.05] [color={rgb, 255:red, 0; green, 0; blue, 0 }  ,draw opacity=1 ][line width=0.75]    (10.93,-3.29) .. controls (6.95,-1.4) and (3.31,-0.3) .. (0,0) .. controls (3.31,0.3) and (6.95,1.4) .. (10.93,3.29)   ;
		\draw [shift={(364.67,226.12)}, rotate = 300.05] [color={rgb, 255:red, 0; green, 0; blue, 0 }  ,draw opacity=1 ][line width=0.75]    (10.93,-3.29) .. controls (6.95,-1.4) and (3.31,-0.3) .. (0,0) .. controls (3.31,0.3) and (6.95,1.4) .. (10.93,3.29)   ;
		
		\draw [color={rgb, 255:red, 0; green, 0; blue, 0 }  ,draw opacity=1 ]   (572.97,229.78) -- (496.29,99.57) ;
		\draw [shift={(495.27,97.85)}, rotate = 59.51] [color={rgb, 255:red, 0; green, 0; blue, 0 }  ,draw opacity=1 ][line width=0.75]    (10.93,-3.29) .. controls (6.95,-1.4) and (3.31,-0.3) .. (0,0) .. controls (3.31,0.3) and (6.95,1.4) .. (10.93,3.29)   ;
		\draw [shift={(573.98,231.5)}, rotate = 239.51] [color={rgb, 255:red, 0; green, 0; blue, 0 }  ,draw opacity=1 ][line width=0.75]    (10.93,-3.29) .. controls (6.95,-1.4) and (3.31,-0.3) .. (0,0) .. controls (3.31,0.3) and (6.95,1.4) .. (10.93,3.29)   ;
		
		\draw [color={rgb, 255:red, 0; green, 0; blue, 0 }  ,draw opacity=1 ]   (496.33,251.01) -- (436.73,251.23) ;
		\draw [shift={(434.73,251.24)}, rotate = 359.79] [color={rgb, 255:red, 0; green, 0; blue, 0 }  ,draw opacity=1 ][line width=0.75]    (10.93,-3.29) .. controls (6.95,-1.4) and (3.31,-0.3) .. (0,0) .. controls (3.31,0.3) and (6.95,1.4) .. (10.93,3.29)   ;
		\draw [shift={(498.33,251)}, rotate = 179.79] [color={rgb, 255:red, 0; green, 0; blue, 0 }  ,draw opacity=1 ][line width=0.75]    (10.93,-3.29) .. controls (6.95,-1.4) and (3.31,-0.3) .. (0,0) .. controls (3.31,0.3) and (6.95,1.4) .. (10.93,3.29)   ;

		\draw (451,63.07) node [anchor=north west][inner sep=0.75pt]    {$HW$};
		
		\draw (444,163.07) node [anchor=north west][inner sep=0.75pt]    {$HW_{trop}$};
		
		\draw (109.12,61.04) node [anchor=north west][inner sep=0.75pt]    {$Alg.\ geometry$};
		
		\draw (23.06,241.47) node [anchor=north west][inner sep=0.75pt]  [rotate=-359.58]  {$Rep.\ theory$};
		
		\draw (216.74,243.43) node [anchor=north west][inner sep=0.75pt]    {$Fock\ space\ $};
		
		\draw (355,243.07) node [anchor=north west][inner sep=0.75pt]    {$HW_{rep.}$};
		
		\draw (540,247.07) node [anchor=north west][inner sep=0.75pt]    {$HW_{Fock}$};
		
		\draw (118.19,158.21) node [anchor=north west][inner sep=0.75pt]    {$Trop.\ geom.$};

	\end{tikzpicture}
}
\caption{ \footnotesize A network for type A theory: Areas supporting Hurwitz numbers ($HW$), on the left, and their realization within each setting, on the right. Double-headed arrows on the right indicate correspondences (equalities) between different realizations. Double-headed arrows on the left represent bidirectional bridges between the different settings.}
\label{figure_networkforHW}
\end{figure}
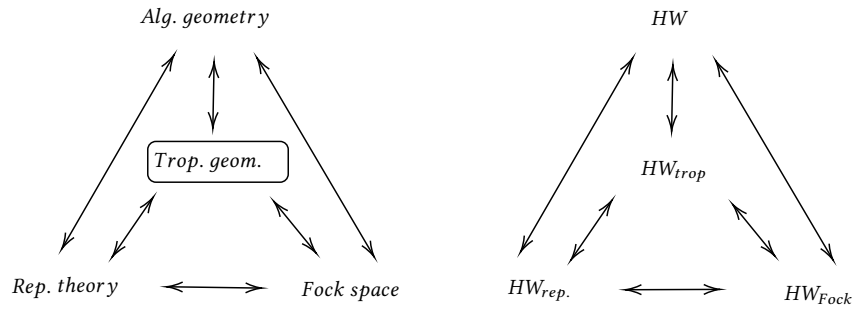

Central settings (algebraic geometry, tropical geometry, representation theory, group theory and Fock space formalism) together with their interconnections are illustrated in \cref{figure_networkforHW}. Among these, tropical geometry enjoys a central position with direct connections to nearly all areas underlying Hurwitz theory. By translating geometric counting problems into combinatorial ones, this perspective provides a framework that facilitates explicit analysis.

Building on the representation-theoretic expression $HW_{rep.}$, one obtains a new family of purely representation-theoretic invariants: \emph{Hurwitz numbers with completed cycles} ($HWC$) (wavy arrow in \cref{figure_networkforHWC}). These invariants connect Hurwitz theory to Gromov-Witten theory through the remarkable GW/H correspondence. They are amenable to explicit computation, are linked to integrable systems, and extend the piecewise-polynomiality and quasi-polynomiality properties of double and single Hurwitz numbers, respectively.

\begin{figure}[H]

\centering
\tikzset{every picture/.style={line width=0.75pt}} 
\scalebox{0.75}{
	
	\begin{tikzpicture}[x=0.75pt,y=0.75pt,yscale=-1,xscale=1]

		\draw [color={rgb, 255:red, 0; green, 0; blue, 0 }  ,draw opacity=1 ]   (427.67,211.52) -- (452.99,183.48) ;
		\draw [shift={(454.33,182)}, rotate = 132.09] [color={rgb, 255:red, 0; green, 0; blue, 0 }  ,draw opacity=1 ][line width=0.75]    (10.93,-3.29) .. controls (6.95,-1.4) and (3.31,-0.3) .. (0,0) .. controls (3.31,0.3) and (6.95,1.4) .. (10.93,3.29)   ;
		\draw [shift={(426.33,213)}, rotate = 312.09] [color={rgb, 255:red, 0; green, 0; blue, 0 }  ,draw opacity=1 ][line width=0.75]    (10.93,-3.29) .. controls (6.95,-1.4) and (3.31,-0.3) .. (0,0) .. controls (3.31,0.3) and (6.95,1.4) .. (10.93,3.29)   ;
		
		\draw [color={rgb, 255:red, 0; green, 0; blue, 0 }  ,draw opacity=1 ]   (544.61,183.86) -- (570.05,214.46) ;
		\draw [shift={(571.33,216)}, rotate = 230.26] [color={rgb, 255:red, 0; green, 0; blue, 0 }  ,draw opacity=1 ][line width=0.75]    (10.93,-3.29) .. controls (6.95,-1.4) and (3.31,-0.3) .. (0,0) .. controls (3.31,0.3) and (6.95,1.4) .. (10.93,3.29)   ;
		\draw [shift={(543.33,182.32)}, rotate = 50.26] [color={rgb, 255:red, 0; green, 0; blue, 0 }  ,draw opacity=1 ][line width=0.75]    (10.93,-3.29) .. controls (6.95,-1.4) and (3.31,-0.3) .. (0,0) .. controls (3.31,0.3) and (6.95,1.4) .. (10.93,3.29)   ;
		
		\draw [color={rgb, 255:red, 0; green, 0; blue, 0 }  ,draw opacity=1 ]   (494.59,133.55) -- (495.5,89.58) ;
		\draw [shift={(495.54,87.58)}, rotate = 91.19] [color={rgb, 255:red, 0; green, 0; blue, 0 }  ,draw opacity=1 ][line width=0.75]    (10.93,-3.29) .. controls (6.95,-1.4) and (3.31,-0.3) .. (0,0) .. controls (3.31,0.3) and (6.95,1.4) .. (10.93,3.29)   ;
		\draw [shift={(494.54,135.55)}, rotate = 271.19] [color={rgb, 255:red, 0; green, 0; blue, 0 }  ,draw opacity=1 ][line width=0.75]    (10.93,-3.29) .. controls (6.95,-1.4) and (3.31,-0.3) .. (0,0) .. controls (3.31,0.3) and (6.95,1.4) .. (10.93,3.29)   ;
		
		\draw [color={rgb, 255:red, 0; green, 0; blue, 0 }  ,draw opacity=1 ]   (388.37,210.99) -- (465.3,83.71) ;
		\draw [shift={(466.33,82)}, rotate = 121.15] [color={rgb, 255:red, 0; green, 0; blue, 0 }  ,draw opacity=1 ][line width=0.75]    (10.93,-3.29) .. controls (6.95,-1.4) and (3.31,-0.3) .. (0,0) .. controls (3.31,0.3) and (6.95,1.4) .. (10.93,3.29)   ;
		\draw [shift={(387.33,212.7)}, rotate = 301.15] [color={rgb, 255:red, 0; green, 0; blue, 0 }  ,draw opacity=1 ][line width=0.75]    (10.93,-3.29) .. controls (6.95,-1.4) and (3.31,-0.3) .. (0,0) .. controls (3.31,0.3) and (6.95,1.4) .. (10.93,3.29)   ;
		
		\draw [color={rgb, 255:red, 0; green, 0; blue, 0 }  ,draw opacity=1 ]   (607.25,207.32) -- (527.42,83.68) ;
		\draw [shift={(526.33,82)}, rotate = 57.15] [color={rgb, 255:red, 0; green, 0; blue, 0 }  ,draw opacity=1 ][line width=0.75]    (10.93,-3.29) .. controls (6.95,-1.4) and (3.31,-0.3) .. (0,0) .. controls (3.31,0.3) and (6.95,1.4) .. (10.93,3.29)   ;
		\draw [shift={(608.33,209)}, rotate = 237.15] [color={rgb, 255:red, 0; green, 0; blue, 0 }  ,draw opacity=1 ][line width=0.75]    (10.93,-3.29) .. controls (6.95,-1.4) and (3.31,-0.3) .. (0,0) .. controls (3.31,0.3) and (6.95,1.4) .. (10.93,3.29)   ;
		
		\draw [color={rgb, 255:red, 0; green, 0; blue, 0 }  ,draw opacity=1 ]   (528.54,233.78) -- (458.54,232.87) ;
		\draw [shift={(456.54,232.85)}, rotate = 0.74] [color={rgb, 255:red, 0; green, 0; blue, 0 }  ,draw opacity=1 ][line width=0.75]    (10.93,-3.29) .. controls (6.95,-1.4) and (3.31,-0.3) .. (0,0) .. controls (3.31,0.3) and (6.95,1.4) .. (10.93,3.29)   ;
		\draw [shift={(530.54,233.81)}, rotate = 180.74] [color={rgb, 255:red, 0; green, 0; blue, 0 }  ,draw opacity=1 ][line width=0.75]    (10.93,-3.29) .. controls (6.95,-1.4) and (3.31,-0.3) .. (0,0) .. controls (3.31,0.3) and (6.95,1.4) .. (10.93,3.29)   ;
		
		\draw [color={rgb, 255:red, 155; green, 155; blue, 155 }  ,draw opacity=1 ]   (89.67,211.52) -- (114.99,183.48) ;
		\draw [shift={(116.33,182)}, rotate = 132.09] [color={rgb, 255:red, 155; green, 155; blue, 155 }  ,draw opacity=1 ][line width=0.75]    (10.93,-3.29) .. controls (6.95,-1.4) and (3.31,-0.3) .. (0,0) .. controls (3.31,0.3) and (6.95,1.4) .. (10.93,3.29)   ;
		\draw [shift={(88.33,213)}, rotate = 312.09] [color={rgb, 255:red, 155; green, 155; blue, 155 }  ,draw opacity=1 ][line width=0.75]    (10.93,-3.29) .. controls (6.95,-1.4) and (3.31,-0.3) .. (0,0) .. controls (3.31,0.3) and (6.95,1.4) .. (10.93,3.29)   ;
		
		\draw [color={rgb, 255:red, 155; green, 155; blue, 155 }  ,draw opacity=1 ]   (190.76,181.73) -- (216.91,207.59) ;
		\draw [shift={(218.33,209)}, rotate = 224.68] [color={rgb, 255:red, 155; green, 155; blue, 155 }  ,draw opacity=1 ][line width=0.75]    (10.93,-3.29) .. controls (6.95,-1.4) and (3.31,-0.3) .. (0,0) .. controls (3.31,0.3) and (6.95,1.4) .. (10.93,3.29)   ;
		\draw [shift={(189.33,180.32)}, rotate = 44.68] [color={rgb, 255:red, 155; green, 155; blue, 155 }  ,draw opacity=1 ][line width=0.75]    (10.93,-3.29) .. controls (6.95,-1.4) and (3.31,-0.3) .. (0,0) .. controls (3.31,0.3) and (6.95,1.4) .. (10.93,3.29)   ;
		
		\draw [color={rgb, 255:red, 155; green, 155; blue, 155 }  ,draw opacity=1 ]   (150.59,133.55) -- (151.5,89.58) ;
		\draw [shift={(151.54,87.58)}, rotate = 91.19] [color={rgb, 255:red, 155; green, 155; blue, 155 }  ,draw opacity=1 ][line width=0.75]    (10.93,-3.29) .. controls (6.95,-1.4) and (3.31,-0.3) .. (0,0) .. controls (3.31,0.3) and (6.95,1.4) .. (10.93,3.29)   ;
		\draw [shift={(150.54,135.55)}, rotate = 271.19] [color={rgb, 255:red, 155; green, 155; blue, 155 }  ,draw opacity=1 ][line width=0.75]    (10.93,-3.29) .. controls (6.95,-1.4) and (3.31,-0.3) .. (0,0) .. controls (3.31,0.3) and (6.95,1.4) .. (10.93,3.29)   ;
		
		\draw [color={rgb, 255:red, 155; green, 155; blue, 155 }  ,draw opacity=1 ]   (54.37,205.99) -- (131.3,78.71) ;
		\draw [shift={(132.33,77)}, rotate = 121.15] [color={rgb, 255:red, 155; green, 155; blue, 155 }  ,draw opacity=1 ][line width=0.75]    (10.93,-3.29) .. controls (6.95,-1.4) and (3.31,-0.3) .. (0,0) .. controls (3.31,0.3) and (6.95,1.4) .. (10.93,3.29)   ;
		\draw [shift={(53.33,207.7)}, rotate = 301.15] [color={rgb, 255:red, 155; green, 155; blue, 155 }  ,draw opacity=1 ][line width=0.75]    (10.93,-3.29) .. controls (6.95,-1.4) and (3.31,-0.3) .. (0,0) .. controls (3.31,0.3) and (6.95,1.4) .. (10.93,3.29)   ;
		
		\draw [color={rgb, 255:red, 155; green, 155; blue, 155 }  ,draw opacity=1 ]   (253.25,203.32) -- (173.42,79.68) ;
		\draw [shift={(172.33,78)}, rotate = 57.15] [color={rgb, 255:red, 155; green, 155; blue, 155 }  ,draw opacity=1 ][line width=0.75]    (10.93,-3.29) .. controls (6.95,-1.4) and (3.31,-0.3) .. (0,0) .. controls (3.31,0.3) and (6.95,1.4) .. (10.93,3.29)   ;
		\draw [shift={(254.33,205)}, rotate = 237.15] [color={rgb, 255:red, 155; green, 155; blue, 155 }  ,draw opacity=1 ][line width=0.75]    (10.93,-3.29) .. controls (6.95,-1.4) and (3.31,-0.3) .. (0,0) .. controls (3.31,0.3) and (6.95,1.4) .. (10.93,3.29)   ;
		
		\draw [color={rgb, 255:red, 155; green, 155; blue, 155 }  ,draw opacity=1 ]   (184.54,230.78) -- (114.54,229.87) ;
		\draw [shift={(112.54,229.85)}, rotate = 0.74] [color={rgb, 255:red, 155; green, 155; blue, 155 }  ,draw opacity=1 ][line width=0.75]    (10.93,-3.29) .. controls (6.95,-1.4) and (3.31,-0.3) .. (0,0) .. controls (3.31,0.3) and (6.95,1.4) .. (10.93,3.29)   ;
		\draw [shift={(186.54,230.81)}, rotate = 180.74] [color={rgb, 255:red, 155; green, 155; blue, 155 }  ,draw opacity=1 ][line width=0.75]    (10.93,-3.29) .. controls (6.95,-1.4) and (3.31,-0.3) .. (0,0) .. controls (3.31,0.3) and (6.95,1.4) .. (10.93,3.29)   ;
		
		\draw [color={rgb, 255:red, 0; green, 0; blue, 0 }  ,draw opacity=1 ]   (297.33,237) .. controls (299,235.33) and (300.66,235.33) .. (302.33,237) .. controls (304,238.67) and (305.66,238.67) .. (307.33,237) .. controls (309,235.33) and (310.66,235.33) .. (312.33,237) .. controls (314,238.67) and (315.66,238.67) .. (317.33,237) .. controls (319,235.33) and (320.66,235.33) .. (322.33,237) .. controls (324,238.67) and (325.66,238.67) .. (327.33,237) .. controls (329,235.33) and (330.66,235.33) .. (332.33,237) .. controls (334,238.67) and (335.66,238.67) .. (337.33,237) -- (345.33,237) ;
		\draw [shift={(347.33,237)}, rotate = 180] [color={rgb, 255:red, 0; green, 0; blue, 0 }  ,draw opacity=1 ][line width=0.75]    (10.93,-3.29) .. controls (6.95,-1.4) and (3.31,-0.3) .. (0,0) .. controls (3.31,0.3) and (6.95,1.4) .. (10.93,3.29)   ;

		\draw (477,50.99) node [anchor=north west][inner sep=0.75pt]    {$GWI$};
		
		\draw (469,155.49) node [anchor=north west][inner sep=0.75pt]    {$HWC_{trop}$};
		
		\draw (362,225.93) node [anchor=north west][inner sep=0.75pt]    {$HWC$};
		
		\draw (578,226.85) node [anchor=north west][inner sep=0.75pt]    {$HWC_{Fock}$};
		\draw (398.37,143.57) node [anchor=north west][inner sep=0.75pt]  [rotate=-303.45]  {\footnotesize {\hypersetup{pdfborder={0 0 0}}\cite{zbMATH05049020}}};
		\draw (476.03,242.36) node [anchor=north west][inner sep=0.75pt]  [rotate=-0.21]  {\footnotesize {\hypersetup{pdfborder={0 0 0}}\cite{zbMATH05049020}}};
		\draw (515.9,87.36) node [anchor=north west][inner sep=0.75pt]  [rotate=-88.93]  {\footnotesize {\hypersetup{pdfborder={0 0 0}}\cite{zbMATH07272609}}};
		\draw (560.24,164.04) node [anchor=north west][inner sep=0.75pt]  [rotate=-52.95]  {\footnotesize {\hypersetup{pdfborder={0 0 0}}\cite{zbMATH07272609}}};
		\draw (139,45.99) node [anchor=north west][inner sep=0.75pt]    {$HW$};
		\draw (129,151.49) node [anchor=north west][inner sep=0.75pt]    {$HW_{trop}$};
		\draw (226,223.93) node [anchor=north west][inner sep=0.75pt]    {$HW_{rep.}$};
		\draw (13,226.85) node [anchor=north west][inner sep=0.75pt]    {$HW_{Fock}$};
		\draw (292.59,212.06) node [anchor=north west][inner sep=0.75pt]  [rotate=-359.51]  {$\text{defines}$};

	\end{tikzpicture}

}
\caption{\footnotesize A network for type A theory: The network for Hurwitz numbers, on the left, is a blueprint for the completed‑cycles variants on the right. The wavy arrow indicates that $HWC$ is defined via the representation‑theoretic expression of $HW$.}
\label{figure_networkforHWC}
\end{figure}
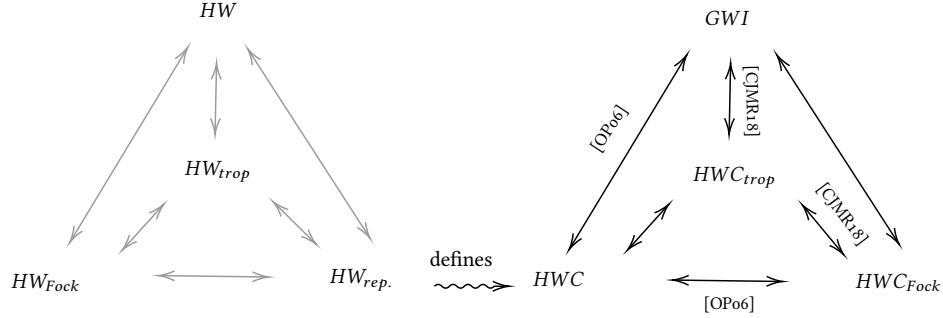

\paragraph{\textbf{Spin theory: type B}} 
A refinement of the classical count: \emph{Spin Hurwitz numbers} enumerate branched covers of a fixed spin curve (a curve with theta characteristic) with specified ramification behavior and sign fixed by a lift of this theta characteristic to the source. The resulting signed invariants retain geometric information that ordinary Hurwitz numbers do not detect. They were introduced by Eskin-Okounkov-Pandharipande in 2008 for certain computations in the moduli space of differentials on a Riemann surface \cite{zbMATH05234510}.

Previous work \cite{zbMATH06196720,MR3152971,gunningham2016spin,zbMATH07184121,arXiv:2104.05697} has established spin counterparts of several components of classical Hurwitz theory, including a spin GW/H correspondence for $\mathbb{P}^1$ \cite{Giacchetto_2025}. The tropical parts of this network (\cref{figure_networkspinHWC}), however, are still missing. We address this gap: 
\begin{itemize}
	\item \textbf{Spin Hurwitz numbers: Geometric setting (\cref{section_tropicalspinHurwitz}):} Starting from \cite{zbMATH05234510}, \cite{zbMATH06196720} and \cite{MR3152971} develop degeneration techniques, which we use to introduce a tropical count for $HW^{spin}$ and prove its equivalence to the classical one. The key new ingredient is a parity function on the source graph, which combinatorially records the sign entering the spin count.

	\item  \textbf{Spin Hurwitz numbers with completed cycles: Representation-theoretic and Fock space setting (\cref{section_tropiclspincompletedcycles}):} Following the representation-theoretic formulation of spin Hurwitz numbers in  \cite{gunningham2016spin}, \cite{arXiv:2104.05697} introduced their completed cycles analogues, specifically spin Hurwitz numbers with $(r+1)$-completed cycles (building on \cite{zbMATH05234510} and \cite{zbMATH07184121}). We use Fock space formalism from \cite{arXiv:2104.05697} to re-geometrize completed cycles spin Hurwitz numbers, and study piecewise polynomiality,their polynomiality domains and wall-crossing formulas from a combinatorial perspective. The restriction to odd edge weights in the spin setting changes the underlying lattice-point problem and, in higher genus, requires an adapted Ehrhart-theoretic approach.
\end{itemize}
Together, these results provide the missing tropical part of the type B network shown in \cref{figure_networkspinHWC}.

\begin{figure}[H]
\centering
\tikzset{every picture/.style={line width=0.75pt}} 
\scalebox{0.75}{
	\begin{tikzpicture}[x=0.75pt,y=0.75pt,yscale=-1,xscale=1]
		
		\draw [color={rgb, 255:red, 0; green, 0; blue, 0 }  ,draw opacity=1 ]   (394.67,204.52) -- (419.99,176.48) ;
		\draw [shift={(421.33,175)}, rotate = 132.09] [color={rgb, 255:red, 0; green, 0; blue, 0 }  ,draw opacity=1 ][line width=0.75]    (10.93,-3.29) .. controls (6.95,-1.4) and (3.31,-0.3) .. (0,0) .. controls (3.31,0.3) and (6.95,1.4) .. (10.93,3.29)   ;
		\draw [shift={(393.33,206)}, rotate = 312.09] [color={rgb, 255:red, 0; green, 0; blue, 0 }  ,draw opacity=1 ][line width=0.75]    (10.93,-3.29) .. controls (6.95,-1.4) and (3.31,-0.3) .. (0,0) .. controls (3.31,0.3) and (6.95,1.4) .. (10.93,3.29)   ;
		\draw [color={rgb, 255:red, 0; green, 0; blue, 0 }  ,draw opacity=1 ]   (516.61,180.86) -- (542.05,211.46) ;
		\draw [shift={(543.33,213)}, rotate = 230.26] [color={rgb, 255:red, 0; green, 0; blue, 0 }  ,draw opacity=1 ][line width=0.75]    (10.93,-3.29) .. controls (6.95,-1.4) and (3.31,-0.3) .. (0,0) .. controls (3.31,0.3) and (6.95,1.4) .. (10.93,3.29)   ;
		\draw [shift={(515.33,179.32)}, rotate = 50.26] [color={rgb, 255:red, 0; green, 0; blue, 0 }  ,draw opacity=1 ][line width=0.75]    (10.93,-3.29) .. controls (6.95,-1.4) and (3.31,-0.3) .. (0,0) .. controls (3.31,0.3) and (6.95,1.4) .. (10.93,3.29)   ;
		\draw [color={rgb, 255:red, 0; green, 0; blue, 0 }  ,draw opacity=1 ] [dash pattern={on 0.84pt off 2.51pt}]  (466.54,132.55) -- (466.34,82.33) ;
		\draw [shift={(466.33,80.33)}, rotate = 89.78] [color={rgb, 255:red, 0; green, 0; blue, 0 }  ,draw opacity=1 ][line width=0.75]    (10.93,-3.29) .. controls (6.95,-1.4) and (3.31,-0.3) .. (0,0) .. controls (3.31,0.3) and (6.95,1.4) .. (10.93,3.29)   ;
		\draw [shift={(466.54,134.55)}, rotate = 269.78] [color={rgb, 255:red, 0; green, 0; blue, 0 }  ,draw opacity=1 ][line width=0.75]    (10.93,-3.29) .. controls (6.95,-1.4) and (3.31,-0.3) .. (0,0) .. controls (3.31,0.3) and (6.95,1.4) .. (10.93,3.29)   ;
		\draw [color={rgb, 255:red, 0; green, 0; blue, 0 }  ,draw opacity=1 ]   (362.37,200.99) -- (439.3,73.71) ;
		\draw [shift={(440.33,72)}, rotate = 121.15] [color={rgb, 255:red, 0; green, 0; blue, 0 }  ,draw opacity=1 ][line width=0.75]    (10.93,-3.29) .. controls (6.95,-1.4) and (3.31,-0.3) .. (0,0) .. controls (3.31,0.3) and (6.95,1.4) .. (10.93,3.29)   ;
		\draw [shift={(361.33,202.7)}, rotate = 301.15] [color={rgb, 255:red, 0; green, 0; blue, 0 }  ,draw opacity=1 ][line width=0.75]    (10.93,-3.29) .. controls (6.95,-1.4) and (3.31,-0.3) .. (0,0) .. controls (3.31,0.3) and (6.95,1.4) .. (10.93,3.29)   ;
		\draw [color={rgb, 255:red, 0; green, 0; blue, 0 }  ,draw opacity=1 ]   (582.25,200.32) -- (502.42,76.68) ;
		\draw [shift={(501.33,75)}, rotate = 57.15] [color={rgb, 255:red, 0; green, 0; blue, 0 }  ,draw opacity=1 ][line width=0.75]    (10.93,-3.29) .. controls (6.95,-1.4) and (3.31,-0.3) .. (0,0) .. controls (3.31,0.3) and (6.95,1.4) .. (10.93,3.29)   ;
		\draw [shift={(583.33,202)}, rotate = 237.15] [color={rgb, 255:red, 0; green, 0; blue, 0 }  ,draw opacity=1 ][line width=0.75]    (10.93,-3.29) .. controls (6.95,-1.4) and (3.31,-0.3) .. (0,0) .. controls (3.31,0.3) and (6.95,1.4) .. (10.93,3.29)   ;
		\draw [color={rgb, 255:red, 0; green, 0; blue, 0 }  ,draw opacity=1 ]   (497.54,230.78) -- (427.54,229.87) ;
		\draw [shift={(425.54,229.85)}, rotate = 0.74] [color={rgb, 255:red, 0; green, 0; blue, 0 }  ,draw opacity=1 ][line width=0.75]    (10.93,-3.29) .. controls (6.95,-1.4) and (3.31,-0.3) .. (0,0) .. controls (3.31,0.3) and (6.95,1.4) .. (10.93,3.29)   ;
		\draw [shift={(499.54,230.81)}, rotate = 180.74] [color={rgb, 255:red, 0; green, 0; blue, 0 }  ,draw opacity=1 ][line width=0.75]    (10.93,-3.29) .. controls (6.95,-1.4) and (3.31,-0.3) .. (0,0) .. controls (3.31,0.3) and (6.95,1.4) .. (10.93,3.29)   ;
		\draw [color={rgb, 255:red, 155; green, 155; blue, 155 }  ,draw opacity=1 ]   (138.71,183.78) -- (168.96,215.88) ;
		\draw [shift={(170.33,217.33)}, rotate = 226.7] [color={rgb, 255:red, 155; green, 155; blue, 155 }  ,draw opacity=1 ][line width=0.75]    (10.93,-3.29) .. controls (6.95,-1.4) and (3.31,-0.3) .. (0,0) .. controls (3.31,0.3) and (6.95,1.4) .. (10.93,3.29)   ;
		\draw [shift={(137.33,182.32)}, rotate = 46.7] [color={rgb, 255:red, 155; green, 155; blue, 155 }  ,draw opacity=1 ][line width=0.75]    (10.93,-3.29) .. controls (6.95,-1.4) and (3.31,-0.3) .. (0,0) .. controls (3.31,0.3) and (6.95,1.4) .. (10.93,3.29)   ;
		\draw [color={rgb, 255:red, 155; green, 155; blue, 155 }  ,draw opacity=1 ]   (104.59,124.55) -- (105.5,80.58) ;
		\draw [shift={(105.54,78.58)}, rotate = 91.19] [color={rgb, 255:red, 155; green, 155; blue, 155 }  ,draw opacity=1 ][line width=0.75]    (10.93,-3.29) .. controls (6.95,-1.4) and (3.31,-0.3) .. (0,0) .. controls (3.31,0.3) and (6.95,1.4) .. (10.93,3.29)   ;
		\draw [shift={(104.54,126.55)}, rotate = 271.19] [color={rgb, 255:red, 155; green, 155; blue, 155 }  ,draw opacity=1 ][line width=0.75]    (10.93,-3.29) .. controls (6.95,-1.4) and (3.31,-0.3) .. (0,0) .. controls (3.31,0.3) and (6.95,1.4) .. (10.93,3.29)   ;
		\draw [color={rgb, 255:red, 155; green, 155; blue, 155 }  ,draw opacity=1 ]   (217.25,199.32) -- (137.42,75.68) ;
		\draw [shift={(136.33,74)}, rotate = 57.15] [color={rgb, 255:red, 155; green, 155; blue, 155 }  ,draw opacity=1 ][line width=0.75]    (10.93,-3.29) .. controls (6.95,-1.4) and (3.31,-0.3) .. (0,0) .. controls (3.31,0.3) and (6.95,1.4) .. (10.93,3.29)   ;
		\draw [shift={(218.33,201)}, rotate = 237.15] [color={rgb, 255:red, 155; green, 155; blue, 155 }  ,draw opacity=1 ][line width=0.75]    (10.93,-3.29) .. controls (6.95,-1.4) and (3.31,-0.3) .. (0,0) .. controls (3.31,0.3) and (6.95,1.4) .. (10.93,3.29)   ;
		\draw [color={rgb, 255:red, 0; green, 0; blue, 0 }  ,draw opacity=1 ]   (268.33,234) .. controls (270.03,232.37) and (271.7,232.4) .. (273.33,234.1) .. controls (274.97,235.79) and (276.64,235.82) .. (278.33,234.19) .. controls (280.03,232.56) and (281.7,232.59) .. (283.33,234.29) .. controls (284.97,235.98) and (286.64,236.01) .. (288.33,234.38) .. controls (290.03,232.75) and (291.7,232.78) .. (293.33,234.48) .. controls (294.96,236.18) and (296.63,236.21) .. (298.33,234.58) .. controls (300.02,232.95) and (301.69,232.98) .. (303.33,234.67) .. controls (304.96,236.37) and (306.63,236.4) .. (308.33,234.77) -- (310.34,234.81) -- (318.33,234.96) ;
		\draw [shift={(320.33,235)}, rotate = 181.1] [color={rgb, 255:red, 0; green, 0; blue, 0 }  ,draw opacity=1 ][line width=0.75]    (10.93,-3.29) .. controls (6.95,-1.4) and (3.31,-0.3) .. (0,0) .. controls (3.31,0.3) and (6.95,1.4) .. (10.93,3.29)   ;
		
		\draw (439,43.99) node [anchor=north west][inner sep=0.75pt]    {$GWI^{spin}$};
		\draw (438,148.49) node [anchor=north west][inner sep=0.75pt]    {$HWC_{trop}^{spin}$};
		\draw (338,221.93) node [anchor=north west][inner sep=0.75pt]    {$HWC^{spin}$};
		\draw (546,224.85) node [anchor=north west][inner sep=0.75pt]    {$HWC_{Fock}^{spin}$};
		\draw (367.37,148.57) node [anchor=north west][inner sep=0.75pt]  [rotate=-303.45]  {\footnotesize {\hypersetup{pdfborder={0 0 0}}\cite{Giacchetto_2025}}};
		
		\draw (420.03,235.36) node [anchor=north west][inner sep=0.75pt]  [rotate=-0.21]  {\footnotesize {\hypersetup{pdfborder={0 0 0}}\cite{zbMATH07184121},\cite{Giacchetto_2025}}};
		\draw (489.9,93.36) node [anchor=north west][inner sep=0.75pt]  [rotate=-88.93]  {$\text{open}$};
		\draw (530.24,160.04) node [anchor=north west][inner sep=0.75pt]  [rotate=-52.95]  {\footnotesize {\hypersetup{pdfborder={0 0 0}} Thm. \ref{theorem_correspondencecompletedcyclesintro}}};
		\draw (87,41.99) node [anchor=north west][inner sep=0.75pt]    {$HW^{spin}$};
		\draw (79,146.49) node [anchor=north west][inner sep=0.75pt]    {$HW_{trop}^{spin}$};
		\draw (191,216.93) node [anchor=north west][inner sep=0.75pt]    {$HW_{rep.}^{spin}$};
		\draw (186.49,93.08) node [anchor=north west][inner sep=0.75pt]  [rotate=-58.5]  {\footnotesize {\hypersetup{pdfborder={0 0 0}}\cite{gunningham2016spin}, \cite{arXiv:2104.05697}}};
		\draw (130.9,80.36) node [anchor=north west][inner sep=0.75pt]  [rotate=-88.93]  {\footnotesize {\hypersetup{pdfborder={0 0 0}}Thm. \ref{theorem_correspondenceintro}}};
		\draw (266.59,212.06) node [anchor=north west][inner sep=0.75pt]  [rotate=-359.51]  {$\text{defines}$};

\end{tikzpicture}}
\caption{\footnotesize A network for type B theory from ordinary spin Hurwitz numbers, on the left, to completed cycles, on the right. Note: The absence of a Fock space expression for ordinary double spin Hurwitz numbers stems from a key difference from the classical theory. Double Hurwitz numbers  (two special ramifications, 2‑cycles elsewhere) agree with their 2‑completed‑cycle analogues. This explains the presence of the operator formalism in Figure 1. This equality fails in the spin setting, as  ordinary double spin Hurwitz numbers (two special ramifications, 3‑cycles elsewhere) differ from their 3‑completed‑cycle analogues. \\
The connection between $GWI^{spin}$ and $HWC^{spin}_{trop}$ (the dashed arrow) is yet to be established. This can be done by modeling \cite{zbMATH07272609} for type A theory and using the spin ingredients developed in this paper.}
\label{figure_networkspinHWC}
\end{figure}
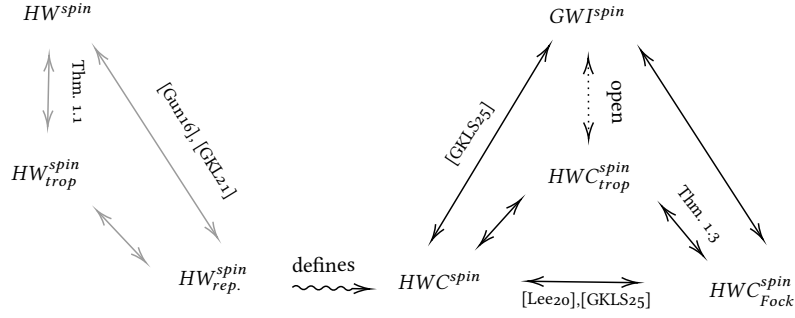 

\subsection{Results}\label{subsection_resultsproject3}
\subsubsection{Spin Hurwitz numbers}
We first define tropical spin Hurwitz numbers (\cref{definition_tropicalspinHurwitznbr}) and prove a correspondence theorem that clarifies their relation to the classical invariants.
\begin{theorem}\label{theorem_correspondenceintro}[\cref{thm-correspondencetheorem1}]
Let $g,h$ be non-negative integers, $d$ a positive integer, $\mu_1,\dots,\mu_n$ odd partitions of $d$, and $p\in{\mathbb{Z}}/{2\mathbb{Z}}$, then we have
\begin{equation*}
	H^{\bullet,p}_{g\to h}(\mu_1,\dots,\mu_n)=\mathbb{T}H^{\bullet,p}_{g\to h}(\mu_1,\dots,\mu_n).
\end{equation*}
\end{theorem}

As an application, we derive the following closed expressions.

\begin{proposition}\label{proposition_closedformulae}
        Fix a positive integer $d$. Then, we obtain

    \begin{enumerate}
        \item $H_{g\to 0}^{+}((31^{d-3})^0))=\frac{1}{d!}$ and $H_{g\to 0}^{+}((31^{d-3})^1))=0$ for all $d$
        \item $H_{g\to 0}^{+}((31^{d-3})^2))=\frac{1}{3(d-3)!}$ and $H_{g\to 0}^{+}((31^{d-3})^2))=\frac{3d-10}{3(d-3)!}$ for $d\ge 4$
        \item $H_{g\to 0}^{+}((31^{d-3})^4))=\frac{2d^3+57d^2-446d+780}{9(d-3)!}$ 
    \end{enumerate}
\end{proposition}
\paragraph{\textbf{Strategy.}} Degeneration techniques developed by \cite{MR3152971} provide a bridge to combinatorics, relevant geometric information (theta characteristics) is encoded as a parity function on the source and target graph. We then use this new combinatorial framework to recover known formulae and produce new spin Hurwitz numbers (\cref{proposition_closedformulae}). 
\subsubsection{Spin Hurwitz numbers with completed cycles}
We next define tropical spin Hurwitz numbers with completed cycles (\cref{definition_tropicalcompletedcyclesspinHurwitznumbers}) and prove a correspondence theorem that re-geometrizes the purely representation-theoretic invariant. 
\begin{theorem}\label{theorem_correspondencecompletedcyclesintro}[\cref{theorem_tropicalclassicalcompletedcyclecorrespondence}]
Fix discrete data: Positive integers $d,g>0$, even integers $ r_1,..., r_b>0$ and odd partitions $\mu, \nu$ of $d\in \mathbb{N}$ such that $r_1+ ... +r_b= 2g + l(\mu) + l(\nu) -2$ holds.  
The classical and the tropical spin double Hurwitz numbers with multi-completed cycles agree:
\begin{align*}
&	\mathbb{T}h^{spin}_{g}(\mu, \overline{(r_1 +1)}, ..., \overline{(r_b +1)}, \nu)=h^{spin}_{g}(\mu, \overline{(r_1 +1)}, ..., \overline{(r_b +1)}, \nu)\\
&\mathbb{T}h^{spin,\bullet}_{g}(\mu,\overline{(r_1 +1)}, ..., \overline{(r_b +1)},\nu)=h^{spin,\bullet}_{g}(\mu,\overline{(r_1 +1)}, ..., \overline{(r_b +1)},\nu).
\end{align*}
If $r:=r_1=...= r_b$, we simply write $\mathbb{T}h^{spin,r}_{g}(\mu, \nu)$/$\mathbb{T}h^{spin,r,\bullet}_{g}(\mu,\nu)$.
\end{theorem}

The discrete data governing both invariants now has a geometric interpretation: The tropical invariant enumerates covers (with suitable multiplicity) of the tropical projective line $\mathbb{TP}^1$ of degree $d$ by genus $g$ curves with odd edge weights and vertex valencies fixed by the integers $r_i$.

\cref{theorem_correspondencecompletedcyclesintro} also paves the way for investigating structural properties through combinatorial methods: Consider the map
\begin{align}\label{equation_inspinintro}
\mathbb{T}h^{spin,r}_{g}: \mathcal{H} \rightarrow \mathbb{Q}, x \mapsto \mathbb{T}h^{spin,r}_{g}(x^-,x^+),
\end{align}
where  $\mathcal{H}:=\{x\in \mathbb{Z}^{n}: \sum^{n}_{i=1}x_i=0 \text{ and } x_i \text{ odd} \}$.

Using the tropical description, we obtain a new proof that (\ref{equation_inspinintro}) is piecewise polynomial and new wall crossing formulae. The wall crossing formula of \cite{arXiv:2104.05697} is a coefficient of a product of vacuum expectations of B-Okounkov-Pandharipande operators and expressions derived from the generating series for multi-completed cycles (\cite{arXiv:2104.05697}, Definition 6.15). In contrast, our formula expresses wall crossing in terms of completed cycles spin Hurwitz numbers for smaller counting problems. This recursive structure mirrors that of formulae obtained for related problems through tropical methods (\cite{zbMATH07272609},\cite{zbMATH06806783},\cite{zbMATH07215113}) and yields a geometric interpretation for wall-crossing: the smaller invariants are inherently tied to the covers we enumerate as they reflect building blocks of covers on either side of the wall (see \cref{section_Pwpolynomialityandwc}).

\begin{theorem}\label{theorem_introspin}[\cref{proposition_piecewisepolynomialitygenus0} and \cref{theorem_wall-crossinggenus0}, \cref{proposition_multGammaispolynomial} and \cref{theorem_wallcrossinggenusg>0beautiful}]
The function $	\mathbb{T}h^{spin,r}_{g}$ is piecewise polynomial and satisfies explicit wall crossing formulae. 
\end{theorem}

\paragraph{\textbf{Strategy.}}
Techniques for proving theorems of this type were introduced by \cite{zbMATH06806783} in the context of double Hurwitz numbers. We generalise these methods to the spin setting, demonstrating their broader applicability. Moreover, we provide in parts a new perspective on the underlying polyhedral geometry. The key issue here is to introduce an additional parity condition when applying Ehrhart theory techniques that correspond to only allowing odd ramification for spin Hurwitz numbers.

We first consider the case where the source is of genus $0$. As for double Hurwitz numbers, the problem is entirely combinatorial in this setting. The vertex multiplicities (local spin Hurwitz numbers, which we determine in \cref{section_tropicalspincompletedcycles}) are constant. Thus, the genus $0$ case reduces to enumerating graphs of fixed type, whose multiplicities depend only on the structure of each graph (see \cref{definition_tropicalcompletedcyclesspinHurwitznumbers}). 

Extending to higher genus adds
\begin{itemize}
\item Combinatorial complexity: Cycles lead to auxiliary variables for edge weights, in addition to variables for end weights.
\item Algebraic complexity: The multiplicity of each cover requires an algebraic input in the form of local spin Hurwitz numbers.
\end{itemize}
These complications require a change of approach. Following \cite{zbMATH06806783}, we introduce a new counting unit, the $x$-graph, that encompasses all covers that share the same underlying undirected graph. Its multiplicity is now independent of all auxiliary variables, since it is computed by integrating over them. In the spin setting, however, this means integrating over odd lattice points of bounded polytopes in a hyperplane arrangement. In order to use the general theory that deals with lattice points in polytopes (Ehrhart theory), we work with rational hyperplane arrangements. Care is needed since summing a polynomial over lattice points of a rational polytope usually yields a quasi-polynomial. After handling this quasi-polynomial behavior using \cref{lemma_qpofperiodatmost2} and \cref{lemma_verticesoflift}, we obtain \cref{theorem_introspin} for $g>0$.

\subsection{Overview}
In \cref{sec-prelim}, we review necessary background around (spin) Hurwitz numbers, the Fock space formalism and tropical covers. We continue in \cref{section_tropicalspinHurwitz} to use degeneration techniques to prove \cref{theorem_correspondenceintro}. As an application, we derive \cref{proposition_closedformulae}. Moreover, in \cref{section_tropiclspincompletedcycles} we use Fock space techniques to tropicalise their completed cycles counterparts and derive \cref{theorem_correspondencecompletedcyclesintro}. Finally, we prove \cref{theorem_introspin} in \cref{section_Pwpolynomialityandwc}, i.e. we prove piecewise polynomiality and find wall--crossing formulae for spin Hurwitz numbers.

\section{Preliminaries}\label{sec-prelim}

\subsection{Notation}
In this subsection, we fix some basic notation.

Let $\mathcal{P}$ be the set of all partitions. We denote by $\mathrm{SP}$ the set of \emph{strict partitions}.

We denote the symmetric group on $d$ elements by $S_d$. For $\lambda,\mu$ partitions of $d$, we denote the character of the representation $\lambda$ evaluated on $\mu$ by $\chi_\lambda(\mu)$. We denote the dimension of the representation by $\mathrm{dim}(\lambda)$. Moreover, in the centre $\mathcal{Z}(d)$ of the symmetric group algebra $\mathbb{Q}[S_d]$, we denote the conjugacy class of $\mu$ by $C_\mu$ and its size by $|C_\mu|$.

Moreover, we define the \emph{central characters} $f_\mu: \mathcal{P} \rightarrow \mathbb{Q}$ by
\begin{itemize}
\item $f_\mu\equiv 1$, if $\mu=\emptyset$.
\item $f_\mu(\lambda):=\binom{|\lambda|}{|\mu|}\frac{|C_{\mu}|\chi_\lambda(\mu)}{\dim(\lambda)}$ else. (If $|\mu|> |\lambda|$, $f_\mu\equiv 0$. If $|\mu|< |\lambda|$, then $\chi_\lambda(\mu)$ makes sense since $S_{|\lambda|}$ includes all permutations of cycle type $\mu$.) 
\end{itemize} 

Further, we consider the shifted action of $S_d$ on the polynomial ring $\mathbb{Q}[x_1,\dots,x_d]$ by permuting $x_i-i$ and denote the invariant ring by
\begin{equation}
    \mathbb{Q}[x_1,\dots,x_d]^{\ast S_d}.
\end{equation}

Then we define the ring of shifted symmetric functions by the following projective limit in the category of filtered algebras
\begin{equation}
    \Lambda^\ast=\varprojlim\mathbb{Q}[x_1,\dots,x_d]^{\ast S_d}
\end{equation}
with respect to the homomorphisms $\mathbb{Q}[x_1,\dots,x_d]^{\ast S_d}\to\mathbb{Q}[x_1,\dots,x_{d-1}]^{\ast S_{d-1}}$ via $x_d\mapsto 0$. The ring $\Lambda^\ast$ is freely generated by the \emph{shifted symmetric power sums}
\begin{equation}
    p_k(x_1,x_2,\dots)=\sum_{i}\left(x_i-i+\frac{1}{2}\right)^k-\left(-i+\frac{1}{2}\right)^k+\left(1-2^{-k}\right)\zeta(-k),
\end{equation}

i.e. we have

\begin{equation}
    \Lambda^\ast=\mathbb{Q}[p_1,p_2,\dots].
\end{equation}

Furthermore, denoting $p_\mu=\prod p_{\mu_i}$, the collection ${p_\mu}_{\mu\in\mathcal{P}}$ forms a basis of $\Lambda^\ast$.

As proved in \cite{zbMATH00663667}, we have $f_\mu\in\Lambda^\ast$, i.e. the central characters are shifted symmetric. In fact, the collection ${f_\mu}_{\mu\in\mathcal{P}}$ forms a basis of $\Lambda^\ast$, as well. Their relation is mutually triangular. In particular, we have
\begin{equation}
    f_\mu=\frac{1}{\prod \mu_i}p_\mu+\dots,
\end{equation}

where the dots denote terms of degree strictly smaller than $|\mu|$.

The Fourier transform

\begin{align}
\bigoplus^\infty_{d=0} \mathcal{Z}(d) \xrightarrow{\phi}  \Lambda^\ast, C_\mu \mapsto {f_\mu}
\end{align}

is a linear isomorphism. While the conjugacy class basis of $\mathcal{Z}(d)$ corresponds to central characters, we denote the basis corresponding to power sums given by $\{\bar{\mu}| \mu \in \mathcal{P}\}$, where $\bar{\mu}:=\frac{1}{\prod_{i}\mu_i}\phi^{-1}(p_\mu)$. We call $\bar{\mu}$ the completed cycles.

Moreover, we introduce the following notation:

\begin{align*}
		\zeta(z):=2 \sinh(\frac{z}{2}), \thickspace \mathcal{S}(z)=\frac{\zeta(z)}{z}, \thickspace  \rho(z)=\frac{1}{2} \cosh(\frac{z}{2})
	\end{align*}

\subsection{Hurwitz numbers and completed cycles}
In this section, we introduce the basics of (spin) Hurwitz numbers and their complected cycle analogues.

We begin with a general definition of classical Hurwitz numbers with targets of arbitrary genus.

\begin{definition}
\label{definition_Hurwitznbr}
Let $d$ be a positive integer and $g,h$ be non--negative integers with $g\ge h$. Let $D$ be a connected compact Riemann surface of genus $h$ and fix pairwise distinct points $q^1, \dots, q^n \in D$. Moreover, we fix a collection $\mu_1, \dots, \mu_n$ of partitions of $d$.

We define a  \emph{(connected/potentially disconnected) Hurwitz cover} of type $(g\rightarrow h,d,\mu_1,\dots,\mu_n)$ to be a holomorphic map $f: C\rightarrow D$ such that
\begin{itemize}
	\item $C$ is a compact and (connected/potentially disconnected) Riemann surface of genus $g$,
	\item $f$ is of degree $d$ and
	\item ramified at $q_i$ as prescribed by the partition $ \mu_i$.
	
\end{itemize} 
We call two such covers $f: C\rightarrow D$ and $\tilde{f}: \tilde{C}\rightarrow D$ to be equivalent if there exists an isomorphism $\phi\colon C\rightarrow \tilde{C}$, such that $f=g\circ \phi$.

We define the \emph{connected/Hurwitz number} $H_{g\xrightarrow[]{d} h }(\mu_1,\dots,\mu_n)/H^{\bullet}_{g\overset{d}{\to} h }(\mu_1,\dots,\mu_n)$ as the weighted sum 
	\begin{align*}
		H_{g\overset{d}{\to} h }(\mu_1,\dots,\mu_n):=\sum_{[f]}\frac{1}{|\text{Aut}(f)|},
	\end{align*}
	of equivalence classes $[f]$ of connected Hurwitz covers of type $(g\rightarrow h,d,\mu_1,\dots,\mu_n)$. We denote the potentially disconnected analogue by $H^{\bullet}_{g\overset{d}{\to} h }(\mu_1,\dots,\mu_n)$.
\end{definition}

\begin{remark}
We call that a necessary condition for the existence of Hurwitz covers of type $(g\rightarrow h,d,\mu_1,\dots,\mu_n)$ is given by the Riemann Hurwitz formula: 
\begin{align*}
	2-2g=d(2-2h)+\sum_{i=1}^{n} (d-l(\mu_i)).
\end{align*}
Optionally, one of the decorations $g, h$ or $d$ will be omitted and understood as determined by the Riemann Hurwitz formula.
\end{remark}

When discussing completed cycles Hurwitz numbers, we need to slightly extend the definition for target $\mathbb{P}^1$ via the following conventions.

\begin{itemize}
    \item $H_{g\overset{0}{\to} 0}(\emptyset,\dots,\emptyset)=1$
    \item If $|\lambda_i|>d$ for some $i$, we set $H^\bullet_{g\overset{d}{\to}0}(\mu_1,\dots,\mu_n):=0$
    \item If $|\lambda_i|\leqslant d$, let $\tilde{\lambda}_1$ be the partition of size $d$ obtained by adding $1'$s to $\lambda_i$. Then, we define
    \begin{equation}        H^\bullet_{g\overset{d}{\to} 0}(\lambda_1,\dots,\lambda_n):=\prod_{i=1}^{n}\binom{m_1(\tilde{\lambda}_i)}{m_1(\lambda_i)}H^\bullet_{g\overset{d}{\to} 0} (\tilde{\lambda}_1,\dots,\tilde{\lambda}_n),    
    \end{equation}
    where $m_1(\lambda)$ counts the number of $1$'s occurring in $\lambda$ and $g$ is determined by the Riemann Hurwitz formula.
\end{itemize}

Hurwitz numbers famously admit an interpretation via the symmetric group. It is most elegantly phrased via Burnside's formula.

\begin{theorem}
\label{theorem_characterformulaHurwitznbrs}
    In the setting of \cref{definition_Hurwitznbr}, we have
    \begin{equation}
        H^{\bullet}_{g\overset{d}{\to}0}(\mu_1,\dots,\mu_n)=\sum_{|\lambda|=d}\left(\frac{\mathrm{dim}(\lambda)}{d!}\right)^{2-2h}\prod_{i=1}^n\frac{|C_{\mu_i}|\chi_\lambda(\mu_i)}{\mathrm{dim}(\lambda)}.
    \end{equation}
\end{theorem}

The Burnside formula may be rewritten as

\begin{equation}
        H^{\bullet}_{g\overset{d}{\to}0}(\mu_1,\dots,\mu_n)=\sum_{|\lambda|=d}\left(\frac{\mathrm{dim}(\lambda)}{d!}\right)^{2-2h}\prod_{i=1}^nf_{\mu_i}(\lambda).
\end{equation}

It is natural to replace the central characters $f_{\mu_i}$ by other shifted symmetric functions. This is the perspective taken in the definition of completed cycles Hurwitz numbers. 

For completed cycles, we now restrict ourselves to target $\mathbb{P}^1$, i.e. $h=0$. We drop the index and label the Hurwitz numbers either by $d$ or $g$ depending on the context.

\begin{definition}\label{definition_completedcyclesHurwitznumber}
Fix positive integers $d,g, k_1,\dots, k_n$ and partitions $\mu, \nu$ of $d\in \mathbb{N}$ such that $k_1+ \dots +k_n= 2g + l(\mu) + l(\nu) -2$ holds.
The \emph{disconnected Hurwitz number with multi-completed cycles} $h^{\bullet}_d(\mu, \overline{(k_1 +1)}, \dots, \overline{(k_n +1)}, \nu)$ is defined as the linear combination of potentially disconnected Hurwitz numbers via
\begin{equation}
	h^{\bullet}_d(\mu, \overline{(k_1 +1)}, \dots, \overline{(k_n +1)}, \nu):= \sum_{\lambda_1,\dots,\lambda_n} \rho_{(k_1 +1), \lambda_1}\cdots \rho_{(k_n +1), \lambda_n} \cdot  H^{\bullet}_d(\mu,\lambda_1,\dots,\lambda_n,\nu),
\end{equation}
where the summation is over all partitions  $\lambda_i$ of integers less than or equal to $k_i + 1$ and the coefficients $\rho_{(k_i +1), \lambda_i}$ come from the expansion $\overline{(k_i +1)}= C_{k_i +1} + \sum_{\lambda_1}\rho_{(k_i +1), \lambda_i}C_{\lambda_i}$ of the completed cycles in ordinary conjugacy classes given in \cite{zbMATH05049020}, (0.21).
\end{definition}

\begin{remark}
    The celebrated Gromov--Witten/Hurwitz correspondence gives a geometric meaning to completed cycles Hurwitz numbers as stationary descendant relative Gromov--Witten invariants \cite{zbMATH05049020}. 
\end{remark}

The character formula for multi--completed cycles Hurwitz numbers becomes

\begin{equation}
\label{equ-charactercomplcycl}
	h^{\bullet}_d(\mu, \overline{(k_1 +1)}, \dots, \overline{(k_1 +1)}, \nu)= \frac{\prod_{i=1}^{n}k_i!}{\mathfrak{z}(\mu) \mathfrak{z}(\nu) } \sum_{\lambda}\chi_\mu^\lambda\chi_\nu ^\lambda \cdot \prod_{i=1}^{n}\frac{p_{k_i +1}(\lambda)}{(k_i + 1)!},
\end{equation}
where $\mathfrak{z}(\mu):=|\mu|!/|C_\mu|(=|\Aut(\mu)|\prod_{i=1}\mu_i)$

We now move on to spin Hurwitz numbers and their completed cycles analogues. To begin with, we recall the notion of a spin structure.

\begin{definition}
    Let $C$ be a compact Riemann surface. A spin structure is a line bundle $L$, such that $L^{\otimes2}\cong \omega_C$. We call $(C,L)$ a spin curve.

    We call
    \begin{equation}
        p(L)\equiv h^0(C,L)\,\mathrm{mod}\, 2
    \end{equation}
    the \emph{parity} of $(C,L)$. We call the spin structure even for $p(L)=0$ and odd for $p(L)=1$. We will also indicate even spin structures by $+$ and odd ones by $-$.
\end{definition}

We now define \emph{spin Hurwitz numbers}.

\begin{definition}
\label{definition_extendeddefinitionspinHurwitznumbers}
    In the same set-up as \cref{definition_Hurwitznbr}, we fix a spin structure $L$ on $D$ with parity $p$. Let $f\colon C\to D$ be a Hurwitz cover of type $(g\to h,d,\mu_1,\dots,\mu_n)$. We define a spin structure on $C$ via
    \begin{equation}
        N_{f,l}=f^\ast(L)\otimes \mathcal{O}\left(\frac{1}{2}R_f\right),
    \end{equation}
    where $R_f$ is the ramification divisor of the cover. We define the parity of $f$ to be the parity of $N_{f,L}$.

    We then define the \emph{connected/disconnected spin Hurwitz number} $H^p_{g\xrightarrow[]{d} h }(\mu_1,\dots,\mu_n)/H^{\bullet, p}_{g\overset{d}{\to} h }(\mu_1,\dots,\mu_n)$ as the weighted sum 
\begin{align}\label{equation_inDefinitionspinHurwitznbr}
	H^p_{g\overset{d}{\to} h }(\mu_1,\dots,\mu_n)/H^{\bullet,p}_{g\overset{d}{\to} h }(\mu_1,\dots,\mu_n):=\sum_{[f]}\frac{{(-1)}^{p(f)}}{|\text{Aut}(f)|},
\end{align}
of equivalence classes $[f]$ of (connected/potentially disconnected) Hurwitz covers of type $(g\rightarrow h,d,\mu_1,\dots,\mu_n)$.
Moreover, we define the sub-counts 
\begin{equation*}
	H_{g\to h}^{(p,+)}/H_{g\to h}^{\bullet,(p,+)}=\sum_{p(f)\,\textrm{even}}\frac{1}{\left|\mathrm{Aut}(f)\right|}\quad\textrm{and}\quad H_{g\to h}^{(p,-)}/H_{g\to h}^{\bullet, (p,-)}=\sum_{p(f)\,\textrm{odd}}\frac{1}{\left|\mathrm{Aut}(f)\right|}
\end{equation*}
that yield enumerations of covers with fixed parity.
\end{definition}

Again, we extend the definition

\begin{itemize}
    \item $H^{-}_{g\overset{0}{\to} 0}(\emptyset,\dots,\emptyset)=1$
    \item If $|\lambda_i|>d$ for some $i$, we set $H^{\bullet,-}_{g\overset{d}{\to}0}(\mu_1,\dots,\mu_n):=0$
    \item If $|\lambda_i|\leqslant d$, let $\tilde{\lambda}_1$ be the partition of size $d$ obtained by adding $1'$s to $\lambda_i$. Then, we define
    \begin{equation}        H^{\bullet,-}_{g\overset{d}{\to} 0}(\lambda_1,\dots,\lambda_n):=\prod_{i=1}^{n}\binom{m_1(\tilde{\lambda}_i)}{m_1(\lambda_i)}H^{\bullet,-}_{g\overset{d}{\to} 0} (\tilde{\lambda}_1,\dots,\tilde{\lambda}_n),    
    \end{equation}
    where $m_1(\lambda)$ counts the number of $1$'s occurring in $\lambda$ and $g$ is determined by the Riemann Hurwitz formula.
\end{itemize}

The representation theory for spin Hurwitz numbers is more involved and moves through the Sergeev group. Consider $B(d)\subset S_{2d}$ the hyperoctahedral group. It admits a central extension via
\begin{equation}
    0 \rightarrow \mathbb{Z}_2 \rightarrow \mathfrak{h}_{d} \rightarrow B(d) \rightarrow 0
\end{equation}

called the Sergeev group. It is the semi-direct product
$\mathfrak{h}_d=S_d\ltimes\mathrm{Cl}_d$ , where
\begin{itemize}
	\item $	\mathrm{Cl}_d:=\langle \epsilon, \xi_1,\dots,\xi_d \mid\xi_i^2=\epsilon^2=1, \xi_i\xi_j=\epsilon \xi_j\xi_i\text{ for } i\neq j\rangle $ is the \emph{Clifford group} and
	\item $\sigma \in S_d$ acts on $\mathrm{Cl}_d$ by $\epsilon \mapsto \epsilon \text{ and } x_i \mapsto x_{\sigma(i)}$.
\end{itemize}

Irreducible spin representations correspond to irreducible modules of the \textit{twisted group algebra of $\mathfrak{h}_d$} (\cite{gunningham2016spin}, Section 1.5)
\begin{equation*}
\mathcal{H}_d=\mathbb{C}[\mathfrak{h}_d]/1+\epsilon.
\end{equation*}

These representations are studied via so-called \emph{superalgebras}. A superalgebra is a $\mathbb{Z}_2$-graded algebra $A=A_0\oplus A_1$ with $A_i A_j\subset A_{i+j}$ for $i,j \in  \mathbb{Z}_2$.

The morphism
\begin{align*}
\mathrm{deg}: \mathfrak{h}_d \rightarrow \mathbb{Z}_2,	\mathrm{deg}(\xi_i):=1 \text{ and } \mathrm{deg}(\epsilon)=\mathrm{deg}(\sigma):=0, \sigma\in S_d
\end{align*}
defines a $\mathbb{Z}_2$--grading on $\mathfrak{h}_d$, which turns $\mathbb{C}[\mathfrak{h}_d]$ into a superalgebra with direct sum decomposition $\mathbb{C}_0[\mathfrak{h}_d] \oplus \mathbb{C}_1[\mathfrak{h}_d]$, where $\mathbb{C}_{i}[\mathfrak{h}_d]:=\mathbb{C}[\{g\in \mathfrak{h}_d| \mathrm{deg}(g)=i\}]
$. The quotient algebra $\mathcal{H}_d$ inherits this superalgebra structure.

Irreducible spin representations arise from irreducible supermodules of $\mathcal{H}_d$ (\cite{gunningham2016spin}, Proposition 1.13). Mirroring the classical setting, these are classified by \emph{strict} partitions of $d$ (\cite{gunningham2016spin}, Propositions 1.16).

Let $\mathcal{Z}_d$ denote the \emph{supercenter} of $\mathcal{H}_d$ (\cite{gunningham2016spin}, Section 1.4). Since $\mathcal{H}_d$ is semisimple (\cite{zbMATH01716688}, by Lemma 3.6), $\mathcal{Z}_d$ is equal to the even part of the ordinary (algebra) center of $\mathcal{H}_d$. 

The algebra $\mathcal{Z}_d$ has two natural vector space bases indexed by \emph{odd} and \emph{strict} partitions of $d$. 
\begin{itemize}
\item The conjugacy class basis: $\{c_\mu| \mu \in \mathrm{OP}(d) \}$, where $c_\mu$ is the image of the sum $\sum_{x}x\in\mathbb{C}[\mathfrak{h}_d]$ over all $x$ in the conjugacy class of a permutation $\sigma\in S_d\subset \mathfrak{h}_d$.
\item The semisimple basis: $\{e^B_\lambda| \lambda \vdash d\}$ of orthogonal idempotents indexed by irreducible supermodules $V^\lambda$ of $\mathcal{H}_d$ for $\lambda\in \mathrm{SP}(d)$ (see \cite{gunningham2016spin}, Remark 1.11 using that $\mathcal{H}_d$ is semisimple). 
\end{itemize}
For $\lambda\in \mathrm{SP}(d)$, let $\zeta_\lambda:\mathfrak{h}_d \rightarrow \mathbb{C}$ be the (usual) character of the spin representation associated to the supermodule $V^\lambda$ by forgetting the grading and write $\zeta_\lambda(\mu)$ for its value at $\sigma\in S_d \subset \mathfrak{h}_d$ of cycle type $\mu \in \mathrm{OP}(d)$.

\begin{theorem}[\cite{gunningham2016spin}, Theorem 1.10] \label{theorem_characterformulaspinHurwitz}
The disconnected spin Hurwitz number $H^{\bullet,0}_{g\to 0}(\mu_1,\dots,\mu_n)$ is equal to
\begin{equation*}
	2^{\left(\sum_i(\ell(\mu_i)-d)-2d\right)}\sum_{\lambda\in\mathrm{SP}(d)}\left(\frac{\mathrm{dim}(V^\lambda)}{2^{\delta(d)/2}d!}\right)^{2}\prod_{i=1}^{n}\frac{|c_{\mu_i}|\zeta_\lambda(\mu_i)}{\dim(V^\lambda)},
\end{equation*}
where $\delta(d)=1$ if $d$ is odd and $\delta(d)=0$ if $d$ is even and $|c_{\mu_i}|$ is the number of elements in the conjugacy class of a permutation $\sigma \in S_d\subset \mathfrak{h}_d$  of cycle type $\mu_i$.  
\end{theorem}

Define \emph{the central character} $f_\mu: \mathrm{SP} \rightarrow \mathbb{C}$ by
\begin{itemize}
\item $f_\mu\equiv 1$, if $\mu=\emptyset$.
\item $f_\mu(\lambda):=\binom{m_1(\mu) + |\lambda|-|\mu|}{|\lambda|-|\mu|}\frac{|c_{\mu}|\zeta_\lambda(\mu)}{\dim(V^\lambda)}$, where $m_1(\mu)$ is the number of parts of $\mu$ equal to 1, else. 
\end{itemize} 
Thus, we can rewrite the formula in \cref{theorem_characterformulaspinHurwitz}
\begin{align}\label{equation_characterformulaspinHurwitznbrs}
2^{\left(\sum_i(\ell(\mu_i)-d)-2d\right)}\sum_{\lambda\in\mathrm{SP}(d)}\left(\frac{\mathrm{dim}(V^\lambda)}{2^{\delta(d)/2}d!}\right)^{2}\prod_{i=1}^{n}f_{\mu_i}(\lambda).
\end{align}

Mirroring the classical situation, we consider the algebra $\mathbb{C}^{\mathrm{SP}}$ of $\mathbb{C}$ valued functions on $\mathrm{SP}$. It has a distinguished subalgebra: The algebra $\Lambda^B$ of supersymmetric functions is defined as the image of the injective linear map (\cite{arXiv:2104.05697}, Proposition 3.14)
\begin{align}
\bigoplus^\infty_{d=0} \mathcal{Z}_d \xrightarrow{\phi}  \mathbb{Q}^{\mathrm{SP}}, c_\mu \mapsto {f_\mu}.
\end{align}
Usually $\Lambda^B$ is defined as a subalgebra of the algebra of symmetric functions, which itself can be viewed as a subalgebra of $\mathbb{C}^{\mathcal{P}}$. Since the central characters are only defined for strict partitions, we work with $\mathbb{C}^{\mathrm{SP}}$ instead using that symmetric functions are uniquely determined by value on strict partitions (see \cite{zbMATH02182845}, Proposition 6.1.).

As an algebra, $\Lambda^B$ is freely generated by the usual \emph{symmetric power sums} $p_k\in \mathbb{C}^{\mathcal{P}}$ with $k$ odd (\cite{zbMATH02182845}), i.e. $\Lambda^B=\mathbb{C}[p_1,p_3,\dots]$.

As a vector space, $\Lambda^B$ admits two key bases for our applications: $\{f_\mu| \mu\in \mathrm{OP}\}$ and  $\{p_\mu| \mu\in \mathrm{OP}\}$, where $p_\mu:=\prod_{i}p_{\mu_i}$.		

This provides another vector space basis for $\bigoplus^\infty_{d=0} \mathcal{Z}_d$ via ${\phi}$: The counterpart of the conjugacy class basis is given by $\{\bar{\mu}\thinspace|\thinspace \mu \in \mathrm{OP}\}$, where $\bar{\mu}:=\frac{1}{\prod_{i}\mu_i}\phi^{-1}(p_\mu)$ is called the \emph{spin completed cycle conjugacy class}. 

We are now ready to define completed cycles spin Hurwitz numbers. Again we restrict ourselves to target $\mathbb{P}^1$ and therefore drop the dependence on $h$.

\begin{definition}\label{definition_completedcyclesspinHurwitznumber}

	Fix positive integers $d,g>0$, even integers $ r_1,\dots, r_b>0$ and odd partitions $\mu, \nu$ of $d\in \mathbb{N}$ such that $r_1+ \dots +r_b= 2g + l(\mu) + l(\nu) -2$ holds.
	The \emph{disconnected spin double Hurwitz number with multi-completed cycles} $h^{spin,\bullet,0}_g(\mu, \overline{(r_1 +1)}, \dots, \overline{(r_b +1)}, \nu)$ is defined as the linear combination of ordinary spin Hurwitz numbers 
	\begin{align}\label{equation_defcompletedcyclespin}
		\frac{|\Aut(\mu)||\Aut(\nu)|}{b!}	 		
		\sum_{\lambda_1,\dots,\lambda_b} \rho_{(r_1 +1), \lambda_1}\cdot \dots \cdot \rho_{(r_b +1), \lambda_b} \cdot  H^{\bullet,0}_{\tilde{g}}(\mu, \lambda_1, \dots, \lambda_b, \nu),
	\end{align}
	where the summation is over all odd partitions $\lambda_i$ of integers smaller than or equal to $r_i + 1$, $\tilde{g}$ is defined by the Riemann Hurwitz formula (for $(d,h=0,\tilde{\lambda}_i)$ with $\tilde{\lambda}_i$ the partition of size $d$ obtained by adding $1'$s to $\lambda_i$), and the coefficients $\rho_{(r_i +1), \lambda_i}$ come from the expansion $\overline{(r_i +1)}= (r_i +1) + \sum_{\lambda_i}\rho_{(r_i +1), \lambda_i}\lambda_i$ of the completed cycles in ordinary conjugacy classes (see \cite{Giacchetto_2025}, Lemma 2.11). 
\end{definition}

\begin{remark}\ \\
	\begin{itemize}
		\item 	\cref{definition_completedcyclesspinHurwitznumber} specializes to Definition 4.5 in \cite{arXiv:2104.05697} by setting $r_i+1=r+1$ for $i=1,\dots,b$.
		\item 
		We may introduce \emph{connected spin double Hurwitz numbers with multi-completed cycles} \begin{equation}
		    h^{spin,0}_d(\mu, \overline{(k_1 +1)}, \dots, \overline{(k_n +1)}, \nu)
		\end{equation}
        by imposing a connectedness condition on the source.
		
	\end{itemize}
\end{remark}

There is again an expression of completed cycles spin Hurwitz numbers via characters.

\begin{theorem}\label{theorem_characterformulacpltcyclesspinHurwitznumbers}
	The degree $d$ disconnected spin double Hurwitz number with multi-completed cycles is given by
	\begin{align*}
		h^{spin,\bullet,0}_g(\mu, \overline{(r_1 +1)}, \dots, \overline{(r_b +1)}, \nu)= \frac{2^{1-g}}{b!\prod_{i}\mu_i\prod_{j}\nu_j} \sum_{\lambda\in \mathrm{SP}(d)} \frac{\zeta_\mu^\lambda\zeta_\nu ^\lambda}{2^{\delta(\lambda)+l(\mu) + l(\nu)}} \cdot \prod_{i=1}^{b}\frac{p_{r_i +1}(\lambda)}{(r_i + 1)},
	\end{align*}
	where $\zeta^\lambda_{\mu}:=\zeta_\lambda(\mu)$ ($\zeta^\lambda_{\mu}$) denotes the usual character of the irreducible representation corresponding to $\lambda$ at $\mu$ ($\nu$).
\end{theorem}

The case, $r_i+1=r+1$ for all $i$, is stated in \cite{arXiv:2104.05697} without proof. For completeness' sake and to present a general formula that mirrors \cref{equ-charactercomplcycl} in the classical case, we include a proof here.
\begin{proof}
	The linear expansion on the right hand side of Equation (\ref{equation_defcompletedcyclespin}) corresponds to replacing the central characters $f_{(r_i+1)}$ by the usual power sum functions $p_{(r_i+1)}/(r_i+1)$. The rest follows from a simple computation, using
	\begin{itemize}
		\item the relation between central and the usual characters to obtain:
		\begin{align*}
			\frac{\dim(V^\lambda)f^\lambda_\mu}{2^{\delta(\lambda)/2}d!}=\frac{2^d \zeta^\lambda_\mu}{2^{\delta(\lambda)/2+l(\mu)}\mathfrak{z}(\mu)},
		\end{align*} 
		where $\mathfrak{z}(\mu)=|\Aut(\mu)|\prod_{i=1}\mu_i$ and $f^\lambda_\mu:=f_\mu(\lambda)$. 
		\item the Riemann-Hurwitz formula for the factor $2^{1-g}$. 
	\end{itemize}
\end{proof}

\subsection{Fock space formalism}
In analogy to classical Hurwitz theory, there is a Fock space formalism for completed cycles spin Hurwitz numbers. For a more detailed discussion, we refer to \cite{arXiv:2104.05697}.

We begin with a brief discussion of the relevant Lie theory. Vector spaces with bilinear forms give rise to Lie algebras. If the bilinear form is symmetric, the result is a \emph{Clifford algebra}.

\begin{definition}
	Let $V$ be a vector space over $\mathbb{K}$ equipped with a quadratic form $q$. The \emph{Clifford algebra} $\mathcal{Cl}(V,q)$ of $(V,q)$ is the unital free algebra on $V$ (i.e. the tensor algebra of $V$ with unit $1$) modulo the relations 
	\begin{equation*}
		v^2=q(v)\cdot 1
	\end{equation*}
	for all $v\in V$. The commutator Lie bracket turns $\mathcal{Cl}(V,q)$ into a Lie algebra.

\end{definition}

\begin{remark}
	If $\text{Char}(\mathbb{K})\neq 2$, then a quadratic form is the same as a symmetric bilinear form $\langle \cdot, \cdot \rangle$. We will write $\mathcal{Cl}(V,\langle\cdot{,}\cdot\rangle)$ in this case. Using $\langle u, v\rangle:=\frac{1}{2}(q(u+v)-q(u)-q(v))$, one rewrites the relations $v^2=q(v)\cdot 1$ as $uv + vu=2 \langle u, v\rangle\cdot 1$, where $ uv + vu$ is called the \emph{anti commutator} of $v$ and $u$.
\end{remark}

If the bilinear form is skew-symmetric, the result is a \emph{Heisenberg algebra}.

\begin{definition}
	Let $V$ be a vector space over $\mathbb{K}$ and $\omega:V \times V \rightarrow \mathbb{K}$ a skew-symmetric bilinear form. The \emph{Heisenberg algebra} $\mathcal{Heis}(V,\omega)$ of $(V,\omega)$ is the Lie algebra with underlying vector space $V\oplus \mathbb{K}$ and Lie bracket 
	\begin{equation}\label{equation_commutatorrelationHeisenberg}
		[(v,k), (v',k')]:=(0,\omega(v,v')), 
	\end{equation}
	where $v,v'\in V$ and $k,k'\in \mathbb{K}$.
\end{definition}

\begin{definition}\label{definition_Liealgebrarepresentation}
	A \emph{representation of a Lie algebra}/\emph{projective representation of a Lie algebra} $\mathcal{g}$ on a $\mathbb{C}$-vector space $V$ is a Lie algebra homomorphism $\rho$ from $\mathcal{g}$ to the algebra $\text{gl}(V)$ of endomorphisms of $V$/the quotient $\gl(V)/\mathbb{C}\cdot \text{Id}$. Moreover, $\rho$ is said to be \emph{irreducible} if $V$ does not contain any non-trivial $\rho$-invariant subspaces.	
\end{definition}

We define the Clifford algebra $\mathcal{Cl}^\mathrm{B} $ generated by the complex vector space with basis of formal symbols $\{\varphi_n, n\in \mathbb{Z}\}$ and with symmetric bilinear form $\langle\varphi_k,\varphi_l\rangle=\frac{(-1)^{k}}{2}\delta_{k,-l}$. We call the symbols $\varphi_n$ fermions and have the anticommutation relations

\begin{equation}
    \{\varphi_k,\varphi_l\}=(-1)^k\delta_{k,-l}.
\end{equation}

The subspace  $L_0^\mathrm{B}=\mathbb{C}(\phi_n)_{n<0}$ is maximally isotropic and $\mathcal{F}^\mathrm{B}=\mathcal{Cl}^\mathrm{B}/\mathcal{Cl}^\mathrm{B}L_0^\mathrm{B}$ is called the \emph{fermionic Fock space} of type $B$. Moreover, we denote its degree $0$ subspace by $\mathcal{F}_0^B$. 

We write $|0\rangle$ or $v_{\emptyset}$ for the class of $1$.  The fermionic Fock space inherits an inner product and we define for $M\in\mathcal{Cl}^\mathrm{B}$ its \emph{vacuum expectation} $\langle M\rangle=\langle v_\emptyset, Mv_\emptyset\rangle$.

\begin{definition}
    We define the normal ordering of quadratic expressions of fermions as
    \begin{equation}
        \colon \varphi_k\varphi_l\colon=\varphi_k\varphi_l-\langle \varphi_k\varphi_l\rangle.
    \end{equation}
\end{definition}

We denote by $a_\infty=\{(a_{m,n})_{m,n\in\mathbb{Z}}\mid a_{m,n}=0\,\textrm{for}\,|m-n|\ll 0\}$ the bi-infinite general linear algebra. Moreover, we define an involution $\iota\colon a_\infty\to a_\infty, E_{i,j}\mapsto (-1)^{i+j}E_{-j,-i}$ and denote the subspace

\begin{equation}
    b_\infty=\{g\in a_\infty\mid \iota(g)=-g\}.
\end{equation}

A basis on $b_\infty$ is given by $B_{i,j}=E_{-i,j}-(-1)^{j+k}E_{-j,-i}$. We have that
\begin{equation}
    \pi\colon B_{i,j}\mapsto (-1)^i\varphi_i\varphi_j
\end{equation}
defines an irreducible linear representation of $b_\infty$ on $\mathcal{F}_0^B$, while 

\begin{equation}
    \hat{\pi}\colon B_{i,j}\mapsto (-1)^i\colon \varphi_i\varphi_j\colon
\end{equation}
defines an irreducible projective one. We denote by $\hat{b}_\infty$ the central extension on which $\hat{\pi}$ is an honest representation. We will write $F_{i,j}=(-1)^i\varphi_i\varphi_j$ and $\hat{F}_{i,j}=(-1)^i\colon\varphi_i\varphi_j\colon$.

We define the \emph{bosonic operators} $\alpha_m^B=-\frac{1}{2}\sum_{k\in\mathbb{Z}}\hat{F}_{k,-k-m}=-\sum_{k>-m/2}F_{k,-k-m}$ and the generating series
\begin{equation}
    H^B=\sum_{m\in 2\mathbb{N}+1}\alpha_m^Bt_m,\quad t_m=\frac{p_m}{m}.
\end{equation}

The operators $\alpha_m^B$ satisfy
\begin{equation}
    [\alpha_m^B,\alpha_n^B]=\frac{m}{2}\delta_{m+n}
\end{equation}

and span the Heisenberg algebra $\mathcal{H}^\mathrm{B}$. It has a unique irreducible representation on the \emph{bosonic Fock space of type $B$} $\mathbb{C}[p_1,p_3,p_5,\dots]$ given by $\rho(\alpha_m^B)=m\partial_{t_m}$ and $\rho(\alpha_{-m}^B)=\frac{1}{2}t_m$ for $m>0$ odd.

\begin{theorem}[{Boson--Fermion correspondence \cite{you1989polynomial}}]
\label{theorem_bosonfermioncorrespondencetypeB}
    As a representation of $\mathcal{H}^B$, the fermionic Fock space $\mathcal{F}_0^B$ is irreducible and therefore isomorphic to the bosonic Fock space $\mathbb{C}[p_1,p_3,p_5,\dots]$. The isomorphism is given by

    \begin{align}
        \mathcal{F}_0^B&\to \mathbb{C}[p_1,p_3,p_5,\dots]\\
        v&\mapsto v_\emptyset^\ast e^{H^B}v
    \end{align}
\end{theorem}
\subsubsection{The vacuum expectation formula}

Using \cref{theorem_characterformulacpltcyclesspinHurwitznumbers} and results from \cite{arXiv:2104.05697} (Corollary 3.33, Lemma 3.35, Proposition 5.4.), the vacuum expectation formula for spin Hurwitz numbers on bosonic Fock space (\cite{arXiv:2104.05697}, (1.5)) extends to multi-completed cycles:

	\begin{lemma}
		The spin Hurwitz number $h^{spin,\bullet,0}_g(\mu, \overline{(r_1 +1)}, ..., \overline{(r_b +1)}, \nu)$ is given by:
		\begin{align*}
			\frac{2^{1-g}}{b!\prod_i \mu_i \prod_j \nu_j} \cdot \langle \prod_{i=1}^{l(\mu)}a^B_{\mu_i} \prod_{i=1}^{n}\mathcal{F}^\mathrm{B}_{r_i + 1}\prod_{i=1}^{l(\nu)}a^B_{-\nu_i} \rangle,
		\end{align*}

        where 
        \begin{itemize}
            \item 	$\mathcal{F}^\mathrm{B}_{r+1}:=$ 
	\begin{equation}\label{equation_bosoniccutandjoin}
		r![z^{r+1}]\bigg(\coth(\frac{z}{2})\sum^{\infty}_{n=1} \sum_{k_1+\dots +k_n=0, k_i \text{ odd}} \frac{2^{n-2}}{n!}:\prod^{n}_{i=1}\zeta(k_iz)\frac{a^B_{k_i}}{k_i}:\bigg)
	\end{equation}
    is the spin completed cycles cut-and-join operator and
    \item $a^B_{k_i}:=\rho(\alpha^B_{k_i})\in \gl(\mathbb{C}[p_1,p_3,...]).$
        \end{itemize}
	\end{lemma}

\subsection{Tropical covers and spin curves}
\label{sec-tropcov}
In this section we gather ingredients for a combinatorial approach to spin Hurwitz numbers and recall the tropical counterparts of curves and spin curves. These are constructed to be compatible with the classical notions: \emph{Stable tropical curves} of genus $g$ with $n$ marked points organize naturally in a moduli space $M^{tr}_{g,n}$ that admits an interpretation as the skeleton of the analytification of  $\overline{\mathcal{M}}_{g,n}$ (\cite{zbMATH06502666}). Likewise, \emph{tropical spin curves} form a moduli space $S^{tr}_{g}$ that corresponds to the skeleton of the analytification of the moduli space $\overline{\mathcal{S}}_{g}$ of stable spin curves constructed by Cornalba (\cite{caporaso2020tropicalizing}). 

We begin with the notion of tropical curves.
  
	\begin{definition}
	\emph{Curves:} A (vertex-weighted) abstract tropical curve $\Gamma$ is a triple $(G,l,g)$ consisting of
	\begin{enumerate}
		\item a metric graph $(G,l)$, i.e. a finite graph $G$ with $E^\infty(G)\subset E(G)$ the subset of \textit{ends} (i.e. edges adjacent to a single vertex), and a function $l:E(G) \rightarrow \mathbb{R}_{\ge0}\cup \infty$ such that $l(e)=\infty$ if and only if $e \in E^\infty(G)$.
		\item a weighting $g:V(G)\rightarrow \mathbb{Z}_{\ge0}$ on the vertices.
	\end{enumerate}
	The pair $(G,g)$ is its \emph{combinatorial type}.
	
	For a vertex $v$, the value $g(v)$ is called the \textit{genus of $v$}. The \emph{genus of $\Gamma$} is defined as
	\begin{equation*}
		g(\Gamma)=b_1(G)+\sum_{v\in V(G)}g(v),
	\end{equation*}
	where $b_1(G)$ denotes the first Betti number of the graph.	
	
	Any $(G',l')$ that is obtained from $(G,l)$ by adding or deleting vertices of valence $2$ (and adapting the length function accordingly) is called a \emph{model} for $\Gamma$. For a standalone tropical curve, we identify $\Gamma$ with its \emph{essential model} $(G,l)$, which is chosen so that $V(G)$ is the set of points with $g(v)>0$ or $\val(v) \neq 2$.
	The curve $\Gamma$ is called 
	\begin{itemize}
		\item \textit{explicit}, if $g(\Gamma)=b_1(\Gamma)$.
		\item  \emph{stable}, if for all $v\in V(\Gamma)$ we have
		\begin{align*}
			2g(v)-2 + \val(v) > 0.
		\end{align*}
		This means that each genus $0$ vertex has to be at least trivalent and 2-valent vertices have to be of higher genus.
	\end{itemize}
	\end{definition}
	
	\emph{Morphisms:} We call a map $\pi: \Gamma \rightarrow B$ a \textit{harmonic morphism/tropical cover} if $\pi$ is a harmonic morphism/tropical cover between the metric graphs (\cite{MR3375652}, Section 2), i.e. $\pi$ is surjective and continuous and there exist models $(G,l)$ for $\Gamma$ and $(\tilde{G},\tilde{l})$ for $B$ such that
	\begin{enumerate}[label={\roman*}]
		\item $\pi(V(G))\subset V(\tilde{G})$ and $\pi^{-1}(E(\tilde{G}))\subset E(G)$, \label{item:1}
		\item locally integer affine linear, i.e. $\pi$ restricts on each edge $e$ to an affine linear function $\pi_{|e}: [0,l(e)] \rightarrow \mathbb{R}: t \mapsto a+\omega(e) t$ whose slope is a positive integer. We call $\omega(e)$ the weight $\omega(e)$.\label{item:2}
		\item harmonic/balanced: Let $v \in \Gamma$ and $e'$ adjacent to $\pi(v)$. We require the integer $d_v:=\sum\limits_{e\mapsto e'} \omega(e)$ to be independent of the choice of $e'$. In this case call $d_v$ the \emph{local degree} of $\pi$ at $v$. \label{item:3}
	\end{enumerate} 
	When working with covers, we identify $\Gamma$ and $B$ with models that fit these conditions.

	\begin{remark}
		\ \\
	\begin{itemize}
		\item Let $\Gamma$ and $B$ be metric graphs. A graph homomorphism $\pi: G \rightarrow \tilde{G}$ (on the underlying graphs) that satisfies Condition \ref{item:1} can be promoted to a morphism of metric graphs (i.e. a graph homomorphism satisfying Condition \ref{item:1} and \ref{item:2}) by a suitable assignment of edge weights: For all contracted edges $e$, we have $\omega(e)=0$. For a non-contracted edge $e$ of finite length, $\omega(e)$ is uniquely determined by the equation $l(\pi(e))=\omega(e)l(e)$. Since $\omega(e)>0$, this relation implies that
		\begin{itemize}
			\item $ l(e)= \infty$ if and only if $l(\pi(e))=\infty$ (i.e. non-contracted ends are mapped to ends).
			\item end weights are not uniquely determined by the metrics.
		\end{itemize} 
		If the balancing condition is satisfied on $E(\Gamma)\setminus E^\infty(\Gamma)$, then $\pi$ can further be promoted to a harmonic morphism by choosing end weights such that the balancing condition is satisfied. If not, $\pi$ remains a morphism of metric graphs for any choices of end weights, but fails to be harmonic.
		\item For a harmonic morphism, we will often drop the distance labelings of $B$ and $\Gamma$ and consider only their combinatorial type. This is due to the fact that for a non-contracted edge $e$ in $\Gamma$ its length is uniquely determined by $\omega(e)$ and the length of $\pi(e)$. 
		\end{itemize}
	
\end{remark}

\textbf{Tropical Hurwitz covers:} We introduce a specific class of tropical covers that emerge from the classical enumeration via degeneration.

\begin{definition}
A tropical cover $\pi$ is a \emph{tropical Hurwitz cover} if at every point $v\in \Gamma$ the local \emph{Riemann Hurwitz condition} is satisfied, i.e. we have 
\begin{align}\label{eq_tropriemannhurw}
r_{\pi,v}\coloneqq d_v (2-2g(h(v)))-\sum(\omega(e)-1)-(2-2g(v))\ge0, 
\end{align}
where the sum is over all edges adjacent to $v$ and the genus function $g$ is extended on $\Gamma$ by mapping points that are not vertices to $0$.

The cover $\pi$ is an \emph{isomorphism} if there exists a tropical Hurwitz cover $\pi'$ such that $\pi\circ\pi'=\pi'\circ \pi=\text{id}$. An isomorphism $\pi: \Gamma \rightarrow \Gamma$ from $\Gamma$ to itself is called an \emph{automorphism} of $\Gamma$. Concretely, an \emph{automorphism} of $\Gamma$ is an automorphism of the underlying graph that preserves edge lengths and vertex genera.

\end{definition}

Let $\pi\colon\Gamma\to B$ be a tropical Hurwitz cover. In analogy to the classical theory, we introduce the following concepts

\begin{itemize}
\item Degree: The \emph{degree} of $\pi$ is given by the integer $\text{deg}(\pi):=\sum_{p'\in {\pi}^{-1}(p)}d_{p'}$, where $p\in B$ is an arbitrary point. The balancing condition ensures that $\text{deg}(\pi)$ is independent of the choice of $p$.
\item Ramification: A subset $E$ of $\Gamma$, where $\pi(E)$ is a point, is called a \emph{ramification component} if $E$ is a connected component of $\pi^{-1}(\pi(E))$, such that
\begin{itemize}
\item $E$ contains an edge $e\in E(\Gamma)\backslash E^\infty(\Gamma)$,
\item $E$ contains a vertex $v\in V(\Gamma)$ with $r_{\pi,v}>0$ or
\item $E$ contains a point that lives on an end $e\in E^\infty(\Gamma)$ with $\omega(e)>1$.
\end{itemize} 
Let $\mu=(\mu^1,\dots,\mu^n)$ be a partition. We say an end $e$ has \emph{ramification profile} $\mu$ with respect to $\pi$ if $\pi^{-1}(e)=
\{e_1,...,e_n\}$ a set of ends of $\Gamma$ with weights $\omega(e_i)=\mu^i$.
\end{itemize}

\textbf{Tropical spin curves:} Spin structures on tropical curves arise from the dual spin graph of the stable spin curve (\cite{zharkov2010tropical,caporaso2020tropicalizing}):

Let $\Gamma$ be a graph, with edge set $E$ and vertex set $V$. Then, we consider the $\mathbb{F}_2$ vector spaces $\mathcal{E}$ and $\mathcal{V}$ spanned by $E$ and $V$ respectively. Every element of $\mathcal{E}$ is canonically identified with a subset of $E$. We consider the linear map
\begin{equation*}
\partial\colon\mathcal{E}\to\mathcal{V}
\end{equation*}
where for an edge $e$ with vertices $v$ and $w$, we have $\partial(e)=v+w$. We call $\mathcal{C}_\Gamma=\mathrm{ker}(\partial)$ the \textit{cycle space} of $\Gamma$. Moreover, for $P\in\mathcal{C}_\Gamma$, we denote by $\Gamma/{P}$ the graph obtained by contracting all edges in $P$.

\begin{definition}
We define a \textit{tropical spin curve} to be a triple $(\Gamma,P,s)$, such that
\begin{itemize}
\item $\Gamma$ is an abstract tropical curve,
\item $P\in\mathcal{C}_\Gamma$ is a collection of cycles in $\Gamma$,
\item $s\colon{\Gamma}/{P}\to{\mathbb{Z}}/{2\mathbb{Z}}$ is the \emph{sign (or parity) function}, such that $s(v)=0$ if $g(v)=0$.
\end{itemize}
The parity of a tropical spin curve is defined as
\begin{equation*}
s(\Gamma)=\sum_{v\in V(G/P)}s(v).
\end{equation*}
\end{definition}

\section{Tropical spin Hurwitz numbers}\label{section_tropicalspinHurwitz}

We enhance the covers from \cref{sec-tropcov} with spin structures, obtaining a class of covers that emerges naturally from the degeneration of their classical (spin) counterparts.

\begin{definition}
A \textit{tropical spin cover} is a tuple $(\pi,s,s_B)$, such that
\begin{itemize}
	\item $\pi\colon\Gamma\to B$ is a tropical Hurwitz cover,
	\item all edge weights are odd,
	\item $(\Gamma,[0],s)$ is a tropical spin curve, i.e. we choose the empty cycle for $P$,
	\item $(B,[0],s_B)$ is a tropical spin curve as well.
\end{itemize}
We define the parity of $(\pi,s)$ to be
\begin{equation*}
	p(\pi,s)=\sum_{v\in V(\Gamma)}s(v)\in{\mathbb{Z}}/{2\mathbb{Z}}.
\end{equation*}
\end{definition}

Fix 
\begin{itemize}
\item a tropical spin curve $(B,[0],s_B)$, where $B$ is an abstract tropical curve of genus $h$ with $n$ ends and  parity $p:=s_B(B)$.
\item positive integers $d$ and $g$.
\item  a collection $\mu_1, \dots, \mu_n$ of \emph{odd} partitions of $d$.
\end{itemize}
A  \emph{(connected/potentially disconnected) tropical spin Hurwitz cover} of type $(g\rightarrow (h,p),d,\mu_1,\dots,\mu_n)$ is a tropical spin Hurwitz cover $\pi:(\Gamma,s) \rightarrow(B,s_B)$, where $\pi:\Gamma \rightarrow B$ is a (connected/potentially disconnected) tropical Hurwitz cover of type $(g\rightarrow h,d,\mu_1,\dots,\mu_n)$ and $s_B(B)=p$.
Introduce an equivalence relation on the set of tropical spin Hurwitz covers: Let $(\pi,s): (\Gamma,s)\rightarrow (B,s_B)$ and $(\tilde{\pi},\tilde{s}): (\tilde{\Gamma},\tilde{s})\rightarrow (B,s_B)$ be of type $(g\rightarrow (h,p),d,\mu_1,\dots,\mu_n)$. Then
\begin{align*}
\pi \text{ is equivalent to } \tilde{\pi} :\Leftrightarrow &\exists \text{ isomorphism } \phi: \Gamma \rightarrow \tilde{\Gamma} \text{ such that } \pi= \tilde{\pi}\circ \phi \\ &\text{ and } s(v)=\tilde{s}(\phi(v))
\text{ for all } v\in V(\Gamma).
\end{align*}
As usual, we want to identify equivalent spin Hurwitz covers and introduce the \emph{multiplicity of $(\pi,s,s_B)$} as a record of the number of ways in which $(\pi,s,s_B)$ may be promoted to a degenerate spin Hurwitz cover of nodal curves with the same local parities. Using the same notation as before we define:

\begin{definition}
Let $B$ be a tropical curve and $(\pi,s,s_B)$ a tropical spin cover of $B$, that is unramified over $B\backslash E^\infty(B)$. We define its multiplicity 
\begin{equation}\label{def:localhurwitz}
	\mathrm{mult}(\pi,s,s_B)=(-1)^{\sum_{v\in V(\Gamma)}s(v)}\frac{1}{|\mathrm{Aut}(\pi)|}\prod_{e\in E(\Gamma)\backslash E^\infty(\Gamma)}\omega(e)\prod_{v\in V(\Gamma)}|\mathrm{Aut}(\mu_v)|\cdot H_{g(v)\to g(v')}^{s_B(v'),s(v)}(\mu_v),
\end{equation}
where we use $v'$ to denote $\pi(v)$ and $H_{g(v)\to g(v')}^{s_B(v'),s(v)}(\mu_v)$.
\end{definition}

\begin{definition}\label{definition_tropicalspinHurwitznbr}
The \emph{connected/disconnected tropical spin Hurwitz number} $\mathbb{T}H^p_{g\xrightarrow[]{d} h }(\mu_1,\dots,\mu_n)/\mathbb{T}H^{\bullet, p}_{g\overset{d}{\to} h }(\mu_1,\dots,\mu_n)$ is the weighted sum 
\begin{align}
	\mathbb{T}H^p_{g\overset{d}{\to} h }(\mu_1,\dots,\mu_n)/\mathbb{T}H^{\bullet,p}_{g\overset{d}{\to} h }(\mu_1,\dots,\mu_n):=\sum_{(\pi,s,s_B)}\mathrm{mult}(\pi),
\end{align}
of equivalence classes of (connected/potentially disconnected) tropical spin Hurwitz covers of type $(g\rightarrow (h,p),d,\mu_1,\dots,\mu_n)$.
\end{definition}

\begin{remark}
\begin{itemize}
	\item We note that in the definition of tropical spin Hurwitz numbers, we do not impose an explicit compatibility condition between the sign functions $s$ on $\Gamma$ and $s_B$ on $B$. This contrasts \cite{MasterthesisCobigo}, where the parity function $s:V(\Gamma) \rightarrow \mathbb{Z}/2\mathbb{Z}$ on $\Gamma$ is required to be \emph{admissible}. That is a function $s$, such that for all $v\in V(\Gamma)$ there exists a spin Hurwitz Cover $f$ (with connected domain curve) contributing to the local Hurwitz number associated to $v$ with $p(f) = s(v)$. Here, compatibility issues are encoded in the multiplicity:  If the parity function is not admissible, then the cover contributes $0$ to the count.
	\item Moreover, as in the classical case, we do not include $B$ in the notation of the tropical Hurwitz numbers. Again, the tropical spin Hurwitz number is independent of the choice of the genus $g$ curve $B$, however it does depend on the parity of the spin structure on $B$. This is clarified in \cref{thm-correspondencetheorem1}.
	\item A collection of examples and closed formulae for special cases is available in \cite{MasterthesisCobigo}. 
\end{itemize}
\end{remark}

\subsection{Correspondence theorem}
\label{section_correspondencespinclassicaltropical}\label{sec-degproof}
The following correspondence theorem clarifies the relation between classical and tropical spin Hurwitz numbers.
\begin{theorem}[\cite{MasterthesisCobigo}]
\label{thm-correspondencetheorem1}
Let $g,h$ be non-negative integers, $d$ a positive integer, $\mu_1,\dots,\mu_n$ partitions of $d$. Moreover, for $p\in{\mathbb{Z}}/{2\mathbb{Z}}$ we have
\begin{equation*}
	H^{\bullet,p}_{g\to h}(\mu_1,\dots,\mu_n)=\mathbb{T}H^{\bullet,p}_{g\to h}(\mu_1,\dots,\mu_n).
\end{equation*}
\end{theorem}

The key ingredient for the proof of the correspondence theorem is the following degeneration formula for spin Hurwitz numbers.

\begin{theorem}[{\cite[Theorem 1.1]{MR3152971},\cite[Theorem 4.5]{gunningham2016spin}}]
\label{thm:degenerationformula}
    Let $d$ be a positive integer, $h,g$ non-negative integers, $\mu^1,\dots,\mu^n$ odd partitions of $d$ and $p$ the parity of a spin structure. Then, we have
    \begin{enumerate}
        \item if $g_1,g_2$ with $g_1+g_2=g$ and $p_1,p_2$ parities with $p=p_1+p_2\,\mathrm{mod}\,2$ and any $0\le k_0\le k$ then
        \begin{equation}
            H_{h\to g}^{\bullet,p}(\mu^1,\dots,\mu^{n})=\sum_{\mu} |\mu|\cdot\mathrm{Aut}(\mu)\cdot H_{h_1\to g_1}^{\bullet,p_1}(\mu^1,\dots,\mu^{k_0},\mu)H_{h_2\to g_2}^{\bullet,p_2}(\mu,\mu^{k_0+1},\dots,\mu^{k}),
        \end{equation}
        where $h_1$ and $h_2$ are determined by the Riemann Hurwitz formula and $|\mu|=\prod_i\mu_i$ denotes the product over the entries of $\mu$.
        \item if $g\ge2$ or $(g,p)=(1,0)$ then
        \begin{equation}
            H_{h\to g}^{\bullet,p}(\mu^1,\dots,\mu^{n})=\sum_{\mu}|\mu|\cdot\mathrm{Aut}(\mu)\cdot H_{\tilde{h}\to g-1}^{\bullet,p}(\mu,\mu,\mu^1,\dots,\mu^{k}),
        \end{equation}
        where $\tilde{h}$ is determined by the Riemann Hurwitz formula and $|\mu|=\prod_i\mu_i$ denotes the product over the entries of $\mu$
    \end{enumerate}
     The sums are taken over all odd partitions $\mu$ of $d$.
\end{theorem}

We begin by providing an outline of the proof of \cref{thm:degenerationformula} given in \cite{MR3152971} and introduce all notions required for our consideration. Indeed our proof of \cref{thm:degenerationformula} will proceed by applying these techniques iteratively and capturing all data in a tropical spin cover.

The idea of degenerating curves for the purpose of enumeration is well--established. The main challenge in the case of spin Hurwitz numbers is the degenerative behaviour of spin structures. We begin with the following definition.

\begin{definition}
    Let $\pi : \mathcal{C} \rightarrow \Delta$ be a smooth family of curves and $\omega_\pi$ the relative cotangent bundle of $\pi$. Then a line bundle $\mathcal{L}$ on $\mathcal{C}$, such that $\mathcal{L}^{\otimes 2}\cong \omega_\pi$, is called a \textit{family of theta characteristics}. If we fix $\alpha : \mathcal{L}^{\otimes 2} \rightarrow \omega_\pi$, then the triplet $(\mathcal{C}, \mathcal{L}, \alpha)$ is called a \textit{family of spin curves}.
\end{definition}

On each fibre $C_t:=\pi^{-1}(t)$, for $t\in \Delta$, the relative cotangent bundle  $\omega_{\pi}$ restricts to the cotangent bundle $\omega_{C_t}$ of the fibre and $\mathcal{L}$ to a theta characteristic thereon.

The following classical result due to Atiyah and Mumford is crucial for our study since we aim to study spin Hurwitz numbers via degeneration.

	\begin{theorem} \label{thm_paritydeforminvar}
		The parity of a spin curve of genus $g$ is a deformation invariant.
	\end{theorem}

 So far, we have restricted ourselves to families of smooth curves. Our aim however is to degenerate to curves with nodal singularities. A conceptual framework for this is given by moduli theory. More precisely, let $M_{g,n}$ be the moduli space of smooth genus $g$ curves with $n$ marked points and let $S_{g,n}$ be the moduli space of smooth genus $g$ spin curves with $n$ marked points. There is a natural finite forgetful map
 \begin{equation}
     \phi\colon S_{g,n}\to M_{g,n},\, (C,L)\mapsto C
 \end{equation}
 that forgets the spin structure. Thus one would like to find a compactification $S_{g,n}\subset \overline{S}_{g,n}$, such that $\phi$ extends to a finite map $\overline{\phi}\colon\overline{S}_{g,n}\to\overline{M}_{g,n}$. It turns out that this is not quite the correct way to frame the question on degenerations of spin curves but rather one has to enlarge $\overline{M}_{g,n}$ to include so-called \textit{quasistable} curves.

\begin{definition}
     Let $(C,\sigma)$ be an $n$-marked curve, where $C$ has at most nodes as singularities and $\sigma=\{p_1,...,p_n\}$ is a set of distinct and smooth points of $C$. We say that a component $E\subset X$ is exceptional, if $E\cong\mathbb{P}^1$, $|E\cap \bar{C\setminus E}|=2$ and $E\cap\sigma=\emptyset$.
		Then $(C,\sigma)$ is
		\begin{enumerate}
			\item \emph{stable}, if $C$ does not contain any exceptional components.
			\item \emph{semistable}, if a smooth rational component meets other components in at least two points.
			\item \emph{quasistable}, if $C$ is semistable and two exceptional components are disjoint.
		\end{enumerate}
 \end{definition}

Given a  quasistable curve $\hat{C}$ we obtain a stable one by contracting all the exceptional components. Conversely, take a stable curve $C$ and a set of nodes $R$ of $C$. The blow up of $C$ at $R$ is a quasistable curve $\hat{ C}$:\\
	Let $\nu_R: C^{\nu}_R \rightarrow C$ be the partial normalization map, where ${\nu_r}^{-1} (r)=\{r_1,r_2\}$ for $r\in R$. Then $\hat{ C}$ is defined as 
	\begin{align}
		\hat{ C}:=C^{\nu}_R\cup \bigcup_{r\in R}E_r\text{ with exceptional components }E_r.
	\end{align} 

The presence of quasistable curves is explained in \cite{Farkas_thetacharacteristics}. Indeed if one insists on $\overline{S}_{g,n}$ to be fibered over $\overline{M}_{g,n}$, then one needs to allow for singularities in the sheaves in the compactification of $S_{g,n}$. By allowing quasi-stable curves, one can actually think of degenerating line bundles to line bundles.\\
We are now ready to define quasi-stable spin curves.

\begin{definition}
    Let $\hat{C}$ be a quasistable curve. Consider a triple $(\hat{C},L,\alpha)$, where $L$ is a line bundle of degree $g-1$ and $\alpha:L^{\otimes2} \rightarrow \omega_{\hat{C}}$ a homomorphism satisfying the following conditions:
		\begin{enumerate}
			\item $L_{|E}=\mathcal{O}(1)$ on each exceptional component of $\hat{C}$.
			\item $\alpha$ induces an isomorphism $\tilde{\alpha}: L^{\otimes2}_{|C}\rightarrow \omega_{C}$ , where $C$ denotes the curve $\hat{C}$ without exceptional components.
		\end{enumerate} 
		Then $(L,\alpha)$ is called a \emph{spin structure} on $\hat{ C}$ and  $(\hat{C},L,\alpha)$ a spin curve. We denote by $\overline{S}_{g,n}$ the moduli space of spin quasistable curves.
\end{definition}

If $\hat{C}$ is smooth, $L$ is just a theta characteristic.\\
As first evidence that this is a reasonable definition, we note that the classical enumeration of spin structures (i.e. there exist $2^{2g}$ spin structures on a genus $g$ curve) extends to spin structures on quasistable curves in the following sense: Let $\overline{\phi}\colon \overline{S}_{g,n}\to\overline{M}_{g,n}$ be the map that sends $(\hat{C},L,\alpha)$ to $C$ the underlying stable curves obtained by contracting all exceptional components. Then $\overline{\phi}$ has degree $2^{2g}$. \vspace{\baselineskip}

Now, we study degenerations of theta characteristics to quasistable curves.

\begin{definition}
\label{def-famspincurves}
    A one-parameter family of curves is a family of (pointed) nodal curves $\pi : \mathcal{C} \rightarrow \Delta$ over a regular, connected curve $\Delta$ with a marked point, $t_0 \in \Delta$. We denote by $C_0$ the fibre over $t_0$ and	 we will always assume that every fibre $C_t$ for $t\in \Delta \setminus \{t_0\}$ has the same dual graph. We shall call $C_t$ the ``generic'' fibre.\\
	A \emph{one-parameter family of stable spin curves} is a triplet $(\mathcal{C}, \mathcal{L}, \alpha)$, where $\mathcal{L}$ is a line bundle and $\alpha : \mathcal{L}^{\otimes2} \rightarrow \omega_\pi$ a homomorphism such that for every $t \in \Delta$ the fibre $(C_t,{\mathcal{L}}_{|C_t},{\alpha}_{|C_t})$ is a stable spin curve.
\end{definition}

The proof of \cref{thm:degenerationformula} given in \cite{MR3152971} proceeds in several steps.

\subsubsection*{Step 1}
For the first step, we express spin Hurwitz numbers in terms of moduli spaces.

\begin{definition}
    Let $D$ be a smooth curve of genus $h$ and let $V = \{q^1,..., q^k\}$ be a fixed set of points	on $D$. Given partitions $\mu^1,...,\mu^k$ of $d$, a degree $d$ holomorphic map $f:C\rightarrow D$ from a (possibly disconnected) curve $C$ is called $ V$-regular with contact partitions $\mu^1,...,\mu^k$ if for each $i=1,...,k$, we have that $f^{-1}(q^i)$ consists of $l(\mu^i)$ points $q^i_j$ so that the ramification index of $f$ at $q^i_j$ is $(\mu^i)_j$. We call the number $j$ the label of $q^i_j$.\\
    Two $V$-regular maps $(f, C; \{q^i_j\})$ and $(\tilde{f}, \tilde{C}; \{\tilde{q}^i_j\})$ are isomorphic if there is a biholomorphism $\sigma: C\rightarrow \tilde{C}$ with $\tilde{f}\circ \sigma=f$ and $\sigma(q^i_j)=\tilde{q}^i_j$	for all $i,j$.\\
    The relative moduli space
	\begin{equation}\label{def_relgromovwittenspace}
		M^V_{\chi,\mu^1,...,\mu^k}(D,d),
	\end{equation}
	consists of isomorphism classes of $V$-regular maps $(f, C; \{q^i_j\})$ as above with Euler characteristic $\chi\coloneqq\chi(C)$.
\end{definition}

Relating $V$-regular maps to spin Hurwitz covers allows us to express spin Hurwitz numbers in terms of $V$-regular maps.
	\begin{lemma}[{\cite[Lemma 1.1]{MR3068550}}] \label{lem_spinhnbrGW}\ \\
		Fix genera $g,h$, degree $d$ and partitions $\mu^1,...,\mu^n$ of an integer $d$.
		\begin{enumerate}
			\item If $\mu^1,...,\mu^n$ are all odd partitions, then
			\begin{align}
				H^{\bullet,p}_{h\to g}(\mu^1,...,\mu^n)=\frac{1}{\prod_{i=1}^{n}|\mathrm{Aut}(\mu^i)|}\sum \frac{{(-1)}^{p(f)}}{|Aut(f,V)|},
			\end{align} where the sum goes over all $(f, C; \{q^i_j\}) \in {\mathcal{M}}_{\chi,\mu^1,...,\mu^n}^V(D,d)$ with $\chi=2-2g$.
			\item If $\mu^i=(1^d)$, then $|Aut(f,V)|=1$.
		\end{enumerate}
	\end{lemma}

\subsubsection*{Step 2}
For the second step, we consider the following set-up. Fix discrete data $d,g,h, \mu^1,...,\mu^n$ as before. We set $\chi=2-2h$ and consider $D_0=D_1\cup E \cup D_2$ of genus $g$ such that for $i=1,2$:
	\begin{itemize}
		\item $D_0$ contains two nodes $p^1$ and $p^2$, such that $D_i$ meets $E$ at node $p_i$.
		\item $D_i$ is smooth with $g(D_i)=g_i$ and $g_1+g_2=g$.
		\item $E$ is an exceptional component.
	\end{itemize}

Let $\mu^{n+1}=\mu^{n+2}=\mu^{n+3}=(1^d)$, $m=(m_1,...,m_l)$ be a partition of $d$ and $1\leqslant k_0\leqslant n$. Consider the product space
	\begin{equation}
		\mathcal{P}_m=\mathcal{M}^{V_1}_{\chi_1,\mu^{n+1},\mu^1,...,\mu^{k_0},m}(D_1,d) \times \mathcal{M}^{V_e}_{\chi_e,m,\mu^{n+2},m}(E,d) \times \mathcal{M}^{V_2}_{\chi_2,m,\mu^{n+3},\mu^{k_0+1}...,\mu^k}(D_2,d),
	\end{equation}
 where $V_1=\{q^{n+1}, q^1,..., q^{k_0}, p^1\}$, $V_e = \{p^1,q^{n+2}, p^2\}$, $V_2 = \{p^2, q^{k_0+1},..., q^k, q^{n+3}\}$ and $\chi_1+\chi_e+\chi_2-4l(m)=\chi$. For later use we call the first, second and third factors of $\mathcal{P}_m$: $\mathcal{M}_m^1, \mathcal{M}_m^e, \mathcal{M}_m^2$.\vspace{\baselineskip}
 
 It follows from Lemma 2.1 and from Lemma 2.2 in \cite{MR3068550} that a map $(f_1,f_e,f_2)\in \mathcal{P}_m$ has the following structure:
	\begin{itemize}
		\item $f_1,f_2$ are ramified at contact marked points $q^i$ and nodal points only.
		\item the domain of $f_e$ is a disjoint union of $l(m)$ rational curves $E_i$, i.e. $f_e:\bigcup_{i=1}^{l(m)}E_i \rightarrow E$, where $f_{|E_i}$ has degree $m_i$.
	\end{itemize}
\cref{fig19} shows a cartoon of such a map.
	\begin{figure}[H]
		\centering
		\tikzset{every picture/.style={line width=0.75pt}} 
		
		\begin{tikzpicture}[x=0.75pt,y=0.75pt,yscale=-1,xscale=1]
			
			\draw   (330,206) .. controls (330,196.61) and (350.82,189) .. (376.5,189) .. controls (402.18,189) and (423,196.61) .. (423,206) .. controls (423,215.39) and (402.18,223) .. (376.5,223) .. controls (350.82,223) and (330,215.39) .. (330,206) -- cycle ;
			\draw   (95,203.5) .. controls (95,193.28) and (113.8,185) .. (137,185) .. controls (160.2,185) and (179,193.28) .. (179,203.5) .. controls (179,213.72) and (160.2,222) .. (137,222) .. controls (113.8,222) and (95,213.72) .. (95,203.5) -- cycle ;
			\draw    (152,204) -- (354,202) ;
			\draw   (320,85) .. controls (320,67.88) and (346.42,54) .. (379,54) .. controls (411.58,54) and (438,67.88) .. (438,85) .. controls (438,102.12) and (411.58,116) .. (379,116) .. controls (346.42,116) and (320,102.12) .. (320,85) -- cycle ;
			\draw   (74,81) .. controls (74,65.54) and (98.62,53) .. (129,53) .. controls (159.38,53) and (184,65.54) .. (184,81) .. controls (184,96.46) and (159.38,109) .. (129,109) .. controls (98.62,109) and (74,96.46) .. (74,81) -- cycle ;
			\draw    (143,105) -- (348,105) ;
			\draw    (174,85) -- (328,85) ;
			\draw    (164,62) -- (342,62) ;
			\draw    (253,117) -- (252.04,171) ;
			\draw [shift={(252,173)}, rotate = 271.02] [color={rgb, 255:red, 0; green, 0; blue, 0 }  ][line width=0.75]    (10.93,-3.29) .. controls (6.95,-1.4) and (3.31,-0.3) .. (0,0) .. controls (3.31,0.3) and (6.95,1.4) .. (10.93,3.29)   ;
			
			\draw (119,72.4) node [anchor=north west][inner sep=0.75pt]    {$C_{1}$};
			\draw (373,197.4) node [anchor=north west][inner sep=0.75pt]    {$D_{2}$};
			\draw (120,194.4) node [anchor=north west][inner sep=0.75pt]    {$D_{1}$};
			\draw (372,77.4) node [anchor=north west][inner sep=0.75pt]    {$C_{2}$};
			\draw (250,215.4) node [anchor=north west][inner sep=0.75pt]    {$E$};
			\draw (242,85.4) node [anchor=north west][inner sep=0.75pt]    {$E_{3}$};
			\draw (241,66.4) node [anchor=north west][inner sep=0.75pt]    {$E_{2}$};
			\draw (242,41.4) node [anchor=north west][inner sep=0.75pt]    {$E_{1}$};
			\draw (320,180.4) node [anchor=north west][inner sep=0.75pt]    {$p^{2}$};
			\draw (181,180.4) node [anchor=north west][inner sep=0.75pt]    {$p^{1}$};
			\draw (115,127.4) node [anchor=north west][inner sep=0.75pt]    {$f_{1}$};
			\draw (371,131.4) node [anchor=north west][inner sep=0.75pt]    {$f_{2}$};
			\draw (213,126.4) node [anchor=north west][inner sep=0.75pt]    {$f_{e}$};

		\end{tikzpicture}

		\caption{Sketch of a map into $D_0$.  }
		\label{fig19}
        \end{figure}
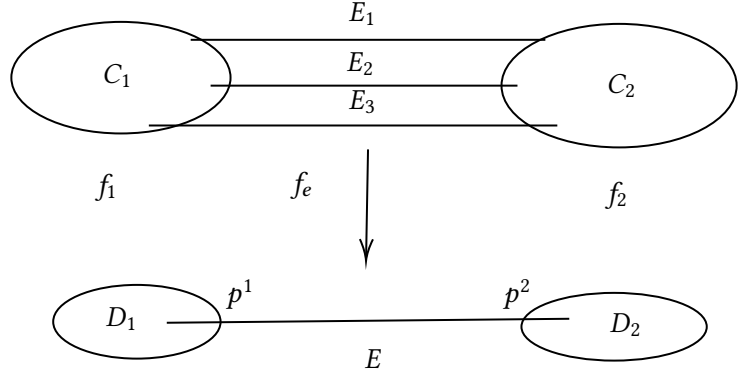

By Lemma 2.2 \cite{MR3068550}) we have that
	\begin{align}
		|\mathcal{M}^{V_e}_{\chi_e,\mu^1,...,\mu^k}(E,d)|=\frac{d! |\mathrm{Aut}(m)|}{|m|},
	\end{align}
 where $|m|=\prod_{i=1}^lm_i$.
	Note, this number is in some sense artificial since forgetting contact marked points yields exactly one map.\\
	Matching ramification orders over the nodes ensures that we can glue the domains of $f_1$ and $f_e$ and of $f_2$ and $f_e$, respectively, together to obtain a map $f:C_0 \rightarrow D_0$. Denote by $\mathcal{M}_{m,0}$ the space of such glued maps $f = (f_1,f_e,f_2)$ and discard the information about labelling above the nodes. Thus,
	\begin{align}
		\mathcal{P}_{m} \to \mathcal{M}_{m,0}, (f_1,f_e,f_2) \mapsto f
	\end{align}
	is a  cover of degree $|\mathrm{Aut}(m)|^2$  (Lemma 2.3 \cite{MR3068550}). The factor $|\mathrm{Aut}(m)|^2$ appears because we have $|\mathrm{Aut}(m)|$ possibilities to label the ramification points above each node.\\

\subsubsection*{Step 3}
The third step is concerned with identifying the space $\mathcal{M}_{m,0}$ as the space of limits of maps to a smooth target where the target degenerates to $D_0$.

Fix $D_0$ as in step 1 and construct a 1-parameter family of curves $\mathcal{D} \rightarrow \Delta$ (with parameter $r$) as in section 4 of \cite{MR3068550} together with sections $Q^i$ that select $n+3$ special points on each fibre such that $Q^i(0)=q^i$ for $i=1,...,n+3$. The generic fibre $D_r$ for $r\neq 0$ is smooth and of genus $h$ and the special fibre is given by $D_0$. For $r\neq 0$ consider the moduli space
	\begin{align}
		\mathcal{M}_r:=\mathcal{M}^{V_r}_{\chi,\mu^1,...,\mu^{n+3}}(D_r,d) \text{ with } V_r=\{Q^1(r),...,Q^{n+3}(r)\}.
	\end{align}
Of course we can compute the spin Hurwitz number $H^{(h,p)}_{\mu^1,...,\mu^n}$  by counting maps in $\mathcal{M}_r$ for any $r\neq 0$ (see  \cref{lem_spinhnbrGW}). For $r=0$ the role of $M_r$ is replaced by $\mathcal{M}_{m,0}$. More precisely, we have by \cite[Lemma 3.1]{MR3068550}) that

\begin{align}
    		\mathrm{lim}_{r\rightarrow 0}\,\mathcal{M}_r \subset \bigcup_{m}\mathcal{M}_{m,0},
\end{align}
where the union is over all partitions $m$ of $d$ with $\mathcal{P}_m\neq \emptyset$ and $\mathrm{lim}_{r\rightarrow 0}\thinspace\mathcal{M}_r$ denotes the set of limits of sequences of maps in $\mathcal{M}_r $ as $r\rightarrow 0$ (for a justification of the existence of such limit maps see \cite{MR3068550}). On the other hand, we can reconstruct all maps in $\mathcal{M}_r $ by choosing $f\in \mathcal{M}_{m,0}$ and smoothing the domain curve $C_0$ (\cite{MR3068550}). Recall, for $i=1,2$ the preimage of $p^i\in D_i$ consists of $l(m)$ points $x^i_j$. Each one is a node of $C_0$ and can be smoothed in $m_j$ ways (\cite{MR3068550}). This yields

\begin{align}
		\mathcal{M}_r=\bigcup_{m}\bigcup_{f\in \mathcal{M}_{m,0}}\mathcal{Z}_{m,f,r} \text{ for } r\neq 0,
\end{align}

where $\mathcal{Z}_{m,f,r}$ (informally) denotes the set of maps in $\mathcal{M}_r$ having $f\in\mathcal{M}_{m,0}$ as limit and $|\mathcal{Z}_{m,f,r}|={|m|}^2$. For a more precise definition of $\mathcal{Z}_{m,f,r}$ see \cite{MR3068550}.

\subsubsection*{Step 4}
We now degenerate the base and source curves simultaneously. Fix $f: C_0 \rightarrow D_0 \in\mathcal{M}_{m,0}$. We want to make the statement that $f$ is the limit of maps in $M_r$ precise.

This means that we need, in addition to the deformation of $D_0$ given by $\mathcal{D}$, a deformation of $C_0$. This construction is done in Theorem 5.1. \cite{MR3152971}. For each 
\begin{equation}
\zeta\in Q_m:=\{\zeta=(\zeta^1_1,\zeta^2_1,...,\zeta^1_{\ell(m)},\zeta^2_{\ell(m)}) \text{ with } {\left(\zeta^1_j\right)}^{m_j}=1 \text{ and } {\left(\zeta^2_j\right)}^{m_j}=1 \}
\end{equation}
there is a 1-parameter family of quasistable curves $\mathcal{C}_\zeta\rightarrow \Delta$ (with parameter $s$) together with a holomorphic map $\mathcal{F}_\zeta: \mathcal{C}_\zeta\rightarrow \mathcal{D}$, such that the generic fibre $\mathcal{C}_{\zeta,s}$  for $s\neq 0$ is smooth and the map $f _{\zeta,s}={\mathcal{F}_\zeta}_{|\mathcal{C}_{\zeta,s}}$ is contained in $\mathcal{M}_r$ (where $s^{|m|}=r$).
The special fibre is given by $C_0$ in the case $m$ is an odd partition of $d$ and $f _{\zeta,0}=f$. If $m$ is even the special fibre is a little bit more complicated (see \cite{MR3152971} for a detailed account). However, we do not need to worry about that since we will see later on when introducing spin structures on both families that the contributions of these maps cancel out. \\
Note, the families $\mathcal{C}_\zeta$ are indexed by vectors $\zeta\in Q_m$. This accounts for the different possibilities of smoothing the nodes of $C_0$ and provides us with a way to get all maps in the set $\mathcal{Z}_{m,f,r}$. By Lemma 5.2. (\cite{MR3152971}) we have for $s\neq 0$ and $s^{|m|}=r$ that
	\begin{align}
		\mathcal{Z}_{m,f,r}=	\bigcup_{\zeta\in Q_m}\{f_{\zeta,s}\}.
	\end{align}

\subsubsection*{Step 5}
Now is the time for theta characteristics to enter the story. A spin structure on $D_0$ is given by a collection of theta characteristics on each of its  smooth components. Thus, choose parities $p_1$ and $p_2$ with $p=p_1+p_2 \mod 2$. In \cite{MR3068550}, a family $\mathcal{D} \rightarrow \Delta$ is endowed with a spin structure $(\mathcal{N},\Phi)$ such that:
	\begin{itemize}
		\item for $r\neq 0$, $\mathcal{N}$ restricts to a theta characteristic on $D_r$ with a parity $p$.
		\item for $i=1,2$ the restriction of $\mathcal{N}$ to the smooth component $D_i$ of the special fibre $\mathcal{N}_{|D_i}$ is a theta characteristic of parity $p_i$.
	\end{itemize}  

 The tuple $(\mathcal{D}, \mathcal{N},\Phi)$ is a family of stable spin curves as in \cref{def-famspincurves}. The same can be done for the family $\mathcal{C}_\zeta$. We restrict to the case $m$ is odd.

 \begin{theorem}[ Theorem 5.1. \cite{MR3152971}]
For $f = (f_1, f_e, f_2) \in\mathcal{M}_{m,0}$ let $\mathcal{C}_\zeta\rightarrow \Delta$ and $\mathcal{F}_\zeta: \mathcal{C}_\zeta\rightarrow \mathcal{D}$ be as above. Then, there exists a line bundle $\mathcal{L}_\zeta$ over $\mathcal{C}_\zeta$ satisfying:
\begin{enumerate}
			\item the restriction ${\mathcal{L}_\zeta}_{|\mathcal{C}_{\zeta,s}}$ is a theta characteristic on ${\mathcal{C}_{\zeta,s}}$ and the restriction map  $f _{\zeta,s}={\mathcal{F}_\zeta}_{|\mathcal{C}_{\zeta,s}}$ has the associated parity $p(f _{\zeta,s})=p({\mathcal{L}_\zeta}_{|\mathcal{C}_{\zeta,s}})$ since ${\mathcal{L}_\zeta}_{|\mathcal{C}_{\zeta,s}}=L_{f _{\zeta,s}}$.
			\item the restriction to the special fibre is a collection of theta characteristics ${\mathcal{L}_\zeta}_{|C_1}=L_{f_1}$ and ${\mathcal{L}_\zeta}_{|C_2}=L_{f_2}$.
\end{enumerate}    
\end{theorem}

\subsubsection*{Step 6}
The following theorem describes the parity of maps $f = (f_1, f_e, f_2) \in\mathcal{M}_{m,0}$.

\begin{theorem}[\cite{MR3152971},  Theorem 4.2]\label{thm_factsparity}
Let $f = (f_1, f_e, f_2) \in\mathcal{M}_{m,0}$ and $r\neq 0$.
\begin{enumerate}
	\item If $m$ is odd, then $p(f_r) = p(f_1)+ p(f_2)\mod 2$ for all $ f_r \in  \mathcal{Z}_{m,f,r}$.
	\item If $m$ is even, then $\sum_{f_r\in \mathcal{Z}_{m,f,r}}{(-1)}^{p(f)}=0$.
\end{enumerate}
\end{theorem}

\subsubsection*{Step 7}
 We are finally ready to give a proof of \cref{thm:degenerationformula}. We will only prove the first part, for the second part a different degeneration is employed which we comment on in \cref{rem-casetwo}.

 Spin Hurwitz numbers depend only on discrete data, i.e. the genus of $D$, the parity of the theta characteristic $N$ and the ramification profile over the $q^i$, but neither on the configuration of the points nor on the curve $D$. This means we can replace $\mathcal{M}^V_{\chi,\mu^1,...,\mu^n}(D,d)$ (step 0) by any one of the spaces $\mathcal{M}_r$ for $r\neq 0$. By \cref{lem_spinhnbrGW} we have
	\begin{align}
		\label{eqn:proofdegformspin}
		H^{\bullet,p}_{h\to g}(\mu^1,...,\mu^n)&=H^{\bullet,p}_{h\to g}(\mu^1,...,\mu^{k},(1^d),(1^d),(1^d))=\frac{1}{{d!}^3\prod_{i=1}^{n}|\mathrm{Aut}(\mu^i)|}\sum_{f_r\in \mathcal{M}_r}{(-1)}^{p(f_r)} \\
		&=\frac{1}{{d!}^3\prod_{i=1}^{n}|\mathrm{Aut}(\mu^i)|} \sum_{m} \sum_{f\in \mathcal{M}_{m,0}}\sum_{f_r\in \mathcal{Z}_{m,f,r}}{(-1)}^{p(f_r)}
	\end{align}
	where we use $\mathrm{|Aut}((1^d))|=d!$ for the first equality and $\mathcal{M}_r=\bigcup_{m}\bigcup_{f\in \mathcal{M}_{m,0}}\mathcal{Z}_{m,f,r}$ for the second. Thanks to \cref{thm_factsparity} the last sum vanishes whenever $m$ is an even partition of $d$.
	Using $|\mathcal{Z}_{m,f,r}|={|m|}^2$ we compute the spin Hurwitz number by counting degenerate coverings instead:
	\begin{align}
		H^{\bullet,p}_{h\to g}(\mu^1,...,\mu^n)=\frac{1}{{d!}^3\prod_{i=1}^{k}|\mathrm{Aut}(\mu^i)|}\sum_{m\colon \textrm{odd}}{|m|}^2\sum_{f=(f_2,f_e,f_1)\in \mathcal{M}_{m,0}}{(-1)}^{p(f_1)+p(f_2)}.
	\end{align}
	Recall  \begin{align}
		\mathcal{P}_{m}\thinspace \rightarrow \mathcal{M}_{m,0}, (f_1,f_e,f_2) \mapsto f
	\end{align}
	is a  cover of degree ${|\mathrm{Aut}(m)|}^2$. We can sum over disjoint tuples in $\mathcal{P}_{m}$ as well.
	\begin{align}
		\sum_{f=(f_2,f_e,f_1)\in \mathcal{M}_{m,0}}{(-1)}^{p(f_1)+p(f_2)}&=\frac{1}{|\mathrm{Aut}(m)|^2}\sum_{f=(f_2,f_e,f_1)\in \mathcal{P}_{m}}{(-1)}^{p(f_1)+p(f_2)}\\
		&=\frac{1}{{|\mathrm{Aut}(m)|}^2}\sum_{f_e\in \mathcal{M}^{e}_m }\left(\sum_{f_1\in \mathcal{M}^{1}_m }{(-1)}^{p(f_1)}\right)\left(\sum_{f_2\in \mathcal{M}^{2}_m}{(-1)}^{p(f_2)}\right)\\
		&\underbrace{=}_{|\mathcal{M}^{e}_m|=\frac{d!|\mathrm{Aut}(m)|}{|m|}}\frac{d!}{|\mathrm{Aut}(m)||m|}\left(\sum_{f_1\in \mathcal{M}^{1}_m }{(-1)}^{p(f_1)}\right)\left(\sum_{f_2\in \mathcal{M}^{2}_m}{(-1)}^{p(f_2)}\right),
	\end{align}
	where 
	\begin{align}
		\mathcal{M}^{1}_m:=&\mathcal{M}^{V_1}_{\chi_1,(1^d),\mu^1,...,\mu^{k_0},m}(D_1,d), \thinspace\mathcal{M}^{e}_m:=\mathcal{M}^{V_e}_{\chi_e,m,(1^d),m}(E,d), \\
		& \mathcal{M}^{2}_m:=\mathcal{M}^{V_2}_{\chi_2,m,(1^d),\mu^{k_0+1}...,\mu^k}(D_2,d).
	\end{align}
	To write $\sum_{f_1\in \mathcal{M}^{1}_m }{(-1)}^{p(f_1)}$ as contribution to the spin Hurwitz number $H^{\bullet,p_1}_{h_1\to g_1}(\mu^{k+1},\mu^1,...,\mu^{k_0},m)$ we need to correct by a factor of $d!|\mathrm{Aut}(m)|\prod_{i=1}^{k_0}|\mathrm{Aut}(\mu^i)|$, analogously for $\sum_{f_2\in \mathcal{M}^{2}_m }{(-1)}^{p(f_2)}$.
	\begin{align}
		\sum_{f=(f_2,f_e,f_1)\in \mathcal{M}_{m,0}}{(-1)}^{p(f_1)+p(f_2)}=\frac{{(d!)}^3|\mathrm{Aut}(m)|\prod_{i=1}^{n}|\mathrm{Aut}(\mu^i)|}{|m|}H^{\bullet,p_1}_{h_1\to g_1}({\mu^1,...,\mu^{k_0},m})H^{\bullet,p_2}_{h_2\to g_2}({{m,\mu^{k_0+1},...,\mu^n}}).
	\end{align}
	Substitute back in equation \ref{eqn:proofdegformspin} and the result follows.

 This completes our sketch of the proof of the first part of \cref{thm:degenerationformula}. 
 
 \subsection{Proof of \cref{thm-correspondencetheorem1}}
 In order to prove \cref{thm-correspondencetheorem1}, we need to introduce the notion of dual spin graphs. We first state the following remark

 \begin{remark}
 \label{rem-casetwo}
     As noted above the proof of the second part of \cref{thm:degenerationformula} involves a different degeneration. More precisely, one degenerates the target curve $D$ of genus $g$ to a nodal curve $D_0=\overline{D}\cap E$, where the exceptional component $E$ is isomorphic to $\mathbb{P}^1$, $\overline{D}$ is a spin curve of genus $g-1$ and $D$ meets $E$ in two distinct points. All arguments work analogously.
 \end{remark}

\begin{definition}
    		Let $C$ be a  nodal curve. We construct the dual graph $\Gamma_C$ in the following way::
		\begin{itemize}
			\item For each irreducible component $C_v$ draw a vertex $v$ with weight $g(v)$ defined as the genus of $C_v$.
			\item For each node draw an edge $e$, where $e$ is incident to a vertex $v$ if the node is contained in the irreducible component dual to $v$.
		\end{itemize} 
		If $C$ is a marked curve, we add an end for each marked point $p$ and attach it to the vertex $v$ dual to the component containing $p$. 
\end{definition}

An illustration of a dual graph may be found in \cref{fig8}

 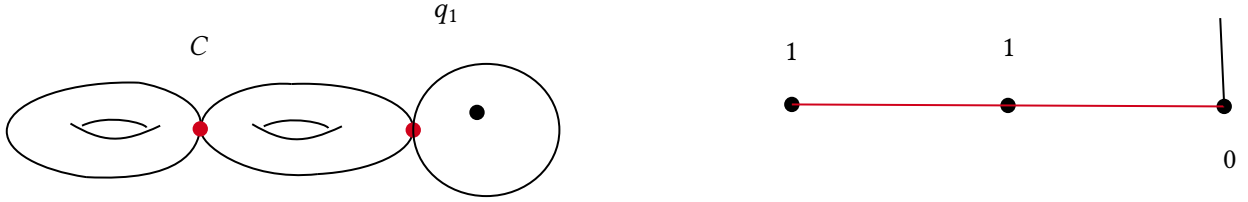
\begin{figure}[H]
		\centering
		\tikzset{every picture/.style={line width=0.75pt}} 
		
		\begin{tikzpicture}[x=0.75pt,y=0.75pt,yscale=-1,xscale=1]
			
			\draw  [draw opacity=0] (54.82,135.47) .. controls (31.57,132) and (15,123.03) .. (15,112.5) .. controls (15,98.97) and (42.37,88) .. (76.13,88) .. controls (77.89,88) and (79.62,88.03) .. (81.34,88.09) -- (76.13,112.5) -- cycle ; \draw   (54.82,135.47) .. controls (31.57,132) and (15,123.03) .. (15,112.5) .. controls (15,98.97) and (42.37,88) .. (76.13,88) .. controls (77.89,88) and (79.62,88.03) .. (81.34,88.09) ;  
			\draw  [draw opacity=0] (171.12,133.95) .. controls (198.13,132.5) and (219,122.95) .. (219,111.38) .. controls (219,98.81) and (194.38,88.62) .. (164,88.62) .. controls (161.92,88.62) and (159.88,88.67) .. (157.86,88.76) -- (164,111.38) -- cycle ; \draw   (171.12,133.95) .. controls (198.13,132.5) and (219,122.95) .. (219,111.38) .. controls (219,98.81) and (194.38,88.62) .. (164,88.62) .. controls (161.92,88.62) and (159.88,88.67) .. (157.86,88.76) ;  
			\draw    (81.34,88.09) .. controls (93.51,90.09) and (114.9,97.86) .. (112.11,111.26) ;
			\draw    (54.82,135.47) .. controls (103,138) and (110.71,121.52) .. (112.11,111.26) ;
			\draw    (138.18,108.91) .. controls (154.27,118.34) and (163.59,119.27) .. (183.2,109.98) ;
			\draw    (143.97,109.98) .. controls (157.21,103.79) and (175.85,108.75) .. (176.34,110.91) ;
			\draw    (46.97,108.6) .. controls (63.06,118.03) and (72.38,118.96) .. (91.99,109.67) ;
			\draw    (52.76,109.67) .. controls (66,103.48) and (84.63,108.44) .. (85.13,110.6) ;
			\draw    (157.86,88.76) .. controls (137,87) and (110.05,98.65) .. (112.8,112.05) ;
			\draw    (171.12,133.95) .. controls (135,138) and (110.74,117.97) .. (112.8,112.05) ;
			\draw [color={rgb, 255:red, 208; green, 2; blue, 27 }  ,draw opacity=1 ]   (112.11,111.26) ;
			\draw [shift={(112.11,111.26)}, rotate = 0] [color={rgb, 255:red, 208; green, 2; blue, 27 }  ,draw opacity=1 ][fill={rgb, 255:red, 208; green, 2; blue, 27 }  ,fill opacity=1 ][line width=0.75]      (0, 0) circle [x radius= 3.35, y radius= 3.35]   ;
			\draw [color={rgb, 255:red, 208; green, 2; blue, 27 }  ,draw opacity=1 ]   (219,111.83) ;
			\draw [shift={(219,111.83)}, rotate = 0] [color={rgb, 255:red, 208; green, 2; blue, 27 }  ,draw opacity=1 ][fill={rgb, 255:red, 208; green, 2; blue, 27 }  ,fill opacity=1 ][line width=0.75]      (0, 0) circle [x radius= 3.35, y radius= 3.35]   ;
			\draw   (219,111.83) .. controls (219,93.5) and (235.42,78.65) .. (255.68,78.65) .. controls (275.93,78.65) and (292.35,93.5) .. (292.35,111.83) .. controls (292.35,130.15) and (275.93,145) .. (255.68,145) .. controls (235.42,145) and (219,130.15) .. (219,111.83) -- cycle ;
			\draw [color={rgb, 255:red, 0; green, 0; blue, 0 }  ,draw opacity=1 ]   (409,99) ;
			\draw [shift={(409,99)}, rotate = 0] [color={rgb, 255:red, 0; green, 0; blue, 0 }  ,draw opacity=1 ][fill={rgb, 255:red, 0; green, 0; blue, 0 }  ,fill opacity=1 ][line width=0.75]      (0, 0) circle [x radius= 3.35, y radius= 3.35]   ;
			\draw [color={rgb, 255:red, 0; green, 0; blue, 0 }  ,draw opacity=1 ]   (626,100) ;
			\draw [shift={(626,100)}, rotate = 0] [color={rgb, 255:red, 0; green, 0; blue, 0 }  ,draw opacity=1 ][fill={rgb, 255:red, 0; green, 0; blue, 0 }  ,fill opacity=1 ][line width=0.75]      (0, 0) circle [x radius= 3.35, y radius= 3.35]   ;
			\draw [color={rgb, 255:red, 0; green, 0; blue, 0 }  ,draw opacity=1 ]   (517.5,99.5) ;
			\draw [shift={(517.5,99.5)}, rotate = 0] [color={rgb, 255:red, 0; green, 0; blue, 0 }  ,draw opacity=1 ][fill={rgb, 255:red, 0; green, 0; blue, 0 }  ,fill opacity=1 ][line width=0.75]      (0, 0) circle [x radius= 3.35, y radius= 3.35]   ;
			\draw [color={rgb, 255:red, 208; green, 2; blue, 27 }  ,draw opacity=1 ]   (409,99) -- (626,100) ;
			\draw [color={rgb, 255:red, 0; green, 0; blue, 0 }  ,draw opacity=1 ]   (251,103) ;
			\draw [shift={(251,103)}, rotate = 0] [color={rgb, 255:red, 0; green, 0; blue, 0 }  ,draw opacity=1 ][fill={rgb, 255:red, 0; green, 0; blue, 0 }  ,fill opacity=1 ][line width=0.75]      (0, 0) circle [x radius= 3.35, y radius= 3.35]   ;
			\draw [color={rgb, 255:red, 0; green, 0; blue, 0 }  ,draw opacity=1 ]   (624,55.67) -- (626,100) ;
			
			\draw (105.66,63.07) node [anchor=north west][inner sep=0.75pt]    {$C$};
			\draw (228,47.4) node [anchor=north west][inner sep=0.75pt]    {$q_{1}$};
			\draw (624,120.4) node [anchor=north west][inner sep=0.75pt]    {$0$};
			\draw (404,66.4) node [anchor=north west][inner sep=0.75pt]    {$1$};
			\draw (513,64.4) node [anchor=north west][inner sep=0.75pt]    {$1$};

		\end{tikzpicture}
		
		\caption{Dual graph of a nodal curve with one marked point.  }
		\label{fig8}
		
	\end{figure}

 Recall that a stable spin curve $\widehat{C_R}$ over a stable curve $C$ is given by
	\begin{enumerate}
		\item a partial normalization $\widehat{{C^\nu}}_R$ of $C$ at a set of nodes $R$ of $C$.
		\item a theta characteristic  $\widehat{{L}}_R$ on $\widehat{{C^\nu}}_R$ .
	\end{enumerate}
	The existence of $\widehat{{L}_R}$ forces each irreducible component $Z$ of $\widehat{{C^\nu}}_R$  to have only an even number of nodes. 
 Otherwise the restriction of the dualizing sheaf $\omega_{\widehat{{C^\nu}}_R}$ on $Z$ would have odd degree since
	\begin{align}
		deg(\omega_{Z})=2-2g(Z)+n, \text{ where }n=|R|.
	\end{align}
Let $\Gamma_C$ be the dual graph of $C$ and denote by $N$ the set of nodes. We provide an (informal) dictionary between the algebraic operations above and corresponding operations on the dual graph (see \cref{fig25} for a visualization):
	\begin{enumerate}
		\item \emph{Blow up} at a node $r\in N$ $\leftrightarrow$ \emph{Adding} a genus $0$ vertex on the edge dual to $r$.
		\item \emph{Partial normalization} at $R\subset N$ $\leftrightarrow$ \emph{Erasing} edges dual to $R$.
	\end{enumerate}
	We denote the graph obtained by removing edges dual to $R$ by $\Gamma_C-R$.
 
	\begin{figure}[H]
		\tikzset{every picture/.style={line width=0.75pt}} 
		
		\begin{tikzpicture}[x=0.75pt,y=0.75pt,yscale=-1,xscale=1]
			
			\draw    (66,44.88) -- (113.33,108.09) ;
			\draw [shift={(89.66,76.48)}, rotate = 53.18] [color={rgb, 255:red, 0; green, 0; blue, 0 }  ][fill={rgb, 255:red, 0; green, 0; blue, 0 }  ][line width=0.75]      (0, 0) circle [x radius= 3.35, y radius= 3.35]   ;
			\draw    (88.88,83.06) -- (183.53,104.27) ;
			\draw [shift={(136.2,93.67)}, rotate = 12.63] [color={rgb, 255:red, 0; green, 0; blue, 0 }  ][fill={rgb, 255:red, 0; green, 0; blue, 0 }  ][line width=0.75]      (0, 0) circle [x radius= 3.35, y radius= 3.35]   ;
			\draw    (77.04,64.82) -- (155.93,27.06) ;
			\draw [shift={(116.48,45.94)}, rotate = 334.42] [color={rgb, 255:red, 0; green, 0; blue, 0 }  ][fill={rgb, 255:red, 0; green, 0; blue, 0 }  ][line width=0.75]      (0, 0) circle [x radius= 3.35, y radius= 3.35]   ;
			\draw    (114.12,14.97) -- (193,78.61) ;
			\draw [shift={(153.56,46.79)}, rotate = 38.89] [color={rgb, 255:red, 0; green, 0; blue, 0 }  ][fill={rgb, 255:red, 0; green, 0; blue, 0 }  ][line width=0.75]      (0, 0) circle [x radius= 3.35, y radius= 3.35]   ;
			\draw    (178.01,32.36) -- (148.04,119.33) ;
			\draw [shift={(163.02,75.85)}, rotate = 109.02] [color={rgb, 255:red, 0; green, 0; blue, 0 }  ][fill={rgb, 255:red, 0; green, 0; blue, 0 }  ][line width=0.75]      (0, 0) circle [x radius= 3.35, y radius= 3.35]   ;
			\draw    (266,54.49) -- (312.96,127.37) ;
			\draw [shift={(289.48,90.93)}, rotate = 57.21] [color={rgb, 255:red, 0; green, 0; blue, 0 }  ][fill={rgb, 255:red, 0; green, 0; blue, 0 }  ][line width=0.75]      (0, 0) circle [x radius= 3.35, y radius= 3.35]   ;
			\draw    (288.7,98.51) -- (382.61,122.97) ;
			\draw [shift={(335.65,110.74)}, rotate = 14.6] [color={rgb, 255:red, 0; green, 0; blue, 0 }  ][fill={rgb, 255:red, 0; green, 0; blue, 0 }  ][line width=0.75]      (0, 0) circle [x radius= 3.35, y radius= 3.35]   ;
			\draw    (276.96,77.48) -- (300,11.33) ;
			\draw [shift={(288.48,44.4)}, rotate = 289.21] [color={rgb, 255:red, 0; green, 0; blue, 0 }  ][fill={rgb, 255:red, 0; green, 0; blue, 0 }  ][line width=0.75]      (0, 0) circle [x radius= 3.35, y radius= 3.35]   ;
			\draw    (366,19.33) -- (392,93.37) ;
			\draw [shift={(379,56.35)}, rotate = 70.65] [color={rgb, 255:red, 0; green, 0; blue, 0 }  ][fill={rgb, 255:red, 0; green, 0; blue, 0 }  ][line width=0.75]      (0, 0) circle [x radius= 3.35, y radius= 3.35]   ;
			\draw    (377.13,40.06) -- (347.39,140.33) ;
			\draw [shift={(362.26,90.19)}, rotate = 106.52] [color={rgb, 255:red, 0; green, 0; blue, 0 }  ][fill={rgb, 255:red, 0; green, 0; blue, 0 }  ][line width=0.75]      (0, 0) circle [x radius= 3.35, y radius= 3.35]   ;
			\draw    (275,20) -- (383,26.33) ;
			\draw    (483,52.48) -- (529.96,125.36) ;
			\draw [shift={(506.48,88.92)}, rotate = 57.21] [color={rgb, 255:red, 0; green, 0; blue, 0 }  ][fill={rgb, 255:red, 0; green, 0; blue, 0 }  ][line width=0.75]      (0, 0) circle [x radius= 3.35, y radius= 3.35]   ;
			\draw    (499.7,98.5) -- (593.61,122.96) ;
			\draw [shift={(546.65,110.73)}, rotate = 14.6] [color={rgb, 255:red, 0; green, 0; blue, 0 }  ][fill={rgb, 255:red, 0; green, 0; blue, 0 }  ][line width=0.75]      (0, 0) circle [x radius= 3.35, y radius= 3.35]   ;
			\draw    (487.96,77.47) -- (511,11.33) ;
			\draw [shift={(499.48,44.4)}, rotate = 289.21] [color={rgb, 255:red, 0; green, 0; blue, 0 }  ][fill={rgb, 255:red, 0; green, 0; blue, 0 }  ][line width=0.75]      (0, 0) circle [x radius= 3.35, y radius= 3.35]   ;
			\draw    (577,19.33) -- (603,93.37) ;
			\draw [shift={(590,56.35)}, rotate = 70.65] [color={rgb, 255:red, 0; green, 0; blue, 0 }  ][fill={rgb, 255:red, 0; green, 0; blue, 0 }  ][line width=0.75]      (0, 0) circle [x radius= 3.35, y radius= 3.35]   ;
			\draw    (588.13,40.05) -- (558.39,140.33) ;
			\draw [shift={(573.26,90.19)}, rotate = 106.52] [color={rgb, 255:red, 0; green, 0; blue, 0 }  ][fill={rgb, 255:red, 0; green, 0; blue, 0 }  ][line width=0.75]      (0, 0) circle [x radius= 3.35, y radius= 3.35]   ;
			\draw   (91.67,210.67) .. controls (91.67,189.31) and (108.98,172) .. (130.33,172) .. controls (151.69,172) and (169,189.31) .. (169,210.67) .. controls (169,232.02) and (151.69,249.33) .. (130.33,249.33) .. controls (108.98,249.33) and (91.67,232.02) .. (91.67,210.67) -- cycle ;
			\draw    (76,146.33) -- (110,178) ;
			\draw [shift={(110,178)}, rotate = 42.96] [color={rgb, 255:red, 0; green, 0; blue, 0 }  ][fill={rgb, 255:red, 0; green, 0; blue, 0 }  ][line width=0.75]      (0, 0) circle [x radius= 3.35, y radius= 3.35]   ;
			\draw    (52,233.33) -- (91.67,210.67) ;
			\draw [shift={(91.67,210.67)}, rotate = 330.26] [color={rgb, 255:red, 0; green, 0; blue, 0 }  ][fill={rgb, 255:red, 0; green, 0; blue, 0 }  ][line width=0.75]      (0, 0) circle [x radius= 3.35, y radius= 3.35]   ;
			\draw    (79,270.33) -- (103,240) ;
			\draw [shift={(103,240)}, rotate = 308.35] [color={rgb, 255:red, 0; green, 0; blue, 0 }  ][fill={rgb, 255:red, 0; green, 0; blue, 0 }  ][line width=0.75]      (0, 0) circle [x radius= 3.35, y radius= 3.35]   ;
			\draw    (198,164.33) -- (164,191) ;
			\draw [shift={(164,191)}, rotate = 141.89] [color={rgb, 255:red, 0; green, 0; blue, 0 }  ][fill={rgb, 255:red, 0; green, 0; blue, 0 }  ][line width=0.75]      (0, 0) circle [x radius= 3.35, y radius= 3.35]   ;
			\draw    (187,275.33) -- (160,236) ;
			\draw [shift={(160,236)}, rotate = 235.53] [color={rgb, 255:red, 0; green, 0; blue, 0 }  ][fill={rgb, 255:red, 0; green, 0; blue, 0 }  ][line width=0.75]      (0, 0) circle [x radius= 3.35, y radius= 3.35]   ;
			\draw   (301.67,213.33) .. controls (301.67,191.98) and (318.98,174.67) .. (340.33,174.67) .. controls (361.69,174.67) and (379,191.98) .. (379,213.33) .. controls (379,234.69) and (361.69,252) .. (340.33,252) .. controls (318.98,252) and (301.67,234.69) .. (301.67,213.33) -- cycle ;
			\draw    (286,149) -- (320,180.67) ;
			\draw [shift={(320,180.67)}, rotate = 42.96] [color={rgb, 255:red, 0; green, 0; blue, 0 }  ][fill={rgb, 255:red, 0; green, 0; blue, 0 }  ][line width=0.75]      (0, 0) circle [x radius= 3.35, y radius= 3.35]   ;
			\draw    (262,236) -- (301.67,213.33) ;
			\draw [shift={(301.67,213.33)}, rotate = 330.26] [color={rgb, 255:red, 0; green, 0; blue, 0 }  ][fill={rgb, 255:red, 0; green, 0; blue, 0 }  ][line width=0.75]      (0, 0) circle [x radius= 3.35, y radius= 3.35]   ;
			\draw    (289,273) -- (313,242.67) ;
			\draw [shift={(313,242.67)}, rotate = 308.35] [color={rgb, 255:red, 0; green, 0; blue, 0 }  ][fill={rgb, 255:red, 0; green, 0; blue, 0 }  ][line width=0.75]      (0, 0) circle [x radius= 3.35, y radius= 3.35]   ;
			\draw    (408,167) -- (374,193.67) ;
			\draw [shift={(374,193.67)}, rotate = 141.89] [color={rgb, 255:red, 0; green, 0; blue, 0 }  ][fill={rgb, 255:red, 0; green, 0; blue, 0 }  ][line width=0.75]      (0, 0) circle [x radius= 3.35, y radius= 3.35]   ;
			\draw    (397,278) -- (370,238.67) ;
			\draw [shift={(370,238.67)}, rotate = 235.53] [color={rgb, 255:red, 0; green, 0; blue, 0 }  ][fill={rgb, 255:red, 0; green, 0; blue, 0 }  ][line width=0.75]      (0, 0) circle [x radius= 3.35, y radius= 3.35]   ;
			\draw    (498.34,141.53) -- (532.34,173.2) ;
			\draw [shift={(532.34,173.2)}, rotate = 42.96] [color={rgb, 255:red, 0; green, 0; blue, 0 }  ][fill={rgb, 255:red, 0; green, 0; blue, 0 }  ][line width=0.75]      (0, 0) circle [x radius= 3.35, y radius= 3.35]   ;
			\draw    (481,237) -- (501.1,225.51) -- (520.67,214.33) ;
			\draw [shift={(520.67,214.33)}, rotate = 330.26] [color={rgb, 255:red, 0; green, 0; blue, 0 }  ][fill={rgb, 255:red, 0; green, 0; blue, 0 }  ][line width=0.75]      (0, 0) circle [x radius= 3.35, y radius= 3.35]   ;
			\draw    (519,273) -- (543,242.67) ;
			\draw [shift={(543,242.67)}, rotate = 308.35] [color={rgb, 255:red, 0; green, 0; blue, 0 }  ][fill={rgb, 255:red, 0; green, 0; blue, 0 }  ][line width=0.75]      (0, 0) circle [x radius= 3.35, y radius= 3.35]   ;
			\draw    (637.39,177.07) -- (603.39,203.74) ;
			\draw [shift={(603.39,203.74)}, rotate = 141.89] [color={rgb, 255:red, 0; green, 0; blue, 0 }  ][fill={rgb, 255:red, 0; green, 0; blue, 0 }  ][line width=0.75]      (0, 0) circle [x radius= 3.35, y radius= 3.35]   ;
			\draw    (614,276) -- (587,236.67) ;
			\draw [shift={(587,236.67)}, rotate = 235.53] [color={rgb, 255:red, 0; green, 0; blue, 0 }  ][fill={rgb, 255:red, 0; green, 0; blue, 0 }  ][line width=0.75]      (0, 0) circle [x radius= 3.35, y radius= 3.35]   ;
			\draw    (348,174.67) ;
			\draw [shift={(348,174.67)}, rotate = 0] [color={rgb, 255:red, 0; green, 0; blue, 0 }  ][fill={rgb, 255:red, 0; green, 0; blue, 0 }  ][line width=0.75]      (0, 0) circle [x radius= 3.35, y radius= 3.35]   ;
			\draw  [draw opacity=0] (603.39,203.74) .. controls (602.96,208.36) and (601.8,213.01) .. (599.85,217.54) .. controls (590.39,239.56) and (565.9,250.18) .. (545.17,241.27) .. controls (524.43,232.36) and (515.29,207.28) .. (524.76,185.26) .. controls (526.7,180.74) and (529.28,176.69) .. (532.34,173.2) -- (562.31,201.4) -- cycle ; \draw   (603.39,203.74) .. controls (602.96,208.36) and (601.8,213.01) .. (599.85,217.54) .. controls (590.39,239.56) and (565.9,250.18) .. (545.17,241.27) .. controls (524.43,232.36) and (515.29,207.28) .. (524.76,185.26) .. controls (526.7,180.74) and (529.28,176.69) .. (532.34,173.2) ;
			
			\draw (134.57,12.92) node [anchor=north west][inner sep=0.75pt]    {$r$};
			\draw (329,5.41) node [anchor=north west][inner sep=0.75pt]    {$E_{r}$};
			\draw (35,12.4) node [anchor=north west][inner sep=0.75pt]    {$C$};
			\draw (229,9.4) node [anchor=north west][inner sep=0.75pt]    {$\widehat{C_{R}}$};
			\draw (442,13.4) node [anchor=north west][inner sep=0.75pt]    {$\widehat{C_{R}^{\nu }}$};
			\draw (37,153.4) node [anchor=north west][inner sep=0.75pt]    {$\Gamma _{C}$};
			\draw (247,156.07) node [anchor=north west][inner sep=0.75pt]    {$\Gamma _{\widehat{C_{R}}}$};
			\draw (458,163.07) node [anchor=north west][inner sep=0.75pt]    {$\Gamma _{C} -R$};
			\draw (355.57,153.4) node [anchor=north west][inner sep=0.75pt]    {$v_{r}$};

		\end{tikzpicture}
		
		\caption{Illustration of the algebraic operations and their graph-theoretic counterpart where $R:=\{r\}$}
  		\label{fig25}
  \end{figure}
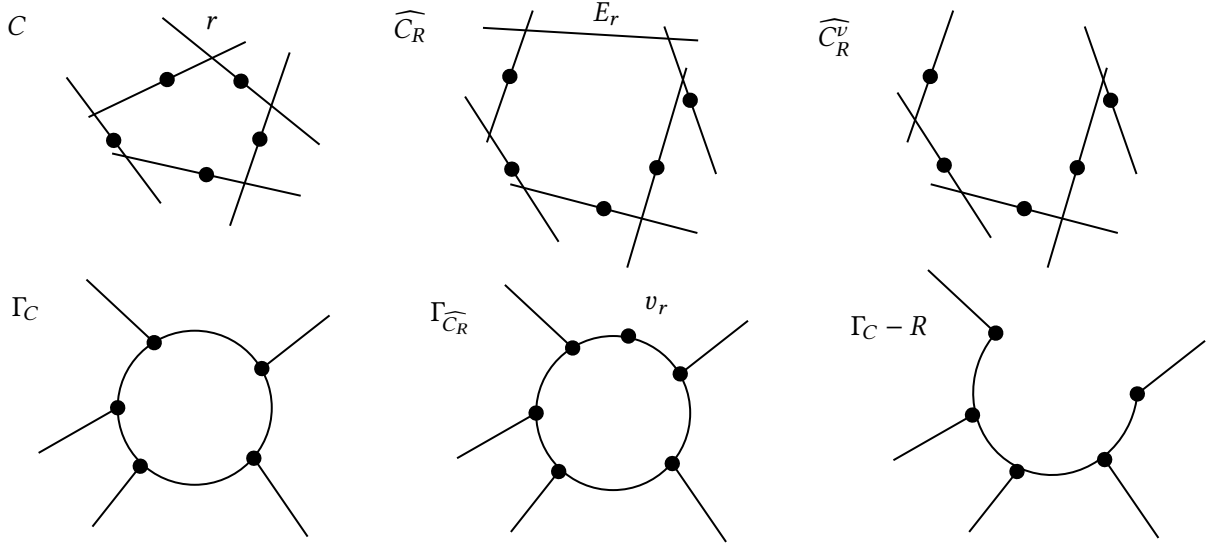

Recall the map $\phi\colon\overline{S}_{g,n}\to\overline{M}_{g,n}$. For a given curve $C\in\overline{M}_{g,n}$ its preimage $\phi^{-1}(C)$ consists of quasistable curves $\widehat{C_R}$ such that the dual graph  $\Gamma_C-R$ of $\widehat{{C^\nu}}_R$ is cyclic, i.e. $\Gamma_C-R$ is an element of the cycle space of $\Gamma_C-R$ denoted by $\mathcal{C}_{\Gamma_C-R}$ (\cite{caporaso2020tropicalizing}, section 1.4). Indeed, requiring the number of nodes that lie on each connected component of  $\widehat{{C^\nu}}_R$ to be even implies that  each vertex of  $\Gamma_C-R$ has an even number of adjacent edges. This is the case if and only if  $\Gamma_C-R$ is cyclic. Counting the number of stable spin curves in $\phi^{-1}(C)$ may be expressed in graph-theoretic language.

\begin{proposition}[{\cite[Proposition 3.5]{Farkas_thetacharacteristics}}]
    		Let $C\in \overline{\mathcal{M}_g}$ and $b:= b_1(\Gamma_C)$ be the Betti number of the dual graph. Then
		the number of components of the zero-dimensional scheme $\phi^{-1}(C)$ is equal to
		\begin{align}
			2^{2g(\Gamma_C)-2b}\sum_{P}2^{b_1(P)}
		\end{align}
  where the sum is over all cyclic subgraphs $P$ of $\Gamma_C$. A component corresponding to a cycle $P$ of $\Gamma_{C}$ appears with multiplicity $2^{b-b_1(P)}$.
\end{proposition}

Thus, we get a total of $2^{2g(\Gamma_C)}$ spin structures counted with multiplicity.

Since connected components of $\widehat{ C^\nu_R}$ are in natural bijection with the vertices of the graph $G/P$ obtained after contracting the edges of the cycle $P:=E(\Gamma_{C})\setminus R$ we write (following \cite{caporaso2020tropicalizing})  the decomposition into connected components of $\widehat{{C^\nu_R}}$  as $\widehat{ C^\nu_R}=\bigcup_{v\in V(G/P)}Z_v$.
	
\begin{definition}[{\cite[Definition 3.2.1.]{caporaso2020tropicalizing}}]\label{def_dualspingraph}
		Let $(\widehat{ C_R},\widehat{ L_R})$ be a stable $n$-pointed spin curve, where $C$ the stable model of $\widehat{ C_R}$ and $\widehat{ C_R}$ the quasistable curve associated to a set of nodes $R$ of $C$. Write $\widehat{C_R}=\widehat{{C^\nu_R}}\cup \bigcup_{p\in R}E_r$ and $\widehat{ C^\nu_R}=\bigcup_{v\in V(G/P)}Z_v$ the decomposition into connected components of $\widehat{{C^\nu_R}}$.
		The \emph{dual spin graph} of $(\widehat{ C_R},\widehat{L_R})$  is the spin graph $(\Gamma_C, P, s)$ defined as follows:
		\begin{itemize}
			\item $\Gamma_C$ is the dual graph of $C$.
			\item $P=E(\Gamma_C)\setminus R_C$, where $R_C\subset E(\Gamma_C)$ corresponds to the set of nodes $R$.
			\item $s(v)$ is the parity of $h^0(Z_v,{\widehat{L_R}}_{|Z_v})$ for all $v\in V(G/P)$. 
		\end{itemize}
\end{definition}

We are finally ready to prove \cref{thm-correspondencetheorem1}. The idea of the proof is simple: We iterate the degeneration procedure described above until the base curve $D$ is dual to a given tropical base curve $B$. Then, the dual picture of a spin cover becomes a tropical spin cover.

 For the discrete data in \cref{thm-correspondencetheorem1}, i.e. $d,h,k$ and odd partitions $\mu^1,...,\mu^k$, fix 
		\begin{itemize}
			\item a spin curve $(D,L)$ of genus $h$ with parity $p:=h^0(L)\mod 2=0\mod 2$.
			\item a collection of distinct points $V:=\{q^1, ..., q^k\}\subset D$ , the prescribed branch points.
		\end{itemize}
		Recall that the spin curve degeneration in \cref{sec-degproof} relates the count of spin Hurwitz covers of $(D,N)$ to the count of degenerate spin Hurwitz covers of a nodal curve $D_0=D_2\cup E \cup D_1$ that was fixed in advance. In fact,  \cite[Theorem 5.1.]{MR3152971} produces (provided we choose a ramification behaviour, i.e. an odd partition $m$ of $d$, over the nodes) for any degenerate map $f_0=(f_2,f_e,f_1)\in \mathcal{M}_{m,0}$, a family of stable spin curves $(\mathcal{D},\mathcal{N})$ and ``the right number'' of families of covering spin curves $(\mathcal{C}_\zeta,\mathcal{L}_\zeta)$ (with parameter $s$) together with holomorphic maps $\mathcal{F}_\zeta: \mathcal{C}_\zeta\rightarrow \mathcal{D}$, such that the map between special fibres is given by $f_0$ and $p(f _{\zeta,s})=p(f_2)+p(f_1)$ , where $f _{\zeta,s}:={\mathcal{F}_\zeta}_{|\mathcal{C}_{\zeta,s}}$. 
   Twisting the construction around we can think of taking $f :C\rightarrow D\in {\mathcal{M}}^V_{\chi, \mu^1,...,,\mu^k}$ in the relative moduli space of $V-$regular maps defined above together with the theta characteristic $L_f$ as taking its corresponding degenerate spin Hurwitz cover $f_0:=(f_2,f_e,f_1) :\thickspace C_0=C_2 \cup \bigcup\limits_{i=1}^{l}E_i\cup C_1 \rightarrow D_0$ with spin structure on $C_0$ that restricted to the smooth curve $C_i$ is given by $L_{f_i}$ for $i=1,2$ and  $p(f)=p(f_2)+p(f_1)$, where $D_0=D_1\cup E\cup D_2$. Alternatively, when we degenerate as in \cref{rem-casetwo}, we consider a map $(\overline{f},f_e)\colon C_0\overline{C}\cup\bigcup_{i=1}^{l}E_i\to D_0$ with $D_0=\overline{D}\cup E$. \\
   We fix a tropical spin curve $(B,[0],s_B)$ of genus $g$. Employing both types of degenerations, we may clearly iteratively degenerate a given target curve $D$ to a nodal curve $D_0=\bigcup_{i=1}^{|V(B)|}D_i\cup\bigcup_{i=1}^{|E(B)\backslash E^{\infty}(B)|}$ whose dual tropical spin curve is $(B,[0],s_B)$. Now, given a spin cover $f\in M_{\chi,\mu^1,\dots,\mu^k}^V$ with $f\colon C\to D$. Degenerating $f$ along the degeneration of $D$ to $D_0$, we obtain a map $f_0\colon C_0=\bigcup_{i=1}^{s}C_i\cup\bigcup_{i=1}^t \tilde{E}_i\to D_0$ which is considered glued along the nodes. Let $\underline{m}=(m_1,\dots,m_{|E(B)\backslash E^\infty(B)|})$ denote the tuple of partitions of the ramification profiles of $f_0$ over the nodes of $D_0$. In analogy to the sets $\mathcal{M}_{m,0}$, we thus denote the set of the glued maps $f_0$ by $\mathcal{M}_{\underline{m},0}$. As in \cref{sec-degproof}, we see that
   \begin{equation}
       H^{\bullet,p}_{h\to g}(\mu^1,\dots,\mu^n)=\frac{1}{(d!)^{|V(B)|+|E(B)\backslash E^\infty(B)|}\prod_i |\mathrm{Aut}(\mu^i)|}\sum_{\underline{m}}\prod_{j=1}^{|E(B)\backslash E^{\infty}(B)|}|m_j|^2\sum_{f_0\in\mathcal{M}_{\underline{m},0}}(-1)^{p(f_0)}
   \end{equation}
   where the first sum runs over all $|E(B)\backslash E^\infty(B)|$-tuples of partitions $(m_i)_i$ of $d$ and $p(f_0)=\sum_{C_i\subset C_0}p({f_0}_{|C_i})$.\\
   Let $(\Gamma,[0],s)$ the dual tropical spin curve of $C_0$. The map $f_0$ induces a morphism of graphs $\pi_{f_0}\colon \Gamma\to B$, where the vertex corresponding to $C_i$ is mapped to the vertex corresponding to $f_0(C_i)$. Let $e$ be an edge of $\Gamma$ corresponding to an exceptional component $E$, then we define the weight $\omega(e)$ of $e$ to be the order of $f_0$ at the nodes of $E$. Adding this structure, we obtain that $\pi_{f_0}$ is a tropical spin cover. It remains to prove that
   \begin{equation}
   \label{equ-multiplicityinproof1}
       \mathrm{mult}(\pi,s,s_B)=\frac{1}{(d!)^{|V(B)|+|E(B)\backslash E^\infty(B)|}\prod_i |\mathrm{Aut}(\mu^i)|}\sum_{\underline{m}}\prod_{e\in E(B)\backslash E^\infty(B)}|m_e|^2\sum_{f_0}(-1)^{p(f_0)},
   \end{equation}
   where the second sum is over all $f_0\in\mathcal{M}_{\underline{m}}$ that induce $(\pi,s_B,s)$. In other words, we need to show that the multiplicity of the tropical spin cover is the sum over all contributions of spin covers that degenerate to $(\pi,s,s_B)$. However, this follows from iterating the computation in step 7 in \cref{sec-degproof}. To make this precise, we fix a tropical spin cover $(\pi,s_B,s)$ with $\pi\colon\Gamma\to B$. Let $v$ be a vertex of $B$ with set of adjacent edges and ends $e^v$. For each $e\in e^v$, we denote by $m_e$ the partition of the degree of $\pi$ which is the tuple of weights of edges in $\pi^{-1}(e)$. Moreover, we set $\underline{m}_{e^v}=(m_e)_{e\in e^v}$.  Then, in analogy with the sets $\mathcal{M}^i_m$ we define $\mathcal{M}_{\underline{m}_{e^v}}$ as the set of maps $f\colon C\to C_v$, where
   \begin{itemize}
       \item the genus of $C_v$ is $g(v)$,
       \item $f$ has ramification profiles $m_e$ for $e\in e_v$ over $|e_v|$ fixed points and ramification profile $(1)^d$ over an additional point,
       \item the ramification points of $f$ are labelled in accordance with the partitions $m_e$
       \item the components of $C$ are in bijection with vertices in $\pi^{-1}(v)$, 
       \item for $w\in\pi^{-1}(v)$, the corresponding component $C_w$ of $C$ has genus $g(w)$,
       \item the ramification points on $C_w$ are in bijection with the edges adjacent to $w$,
       \item the degree of $f$ at a ramification point is the weight of the corresponding edge,
       \item $p(f_{|C_w})=s(w)$.
       \end{itemize}
   Similarly, we define for any edge $e$ in $B$ with $m_e$ as above the set $\mathcal{M}_{m_e}$ of maps $f\colon E\to E_v$, where
   \begin{itemize}
       \item $E_v$ is isomorphic to $\mathbb{P}^1$,
       \item each component of $E$ is isomorphic to $\mathbb{P}^1$,
       \item the components of $E$ are in bijection with the edges in $\pi^{-1}(e)$,
       \item for $\tilde{e}\pi^{-1}(e)$ an edge, f has three ramification points on the corresponding component $E_{\tilde{e}}$, one of local degree $1$ and two of local degree $\omega(\tilde{e})$,
       \item $f$ has ramification profile $m_e$ over two fixed points and ramification profile $(1)^d$ over an additional point.
   \end{itemize}

   Then, we define
   \begin{equation}
       \mathcal{P}_{\pi}=\prod_{v\in V(B)}\mathcal{M}_{\underline{m}_{e^v}}\times\prod_{e\in E(B)\backslash E^\infty(B)}\mathcal{M}_{m_e}.
   \end{equation}
   We will denote a tuple in $\mathcal{P}$ by $\underline{f}$. By the same argument as in step 7 of \cref{sec-degproof}, we obtain

   \begin{equation}
   \label{equ-mult1}
       \sum_{f_0}(-1)^{p(f_0)}=\frac{1}{\prod_{e\in E(B)\backslash E^\infty(B)}|\mathrm{Aut}(m_e)|^2}\sum_{\underline{f}\in  \mathcal{P}_{\pi}}(-1)^{p(\underline{f})},
   \end{equation}

   where the sum on the left is over all maps $f_0$ that induce $(\pi,s,s_B)$ and $p(\underline{f})$ is the sum of the parities of all elements of the tuple $\underline{f}$. Since by \cite[Lemma 2.2]{MR3152971}, we have
   \begin{equation}
       \mathcal{M}_{m_e}=\frac{d!\mathrm{Aut}(m_e)}{|m_e|}
   \end{equation}
    and because any spin structure on $\mathbb{P}^1$ is even, we obtain that \cref{equ-mult1} equals
    \begin{equation}
        \frac{(d!)^{|E(B)\backslash E^\infty(B)|}}{\prod_{e\in E(B)\backslash E^\infty(B)}|\mathrm{Aut}(m_e)||m_e|}\sum_{\underline{f}\in  \prod_{v\in V(B)}\mathcal{M}_{\underline{m}_{e^v}}}(-1)^{p(\underline{f})}.
    \end{equation}
    For a tuple $\underline{f}\in\prod_{v\in V(B)}\mathcal{M}_{\underline{m}_{e^v}}$, we denote the corresponding tuple for which we drop the label of ramification points by $\underline{f}^\circ$. We denote the set of unlabelled such tuples by $\prod_{v\in V(B)}\mathcal{M}_{\underline{m}_{e^v}}^\circ$. Again analogously to the case considered in step 7 in \cref{sec-degproof}, we obtain
    \begin{align}
        &\sum_{f_0}(-1)^{p(f_0)}= 
        \\&\frac{(d!)^{|V(B)|+|E(B)\backslash E^\infty(B)|}\prod_{e\in E(B)\backslash E^\infty(B)}|\mathrm{Aut}(m_e)|\prod_{i=1}^n|\mathrm{Aut}(\mu^i)|}{\prod_{e\in E(B)\backslash E^\infty(B)}|m_e|}\sum_{\underline{f}^\circ\in  \prod_{v\in V(B)}\mathcal{M}_{\underline{m}_{e^v}}^\circ}\frac{(-1)^{p(\underline{f}^\circ)}}{|\mathrm{Aut}(\underline{f}^\circ)|}.
    \end{align}
    
   Therefore, we obtain

   \begin{equation}
        \frac{\prod_{e\in E(B)\backslash E^\infty(B)}|m_e|^2}{(d!)^{|E(B)\backslash E^\infty(B)|}}\sum_{f_0}(-1)^{p(f_0)}=\prod_{e\in E(B)\backslash E^\infty(B)}|\mathrm{Aut}(m_e)||m_e|\sum_{\underline{f}^\circ\in  \prod_{v\in V(B)}\mathcal{M}_{\underline{m}_{e^v}}^\circ}\frac{(-1)^{p(\underline{f}^\circ)}}{{|\mathrm{Aut}(\underline{f}^\circ)|}}.
   \end{equation}

   In order to derive \cref{equ-multiplicityinproof1}, we now simply observe that
   \begin{equation}
          \sum_{\underline{f}^\circ\in  \prod_{v\in V(B)}\mathcal{M}_{\underline{m}_{e^v}}^\circ}(-1)^{p(\underline{f}^\circ)}.
        =\prod_{v\in V(B)}\sum_{f_v\in \mathcal{M}_{\underline{m}_{e^v}}^\circ}\frac{(-1)^{p(f_v)}}{{|\mathrm{Aut}(\underline{f}^\circ)|}}
   \end{equation}
   and that
   \begin{equation}
       \prod_{v\in V(B)}\sum_{f_v\in \mathcal{M}_{\underline{m}_{e^v}}^\circ}(-1)^{p(f_v)}
   \end{equation}
   agrees with the local Hurwitz numbers defined in \cref{def:localhurwitz} and that $\prod_{e\in E(B)\backslash E^\infty(B)}|m_e|=\prod_{e\in  E(B)\backslash E^\infty(B)}\omega(e)$. Substituting into \cref{equ-multiplicityinproof1} completes the proof.

\subsection{Application: Closed formulae}
In \cite{lee2013degree,lee2012recursion}, closed formulae for degree $3$ and $4$ spin Hurwitz numbers were derived in genus $0$ and $1$. More precisely the following were proved.

\begin{theorem}[{\cite{lee2013degree,lee2012recursion}}]
\label{thm-closedformleeparker}
   We fix non--negative integers $g\ge h$ and a positive integer $k$. For $g,h,k$ related by the Riemann--Hurwitz formula, we have the following:
    \begin{align}
    &H_{g\to h}^{\pm}((3)^k)=3^{2h-2}((-1)^k2^{k-1}\pm1)\quad\textrm{and}\quad\quad H_{g\to h}^{\pm}((31)^k)=(3!)^{2h-2}2^k(\pm 2^{h+k-1}+(-1)^k).
    \end{align}
\end{theorem}

In this section, we utilise the tropical correspondence theorem to make progress towards a closed formula for genus $0$ target, but arbitrary degree $d$. More precisely, we prove the following.

\begin{theorem}
    Fix a positive integer $d$. Then, we obtain

    \begin{enumerate}
        \item $H_{g\to 0}^{+}((31^{d-3})^0))=\frac{1}{d!}$ and $H_{g\to 0}^{+}((31^{d-3})^1))=0$ for all $d$
        \item $H_{g\to 0}^{+}((31^{d-3})^2))=\frac{1}{3(d-3)!}$ and $H_{g\to 0}^{+}((31^{d-3})^2))=\frac{3d-10}{3(d-3)!}$ for $d\ge 4$
        \item $H_{g\to 0}^{+}((31^{d-3})^4))=\frac{2d^3+57d^2-446d+780}{9(d-3)!}$ 
    \end{enumerate}
\end{theorem}

\begin{proof}
    We treat the cases $H_{g\to 0}^{+}((31^{d-3})^k))$ individually. The vertex multiplicities are quickly computed using \cref{thm-closedformleeparker} and via factorisations in the symmetric group.
    
    For $k=1$, the set of relevant tropical spin covers is empty. \cref{fig42} shows the computation for $k=0$ and $k=2$. 
		\begin{figure}[H]
			\tikzset{every picture/.style={line width=0.75pt}} 
			
			\begin{tikzpicture}[x=0.75pt,y=0.75pt,yscale=-1,xscale=1]
				
				\draw    (434,111) -- (577,111) ;
				\draw    (434,143) -- (577,143) ;
				\draw    (428,50.33) -- (575,52) ;
				\draw    (89,56) -- (232,56) ;
				\draw    (89,88) -- (232,88) ;
				
				\draw (418,170.4) node [anchor=north west][inner sep=0.75pt]    {$mult( \pi ,s) =\frac{1}{3( d-3) !}$};
				\draw (371,110.4) node [anchor=north west][inner sep=0.75pt]    {$d-3\Bigl\{$};
				\draw (390,42.4) node [anchor=north west][inner sep=0.75pt]    {$3$};
				\draw (597.77,40.4) node [anchor=north west][inner sep=0.75pt]    {$3$};
				\draw (494.78,83.4) node [anchor=north west][inner sep=0.75pt]    {$1$};
				\draw (493.78,120.4) node [anchor=north west][inner sep=0.75pt]    {$1$};
				\draw (423,216.4) node [anchor=north west][inner sep=0.75pt]    {$p( \pi ,s) =\textcolor[rgb]{0.72,0.91,0.53}{0}$};
				\draw (89,174.4) node [anchor=north west][inner sep=0.75pt]    {$mult( \pi ,s) =\frac{1}{d!}$};
				\draw (47,54.4) node [anchor=north west][inner sep=0.75pt]    {$d\Bigl\{$};
				\draw (149.78,28.4) node [anchor=north west][inner sep=0.75pt]    {$1$};
				\draw (148.78,65.4) node [anchor=north west][inner sep=0.75pt]    {$1$};
				\draw (94,220.4) node [anchor=north west][inner sep=0.75pt]    {$p( \pi ,s) =\textcolor[rgb]{0.72,0.91,0.53}{0}$};
				\draw (18,15.4) node [anchor=north west][inner sep=0.75pt]    {$k=0:$};
				\draw (321,16.4) node [anchor=north west][inner sep=0.75pt]    {$k=2:$};

			\end{tikzpicture}
			\caption{Computation of $\mathbb{T}H^{(0,+)}_{{(31...1)}^0}$ and $\mathbb{T}H^{(0,+)}_{{(31...1)}^2}$.}
			\label{fig42}
		\end{figure}
		If $d\geqslant4$, the two tropical spin curves (together with their multiplicity) that contribute to  $\mathbb{T}H^{(0,+)}_{{(31...1)}^3}$ are shown in \cref{sec-degproof}. We conclude $\mathbb{T}H^{(0,+)}_{{(31...1)}^3}=\frac{1}{(d-4)!}+\frac{-1}{3(d-3)!}=\frac{3d-10}{3(d-3)!}$.
		\begin{figure}
			\tikzset{every picture/.style={line width=0.75pt}} 
			
			\begin{tikzpicture}[x=0.75pt,y=0.75pt,yscale=-1,xscale=1]
				
				\draw  [draw opacity=0] (201.17,82.94) .. controls (199.13,82.98) and (197.08,83) .. (195,83) .. controls (153.58,83) and (120,75.39) .. (120,66) .. controls (120,56.61) and (153.58,49) .. (195,49) .. controls (195.49,49) and (195.97,49) .. (196.46,49) -- (195,66) -- cycle ; \draw   (201.17,82.94) .. controls (199.13,82.98) and (197.08,83) .. (195,83) .. controls (153.58,83) and (120,75.39) .. (120,66) .. controls (120,56.61) and (153.58,49) .. (195,49) .. controls (195.49,49) and (195.97,49) .. (196.46,49) ;  
				\draw  [draw opacity=0] (50.02,82.93) .. controls (51.99,82.98) and (53.98,83) .. (56,83) .. controls (91.9,83) and (121,75.61) .. (121,66.5) .. controls (121,57.39) and (91.9,50) .. (56,50) .. controls (53.98,50) and (51.99,50.02) .. (50.02,50.07) -- (56,66.5) -- cycle ; \draw   (50.02,82.93) .. controls (51.99,82.98) and (53.98,83) .. (56,83) .. controls (91.9,83) and (121,75.61) .. (121,66.5) .. controls (121,57.39) and (91.9,50) .. (56,50) .. controls (53.98,50) and (51.99,50.02) .. (50.02,50.07) ;  
				\draw    (52,119) -- (195,119) ;
				\draw    (52,151) -- (195,151) ;
				\draw    (121,67) ;
				\draw [shift={(121,67)}, rotate = 0] [color={rgb, 255:red, 0; green, 0; blue, 0 }  ][fill={rgb, 255:red, 0; green, 0; blue, 0 }  ][line width=0.75]      (0, 0) circle [x radius= 3.35, y radius= 3.35]   ;
				\draw    (434,120) -- (577,120) ;
				\draw    (434,152) -- (577,152) ;
				\draw    (421,61) -- (575,61) ;
				\draw    (498,61) ;
				\draw [shift={(498,61)}, rotate = 0] [color={rgb, 255:red, 0; green, 0; blue, 0 }  ][fill={rgb, 255:red, 0; green, 0; blue, 0 }  ][line width=0.75]      (0, 0) circle [x radius= 3.35, y radius= 3.35]   ;
				
				\draw (492.72,30.19) node [anchor=north west][inner sep=0.75pt]    {$\textcolor[rgb]{0.72,0.91,0.53}{1}$};
				\draw (116,31.4) node [anchor=north west][inner sep=0.75pt]    {$\textcolor[rgb]{0.72,0.91,0.53}{0}$};
				\draw (49.77,23.19) node [anchor=north west][inner sep=0.75pt]    {$3$};
				\draw (26.28,80.36) node [anchor=north west][inner sep=0.75pt]    {$1$};
				\draw (204,17.4) node [anchor=north west][inner sep=0.75pt]    {$3$};
				\draw (219.78,82.4) node [anchor=north west][inner sep=0.75pt]    {$1$};
				\draw (1,117.4) node [anchor=north west][inner sep=0.75pt]    {$d-4\Bigl\{$};
				\draw (34,166.4) node [anchor=north west][inner sep=0.75pt]    {$mult( \pi ,s) =\frac{1}{( d-4) !}$};
				\draw (418,170.4) node [anchor=north west][inner sep=0.75pt]    {$mult( \pi ,s) =\frac{1}{3( d-3) !}$};
				\draw (35,212.4) node [anchor=north west][inner sep=0.75pt]    {$p( \pi ,s) =\textcolor[rgb]{0.72,0.91,0.53}{0}$};
				\draw (383,118.4) node [anchor=north west][inner sep=0.75pt]    {$d-3\Bigl\{$};
				\draw (369,46.4) node [anchor=north west][inner sep=0.75pt]    {$3$};
				\draw (595.77,41.4) node [anchor=north west][inner sep=0.75pt]    {$3$};
				\draw (494.78,92.4) node [anchor=north west][inner sep=0.75pt]    {$1$};
				\draw (493.78,129.4) node [anchor=north west][inner sep=0.75pt]    {$1$};
				\draw (116.78,97.4) node [anchor=north west][inner sep=0.75pt]    {$1$};
				\draw (114.78,129.4) node [anchor=north west][inner sep=0.75pt]    {$1$};
				\draw (423,216.4) node [anchor=north west][inner sep=0.75pt]    {$p( \pi ,s) =\textcolor[rgb]{0.72,0.91,0.53}{1}$};

			\end{tikzpicture}
			\caption{}
            \label{fig-closedform1}
		\end{figure} 

        For $k=4$, we fix $B=\mathbb{TP}^1$ with two inner vertices and $s_B$ the parity function with $s_B(v)=0$ for all vertices $v$.\\
		\emph{Construction of $(\pi,s): (\Gamma,s) \rightarrow (B,s_B)$.} We start with a strand of weight 3 and $(d-3)$ strands of weight 1 over $-\infty$. Above the first vertex the following options are available to us, i.e. we can
        
		\begin{enumerate}
			\item split the strand of weight $3$ into three strands of weight $1$.
			\item form a butterfly vertex, i.e. two incoming and two outgoing edges.
			\item form a genus 1 vertex.
			\item join $3$ strands.
		\end{enumerate} 
		\emph{Option 1.} After splitting the strand of weight $3$ into three strands of weight $1$, we are forced to join three strands of weight 1 above the second vertex to create the ramification profile $(3,1,...,1)$ over $+\infty$. All covers with multiplicity and parity are shown in \cref{fig-closedform2}
		\begin{figure}
			\centering\tikzset{every picture/.style={line width=0.75pt}} 
			
			\begin{tikzpicture}[x=0.75pt,y=0.75pt,yscale=-1,xscale=1]
				
				\draw    (49,74.67) -- (196,73.67) ;
				\draw   (78.72,74) .. controls (78.72,62.95) and (96.63,54) .. (118.72,54) .. controls (140.81,54) and (158.72,62.95) .. (158.72,74) .. controls (158.72,85.05) and (140.81,94) .. (118.72,94) .. controls (96.63,94) and (78.72,85.05) .. (78.72,74) -- cycle ;
				\draw    (237.72,73.67) -- (306.72,74.67) -- (412.72,73.67) ;
				\draw [shift={(272.22,74.17)}, rotate = 0.83] [color={rgb, 255:red, 0; green, 0; blue, 0 }  ][fill={rgb, 255:red, 0; green, 0; blue, 0 }  ][line width=0.75]      (0, 0) circle [x radius= 3.35, y radius= 3.35]   ;
				\draw [shift={(359.72,74.17)}, rotate = 359.46] [color={rgb, 255:red, 0; green, 0; blue, 0 }  ][fill={rgb, 255:red, 0; green, 0; blue, 0 }  ][line width=0.75]      (0, 0) circle [x radius= 3.35, y radius= 3.35]   ;
				\draw  [draw opacity=0] (329.37,91.47) .. controls (327.03,91.6) and (324.64,91.67) .. (322.22,91.67) .. controls (294.33,91.67) and (271.72,82.79) .. (271.72,71.83) .. controls (271.72,60.88) and (294.33,52) .. (322.22,52) .. controls (324.64,52) and (327.03,52.07) .. (329.37,52.2) -- (322.22,71.83) -- cycle ; \draw   (329.37,91.47) .. controls (327.03,91.6) and (324.64,91.67) .. (322.22,91.67) .. controls (294.33,91.67) and (271.72,82.79) .. (271.72,71.83) .. controls (271.72,60.88) and (294.33,52) .. (322.22,52) .. controls (324.64,52) and (327.03,52.07) .. (329.37,52.2) ;  
				\draw    (329.37,91.47) .. controls (329.72,90.67) and (357.72,93.67) .. (359.61,73) ;
				\draw    (244.72,103.67) .. controls (253.72,101.67) and (354.72,113.67) .. (359.61,73) ;
				\draw    (329.37,52.2) -- (412.72,52.67) ;
				\draw    (435.72,74.67) -- (504.72,75.67) -- (610.72,75.67) ;
				\draw  [draw opacity=0] (528.87,94.48) .. controls (526.53,94.6) and (524.14,94.67) .. (521.72,94.67) .. controls (493,94.67) and (469.72,85.79) .. (469.72,74.83) .. controls (469.72,63.88) and (493,55) .. (521.72,55) .. controls (524.14,55) and (526.53,55.06) .. (528.87,55.19) -- (521.72,74.83) -- cycle ; \draw   (528.87,94.48) .. controls (526.53,94.6) and (524.14,94.67) .. (521.72,94.67) .. controls (493,94.67) and (469.72,85.79) .. (469.72,74.83) .. controls (469.72,63.88) and (493,55) .. (521.72,55) .. controls (524.14,55) and (526.53,55.06) .. (528.87,55.19) ;  
				\draw    (440.72,127.67) .. controls (450.72,127.67) and (509.72,127.67) .. (528.87,128.52) ;
				\draw    (527.37,55.2) -- (610.72,55.67) ;
				\draw    (158.72,74) ;
				\draw [shift={(158.72,74)}, rotate = 0] [color={rgb, 255:red, 0; green, 0; blue, 0 }  ][fill={rgb, 255:red, 0; green, 0; blue, 0 }  ][line width=0.75]      (0, 0) circle [x radius= 3.35, y radius= 3.35]   ;
				\draw  [draw opacity=0] (528.87,128.52) .. controls (530.9,128.59) and (532.97,128.62) .. (535.06,128.62) .. controls (566.89,128.62) and (592.69,120.96) .. (592.69,111.5) .. controls (592.69,102.05) and (566.89,94.38) .. (535.06,94.38) .. controls (532.97,94.38) and (530.9,94.42) .. (528.87,94.48) -- (535.06,111.5) -- cycle ; \draw   (528.87,128.52) .. controls (530.9,128.59) and (532.97,128.62) .. (535.06,128.62) .. controls (566.89,128.62) and (592.69,120.96) .. (592.69,111.5) .. controls (592.69,102.05) and (566.89,94.38) .. (535.06,94.38) .. controls (532.97,94.38) and (530.9,94.42) .. (528.87,94.48) ;  
				\draw    (470.22,75.17) ;
				\draw [shift={(470.22,75.17)}, rotate = 0] [color={rgb, 255:red, 0; green, 0; blue, 0 }  ][fill={rgb, 255:red, 0; green, 0; blue, 0 }  ][line width=0.75]      (0, 0) circle [x radius= 3.35, y radius= 3.35]   ;
				\draw    (591.72,112.67) ;
				\draw [shift={(591.72,112.67)}, rotate = 0] [color={rgb, 255:red, 0; green, 0; blue, 0 }  ][fill={rgb, 255:red, 0; green, 0; blue, 0 }  ][line width=0.75]      (0, 0) circle [x radius= 3.35, y radius= 3.35]   ;
				\draw    (78.72,74) ;
				\draw [shift={(78.72,74)}, rotate = 0] [color={rgb, 255:red, 0; green, 0; blue, 0 }  ][fill={rgb, 255:red, 0; green, 0; blue, 0 }  ][line width=0.75]      (0, 0) circle [x radius= 3.35, y radius= 3.35]   ;
				\draw    (55,128) -- (198,128) ;
				\draw    (55,160) -- (198,160) ;
				\draw    (253,124) -- (396,124) ;
				\draw    (253,156) -- (396,156) ;
				\draw    (467,169) -- (610,169) ;
				\draw    (467,201) -- (610,201) ;
				\draw    (443,112.33) -- (616,112.33) ;
				
				\draw (396.72,29.4) node [anchor=north west][inner sep=0.75pt]    {$1$};
				\draw (202.34,67.52) node [anchor=north west][inner sep=0.75pt]    {$3$};
				\draw (62.29,50.52) node [anchor=north west][inner sep=0.75pt]    {$\textcolor[rgb]{0.72,0.91,0.53}{0}$};
				\draw (114.72,57.4) node [anchor=north west][inner sep=0.75pt]    {$1$};
				\draw (114.72,77.4) node [anchor=north west][inner sep=0.75pt]    {$1$};
				\draw (114.6,33.58) node [anchor=north west][inner sep=0.75pt]    {$1$};
				\draw (161.72,53.4) node [anchor=north west][inner sep=0.75pt]    {$\textcolor[rgb]{0.72,0.91,0.53}{0}$};
				\draw (24.5,66.4) node [anchor=north west][inner sep=0.75pt]    {$3$};
				\draw (309.6,54.58) node [anchor=north west][inner sep=0.75pt]    {$1$};
				\draw (231.72,95.07) node [anchor=north west][inner sep=0.75pt]    {$1$};
				\draw (257.94,53.63) node [anchor=north west][inner sep=0.75pt]    {$\textcolor[rgb]{0.72,0.91,0.53}{0}$};
				\draw (361.72,77.57) node [anchor=north west][inner sep=0.75pt]    {$\textcolor[rgb]{0.72,0.91,0.53}{0}$};
				\draw (231.72,47.4) node [anchor=north west][inner sep=0.75pt]    {$3$};
				\draw (399.72,80.4) node [anchor=north west][inner sep=0.75pt]    {$3$};
				\draw (316.37,74.87) node [anchor=north west][inner sep=0.75pt]    {$1$};
				\draw (616.72,47.4) node [anchor=north west][inner sep=0.75pt]    {$1$};
				\draw (463.64,138.78) node [anchor=north west][inner sep=0.75pt]    {$\textcolor[rgb]{0.72,0.91,0.53}{0}$};
				\draw (617.6,67.58) node [anchor=north west][inner sep=0.75pt]    {$1$};
				\draw (455.94,56.63) node [anchor=north west][inner sep=0.75pt]    {$\textcolor[rgb]{0.72,0.91,0.53}{0}$};
				\draw (595.72,89.57) node [anchor=north west][inner sep=0.75pt]    {$\textcolor[rgb]{0.72,0.91,0.53}{0}$};
				\draw (618.87,101.88) node [anchor=north west][inner sep=0.75pt]    {$3$};
				\draw (514.37,77.87) node [anchor=north west][inner sep=0.75pt]    {$1$};
				\draw (518.87,129.88) node [anchor=north west][inner sep=0.75pt]    {$1$};
				\draw (442.87,136.88) node [anchor=north west][inner sep=0.75pt]    {$1$};
				\draw (439.87,89.88) node [anchor=north west][inner sep=0.75pt]    {$1$};
				\draw (441.87,49.4) node [anchor=north west][inner sep=0.75pt]    {$3$};
				\draw (32,178.4) node [anchor=north west][inner sep=0.75pt]    {$mult( \pi ,s) =\frac{2^{2}}{3!( d-3) !}$};
				\draw (4,126.4) node [anchor=north west][inner sep=0.75pt]    {$d-3\Bigl\{$};
				\draw (118.78,103.4) node [anchor=north west][inner sep=0.75pt]    {$1$};
				\draw (114.78,137.4) node [anchor=north west][inner sep=0.75pt]    {$1$};
				\draw (230,174.4) node [anchor=north west][inner sep=0.75pt]    {$mult( \pi ,s) =\frac{2^{2}}{2( d-4) !}$};
				\draw (202,126.4) node [anchor=north west][inner sep=0.75pt]    {$d-4\Bigl\{$};
				\draw (316.78,99.4) node [anchor=north west][inner sep=0.75pt]    {$1$};
				\draw (312.78,133.4) node [anchor=north west][inner sep=0.75pt]    {$1$};
				\draw (444,219.4) node [anchor=north west][inner sep=0.75pt]    {$mult( \pi ,s) =\frac{2^{2}}{2^{2}( d-5) !}$};
				\draw (416,167.4) node [anchor=north west][inner sep=0.75pt]    {$d-5\Bigl\{$};
				\draw (530.78,144.4) node [anchor=north west][inner sep=0.75pt]    {$1$};
				\draw (526.78,178.4) node [anchor=north west][inner sep=0.75pt]    {$1$};

			\end{tikzpicture}
			\tikzset{every picture/.style={line width=0.75pt}} 
			
			\begin{tikzpicture}[x=0.75pt,y=0.75pt,yscale=-1,xscale=1]
				
				\draw    (55,128) -- (230,127.67) ;
				\draw    (55,160) -- (230,159.67) ;
				\draw  [draw opacity=0] (225.31,55.62) .. controls (222.65,55.65) and (219.97,55.67) .. (217.28,55.67) .. controls (151.71,55.67) and (98.55,45.77) .. (98.55,33.55) .. controls (98.55,21.34) and (151.71,11.44) .. (217.28,11.44) .. controls (217.91,11.44) and (218.54,11.44) .. (219.17,11.45) -- (217.28,33.55) -- cycle ; \draw   (225.31,55.62) .. controls (222.65,55.65) and (219.97,55.67) .. (217.28,55.67) .. controls (151.71,55.67) and (98.55,45.77) .. (98.55,33.55) .. controls (98.55,21.34) and (151.71,11.44) .. (217.28,11.44) .. controls (217.91,11.44) and (218.54,11.44) .. (219.17,11.45) ;  
				\draw    (57,33.67) -- (228,34.67) ;
				\draw    (99,33.67) ;
				\draw [shift={(99,33.67)}, rotate = 0] [color={rgb, 255:red, 0; green, 0; blue, 0 }  ][fill={rgb, 255:red, 0; green, 0; blue, 0 }  ][line width=0.75]      (0, 0) circle [x radius= 3.35, y radius= 3.35]   ;
				\draw    (55,92) -- (230,92.67) ;
				\draw  [draw opacity=0] (55.38,114.63) .. controls (58.07,114.65) and (60.78,114.67) .. (63.5,114.67) .. controls (140.54,114.67) and (203,104.67) .. (203,92.33) .. controls (203,80) and (140.54,70) .. (63.5,70) .. controls (60.78,70) and (58.07,70.01) .. (55.38,70.04) -- (63.5,92.33) -- cycle ; \draw   (55.38,114.63) .. controls (58.07,114.65) and (60.78,114.67) .. (63.5,114.67) .. controls (140.54,114.67) and (203,104.67) .. (203,92.33) .. controls (203,80) and (140.54,70) .. (63.5,70) .. controls (60.78,70) and (58.07,70.01) .. (55.38,70.04) ;  
				\draw    (204,92.67) ;
				\draw [shift={(204,92.67)}, rotate = 0] [color={rgb, 255:red, 0; green, 0; blue, 0 }  ][fill={rgb, 255:red, 0; green, 0; blue, 0 }  ][line width=0.75]      (0, 0) circle [x radius= 3.35, y radius= 3.35]   ;
				
				\draw (57,196.4) node [anchor=north west][inner sep=0.75pt]    {$mult( \pi ,s) =\frac{2^{2}}{( 3!)^{2}( d-6) !}$};
				\draw (4,129.4) node [anchor=north west][inner sep=0.75pt]    {$d-6\Bigl\{$};
				\draw (240.78,50.4) node [anchor=north west][inner sep=0.75pt]    {$1$};
				\draw (67.59,13.52) node [anchor=north west][inner sep=0.75pt]    {$3$};
				\draw (239.52,32.4) node [anchor=north west][inner sep=0.75pt]    {$1$};
				\draw (238.53,8.52) node [anchor=north west][inner sep=0.75pt]    {$1$};
				\draw (90.79,11.52) node [anchor=north west][inner sep=0.75pt]    {$\textcolor[rgb]{0.72,0.91,0.53}{0}$};
				\draw (204.12,70.4) node [anchor=north west][inner sep=0.75pt]    {$\textcolor[rgb]{0.72,0.91,0.53}{0}$};
				\draw (241.43,84.4) node [anchor=north west][inner sep=0.75pt]    {$3$};
				\draw (34.78,108.4) node [anchor=north west][inner sep=0.75pt]    {$1$};
				\draw (34.78,86.4) node [anchor=north west][inner sep=0.75pt]    {$1$};
				\draw (34.78,62.4) node [anchor=north west][inner sep=0.75pt]    {$1$};
				\draw (241.53,121.4) node [anchor=north west][inner sep=0.75pt]    {$1$};
				\draw (239.53,152.4) node [anchor=north west][inner sep=0.75pt]    {$1$};

			\end{tikzpicture}
			\caption{}
            \label{fig-closedform2}
		\end{figure}
		\emph{Option 2.} The only possibility for a butterfly vertex of non-zero multiplicity has in-going and out-going edges of weight 3 and 1. Above the second vertex we are left with the choice of either creating a genus one vertex on the strand of weight 3 or forming another butterfly vertex.
		
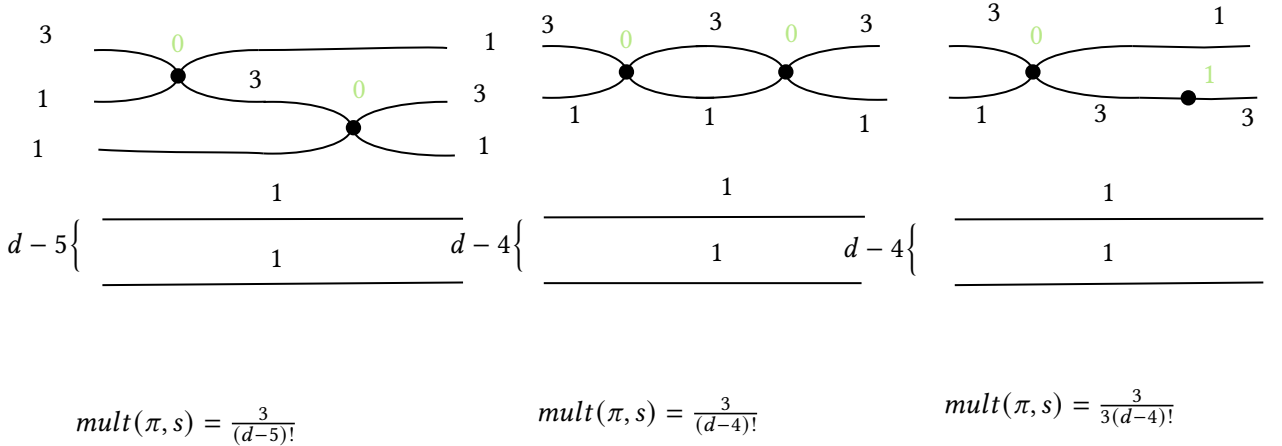
\begin{figure}[H]
			\tikzset{every picture/.style={line width=0.75pt}} 
			
			\begin{tikzpicture}[x=0.75pt,y=0.75pt,yscale=-1,xscale=1]
				
				\draw    (52,139) -- (233,138.67) ;
				\draw    (52,172) -- (233,170.67) ;
				\draw  [draw opacity=0] (133.99,79.91) .. controls (132.45,79.97) and (130.89,80) .. (129.3,80) .. controls (107.92,80) and (90.6,74.19) .. (90.6,67.01) .. controls (90.6,59.84) and (107.92,54.03) .. (129.3,54.03) .. controls (129.67,54.03) and (130.04,54.03) .. (130.41,54.03) -- (129.3,67.01) -- cycle ; \draw   (133.99,79.91) .. controls (132.45,79.97) and (130.89,80) .. (129.3,80) .. controls (107.92,80) and (90.6,74.19) .. (90.6,67.01) .. controls (90.6,59.84) and (107.92,54.03) .. (129.3,54.03) .. controls (129.67,54.03) and (130.04,54.03) .. (130.41,54.03) ;  
				\draw  [draw opacity=0] (47.81,79.9) .. controls (49.34,79.97) and (50.91,80) .. (52.5,80) .. controls (73.21,80) and (90,74.18) .. (90,67) .. controls (90,59.82) and (73.21,54) .. (52.5,54) .. controls (50.91,54) and (49.34,54.03) .. (47.81,54.1) -- (52.5,67) -- cycle ; \draw   (47.81,79.9) .. controls (49.34,79.97) and (50.91,80) .. (52.5,80) .. controls (73.21,80) and (90,74.18) .. (90,67) .. controls (90,59.82) and (73.21,54) .. (52.5,54) .. controls (50.91,54) and (49.34,54.03) .. (47.81,54.1) ;  
				\draw    (90,67) ;
				\draw [shift={(90,67)}, rotate = 0] [color={rgb, 255:red, 0; green, 0; blue, 0 }  ][fill={rgb, 255:red, 0; green, 0; blue, 0 }  ][line width=0.75]      (0, 0) circle [x radius= 3.35, y radius= 3.35]   ;
				\draw  [draw opacity=0] (228.89,106.92) .. controls (227.28,106.97) and (225.65,107) .. (224,107) .. controls (198.59,107) and (178,100.96) .. (178,93.5) .. controls (178,86.04) and (198.59,80) .. (224,80) .. controls (224.39,80) and (224.77,80) .. (225.16,80) -- (224,93.5) -- cycle ; \draw   (228.89,106.92) .. controls (227.28,106.97) and (225.65,107) .. (224,107) .. controls (198.59,107) and (178,100.96) .. (178,93.5) .. controls (178,86.04) and (198.59,80) .. (224,80) .. controls (224.39,80) and (224.77,80) .. (225.16,80) ;  
				\draw  [draw opacity=0] (132.8,105.91) .. controls (134.34,105.97) and (135.91,106) .. (137.5,106) .. controls (159.32,106) and (177,100.18) .. (177,93) .. controls (177,85.82) and (159.32,80) .. (137.5,80) .. controls (135.91,80) and (134.34,80.03) .. (132.8,80.09) -- (137.5,93) -- cycle ; \draw   (132.8,105.91) .. controls (134.34,105.97) and (135.91,106) .. (137.5,106) .. controls (159.32,106) and (177,100.18) .. (177,93) .. controls (177,85.82) and (159.32,80) .. (137.5,80) .. controls (135.91,80) and (134.34,80.03) .. (132.8,80.09) ;  
				\draw    (178,93) ;
				\draw [shift={(178,93)}, rotate = 0] [color={rgb, 255:red, 0; green, 0; blue, 0 }  ][fill={rgb, 255:red, 0; green, 0; blue, 0 }  ][line width=0.75]      (0, 0) circle [x radius= 3.35, y radius= 3.35]   ;
				\draw    (130.41,54.03) .. controls (164,54) and (215,52) .. (225,53) ;
				\draw    (50,104) .. controls (85,106) and (122.8,104.91) .. (132.8,105.91) ;
				\draw  [draw opacity=0] (358.99,77.91) .. controls (357.45,77.97) and (355.89,78) .. (354.3,78) .. controls (332.92,78) and (315.6,72.19) .. (315.6,65.01) .. controls (315.6,57.84) and (332.92,52.03) .. (354.3,52.03) .. controls (354.67,52.03) and (355.04,52.03) .. (355.41,52.03) -- (354.3,65.01) -- cycle ; \draw   (358.99,77.91) .. controls (357.45,77.97) and (355.89,78) .. (354.3,78) .. controls (332.92,78) and (315.6,72.19) .. (315.6,65.01) .. controls (315.6,57.84) and (332.92,52.03) .. (354.3,52.03) .. controls (354.67,52.03) and (355.04,52.03) .. (355.41,52.03) ;  
				\draw  [draw opacity=0] (272.81,77.9) .. controls (274.34,77.97) and (275.91,78) .. (277.5,78) .. controls (298.21,78) and (315,72.18) .. (315,65) .. controls (315,57.82) and (298.21,52) .. (277.5,52) .. controls (275.91,52) and (274.34,52.03) .. (272.81,52.1) -- (277.5,65) -- cycle ; \draw   (272.81,77.9) .. controls (274.34,77.97) and (275.91,78) .. (277.5,78) .. controls (298.21,78) and (315,72.18) .. (315,65) .. controls (315,57.82) and (298.21,52) .. (277.5,52) .. controls (275.91,52) and (274.34,52.03) .. (272.81,52.1) ;  
				\draw    (315,65) ;
				\draw [shift={(315,65)}, rotate = 0] [color={rgb, 255:red, 0; green, 0; blue, 0 }  ][fill={rgb, 255:red, 0; green, 0; blue, 0 }  ][line width=0.75]      (0, 0) circle [x radius= 3.35, y radius= 3.35]   ;
				\draw  [draw opacity=0] (445.89,78.92) .. controls (444.28,78.97) and (442.65,79) .. (441,79) .. controls (415.59,79) and (395,72.96) .. (395,65.5) .. controls (395,58.04) and (415.59,52) .. (441,52) .. controls (441.39,52) and (441.77,52) .. (442.16,52) -- (441,65.5) -- cycle ; \draw   (445.89,78.92) .. controls (444.28,78.97) and (442.65,79) .. (441,79) .. controls (415.59,79) and (395,72.96) .. (395,65.5) .. controls (395,58.04) and (415.59,52) .. (441,52) .. controls (441.39,52) and (441.77,52) .. (442.16,52) ;  
				\draw  [draw opacity=0] (349.8,77.91) .. controls (351.34,77.97) and (352.91,78) .. (354.5,78) .. controls (376.32,78) and (394,72.18) .. (394,65) .. controls (394,57.82) and (376.32,52) .. (354.5,52) .. controls (352.91,52) and (351.34,52.03) .. (349.8,52.09) -- (354.5,65) -- cycle ; \draw   (349.8,77.91) .. controls (351.34,77.97) and (352.91,78) .. (354.5,78) .. controls (376.32,78) and (394,72.18) .. (394,65) .. controls (394,57.82) and (376.32,52) .. (354.5,52) .. controls (352.91,52) and (351.34,52.03) .. (349.8,52.09) ;  
				\draw    (395,65) ;
				\draw [shift={(395,65)}, rotate = 0] [color={rgb, 255:red, 0; green, 0; blue, 0 }  ][fill={rgb, 255:red, 0; green, 0; blue, 0 }  ][line width=0.75]      (0, 0) circle [x radius= 3.35, y radius= 3.35]   ;
				\draw  [draw opacity=0] (572.5,77.94) .. controls (570.95,77.98) and (569.39,78) .. (567.8,78) .. controls (541.18,78) and (519.6,72.19) .. (519.6,65.01) .. controls (519.6,57.84) and (541.18,52.03) .. (567.8,52.03) .. controls (568.17,52.03) and (568.54,52.03) .. (568.91,52.03) -- (567.8,65.01) -- cycle ; \draw   (572.5,77.94) .. controls (570.95,77.98) and (569.39,78) .. (567.8,78) .. controls (541.18,78) and (519.6,72.19) .. (519.6,65.01) .. controls (519.6,57.84) and (541.18,52.03) .. (567.8,52.03) .. controls (568.17,52.03) and (568.54,52.03) .. (568.91,52.03) ;  
				\draw  [draw opacity=0] (476.81,77.9) .. controls (478.34,77.97) and (479.91,78) .. (481.5,78) .. controls (502.21,78) and (519,72.18) .. (519,65) .. controls (519,57.82) and (502.21,52) .. (481.5,52) .. controls (479.91,52) and (478.34,52.03) .. (476.81,52.1) -- (481.5,65) -- cycle ; \draw   (476.81,77.9) .. controls (478.34,77.97) and (479.91,78) .. (481.5,78) .. controls (502.21,78) and (519,72.18) .. (519,65) .. controls (519,57.82) and (502.21,52) .. (481.5,52) .. controls (479.91,52) and (478.34,52.03) .. (476.81,52.1) ;  
				\draw    (597,78) ;
				\draw [shift={(597,78)}, rotate = 0] [color={rgb, 255:red, 0; green, 0; blue, 0 }  ][fill={rgb, 255:red, 0; green, 0; blue, 0 }  ][line width=0.75]      (0, 0) circle [x radius= 3.35, y radius= 3.35]   ;
				\draw    (568.91,52.03) -- (604.94,52.73) -- (628,52) ;
				\draw    (572.5,77.94) -- (608.53,78.64) -- (631.59,77.91) ;
				\draw    (519,65) ;
				\draw [shift={(519,65)}, rotate = 0] [color={rgb, 255:red, 0; green, 0; blue, 0 }  ][fill={rgb, 255:red, 0; green, 0; blue, 0 }  ][line width=0.75]      (0, 0) circle [x radius= 3.35, y radius= 3.35]   ;
				\draw    (273.52,138) -- (434.68,137.67) ;
				\draw    (273.52,171) -- (433,171) ;
				\draw    (479.78,139) -- (634.64,138.67) ;
				\draw    (479.78,172) -- (634.64,170.67) ;
				
				\draw (37,231.4) node [anchor=north west][inner sep=0.75pt]    {$mult( \pi ,s) =\frac{3}{( d-5) !}$};
				\draw (3,137.4) node [anchor=north west][inner sep=0.75pt]    {$d-5\Bigl\{$};
				\draw (134.53,152.4) node [anchor=north west][inner sep=0.75pt]    {$1$};
				\draw (85,44.4) node [anchor=north west][inner sep=0.75pt]    {$\textcolor[rgb]{0.72,0.91,0.53}{0}$};
				\draw (18.77,40.19) node [anchor=north west][inner sep=0.75pt]    {$3$};
				\draw (17.28,72.36) node [anchor=north west][inner sep=0.75pt]    {$1$};
				\draw (238.62,95.52) node [anchor=north west][inner sep=0.75pt]    {$1$};
				\draw (124,62.4) node [anchor=north west][inner sep=0.75pt]    {$3$};
				\draw (236.62,69.52) node [anchor=north west][inner sep=0.75pt]    {$3$};
				\draw (176,67.4) node [anchor=north west][inner sep=0.75pt]    {$\textcolor[rgb]{0.72,0.91,0.53}{0}$};
				\draw (134.62,119.52) node [anchor=north west][inner sep=0.75pt]    {$1$};
				\draw (241.8,43.31) node [anchor=north west][inner sep=0.75pt]    {$1$};
				\draw (14.8,97.31) node [anchor=north west][inner sep=0.75pt]    {$1$};
				\draw (310,42.4) node [anchor=north west][inner sep=0.75pt]    {$\textcolor[rgb]{0.72,0.91,0.53}{0}$};
				\draw (270.77,35.19) node [anchor=north west][inner sep=0.75pt]    {$3$};
				\draw (284.28,80.36) node [anchor=north west][inner sep=0.75pt]    {$1$};
				\draw (429.62,84.52) node [anchor=north west][inner sep=0.75pt]    {$1$};
				\draw (355,33.4) node [anchor=north west][inner sep=0.75pt]    {$3$};
				\draw (351.8,81.31) node [anchor=north west][inner sep=0.75pt]    {$1$};
				\draw (430.62,34.52) node [anchor=north west][inner sep=0.75pt]    {$3$};
				\draw (393,39.4) node [anchor=north west][inner sep=0.75pt]    {$\textcolor[rgb]{0.72,0.91,0.53}{0}$};
				\draw (516,40.4) node [anchor=north west][inner sep=0.75pt]    {$\textcolor[rgb]{0.72,0.91,0.53}{0}$};
				\draw (494.77,29.19) node [anchor=north west][inner sep=0.75pt]    {$3$};
				\draw (488.28,80.36) node [anchor=north west][inner sep=0.75pt]    {$1$};
				\draw (607.62,30.52) node [anchor=north west][inner sep=0.75pt]    {$1$};
				\draw (622.06,81.68) node [anchor=north west][inner sep=0.75pt]    {$3$};
				\draw (548,79.4) node [anchor=north west][inner sep=0.75pt]    {$3$};
				\draw (602.72,60.19) node [anchor=north west][inner sep=0.75pt]    {$\textcolor[rgb]{0.72,0.91,0.53}{1}$};
				\draw (225.31,137.4) node [anchor=north west][inner sep=0.75pt]    {$d-4\Bigl\{$};
				\draw (355.53,149.4) node [anchor=north west][inner sep=0.75pt]    {$1$};
				\draw (360.62,116.52) node [anchor=north west][inner sep=0.75pt]    {$1$};
				\draw (422.69,139.4) node [anchor=north west][inner sep=0.75pt]    {$d-4\Bigl\{$};
				\draw (551.93,149.4) node [anchor=north west][inner sep=0.75pt]    {$1$};
				\draw (552.01,119.52) node [anchor=north west][inner sep=0.75pt]    {$1$};
				\draw (269,226.4) node [anchor=north west][inner sep=0.75pt]    {$mult( \pi ,s) =\frac{3}{( d-4) !}$};
				\draw (473,222.4) node [anchor=north west][inner sep=0.75pt]    {$mult( \pi ,s) =\frac{3}{3( d-4) !}$};

			\end{tikzpicture}
			
		\end{figure}
		\emph{Option 3.} After creating a vertex of genus 1 on the strand of weight 3, we can either create another genus 1 vertex or form a butterfly vertex above the second branch point.
		
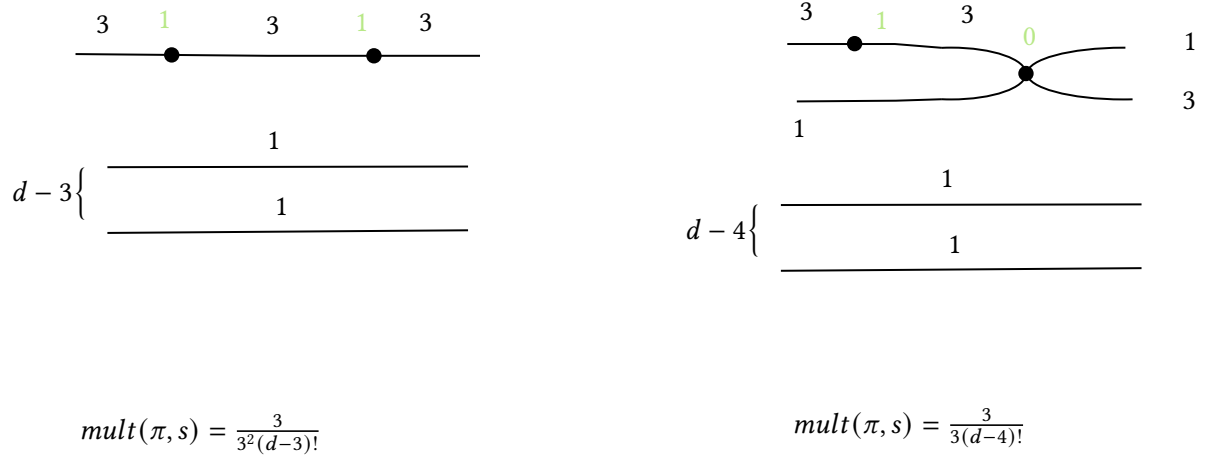
\begin{figure}[H]
			\tikzset{every picture/.style={line width=0.75pt}} 
			
			\begin{tikzpicture}[x=0.75pt,y=0.75pt,yscale=-1,xscale=1]
				
				\draw    (52,109) -- (233,108.67) ;
				\draw    (52,142) -- (233,140.67) ;
				\draw  [draw opacity=0] (566.5,74.94) .. controls (564.95,74.98) and (563.39,75) .. (561.8,75) .. controls (535.18,75) and (513.6,69.19) .. (513.6,62.01) .. controls (513.6,54.84) and (535.18,49.03) .. (561.8,49.03) .. controls (562.17,49.03) and (562.54,49.03) .. (562.91,49.03) -- (561.8,62.01) -- cycle ; \draw   (566.5,74.94) .. controls (564.95,74.98) and (563.39,75) .. (561.8,75) .. controls (535.18,75) and (513.6,69.19) .. (513.6,62.01) .. controls (513.6,54.84) and (535.18,49.03) .. (561.8,49.03) .. controls (562.17,49.03) and (562.54,49.03) .. (562.91,49.03) ;  
				\draw  [draw opacity=0] (470.81,74.9) .. controls (472.34,74.97) and (473.91,75) .. (475.5,75) .. controls (496.21,75) and (513,69.18) .. (513,62) .. controls (513,54.82) and (496.21,49) .. (475.5,49) .. controls (473.91,49) and (472.34,49.03) .. (470.81,49.1) -- (475.5,62) -- cycle ; \draw   (470.81,74.9) .. controls (472.34,74.97) and (473.91,75) .. (475.5,75) .. controls (496.21,75) and (513,69.18) .. (513,62) .. controls (513,54.82) and (496.21,49) .. (475.5,49) .. controls (473.91,49) and (472.34,49.03) .. (470.81,49.1) ;  
				\draw    (427,47) ;
				\draw [shift={(427,47)}, rotate = 0] [color={rgb, 255:red, 0; green, 0; blue, 0 }  ][fill={rgb, 255:red, 0; green, 0; blue, 0 }  ][line width=0.75]      (0, 0) circle [x radius= 3.35, y radius= 3.35]   ;
				\draw    (398,76) -- (447.74,75.63) -- (470.81,74.9) ;
				\draw    (393.21,47.19) -- (447.21,47.19) -- (470.81,49.1) ;
				\draw    (513,62) ;
				\draw [shift={(513,62)}, rotate = 0] [color={rgb, 255:red, 0; green, 0; blue, 0 }  ][fill={rgb, 255:red, 0; green, 0; blue, 0 }  ][line width=0.75]      (0, 0) circle [x radius= 3.35, y radius= 3.35]   ;
				\draw    (36,52.23) -- (132.55,53.1) -- (239,53.1) ;
				\draw [shift={(84.27,52.66)}, rotate = 0.51] [color={rgb, 255:red, 0; green, 0; blue, 0 }  ][fill={rgb, 255:red, 0; green, 0; blue, 0 }  ][line width=0.75]      (0, 0) circle [x radius= 3.35, y radius= 3.35]   ;
				\draw [shift={(185.77,53.1)}, rotate = 0] [color={rgb, 255:red, 0; green, 0; blue, 0 }  ][fill={rgb, 255:red, 0; green, 0; blue, 0 }  ][line width=0.75]      (0, 0) circle [x radius= 3.35, y radius= 3.35]   ;
				\draw    (390,128) -- (571,127.67) ;
				\draw    (390,161) -- (571,159.67) ;
				
				\draw (37,231.4) node [anchor=north west][inner sep=0.75pt]    {$mult( \pi ,s) =\frac{3}{3^{2}( d-3) !}$};
				\draw (3,107.4) node [anchor=north west][inner sep=0.75pt]    {$d-3\Bigl\{$};
				\draw (134.53,122.4) node [anchor=north west][inner sep=0.75pt]    {$1$};
				\draw (510,37.4) node [anchor=north west][inner sep=0.75pt]    {$\textcolor[rgb]{0.72,0.91,0.53}{0}$};
				\draw (478.77,26.19) node [anchor=north west][inner sep=0.75pt]    {$3$};
				\draw (394.28,83.36) node [anchor=north west][inner sep=0.75pt]    {$1$};
				\draw (590.62,39.52) node [anchor=north west][inner sep=0.75pt]    {$1$};
				\draw (398.06,24.68) node [anchor=north west][inner sep=0.75pt]    {$3$};
				\draw (590,69.4) node [anchor=north west][inner sep=0.75pt]    {$3$};
				\draw (435.72,29.19) node [anchor=north west][inner sep=0.75pt]    {$\textcolor[rgb]{0.72,0.91,0.53}{1}$};
				\draw (130.62,89.52) node [anchor=north west][inner sep=0.75pt]    {$1$};
				\draw (395,228.4) node [anchor=north west][inner sep=0.75pt]    {$mult( \pi ,s) =\frac{3}{3( d-4) !}$};
				\draw (44.85,31.05) node [anchor=north west][inner sep=0.75pt]    {$3$};
				\draw (207.01,30.19) node [anchor=north west][inner sep=0.75pt]    {$3$};
				\draw (75.8,28.45) node [anchor=north west][inner sep=0.75pt]    {$\textcolor[rgb]{0.72,0.91,0.53}{1}$};
				\draw (174.82,30.19) node [anchor=north west][inner sep=0.75pt]    {$\textcolor[rgb]{0.72,0.91,0.53}{1}$};
				\draw (130.26,31.92) node [anchor=north west][inner sep=0.75pt]    {$3$};
				\draw (341,126.4) node [anchor=north west][inner sep=0.75pt]    {$d-4\Bigl\{$};
				\draw (472.53,141.4) node [anchor=north west][inner sep=0.75pt]    {$1$};
				\draw (468.62,108.52) node [anchor=north west][inner sep=0.75pt]    {$1$};

			\end{tikzpicture}
			
		\end{figure}
		\emph{Option 4.1.} We join 3 strands of weight 1. Above the second branch point we can either split this strand into 3 strands of weight 1 or choose to split the other strand of weight 3 and obtain a cover as in option 1 with two disconnected balanced double forks.\\
		\emph{Option 4.2.} We join 3 strands, two of weight 1 and one of weight 3. The only possibility above the second vertex is to split again.
		
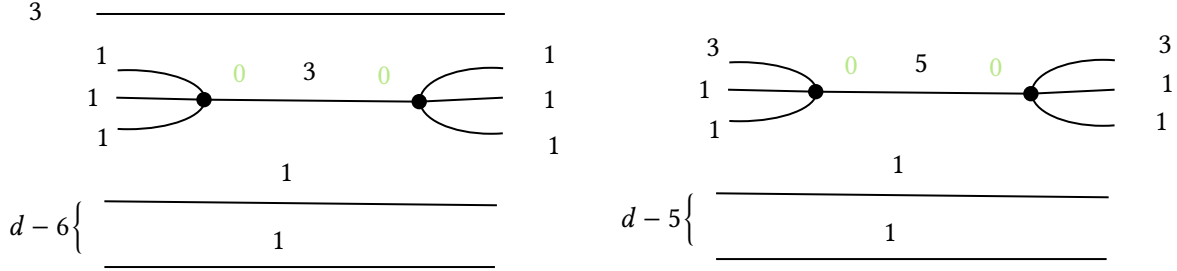
\begin{figure}[H]
			\tikzset{every picture/.style={line width=0.75pt}} 
			
			\begin{tikzpicture}[x=0.75pt,y=0.75pt,yscale=-1,xscale=1]
				
				\draw    (52,109) -- (249,111) ;
				\draw    (52,142) -- (248,142) ;
				\draw    (210,59) -- (252,57) ;
				\draw  [draw opacity=0] (251.97,74.89) .. controls (250.04,75.04) and (248.05,75.12) .. (246.03,75.12) .. controls (226.36,75.12) and (210.42,67.7) .. (210.42,58.56) .. controls (210.42,49.41) and (226.36,42) .. (246.03,42) .. controls (248.1,42) and (250.14,42.08) .. (252.12,42.24) -- (246.03,58.56) -- cycle ; \draw   (251.97,74.89) .. controls (250.04,75.04) and (248.05,75.12) .. (246.03,75.12) .. controls (226.36,75.12) and (210.42,67.7) .. (210.42,58.56) .. controls (210.42,49.41) and (226.36,42) .. (246.03,42) .. controls (248.1,42) and (250.14,42.08) .. (252.12,42.24) ;  
				\draw  [draw opacity=0] (58.65,72.69) .. controls (60.4,72.78) and (62.19,72.83) .. (64,72.83) .. controls (85.26,72.83) and (102.5,66.19) .. (102.5,58) .. controls (102.5,49.81) and (85.26,43.17) .. (64,43.17) .. controls (62.19,43.17) and (60.4,43.22) .. (58.65,43.31) -- (64,58) -- cycle ; \draw   (58.65,72.69) .. controls (60.4,72.78) and (62.19,72.83) .. (64,72.83) .. controls (85.26,72.83) and (102.5,66.19) .. (102.5,58) .. controls (102.5,49.81) and (85.26,43.17) .. (64,43.17) .. controls (62.19,43.17) and (60.4,43.22) .. (58.65,43.31) ;  
				\draw    (102,58) -- (210,59) ;
				\draw [shift={(210,59)}, rotate = 0.53] [color={rgb, 255:red, 0; green, 0; blue, 0 }  ][fill={rgb, 255:red, 0; green, 0; blue, 0 }  ][line width=0.75]      (0, 0) circle [x radius= 3.35, y radius= 3.35]   ;
				\draw    (58,57) -- (102,58) ;
				\draw [shift={(102,58)}, rotate = 1.3] [color={rgb, 255:red, 0; green, 0; blue, 0 }  ][fill={rgb, 255:red, 0; green, 0; blue, 0 }  ][line width=0.75]      (0, 0) circle [x radius= 3.35, y radius= 3.35]   ;
				\draw    (359,105) -- (556,107) ;
				\draw    (359,138) -- (555,138) ;
				\draw    (517,55) -- (559,53) ;
				\draw  [draw opacity=0] (558.97,70.89) .. controls (557.04,71.04) and (555.05,71.12) .. (553.03,71.12) .. controls (533.36,71.12) and (517.42,63.7) .. (517.42,54.56) .. controls (517.42,45.41) and (533.36,38) .. (553.03,38) .. controls (555.1,38) and (557.14,38.08) .. (559.12,38.24) -- (553.03,54.56) -- cycle ; \draw   (558.97,70.89) .. controls (557.04,71.04) and (555.05,71.12) .. (553.03,71.12) .. controls (533.36,71.12) and (517.42,63.7) .. (517.42,54.56) .. controls (517.42,45.41) and (533.36,38) .. (553.03,38) .. controls (555.1,38) and (557.14,38.08) .. (559.12,38.24) ;  
				\draw  [draw opacity=0] (365.65,68.69) .. controls (367.4,68.78) and (369.19,68.83) .. (371,68.83) .. controls (392.26,68.83) and (409.5,62.19) .. (409.5,54) .. controls (409.5,45.81) and (392.26,39.17) .. (371,39.17) .. controls (369.19,39.17) and (367.4,39.22) .. (365.65,39.31) -- (371,54) -- cycle ; \draw   (365.65,68.69) .. controls (367.4,68.78) and (369.19,68.83) .. (371,68.83) .. controls (392.26,68.83) and (409.5,62.19) .. (409.5,54) .. controls (409.5,45.81) and (392.26,39.17) .. (371,39.17) .. controls (369.19,39.17) and (367.4,39.22) .. (365.65,39.31) ;  
				\draw    (409,54) -- (517,55) ;
				\draw [shift={(517,55)}, rotate = 0.53] [color={rgb, 255:red, 0; green, 0; blue, 0 }  ][fill={rgb, 255:red, 0; green, 0; blue, 0 }  ][line width=0.75]      (0, 0) circle [x radius= 3.35, y radius= 3.35]   ;
				\draw    (365,53) -- (409,54) ;
				\draw [shift={(409,54)}, rotate = 1.3] [color={rgb, 255:red, 0; green, 0; blue, 0 }  ][fill={rgb, 255:red, 0; green, 0; blue, 0 }  ][line width=0.75]      (0, 0) circle [x radius= 3.35, y radius= 3.35]   ;
				\draw    (48,15) -- (253,15) ;
				
				\draw (37,231.4) node [anchor=north west][inner sep=0.75pt]    {$mult( \pi ,s) =\frac{2^{2} \cdot 3}{3\cdot ( 3!)^{2}( d-6) !}$};
				\draw (3,107.4) node [anchor=north west][inner sep=0.75pt]    {$d-6\Bigl\{$};
				\draw (134.53,122.4) node [anchor=north west][inner sep=0.75pt]    {$1$};
				\draw (580.54,43.58) node [anchor=north west][inner sep=0.75pt]  [rotate=-359.39]  {$1$};
				\draw (138.62,88.52) node [anchor=north west][inner sep=0.75pt]    {$1$};
				\draw (381,229.4) node [anchor=north west][inner sep=0.75pt]    {$mult( \pi ,s) =\frac{2^{2} \cdot 5}{2^{2}( d-5) !}$};
				\draw (12.53,7.4) node [anchor=north west][inner sep=0.75pt]    {$3$};
				\draw (45.77,29.55) node [anchor=north west][inner sep=0.75pt]    {$1$};
				\draw (41.28,50.67) node [anchor=north west][inner sep=0.75pt]    {$1$};
				\draw (270.62,28.88) node [anchor=north west][inner sep=0.75pt]    {$1$};
				\draw (272.78,74.67) node [anchor=north west][inner sep=0.75pt]    {$1$};
				\draw (46.41,70.64) node [anchor=north west][inner sep=0.75pt]    {$1$};
				\draw (270.62,51.83) node [anchor=north west][inner sep=0.75pt]    {$1$};
				\draw (187.57,39.85) node [anchor=north west][inner sep=0.75pt]    {$\textcolor[rgb]{0.72,0.91,0.53}{0}$};
				\draw (150.11,37.45) node [anchor=north west][inner sep=0.75pt]    {$3$};
				\draw (115,38.73) node [anchor=north west][inner sep=0.75pt]    {$\textcolor[rgb]{0.72,0.91,0.53}{0}$};
				\draw (310,103.4) node [anchor=north west][inner sep=0.75pt]    {$d-5\Bigl\{$};
				\draw (441.53,118.4) node [anchor=north west][inner sep=0.75pt]    {$1$};
				\draw (445.62,84.52) node [anchor=north west][inner sep=0.75pt]    {$1$};
				\draw (352.77,25.55) node [anchor=north west][inner sep=0.75pt]    {$3$};
				\draw (348.28,46.67) node [anchor=north west][inner sep=0.75pt]    {$1$};
				\draw (579.62,23.88) node [anchor=north west][inner sep=0.75pt]    {$3$};
				\draw (353.41,66.64) node [anchor=north west][inner sep=0.75pt]    {$1$};
				\draw (577.62,62.83) node [anchor=north west][inner sep=0.75pt]    {$1$};
				\draw (494.57,35.85) node [anchor=north west][inner sep=0.75pt]    {$\textcolor[rgb]{0.72,0.91,0.53}{0}$};
				\draw (457.11,33.45) node [anchor=north west][inner sep=0.75pt]    {$5$};
				\draw (422,34.73) node [anchor=north west][inner sep=0.75pt]    {$\textcolor[rgb]{0.72,0.91,0.53}{0}$};

			\end{tikzpicture}
			
		\end{figure}
		Adding all contributions together and taking into account the sign provided by the parity functions (in green) yields the claim.
\end{proof}

\section{Tropicalising spin completed cycles}\label{section_tropiclspincompletedcycles}

	\subsection{Spin Hurwitz numbers via bosonic Fock space} \label{section_tropicalspincompletedcycles}
	The action of the cut-and-join operators $\mathcal{F}^\mathrm{B}_{r+1}$ has a geometric interpretation in terms of algebraic degeneration (see \cref{section_regeometrizationtypeB}) that is not visible yet.
	Indeed, $\mathcal{F}^\mathrm{B}_{r+1}$ is the coefficient of a certain power of $z$ in a gigantic sum that mixes products of functions in $z$ together with bosonic operators. Rewriting Equation (\ref{equation_bosoniccutandjoin}) as in \cref{lemma_cutandjoinrewrite} is a first step.
	We introduce
	\begin{align*}
		S\mathbb{Z}^{r+2-2g}:=\{ & x\in \mathbb{Z}^{r+2-2g}: x_1 \leqslant...\leqslant x_l < 0< x_{l+1} \leqslant ...\leqslant x_{r+2-2g} ,\\
		 & \sum^{r+2-2g}_{i=1}x_i=0 \text{ and } x_i \text{ odd}\}
	\end{align*} 
	for the set of ordered integer vectors of length $r+2-2g$ with odd entries that sum up to zero. We denote by $x^+$ ($x^-$) the partitions formed by  the positive (negative) part of an element $x\in S\mathbb{Z}^{r+2-2g}$ and write $h^{spin,0,r}_g(x^+,x^-)$ ($h^{spin,\bullet,0,r}_g(x^+,x^-)$) for $h^{spin,0}_g(\mu, \overline{(r+1)}, \nu)$ ($h^{spin,\bullet,0}_g(x^+, \overline{(r+1)},x^-)$), where $x^+=\mu$ and $x^-=\nu$.
	
	\cref{lemma_cutandjoinrewrite} simplifies $\mathcal{F}^\mathrm{B}_{r+1}$ to a weighted sum of \emph{words} in the bosonic operators $a^B_{k_i}$. The next Lemma provides an interpretation for the coefficients appearing in this sum. 
	
	\begin{lemma}\label{lemma_1point}
		The connected spin Hurwitz numbers with one completed cycle are given by:
		\begin{align*}
			h^{spin,0,r}_g(x^+,x^-)=2^{1-g}r![z^{2g}]\rho(z)\frac{\prod_i\mathcal{S}(x^-_iz)\prod_j\mathcal{S}(x^+_jz)}{\mathcal{S}(z)}.
		\end{align*}
	\end{lemma}
	\begin{proof}		
		An explicit computation of the vacuum expectation (see \cite{arXiv:2104.05697}, Proposition 6.1)
		\begin{align}\label{equation_inproofofclosedformularforconnectedcomplcycles}
			h^{spin,\bullet,0,r}_g(x^+,x^-)=\frac{r!\cdot 2^{1-g}}{\prod_{i}x_i^+\prod_j x_j^-}[z^{r+1}]\langle \prod_i\alpha^B_{x^+_i}\mathcal{E}^B_{0}(z)\prod_j\alpha^B_{-x^-_j}\rangle,
		\end{align} 
		is  carried out by using the commutation relations derived from the properties of the B-Okounkov-Pandharipande operators (\cite{arXiv:2104.05697}, Lemma 5.8):
		\begin{align}\label{equation_commutationrelationsinproof}
			& [\alpha^B_{m},\mathcal{E}^B_{n}(z)]=\bigg{(}\frac{1}{2}+ {\big{(}\frac{-1}{2}\big{)}}^{m+1}\bigg{)}\zeta(mz)\mathcal{E}^B_{n+m}(z)\\
			& [\mathcal{E}^B_{m}(z),\alpha^B_{n}]=\bigg{(}\frac{-1}{2}+ {\big{(}\frac{-1}{2}\big{)}}^{n}\bigg{)}\zeta(nz)\mathcal{E}^B_{n+m}(z)\\
			& [\alpha^B_{m},\alpha^B_{n}]=\frac{m}{2}\delta_{-n,m}
		\end{align}
		to reduce (\ref{equation_inproofofclosedformularforconnectedcomplcycles}) to a linear combination of vacuum expectations of the form $\langle\mathcal{E}^B_{m}(z) \rangle=\delta_m\frac{\rho(z)}{\zeta(z)}$.\\
		More precisely:
		1. Start by moving $\mathcal{E}^B_{0}(z)$ to the left by substituting $\alpha^B_{m}\mathcal{E}^B_{0}(z)$ by
		\begin{align*}
			\mathcal{E}^B_{0}(z)\alpha^B_{m} + \underbrace{[\alpha^B_{m},\mathcal{E}^B_{0}(z)]}_{(...)\mathcal{E}^B_{m}(z)}.
		\end{align*}
		This produces two terms: One in which $\mathcal{E}^B_{0}(z)$ absorbs $\alpha^B_m$ increasing its subscript by $m$ and another in which $\mathcal{E}^B_{0}(z)$ is shifted one step further to the left. 
		Repeat until (\ref{equation_inproofofclosedformularforconnectedcomplcycles}) reduces to a linear combination of terms of the form $\langle\mathcal{E}^B_{k}(z) \prod_{i>0}\alpha^B_{i}\prod_{j<0}\alpha^B_{j}\rangle$.
		
		2. Since $\alpha^B$-operators with positive subscript annihilate the vacuum, we move them to the right using $\alpha^B_{i}\alpha^B_{j}=\alpha^B_{j}\alpha^B_{i} + [\alpha^B_{i},\alpha^B_{j}] $. This \emph{always} yields a term in which the operator with a positive subscript is further to the right, eventually resulting in 0 after iteration. \emph{And}, if $i+j$ happens to be equal to $0$, a second term is created in which both cancel, producing a scalar. 
		
		The surviving terms are of the form $c\cdot \langle\mathcal{E}^B_{k}(z) \prod_{j<0}\alpha^B_{j}\rangle$, where $c$ is a scalar and the subscripts sum up to $0$.
		
		3. Finally, move $\mathcal{E}^B_{k}(z)$ to the right using relations (\ref{equation_commutationrelationsinproof}). From (\ref{equation_commutationrelationsinproof}) also note: Non-zero contributions to (\ref{equation_inproofofclosedformularforconnectedcomplcycles}) come from $\mathcal{E}^B_{0}(z)$ absorbing $\alpha^B$-operators whose subscripts sum up to $0$ such that the remaining operators can pairwise cancel.

		However, \emph{non-trivially} mixed-terms (i.e. obtained by pairwise cancellation \emph{and} absorption) contribute to the disconnected number only (see \cite{zbMATH05049020}, Section 3.2). It then follows that the Formula in \cref{lemma_1point}, the connected part of (\ref{equation_inproofofclosedformularforconnectedcomplcycles}), corresponds to $\mathcal{E}^B_{0}(z)$ absorbing the $\alpha^B$-operator in each step.
		
	\end{proof}
	\emph{A sanity check.} Completed cycles spin Hurwitz numbers can be computed in a number of ways, e.g. by using \cref{lemma_1point} or by extracting the connected part of \cref{definition_completedcyclesspinHurwitznumber}. We verify that both agree.
	\begin{example}\label{example_1completedcyclespinHurwitznumber}
		Consider $h^{spin,0,2}_1((3),(3))$. We have
		\begin{align*}
			\rho(z)\frac{(\mathcal{S}(3z))^2}{\mathcal{S}(z)}= \frac{1}{2} + \frac{5}{12}z^2+... \text{ and } h^{spin,0,2}_g((3),(3))=2^{1-1}\cdot 2![z^2](\rho(z)\frac{(\mathcal{S}(3z))^2}{\mathcal{S}(z)})=\frac{5}{6},
		\end{align*}  by \cref{lemma_1point}.
		On the other hand, $h^{spin, \bullet, 0,2}_g((3),(3))$ is a linear combination of ordinary (extended) spin Hurwitz numbers (\cref{definition_extendeddefinitionspinHurwitznumbers}) whose connected part is
		\begin{align*}
			\rho_{(3),(3)} \binom{0}{0}H^0_{\tilde{g}=1}((3),(3),(3)) + (\rho_{(3),(1,1)}\binom{2+1}{1} + \rho_{(3),(1)}\binom{1+2}{2})H^0_{\tilde{g}=0}((3),(1,1,1),(3))
		\end{align*} 
		with coefficients determined by the expansion of $\bar{(3)}$ in ordinary cycles (\cite {Giacchetto_2025}, Lemma 2.11) 
		\begin{align*}
			\bar{(3)}&= \rho_{(3),(3)}(3) + \rho_{(3),(1,1,1)}(1,1,1) + \rho_{(3),(1,1)}(1,1) + \rho_{(3),(1)}(1) +\rho_{(3),\emptyset}\emptyset\\
			&= (3) + (1,1) + \frac{1}{6}(1) -\frac{1}{720} \emptyset .
		\end{align*}
		
		Excluding the empty partition (since we are interested in the connected number) and using \cite{zbMATH06196720} finally yields:
		\begin{align*}
			h^{spin,0,2}_1((3),(3))= \frac{-1}{3} + (1\cdot 3 + \frac{1}{6}\cdot3)\frac{1}{3}=\frac{5}{6}.
		\end{align*} 
	\end{example}
	
	\begin{lemma}\label{lemma_cutandjoinrewrite}
		The operators $\mathcal{F}^\mathrm{B}_{r+1}$ satisfy
		\begin{align*}
			\mathcal{F}^\mathrm{B}_{r+1}=\sum_{g\geqslant 0} \sum\limits_{x\in S\mathbb{Z}^{r+2-2g}} \frac{2^{l(x)-1+g}}{|\Aut(x^+)||\Aut(x^-)|}h^{spin,0,r}_g(x^+,x^-)\prod_i a^B_{x^-_i} \prod_j a^B_{x^+_j} .
		\end{align*}
	\end{lemma}
	\begin{proof} 
		Recall, the right hand side of (\ref{equation_bosoniccutandjoin}) is a sum of words in the $a^B_{k_i}$ that contribute to the coefficient of $z^{r+1}$. Rewriting it as power series in $z$ we see: The coefficient of $z^{r+1}$ is formed by a choice of summand in the series expansion of $\coth(z)$ together with a word in $a^B_{k_i}$. Note that each individual $a^B_{k_i}$ contributes with a power of $z$ coming from a summand of $\zeta(k_iz)$. The series expansions
		\begin{align*}
			& \coth(z)=\frac{1}{z}+\frac{z}{3}+ \mathcal{O}(z^3)/ \text{h.o.t. in odd powers of }z\\
			& \zeta(z)= z + \mathcal{O}(z^3)/ \text{h.o.t. in odd powers of }z
		\end{align*}
		form the basis for the following observation.
		\begin{enumerate}
			\item The longest word in the $a^B_{k_i}$ has length $r+2$: Choose the l.o.t. in each series.
			\item There are words of length $r+1$ given by choosing the $z-$term in each series.
			\item Any further deviation from the choices above gives rise to a decrease of the word length by 2 since $\coth(z)$ and $\zeta(z)$ are odd in $z$.
		\end{enumerate} 
		Let $g$ remember the difference between word length and power of $z$, i.e. 
		\begin{align*}
			(r+2)-2g= (\text{word length}).
		\end{align*}
		Expression (\ref{equation_bosoniccutandjoin}) becomes 
		\begin{align*}
			r!\sum_{g\geqslant 0} \sum\limits_{x\in S\mathbb{Z}^{r+2-2g}}\frac{2^{l(x)-2}}{|\Aut(x^+)||\Aut(x^-)|}[z^{r+1}]\bigg(\coth(z)\frac{\prod_i \zeta({x_i}z)}{\prod_i x_i}\bigg)\prod_i a^B_{x^-_i} \prod_j  a^B_{x^+_j},
		\end{align*}
		where the factor $|\Aut(x^+)||\Aut(x^-)|$ appears by ordering the positive, respectively negative part of $x$.
		Substituting  $\coth(z)=\frac{\cosh(z)}{\sinh(z)}$ and $\mathcal{S}(z)=\frac{\zeta(z)}{z}$ we obtain
		\begin{align*}
			& r!\sum_{g\geqslant 0} \sum\limits_{x\in S\mathbb{Z}^{r+2-2g}}\frac{2^{l(x)-1}}{|\Aut(x^+)||\Aut(x^-)|}[z^{r+1}]\bigg(\cosh(\frac{z}{2})z^{l(x)-1}\frac{\prod_i \mathcal{S}(x_iz)}{\mathcal{S}(z)}\bigg)\prod_i a^B_{x^-_i} \prod_j a^B_{x^+_j} =\\
			& \sum_{g\geqslant 0} \sum\limits_{x\in S\mathbb{Z}^{r+2-2g}}\frac{2^{l(x)-1}}{|\Aut(x^+)||\Aut(x^-)|}[z^{2g}]\bigg(r! \cosh(\frac{z}{2})\frac{\prod_i \mathcal{S}(x_iz)}{\mathcal{S}(z)}\bigg)\prod_i a^B_{x^-_i} \prod_j a^B_{x^+_j}
		\end{align*}
		Using $\rho(z)=\frac{1}{2}\cosh(\frac{z}{2})$ together with \cref{lemma_1point} yields the result.
	\end{proof}
\subsection{A graph-theoretic computation}

Recall the expression of completed spin Hurwitz numbers as vacuum expectation on bosonic Fock space:
		\begin{align}\label{equation_spinhurwitzinbosonicvacuum}
			h^{spin,\bullet,0}_g(\mu, \overline{(r_1 +1)}, \dots, \overline{(r_b +1)}, \nu)=
			\frac{2^{1-g}}{b!\prod_i \mu_i \prod_j \nu_j} \cdot \langle \prod_{i=1}^{l(\mu)}a^B_{\mu_i} \prod_{i=1}^{n}\mathcal{F}^\mathrm{B}_{r_i + 1}\prod_{i=1}^{l(\nu)}a^B_{-\nu_i} \rangle.
		\end{align} 
What follows is based on two observations:
		\begin{itemize}
			\item The vacuum expectation decomposes into a sum of vacuum expectations of the form $\langle m_0 \dotsm_{b+1} \rangle$, where $m_i$ denotes a normally ordered product of bosonic operators, and $m_0$ ($m_{b+1}$) contains only factors $a^B_k$ with $k>0$ ($k<0$).
			\item The only non-zero contribution comes from polynomials $P$ (in the operators $a^B_k$) of the form $P=m_0\dotsm_{b+1}$, where the $m_i$ fit together nicely (see \cref{proposition_spinwick}).
		\end{itemize}
\begin{construction}\label{construction_Feynmanfragments}[compare to \cite{zbMATH07272609}, Definition 5.4.1]\\
			To a normally ordered product of bosonic operators $m:=(:\prod^{l}_{i=1}a^B_{k_i}:)$ we associated a vertex $v$ with an edge germ of weight $k_i$ for each $a^B_{k_i}$ whose orientation is determined by the sign of $k_i$. If $k_i<0$, the edge germ points towards the left, and if $k_i>0$, towards the right. Each edge germ is provided with a marking (to keep track later on of all the ways different edge germs can be glued together). This is called the \emph{Feynman fragment} associated to $m$ (see \cref{figure_Feynmanfragments}).

			To a polynomial $P=m_0\dotsm_{b+1}$ as above we associate a collection of ordered Feynman fragments $(v_1,\dots,v_b)$ (one for each $m_i$ with $i\neq 0, b+1$) together with a set of left and a set of right pointing ends corresponding to the monomials $m_0$ and $m_{b+1}$.
\end{construction}

\subsection{Via Feynman diagrams}
A collection of Feynman fragments can be merged to a \emph{Feynman graph}, if edge germs can be paired together to form an edge. More precisely, if for each right pointing edge germ, there exists a left pointing edge germ belonging to a fragment of equal weight further to the right. The resulting graph will have edges labeled by tuples of germ markings (see \cref{figure_Feynmanfragments}). This "graphical condition" that allows Feynman fragments to be merged to Feynman graphs characterizes precisely the condition of $P$ making a non-zero contribution. 
		
To keep track of this merging process, we mark all edge germs.
\begin{figure}[H]
\centering
\tikzset{every picture/.style={line width=0.75pt}} 
\scalebox{0.85}{ 
\begin{tikzpicture}[x=0.75pt,y=0.75pt,yscale=-1,xscale=0.9]
					
					\draw    (81.51,106.25) .. controls (88.46,64.95) and (96.04,63.11) .. (146.27,63.11) ;
					\draw [shift={(107.02,65.49)}, rotate = 162.09] [color={rgb, 255:red, 0; green, 0; blue, 0 }  ][line width=0.75]    (10.93,-3.29) .. controls (6.95,-1.4) and (3.31,-0.3) .. (0,0) .. controls (3.31,0.3) and (6.95,1.4) .. (10.93,3.29)   ;
					\draw [color={rgb, 255:red, 0; green, 0; blue, 0 }  ,draw opacity=1 ]   (81.51,106.25) .. controls (100.78,86.06) and (94.14,92.49) .. (145.32,90.65) ;
					\draw [shift={(116.88,91.02)}, rotate = 180.05] [color={rgb, 255:red, 0; green, 0; blue, 0 }  ,draw opacity=1 ][line width=0.75]    (10.93,-3.29) .. controls (6.95,-1.4) and (3.31,-0.3) .. (0,0) .. controls (3.31,0.3) and (6.95,1.4) .. (10.93,3.29)   ;
					\draw    (81.51,106.25) .. controls (95.09,124.61) and (96.99,126.45) .. (146.27,127.37) ;
					\draw [shift={(116.4,126.19)}, rotate = 186.18] [color={rgb, 255:red, 0; green, 0; blue, 0 }  ][line width=0.75]    (10.93,-3.29) .. controls (6.95,-1.4) and (3.31,-0.3) .. (0,0) .. controls (3.31,0.3) and (6.95,1.4) .. (10.93,3.29)   ;
					\draw    (81.51,106.25) ;
					\draw [shift={(81.51,106.25)}, rotate = 0] [color={rgb, 255:red, 0; green, 0; blue, 0 }  ][fill={rgb, 255:red, 0; green, 0; blue, 0 }  ][line width=0.75]      (0, 0) circle [x radius= 3.35, y radius= 3.35]   ;
					\draw    (81.51,106.25) .. controls (63.81,86.06) and (66.66,86.98) .. (15.48,86.98) ;
					\draw [shift={(58.81,88.48)}, rotate = 188.34] [color={rgb, 255:red, 0; green, 0; blue, 0 }  ][line width=0.75]    (10.93,-3.29) .. controls (6.95,-1.4) and (3.31,-0.3) .. (0,0) .. controls (3.31,0.3) and (6.95,1.4) .. (10.93,3.29)   ;
					\draw    (81.51,106.25) .. controls (66.66,124.61) and (66.66,124.61) .. (19.27,126.45) ;
					\draw [shift={(60.39,123.18)}, rotate = 169.19] [color={rgb, 255:red, 0; green, 0; blue, 0 }  ][line width=0.75]    (10.93,-3.29) .. controls (6.95,-1.4) and (3.31,-0.3) .. (0,0) .. controls (3.31,0.3) and (6.95,1.4) .. (10.93,3.29)   ;
					\draw    (81.51,106.25) -- (145.32,104.42) ;
					\draw [shift={(119.41,105.16)}, rotate = 178.35] [color={rgb, 255:red, 0; green, 0; blue, 0 }  ][line width=0.75]    (10.93,-3.29) .. controls (6.95,-1.4) and (3.31,-0.3) .. (0,0) .. controls (3.31,0.3) and (6.95,1.4) .. (10.93,3.29)   ;
					\draw    (218.33,73.86) -- (256.86,73.86) ;
					\draw    (220.08,114.33) -- (256.86,114.33) ;
					\draw    (345.86,92.34) .. controls (352.29,52.75) and (349.66,50.99) .. (396.06,50.99) ;
					\draw [shift={(362.87,53.19)}, rotate = 155.61] [color={rgb, 255:red, 0; green, 0; blue, 0 }  ][line width=0.75]    (10.93,-3.29) .. controls (6.95,-1.4) and (3.31,-0.3) .. (0,0) .. controls (3.31,0.3) and (6.95,1.4) .. (10.93,3.29)   ;
					\draw [color={rgb, 255:red, 0; green, 0; blue, 0 }  ,draw opacity=1 ]   (345.86,92.34) .. controls (363.67,72.98) and (347.03,75.62) .. (394.31,73.86) ;
					\draw [shift={(372.44,74.62)}, rotate = 176.67] [color={rgb, 255:red, 0; green, 0; blue, 0 }  ,draw opacity=1 ][line width=0.75]    (10.93,-3.29) .. controls (6.95,-1.4) and (3.31,-0.3) .. (0,0) .. controls (3.31,0.3) and (6.95,1.4) .. (10.93,3.29)   ;
					\draw    (345.86,92.34) .. controls (358.41,109.93) and (347.03,120.49) .. (392.56,121.37) ;
					\draw [shift={(368.5,119.1)}, rotate = 197.92] [color={rgb, 255:red, 0; green, 0; blue, 0 }  ][line width=0.75]    (10.93,-3.29) .. controls (6.95,-1.4) and (3.31,-0.3) .. (0,0) .. controls (3.31,0.3) and (6.95,1.4) .. (10.93,3.29)   ;
					\draw    (345.86,92.34) ;
					\draw [shift={(345.86,92.34)}, rotate = 0] [color={rgb, 255:red, 0; green, 0; blue, 0 }  ][fill={rgb, 255:red, 0; green, 0; blue, 0 }  ][line width=0.75]      (0, 0) circle [x radius= 3.35, y radius= 3.35]   ;
					\draw    (345.86,92.34) .. controls (329.52,72.98) and (332.15,73.86) .. (284.87,73.86) ;
					\draw [shift={(325.61,75.47)}, rotate = 189.17] [color={rgb, 255:red, 0; green, 0; blue, 0 }  ][line width=0.75]    (10.93,-3.29) .. controls (6.95,-1.4) and (3.31,-0.3) .. (0,0) .. controls (3.31,0.3) and (6.95,1.4) .. (10.93,3.29)   ;
					\draw    (345.86,92.34) .. controls (332.15,109.93) and (332.15,109.93) .. (288.37,111.69) ;
					\draw [shift={(327.06,108.37)}, rotate = 168.15] [color={rgb, 255:red, 0; green, 0; blue, 0 }  ][line width=0.75]    (10.93,-3.29) .. controls (6.95,-1.4) and (3.31,-0.3) .. (0,0) .. controls (3.31,0.3) and (6.95,1.4) .. (10.93,3.29)   ;
					\draw    (345.86,92.34) -- (394.31,94.1) ;
					\draw [shift={(376.08,93.44)}, rotate = 182.08] [color={rgb, 255:red, 0; green, 0; blue, 0 }  ][line width=0.75]    (10.93,-3.29) .. controls (6.95,-1.4) and (3.31,-0.3) .. (0,0) .. controls (3.31,0.3) and (6.95,1.4) .. (10.93,3.29)   ;
					\draw    (513.09,94.1) .. controls (519.51,54.51) and (473.98,50.99) .. (456.47,50.99) ;
					\draw [shift={(502.29,62.81)}, rotate = 213.71] [color={rgb, 255:red, 0; green, 0; blue, 0 }  ][line width=0.75]    (10.93,-3.29) .. controls (6.95,-1.4) and (3.31,-0.3) .. (0,0) .. controls (3.31,0.3) and (6.95,1.4) .. (10.93,3.29)   ;
					\draw [color={rgb, 255:red, 0; green, 0; blue, 0 }  ,draw opacity=1 ]   (513.09,94.1) .. controls (530.89,74.74) and (524.76,80.9) .. (572.04,79.14) ;
					\draw [shift={(546.16,79.5)}, rotate = 180.06] [color={rgb, 255:red, 0; green, 0; blue, 0 }  ,draw opacity=1 ][line width=0.75]    (10.93,-3.29) .. controls (6.95,-1.4) and (3.31,-0.3) .. (0,0) .. controls (3.31,0.3) and (6.95,1.4) .. (10.93,3.29)   ;
					\draw    (513.09,94.1) .. controls (525.63,111.69) and (527.39,113.45) .. (572.91,114.33) ;
					\draw [shift={(545.72,113.24)}, rotate = 186.41] [color={rgb, 255:red, 0; green, 0; blue, 0 }  ][line width=0.75]    (10.93,-3.29) .. controls (6.95,-1.4) and (3.31,-0.3) .. (0,0) .. controls (3.31,0.3) and (6.95,1.4) .. (10.93,3.29)   ;
					\draw    (513.09,94.1) ;
					\draw [shift={(513.09,94.1)}, rotate = 0] [color={rgb, 255:red, 0; green, 0; blue, 0 }  ][fill={rgb, 255:red, 0; green, 0; blue, 0 }  ][line width=0.75]      (0, 0) circle [x radius= 3.35, y radius= 3.35]   ;
					\draw    (513.09,94.1) .. controls (496.74,74.74) and (504.62,73.86) .. (457.35,73.86) ;
					\draw [shift={(495.89,76.35)}, rotate = 193.67] [color={rgb, 255:red, 0; green, 0; blue, 0 }  ][line width=0.75]    (10.93,-3.29) .. controls (6.95,-1.4) and (3.31,-0.3) .. (0,0) .. controls (3.31,0.3) and (6.95,1.4) .. (10.93,3.29)   ;
					\draw    (513.09,94.1) .. controls (499.37,111.69) and (501.12,117.85) .. (457.35,119.61) ;
					\draw [shift={(496.08,113.03)}, rotate = 156.54] [color={rgb, 255:red, 0; green, 0; blue, 0 }  ][line width=0.75]    (10.93,-3.29) .. controls (6.95,-1.4) and (3.31,-0.3) .. (0,0) .. controls (3.31,0.3) and (6.95,1.4) .. (10.93,3.29)   ;
					\draw    (513.09,94.1) -- (457.35,94.1) ;
					\draw [shift={(492.22,94.1)}, rotate = 180] [color={rgb, 255:red, 0; green, 0; blue, 0 }  ][line width=0.75]    (10.93,-3.29) .. controls (6.95,-1.4) and (3.31,-0.3) .. (0,0) .. controls (3.31,0.3) and (6.95,1.4) .. (10.93,3.29)   ;
					\draw    (605.89,79.14) -- (648.33,79) ;
					\draw    (608.52,112.57) -- (646.33,112) ;
					\draw [color={rgb, 255:red, 184; green, 233; blue, 134 }  ,draw opacity=1 ] [dash pattern={on 0.84pt off 2.51pt}]  (396.06,50.99) -- (460.85,50.11) ;
					\draw [color={rgb, 255:red, 184; green, 233; blue, 134 }  ,draw opacity=1 ] [dash pattern={on 0.84pt off 2.51pt}]  (394.31,73.86) -- (457.35,73.86) ;
					\draw [color={rgb, 255:red, 184; green, 233; blue, 134 }  ,draw opacity=1 ] [dash pattern={on 0.84pt off 2.51pt}]  (394.31,94.1) -- (457.35,119.61) ;
					\draw [color={rgb, 255:red, 184; green, 233; blue, 134 }  ,draw opacity=1 ] [dash pattern={on 0.84pt off 2.51pt}]  (392.56,121.37) -- (457.35,94.1) ;
					\draw [color={rgb, 255:red, 184; green, 233; blue, 134 }  ,draw opacity=1 ] [dash pattern={on 0.84pt off 2.51pt}]  (256.86,73.86) -- (284.87,73.86) ;
					\draw [color={rgb, 255:red, 184; green, 233; blue, 134 }  ,draw opacity=1 ] [dash pattern={on 0.84pt off 2.51pt}]  (256.86,114.33) -- (288.37,111.69) ;
					\draw [color={rgb, 255:red, 184; green, 233; blue, 134 }  ,draw opacity=1 ] [dash pattern={on 0.84pt off 2.51pt}]  (572.04,79.14) -- (605.89,79.14) ;
					\draw [color={rgb, 255:red, 184; green, 233; blue, 134 }  ,draw opacity=1 ] [dash pattern={on 0.84pt off 2.51pt}]  (572.91,114.33) -- (608.52,112.57) ;
					
					\draw (23.19,167.43) node [anchor=north west][inner sep=0.75pt]    {$m=a^B_{-5} a^B_{-3} (a^B_{1})^{2}(a^B _{3})^{2}$};
					\draw (4.63,75.71) node [anchor=north west][inner sep=0.75pt]    {$\textcolor[rgb]{0.53,0.96,0.07}{1}$};
					\draw (20.54,65.61) node [anchor=north west][inner sep=0.75pt]    {$3$};
					\draw (19.59,104.16) node [anchor=north west][inner sep=0.75pt]    {$5$};
					\draw (120.05,68.36) node [anchor=north west][inner sep=0.75pt]    {$3$};
					\draw (119.1,43.58) node [anchor=north west][inner sep=0.75pt]    {$3$};
					\draw (159.86,51.84) node [anchor=north west][inner sep=0.75pt]    {$\textcolor[rgb]{0.53,0.96,0.07}{1}$};
					\draw (2.63,115.18) node [anchor=north west][inner sep=0.75pt]    {$\textcolor[rgb]{0.53,0.96,0.07}{2}$};
					\draw (159.86,76.62) node [anchor=north west][inner sep=0.75pt]    {$\textcolor[rgb]{0.53,0.96,0.07}{2}$};
					\draw (160.8,97.74) node [anchor=north west][inner sep=0.75pt]    {$\textcolor[rgb]{0.53,0.96,0.07}{3}$};
					\draw (162.7,117.93) node [anchor=north west][inner sep=0.75pt]    {$\textcolor[rgb]{0.53,0.96,0.07}{4}$};
					\draw (261.57,53.95) node [anchor=north west][inner sep=0.75pt]    {$\textcolor[rgb]{0.53,0.96,0.07}{11}$};
					\draw (289.09,53.07) node [anchor=north west][inner sep=0.75pt]    {$3$};
					\draw (299.01,115.53) node [anchor=north west][inner sep=0.75pt]    {$5$};
					\draw (371.38,59.23) node [anchor=north west][inner sep=0.75pt]    {$3$};
					\draw (373.13,31.96) node [anchor=north west][inner sep=0.75pt]    {$3$};
					\draw (413.03,34.6) node [anchor=north west][inner sep=0.75pt]    {$\textcolor[rgb]{0.53,0.96,0.07}{11}$};
					\draw (261.57,92.66) node [anchor=north west][inner sep=0.75pt]    {$\textcolor[rgb]{0.53,0.96,0.07}{22}$};
					\draw (414.91,84.74) node [anchor=north west][inner sep=0.75pt]    {$\textcolor[rgb]{0.53,0.96,0.07}{34}$};
					\draw (413.91,113.77) node [anchor=north west][inner sep=0.75pt]    {$\textcolor[rgb]{0.53,0.96,0.07}{43}$};
					\draw (456.31,54.83) node [anchor=north west][inner sep=0.75pt]    {$3$};
					\draw (549.99,117.29) node [anchor=north west][inner sep=0.75pt]    {$5$};
					\draw (548.23,57.47) node [anchor=north west][inner sep=0.75pt]    {$3$};
					\draw (456.68,30.2) node [anchor=north west][inner sep=0.75pt]    {$3$};
					\draw (413.91,56.59) node [anchor=north west][inner sep=0.75pt]    {$\textcolor[rgb]{0.53,0.96,0.07}{22}$};
					\draw (493.95,129.6) node [anchor=north west][inner sep=0.75pt]    {$\textcolor[rgb]{0.53,0.96,0.07}{3}$};
					\draw (226.05,51.31) node [anchor=north west][inner sep=0.75pt]    {$3$};
					\draw (228.38,115.97) node [anchor=north west][inner sep=0.75pt]    {$5$};
					\draw (627.03,56.59) node [anchor=north west][inner sep=0.75pt]    {$3$};
					\draw (629.42,115.69) node [anchor=north west][inner sep=0.75pt]    {$5$};
					\draw (581.13,56.59) node [anchor=north west][inner sep=0.75pt]    {$\textcolor[rgb]{0.53,0.96,0.07}{11}$};
					\draw (582.01,91.78) node [anchor=north west][inner sep=0.75pt]    {$\textcolor[rgb]{0.53,0.96,0.07}{22}$};
					\draw (298.19,168.43) node [anchor=north west][inner sep=0.75pt]    {$\langle P:=(a^B _{5} a^B _{3}) (a^B _{-5} a^B _{-3} (a^B_{1})^{2}(a^B _{3})^{2})((a^B _{-1})^{2} (a^B _{-3})^{2}a^B _{5} a^B _{3})( a^B _{-5}a^B_{-3} )\rangle $};				
\end{tikzpicture}}
			
\caption{On the left, a normally ordered product and its Feynman fragment. On the right, a vacuum expectation and Feynman graph contributing to it. The black numbers are edge weights, the green ones markings.}
			\label{figure_Feynmanfragments}
\end{figure}
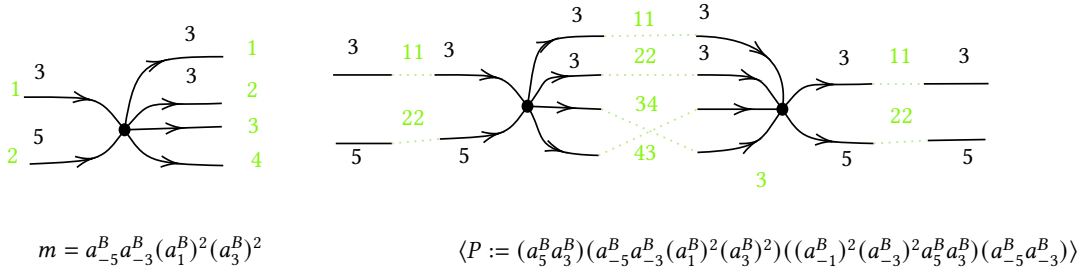

		\begin{proposition}[Graphical condition, Wick's theorem of type B]\label{proposition_spinwick}
			Let  $P$ be a polynomial appearing in the linear expansion of (\ref{equation_spinhurwitzinbosonicvacuum}). The vacuum expectation $\langle P \rangle$ is non-zero if and only if the collection of Feynman fragments associated to $P$ can be connected to a graph. Moreover, $\langle P \rangle$ is equal to the weighted sum of all Feynman graphs associated to $P$, where each graph contributes with the product over all edge weights (including ends) each counted with a factor of $\frac{1}{2}$. 
		\end{proposition}
		
		\begin{proof}
			The proof is analogous to the proof of the type A version of Wick's theorem (see \cite{zbMATH07272609}, Proposition 5.4.3). The only difference consists in the factors of $\frac{1}{2}$ that come from the commutation relations of the type B bosonic operators. 
		\end{proof}
		\begin{remark}
			The connected part $\langle P\rangle^\circ$ of the vacuum expectation  $\langle P\rangle$ is defined by subtracting all possible products of vacuum expectations that are obtained by factorizing $P$ (\cite{zbMATH05049020}, Section 3.3.1).  A Feynman graph that contributes to such a product of vacuum expectations is disconnected, with connected components  corresponding to the individual factors. Thus, we only need to consider connected Feynman graphs when computing connected invariants.
		\end{remark}
		We can immediately transfer this "graphical tracking-technique" for the calculation of vacuum expectations to the case of spin Hurwitz numbers.
		
		\begin{proposition}\label{proposition_spinHequalweightedFeynm}
			The spin Hurwitz number $h^{spin,0}_g(\mu, \overline{(r + 1)}^b, \nu)$ is equal to the weighted count of \emph{connected} Feynman graphs associated to monomials $P$, that appear in the linear expansion of (\ref{equation_spinhurwitzinbosonicvacuum}). Each such graph contributes with the product of its edge weights (without a factor of $\frac{1}{2}$) times a factor of $\frac{h^{spin,0,r}_{g}(x_i^+,x_i^-)}{|\Aut(x^+)||\Aut(x^-)|}$ for every vertex $v_i$, where $x_i^-$ ($x_i^+$) denotes the partition of weights of ingoing (outgoing edges).
		\end{proposition}
		\begin{proof}
			Linear expansion of the right-hand side of Equation (\ref{equation_spinhurwitzinbosonicvacuum}) leaves us with a weighted sum of vacuum expectations of monomials $P=m_0\dotsm_{b+1}$ as in \cref{proposition_spinwick}. 
			Each of the $b$ cut-and-join operator $\mathcal{F}^\mathrm{B}_{r+1}$ donates a normally ordered product of bosonic operators $m_i$ and a coefficient of the form $2^{r+1-g_i}\frac{h^{spin,0,r}_{g}(x_i^+,x_i^-)}{|\Aut(x^+)||\Aut(x^-)|}$. Hence, $P$ appears with coefficient $\prod^b_{i=1}2^{r+1-g_i}\frac{h^{spin,0,r}_{g}(x_i^+,x_i^-)}{|\Aut(x^+)||\Aut(x^-)|}$ in the expansion of (\ref{equation_spinhurwitzinbosonicvacuum}). Let $G$ be a Feynman graph associated to $P$. To facilitate later computations we attach the number $g(v_i):=g_i$ to the vertex $v_i$ corresponding to $m_i$. Taking into consideration \cref{proposition_spinwick} we have to count a graph $G$ with weight
			\begin{align*}
				\underbrace{\prod_{e\in E(G)} \frac{\omega(e)}{2}}_{\text{commutation relation}}\underbrace{\prod_{v_i \in V(G)} 2^{r+1-g(v_i)}\frac{h^{spin,0,r}_{g}(x_i^+,x_i^-)}{|\Aut(x^+)||\Aut(x^-)|}}_{\text{operator contribution}}.
			\end{align*} 
			We claim that factors of $2$ cancel with the prefactor $2^{1-g}$ in (\ref{equation_spinhurwitzinbosonicvacuum}). A look into the proof of \cref{lemma_cutandjoinrewrite} tells us that $r+1-g(v_i)$ is equal to $\val(v_i)-1+g(v_i)$. We get:
			\begin{align*}
				\frac{1}{2^{|E(G)|}}\prod_{v_i \in V(G)} 2^{r+1-g(v_i)}= \frac{1}{2^{|E(G)|}} \prod_{v \in V(G)} 2^{\val(v)-1+g(v)}=\frac{1}{2^{|E(G)|}}2^ {2|E(G)|-|V(G)|+\sum_{v\in V(G)} g(v)}.
			\end{align*}
			Recall that the first Betti number of $G$, $b_1(G)$, satisfies the following identity:\\ $|V(G)|-|E(G)|=1-b_1(G)$. 
			With $|V(G)|=b$ and
			\begin{align*}
				|E(G)|&=\frac{1}{2} (\sum^b_{i=1}l(x_i^-,x_i^+)-l(\mu)-l(\nu))=
				\frac{1}{2}(b(r+2)-2\sum^b_{i=1}g(v_i)-l(\mu)-l(\nu))\\
				&=g-1+b-2\sum^b_{i=1}g(v_i),
			\end{align*} we obtain
			\begin{align*}
				1-b_1(G)=|V(G)|-|E(G)|= g-\sum^b_{i=1}g(v_i)
			\end{align*}
			and hence $|E(G)|-|V(G)|+\sum^b_{i=1} g(v)=g-1$. This proves the claim.
		\end{proof}
		\begin{remark}
			Note, \cref{proposition_spinHequalweightedFeynm} generalizes immediately to spin Hurwitz numbers $h^{spin,0}_g(\mu, \overline{(r_1 +1)},\dots,\overline{(r_b +1)}, \nu)$ and to the disconnected count  $h^{spin,\bullet,0}_g(\mu, \overline{(r_1 +1)},\dots,\overline{(r_b +1)}, \nu)$ by removing the connectedness condition. 
		\end{remark}
		\begin{example}\label{example_tropicalcomputation_cpltdspin}
			\cref{figure_Feynmangraphsinexampled=3} shows all Feynman graphs that count towards (connected) $h^{spin,0}_{g=2}((3), \overline{(3)},\overline{(3)}, (3))$. Note that the second graph occurs $6$ times with markings $\{i_1,i_2,i_3\}=\{j_1,j_2,j_3\}=\{1,2,3\}$.
			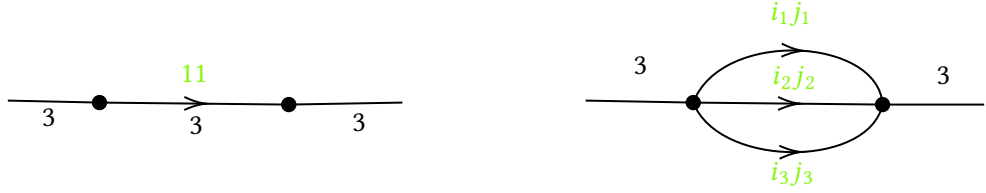
\begin{figure}[H]
				\tikzset{every picture/.style={line width=0.75pt}} 
				
				\begin{tikzpicture}[x=0.75pt,y=0.75pt,yscale=-1,xscale=1]
					
					\draw [color={rgb, 255:red, 0; green, 0; blue, 0 }  ,draw opacity=1 ]   (89.33,98) -- (135.33,99) ;
					\draw    (135.33,99) ;
					\draw [shift={(135.33,99)}, rotate = 0] [color={rgb, 255:red, 0; green, 0; blue, 0 }  ][fill={rgb, 255:red, 0; green, 0; blue, 0 }  ][line width=0.75]      (0, 0) circle [x radius= 3.35, y radius= 3.35]   ;
					\draw    (135.33,99) -- (230.33,100) ;
					\draw [shift={(230.33,100)}, rotate = 0.6] [color={rgb, 255:red, 0; green, 0; blue, 0 }  ][fill={rgb, 255:red, 0; green, 0; blue, 0 }  ][line width=0.75]      (0, 0) circle [x radius= 3.35, y radius= 3.35]   ;
					\draw [shift={(188.83,99.56)}, rotate = 180.6] [color={rgb, 255:red, 0; green, 0; blue, 0 }  ][line width=0.75]    (10.93,-3.29) .. controls (6.95,-1.4) and (3.31,-0.3) .. (0,0) .. controls (3.31,0.3) and (6.95,1.4) .. (10.93,3.29)   ;
					\draw [color={rgb, 255:red, 0; green, 0; blue, 0 }  ,draw opacity=1 ]   (230.33,100) -- (287.33,99) ;
					\draw [color={rgb, 255:red, 0; green, 0; blue, 0 }  ,draw opacity=1 ]   (379.33,98) -- (433.33,99) ;
					\draw    (433.33,99) ;
					\draw [shift={(433.33,99)}, rotate = 0] [color={rgb, 255:red, 0; green, 0; blue, 0 }  ][fill={rgb, 255:red, 0; green, 0; blue, 0 }  ][line width=0.75]      (0, 0) circle [x radius= 3.35, y radius= 3.35]   ;
					\draw    (433.33,99) -- (528.33,100) ;
					\draw [shift={(528.33,100)}, rotate = 0.6] [color={rgb, 255:red, 0; green, 0; blue, 0 }  ][fill={rgb, 255:red, 0; green, 0; blue, 0 }  ][line width=0.75]      (0, 0) circle [x radius= 3.35, y radius= 3.35]   ;
					\draw [shift={(486.83,99.56)}, rotate = 180.6] [color={rgb, 255:red, 0; green, 0; blue, 0 }  ][line width=0.75]    (10.93,-3.29) .. controls (6.95,-1.4) and (3.31,-0.3) .. (0,0) .. controls (3.31,0.3) and (6.95,1.4) .. (10.93,3.29)   ;
					\draw    (528.33,100) .. controls (526.33,66) and (447.33,62) .. (433.33,99) ;
					\draw [shift={(488.58,73.04)}, rotate = 181.88] [color={rgb, 255:red, 0; green, 0; blue, 0 }  ][line width=0.75]    (10.93,-3.29) .. controls (6.95,-1.4) and (3.31,-0.3) .. (0,0) .. controls (3.31,0.3) and (6.95,1.4) .. (10.93,3.29)   ;
					\draw    (433.33,99) .. controls (446.33,132) and (521.33,132) .. (528.33,100) ;
					\draw [shift={(487.53,123.78)}, rotate = 179.24] [color={rgb, 255:red, 0; green, 0; blue, 0 }  ][line width=0.75]    (10.93,-3.29) .. controls (6.95,-1.4) and (3.31,-0.3) .. (0,0) .. controls (3.31,0.3) and (6.95,1.4) .. (10.93,3.29)   ;
					\draw [color={rgb, 255:red, 0; green, 0; blue, 0 }  ,draw opacity=1 ]   (528.33,100) -- (579.33,100) ;
					
					\draw (179,104.4) node [anchor=north west][inner sep=0.75pt]    {$3$};
					\draw (260.83,102.9) node [anchor=north west][inner sep=0.75pt]    {$3$};
					\draw (105,100.4) node [anchor=north west][inner sep=0.75pt]    {$3$};
					\draw (554,79.4) node [anchor=north west][inner sep=0.75pt]    {$3$};
					\draw (402,74.4) node [anchor=north west][inner sep=0.75pt]    {$3$};
					\draw (174.57,77.95) node [anchor=north west][inner sep=0.75pt]    {$\textcolor[rgb]{0.53,0.96,0.07}{11}$};
					\draw (470.57,46.95) node [anchor=north west][inner sep=0.75pt]    {$\textcolor[rgb]{0.53,0.96,0.07}{i_{1} j_{1}}$};
					\draw (471.57,79.4) node [anchor=north west][inner sep=0.75pt]    {$\textcolor[rgb]{0.53,0.96,0.07}{i}\textcolor[rgb]{0.53,0.96,0.07}{_{2}}\textcolor[rgb]{0.53,0.96,0.07}{j}\textcolor[rgb]{0.53,0.96,0.07}{_{2}}$};
					\draw (470.57,127.4) node [anchor=north west][inner sep=0.75pt]    {$\textcolor[rgb]{0.53,0.96,0.07}{i}\textcolor[rgb]{0.53,0.96,0.07}{_{3}}\textcolor[rgb]{0.53,0.96,0.07}{j}\textcolor[rgb]{0.53,0.96,0.07}{_{3}}$};

				\end{tikzpicture}
				
				\caption{Feynman graphs contributing to $h^{spin,0}_{g=2}((3), \overline{(3)},\overline{(3)}, (3))$. The first graph has vertices of genus 1, while the second has vertices of genus 0. To avoid overloading the drawing, these have been omitted.}
				\label{figure_Feynmangraphsinexampled=3}
			\end{figure}
			
		\end{example}

\subsection{Via tropical covers}\label{section_regeometrizationtypeB}
		Let $\mathbb{TP}^1$ be the tropical line with $b$ 
		vertices and $G$ a connected Feynman graph associated to a polynomial $P=m_0\dotsm_{b+1}$ that appears in the expansion of (\ref{equation_spinhurwitzinbosonicvacuum}). Upgrade $G$ to a tropical cover $\pi: \Gamma:=(G,l,g)\rightarrow \mathbb{TP}^1$, where
		\begin{itemize}
			\item the metric $l$ is determined by the metric on $\mathbb{TP}^1$.
			\item the genus function is defined as in the proof of \cref{proposition_spinHequalweightedFeynm} and we forget the marking on the edges of $G$.
		\end{itemize}
		Note that $\Gamma$ is balanced since edge weights at every vertex are given by a vector $(x^-,x^+)\in \mathcal{S}\mathbb{Z}^{r+2-2g(v)}$ and has by \cref{proposition_spinHequalweightedFeynm} genus $g(\Gamma)=b_1(G)+\sum_{v\in V(G)} g(v)=g$. We call $\pi$ a \emph{tropical spin completed cycles Hurwitz cover}. 
		\begin{example}\label{example_upgradedfeynmangraphs}
			The left hand side of \cref{figure_upgradedfeynmandiagrams} shows all tropical spin completed cycles Hurwitz covers that contribute to the connected count $h^{spin,0}_{g=2}((3), \overline{(3)},\overline{(3)}, (3))$, obtained by upgrading the Feynman diagrams from \cref{example_tropicalcomputation_cpltdspin}. Linearly expanding the local completed cycle Hurwitz numbers (\cref{definition_completedcyclesspinHurwitznumber}) at each vertex, yields the ordinary spin Hurwitz covers on the right side of \cref{figure_upgradedfeynmandiagrams} (coefficients of the linear expansion have been omitted). Following \cite{zbMATH07272609}, we add a fat end for each completed $3$-cycle to differentiate between the two.

			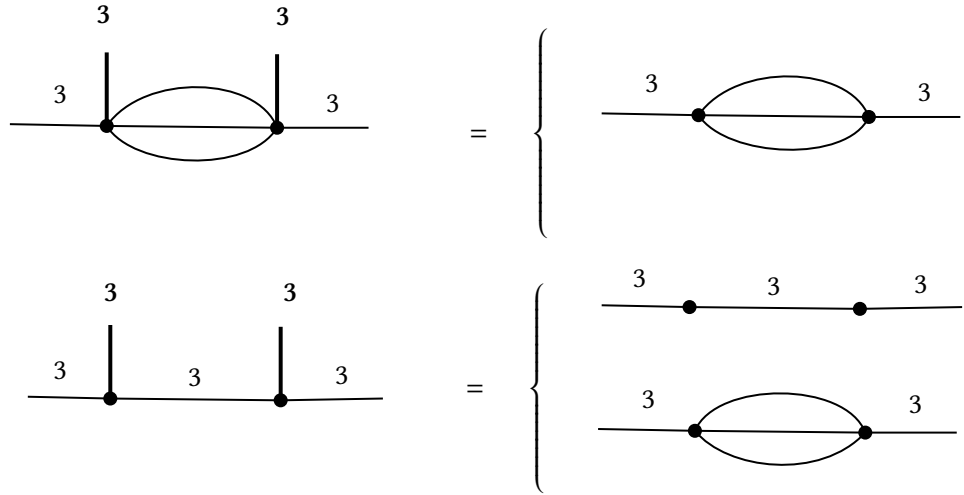
\begin{figure}[H]
				\centering
				\tikzset{every picture/.style={line width=0.75pt}} 
				
				\begin{tikzpicture}[x=0.75pt,y=0.75pt,yscale=-0.9,xscale=0.9]
					
					\draw [color={rgb, 255:red, 0; green, 0; blue, 0 }  ,draw opacity=1 ]   (26.33,88) -- (80.33,89) ;
					\draw    (80.33,89) ;
					\draw [shift={(80.33,89)}, rotate = 0] [color={rgb, 255:red, 0; green, 0; blue, 0 }  ][fill={rgb, 255:red, 0; green, 0; blue, 0 }  ][line width=0.75]      (0, 0) circle [x radius= 3.35, y radius= 3.35]   ;
					\draw    (80.33,89) -- (175.33,90) ;
					\draw [shift={(175.33,90)}, rotate = 0.6] [color={rgb, 255:red, 0; green, 0; blue, 0 }  ][fill={rgb, 255:red, 0; green, 0; blue, 0 }  ][line width=0.75]      (0, 0) circle [x radius= 3.35, y radius= 3.35]   ;
					\draw    (175.33,90) .. controls (164.33,61) and (101.33,59) .. (80.33,89) ;
					\draw    (80.33,89) .. controls (95.33,113) and (159.33,116) .. (175.33,90) ;
					\draw [color={rgb, 255:red, 0; green, 0; blue, 0 }  ,draw opacity=1 ]   (175.33,90) -- (226.33,90) ;
					\draw [color={rgb, 255:red, 0; green, 0; blue, 0 }  ,draw opacity=1 ]   (36.33,240) -- (82.33,241) ;
					\draw    (82.33,241) ;
					\draw [shift={(82.33,241)}, rotate = 0] [color={rgb, 255:red, 0; green, 0; blue, 0 }  ][fill={rgb, 255:red, 0; green, 0; blue, 0 }  ][line width=0.75]      (0, 0) circle [x radius= 3.35, y radius= 3.35]   ;
					\draw    (82.33,241) -- (177.33,242) ;
					\draw [shift={(177.33,242)}, rotate = 0.6] [color={rgb, 255:red, 0; green, 0; blue, 0 }  ][fill={rgb, 255:red, 0; green, 0; blue, 0 }  ][line width=0.75]      (0, 0) circle [x radius= 3.35, y radius= 3.35]   ;
					\draw [color={rgb, 255:red, 0; green, 0; blue, 0 }  ,draw opacity=1 ]   (177.33,242) -- (234.33,241) ;
					\draw [line width=1.5]    (175.33,49) -- (175.33,90) ;
					\draw [line width=1.5]    (177.33,201) -- (177.33,242) ;
					\draw [line width=1.5]    (80.33,48) -- (80.33,89) ;
					\draw [line width=1.5]    (82.33,200) -- (82.33,241) ;
					\draw [color={rgb, 255:red, 0; green, 0; blue, 0 }  ,draw opacity=1 ]   (356.33,82) -- (410.33,83) ;
					\draw    (410.33,83) ;
					\draw [shift={(410.33,83)}, rotate = 0] [color={rgb, 255:red, 0; green, 0; blue, 0 }  ][fill={rgb, 255:red, 0; green, 0; blue, 0 }  ][line width=0.75]      (0, 0) circle [x radius= 3.35, y radius= 3.35]   ;
					\draw    (410.33,83) -- (505.33,84) ;
					\draw [shift={(505.33,84)}, rotate = 0.6] [color={rgb, 255:red, 0; green, 0; blue, 0 }  ][fill={rgb, 255:red, 0; green, 0; blue, 0 }  ][line width=0.75]      (0, 0) circle [x radius= 3.35, y radius= 3.35]   ;
					\draw    (505.33,84) .. controls (494.33,55) and (431.33,53) .. (410.33,83) ;
					\draw    (410.33,83) .. controls (426.33,107) and (493.33,110) .. (505.33,84) ;
					\draw [color={rgb, 255:red, 0; green, 0; blue, 0 }  ,draw opacity=1 ]   (505.33,84) -- (556.33,84) ;
					\draw [color={rgb, 255:red, 0; green, 0; blue, 0 }  ,draw opacity=1 ]   (354.33,258) -- (408.33,259) ;
					\draw    (408.33,259) ;
					\draw [shift={(408.33,259)}, rotate = 0] [color={rgb, 255:red, 0; green, 0; blue, 0 }  ][fill={rgb, 255:red, 0; green, 0; blue, 0 }  ][line width=0.75]      (0, 0) circle [x radius= 3.35, y radius= 3.35]   ;
					\draw    (408.33,259) -- (503.33,260) ;
					\draw [shift={(503.33,260)}, rotate = 0.6] [color={rgb, 255:red, 0; green, 0; blue, 0 }  ][fill={rgb, 255:red, 0; green, 0; blue, 0 }  ][line width=0.75]      (0, 0) circle [x radius= 3.35, y radius= 3.35]   ;
					\draw    (503.33,260) .. controls (492.33,231) and (420.33,231) .. (408.33,259) ;
					\draw    (408.33,259) .. controls (426.33,281) and (478.33,287) .. (503.33,260) ;
					\draw [color={rgb, 255:red, 0; green, 0; blue, 0 }  ,draw opacity=1 ]   (503.33,260) -- (554.33,260) ;
					\draw [color={rgb, 255:red, 0; green, 0; blue, 0 }  ,draw opacity=1 ]   (356.33,189) -- (405.33,190) ;
					\draw    (405.33,190) ;
					\draw [shift={(405.33,190)}, rotate = 0] [color={rgb, 255:red, 0; green, 0; blue, 0 }  ][fill={rgb, 255:red, 0; green, 0; blue, 0 }  ][line width=0.75]      (0, 0) circle [x radius= 3.35, y radius= 3.35]   ;
					\draw    (405.33,190) -- (500.33,191) ;
					\draw [shift={(500.33,191)}, rotate = 0.6] [color={rgb, 255:red, 0; green, 0; blue, 0 }  ][fill={rgb, 255:red, 0; green, 0; blue, 0 }  ][line width=0.75]      (0, 0) circle [x radius= 3.35, y radius= 3.35]   ;
					\draw [color={rgb, 255:red, 0; green, 0; blue, 0 }  ,draw opacity=1 ]   (500.33,191) -- (557.33,190) ;
					
					\draw (201,69.4) node [anchor=north west][inner sep=0.75pt]    {$3$};
					\draw (49,64.4) node [anchor=north west][inner sep=0.75pt]    {$3$};
					\draw (124,222.4) node [anchor=north west][inner sep=0.75pt]    {$3$};
					\draw (206,221.4) node [anchor=north west][inner sep=0.75pt]    {$3$};
					\draw (49,218.4) node [anchor=north west][inner sep=0.75pt]    {$3$};
					\draw (173,21.4) node [anchor=north west][inner sep=0.75pt]    {$\boldsymbol{3}$};
					\draw (73,20.4) node [anchor=north west][inner sep=0.75pt]    {$\boldsymbol{3}$};
					\draw (77,175.4) node [anchor=north west][inner sep=0.75pt]    {$\boldsymbol{3}$};
					\draw (177,175.4) node [anchor=north west][inner sep=0.75pt]    {$\boldsymbol{3}$};
					\draw (531,63.4) node [anchor=north west][inner sep=0.75pt]    {$3$};
					\draw (379,58.4) node [anchor=north west][inner sep=0.75pt]    {$3$};
					\draw (526,237.4) node [anchor=north west][inner sep=0.75pt]    {$3$};
					\draw (377,234.4) node [anchor=north west][inner sep=0.75pt]    {$3$};
					\draw (447,171.4) node [anchor=north west][inner sep=0.75pt]    {$3$};
					\draw (529,170.4) node [anchor=north west][inner sep=0.75pt]    {$3$};
					\draw (372,167.4) node [anchor=north west][inner sep=0.75pt]    {$3$};
					\draw (281,33.4) node [anchor=north west][inner sep=0.75pt]    {$=\ \ \ \ \ \begin{cases}
							& \\
							& \\
							& \\
							& \\
							& 
						\end{cases}$};
					\draw (279,175.4) node [anchor=north west][inner sep=0.75pt]    {$=\ \ \ \ \ \begin{cases}
							& \\
							& \\
							& \\
							& \\
							& 
						\end{cases}$};

				\end{tikzpicture}
				\caption{Tropical spin completed cycles Hurwitz covers with a fat end for each completed $3$-cycle to differentiate them from the ordinary covers on the right (see \cref{example_upgradedfeynmangraphs})  }
				\label{figure_upgradedfeynmandiagrams}
			\end{figure}
		\end{example}
		We give a tropical analogue of the numbers in \cref{definition_completedcyclesspinHurwitznumber}.
		\begin{definition}\label{definition_tropicalcompletedcyclesspinHurwitznumbers}
			Fix positive integers $d,g>0$, even integers $ r_1,\dots, r_b>0$ and odd partitions $\mu, \nu$ of $d\in \mathbb{N}$ such that $r_1+ \dots +r_b= 2g + l(\mu) + l(\nu) -2$ holds. 
			
			Let $\mathbb{TP}^1$ be the tropical line with vertices $p_1,\dots,p_b$.
			The \emph{connected/disconnected tropical spin double Hurwitz number with multi-completed cycles}
			$\mathbb{T}h_{g}(\mu, \overline{(r_1 +1)}, \dots, \overline{(r_b +1)}, \nu)$/$\mathbb{T}h^{spin,\bullet}_{g}(\mu,\overline{(r_1 +1)}, \dots, \overline{(r_b +1)},\nu)$ is given by
			\begin{align*}
				\sum_{\pi: \Gamma \rightarrow \mathbb{TP}^1 } \mult(\pi)
			\end{align*}
			where the sum is over all covers $\pi:\Gamma \rightarrow \mathbb{TP}^1$ such that
			\begin{itemize}
				\item $\Gamma$ is a connected/disconnected tropical curve of genus $g$ with odd edge weights only,
				\item the unbounded left (resp. right) pointing ends of $\Gamma$ have weights as prescribed by $\mu$ (resp. $\nu$),
				\item over each $p_i$ there is a unique vertex $v_i$ of $\Gamma$ whose valence is equal to $2 + r_i -2g(v_i)$,
			\end{itemize}
			and 
			\begin{align*}
				\mult(\pi)=\frac{1}{|\Aut(\pi)|}\prod_{e\in E(\Gamma)}\omega(e) \prod^b_{i=1}h^{spin,0,r_i}_{g}(x_i^+,x_i^-),
			\end{align*} 
			where $x_i^-$ ($x_i^+$) denotes the partition of weights of ingoing (outgoing edges) at $v_i\in V(\Gamma)$ and, as usual, $\Aut(\pi):=\{\pi': \Gamma' \rightarrow \mathbb{TP}^1| \exists \phi\in \Aut(\Gamma) \pi=\pi' \circ \phi \}$ is the automorphism group of $\pi$.
			
			If $r:=r_1=\dots= r_b$, we simply write $\mathbb{T}h^{spin,r}_{g}(\mu, \nu)$/$\mathbb{T}h^{spin,r,\bullet}_{g}(\mu,\nu)$.
		\end{definition}
		The isomorphism class of $\pi$ contains exactly $\frac{1}{|\Aut(\pi)|} \prod_{v\in V(\Gamma)}(|\Aut(x_v^+)||\Aut(x_v^-)|)$ Feynman graphs ( a factor of $\prod_{v\in V(\Gamma)}(|\Aut(x_v^+)||\Aut(x_v^-)|)$ for the number of ways to label edge germs and a factor $\frac{1}{|\Aut(\pi)|}$ to account for overcounting).
		\begin{theorem}\label{theorem_tropicalclassicalcompletedcyclecorrespondence}
			For the discrete data as above, the classical and tropical counts agree:
			\begin{align*}
				&	\mathbb{T}h_{g}(\mu, \overline{(r_1 +1)}, \dots, \overline{(r_b +1)}, \nu)=h_{g}(\mu, \overline{(r_1 +1)}, \dots, \overline{(r_b +1)}, \nu)\\
				&\mathbb{T}h^{spin,\bullet}_{g}(\mu,\overline{(r_1 +1)}, \dots, \overline{(r_b +1)},\nu)=h^{spin,\bullet}_{g}(\mu,\overline{(r_1 +1)}, \dots, \overline{(r_b +1)},\nu).
			\end{align*}
		\end{theorem}
        
We note the following special case of vertex multiplicities which will be relevant in the next section.

		\begin{lemma}\label{lemma_vertexmultg=0}
			The genus $0$ vertex multiplicities are	equal to $h^{spin,0,r}_{g}(x^+,x^-)=r!$.
		\end{lemma}
		\begin{proof}
			We use \cref{lemma_1point}. Recall that the constant term of the quotient of two series in $z$ is the quotient of the constant terms. With
			$\mathcal{S}(xz)= 1 + \frac{x^2}{2^24!}z^2 + \frac{x^4}{2^4 5!}z^4 +\dots$ and $\frac{1}{\mathcal{S}(z)}=1 - \frac{z^2}{24}+ \frac{7z^4}{5760} +..$ $\rho(z)=\frac{1}{2}(1 +\frac{z^2}{2^22!}+ \frac{z^4}{2^4 4!}+\dots)$
			Since $h^{spin,0,r}_{g}(x^+,x^-)$ we obtain $h^{spin,0,r}_{g}(x^+,x^-)=2r!\frac{1}{2}$ .
		\end{proof}
      
\section{Piecewise polynomiality and wall--crossing formulae}\label{section_Pwpolynomialityandwc}
		Thus far, we have analyzed counting problems with fixed data, treating each one in isolation.
		More structure is to be uncovered by a change of perspective. 
		\begin{definition}
			For $k,l\geqslant 0$ consider the subspace $\mathcal{H}(k,l)\subset \mathbb{N}^l \times \mathbb{N}^k $ that consists of pairs of odd partitions $(\mu,\nu)$ of the same number with fixed lengths $l(\mu)=l$ and $l(\nu)=k$, i.e. we have $\mathcal{H}(l,k):=\{(\mu,\nu): \sum^l_{i=1}\mu_i=\sum^k_{j=1}\nu_j\}$. Any choice of a positive integer $g$ and an even integer $r$ yields a map
			\begin{align*}
				\mathbb{T}h^{spin,r}_{g}: \mathcal{H}(k,l) \rightarrow \mathbb{Q}, (\mu,\nu) \mapsto \mathbb{T}h^{spin,r}_{g}(\mu,\nu).
			\end{align*}
			and an assignment $(\mu,\nu)\mapsto \mathbb{T}S(\mu,\nu) $, where $\mathbb{T}S(\mu,\nu)$ is the \emph{set of relevant covers} for the number $\mathbb{T}h^{spin,r}_{g}(\mu,\nu)$ (see \cref{definition_tropicalcompletedcyclesspinHurwitznumbers}).\\
			(The set $\mathbb{T}S(\mu,\nu)$ depends on $g$ and $r$, but we suppress this in the notation as both should be clear from the context.)
		\end{definition}
		\begin{remark}
			We note that
			\begin{itemize}
				\item 	 $\mathbb{T}h^{spin,r}_{g}$ is the trivial map if the data $(k,l,g,r)$ does not satisfy the Riemann-Hurwitz condition. As of now, we implicitly assume that it is met. 
				\item preimages of $0$ and $\infty$ are marked, which is essential to guarantee piecewise polynomiality of the function $\mathbb{T}h^{spin,r}_{g}$. This means that $\Aut(\Gamma)$ is independent of $\mu$ and $\nu$ for any $\Gamma$ that contributes to $\mathbb{T}S(\mu,\nu)$. Moreover, if $g=0$, then $\Aut(\Gamma)$ is trivial. 
			\end{itemize}

		\end{remark} 
		The transition from considering spin Hurwitz numbers for a fixed pair of partitions $(\mu,\nu)$ to viewing them as a function on the space $\mathcal{H}(k,l)$ has uncovered structure that was previously hidden by the static nature of the counting problem: the function $\mathbb{T}h^{spin,r}_{g}$ is piecewise polynomial, i.e. it subdivides the space $\mathcal{H}(k,l)$ into loci of polynomiality $c\subset \mathcal{H}(k,l) $ we call \emph{chambers} separated by hyperplanes $W_{I,J}$ called \emph{walls}. 
		
		\subsection{Genus $0$}
		Among all genera, the genus $0$ case proves most accessible from a combinatorial standpoint (compare to \cref{subsubsection_highergenus}):
		By generalizing the proofs of \cite{zbMATH05824576} (Section 6), we recover similar polynomiality results and wall-crossing formulas \emph{only} using polynomiality of spin-completed cycle Hurwitz numbers with three branch points as algebraic input.
		This follows from \cref{lemma_edgeweightsarepolynomials} (\cite{zbMATH05824576}, Lemma 6.4), which holds for rational tropical curves in general (not only for tri-valent trees), and its Corollary.
		
		\begin{lemma}[\cite{zbMATH05824576}, Lemma 6.4]\label{lemma_edgeweightsarepolynomials}
			Let $\Gamma$ be a rational tropical curve. If we choose the weight $\mu_i$ for the in-end labeled $i$ and $\nu_j$ for the out-end labeled $j$, then the weights $\omega(e)$ of all inner edges $e$ are uniquely determined (but might be negative). The weight $\omega(e)$ equals 
			\begin{align*}
				\omega(e)=\sum_{i\in I}\mu_i-\sum_{j\in J}\nu_j,
			\end{align*}
			where $I\subset \{1,\dots,l\}$ and $J\subset \{1,\dots,k\}$ are the subsets of in- and out-ends belonging to the connected component of $\Gamma\setminus\{e\}$ from which e points away.
		\end{lemma}
		\begin{corollary}\label{corollary_propertyedgeweightsoddindependentfromoddendweightsg=0}
			Let $\Gamma$ be a rational tropical curve.\\
			The parity of the edge weights of $\Gamma$ is entirely determined by the graph structure, i.e. it is independent of any specific choice of odd end weights.\\
			The property "$\omega(e) \text{ odd } \forall e\in \text{E}(\Gamma)$" does not depend on any specific choice $(\mu,\nu)$ of odd end weights.
		\end{corollary}
		\begin{proof}
			By \cref{lemma_edgeweightsarepolynomials} we have
			\begin{align*}
				\omega(e)=\sum_{i\in I_e}\mu_i-\sum_{j\in J_e}\nu_j
			\end{align*}
			for $e\in \text{E}(\Gamma)$ and subsets $I_e\subset \{1,\dots,l\}$ and $J_e\subset \{1,\dots,k\}$ of in- and out-ends. Write $\mu_i=2k_1 +1$ and $\nu_j=2 \hat{k}_j+1$, then
			\begin{align*}
				\omega(e)=2(\sum_{i\in I_e}k_i-\sum_{j\in J_e}\hat{k}_j)+|I_e|-|J_e|
			\end{align*}
			is odd if and only if $|I_e|-|J_e|$ is odd.
		\end{proof}
		We combine the above with the fact that a graph will stay relevant to the count as long as edge weights stay positive.
		\begin{lemma}\label{lemma_relevantsetsareconstantinsidechamber}
			The assignment $(\mu,\nu)\mapsto \mathbb{T}S(\mu,\nu) $ naturally gives rise to a hyperplane arrangement $\mathcal{W}(l,k)$ in $\mathcal{H}(k,l)$ consisting of hyperplanes of the form $\mathcal{W}_{I,J}=\{(\mu,\nu): \sum_{i\in I} \mu_i=\sum_{j\in J} \nu_i \}$, where $I\subset \{1,\dots,l\}$ and $J\subset \{1,\dots,k\}$ are non-empty subsets. The connected components of the complement $\mathcal{H}(k,l)\setminus \mathcal{W}(l,k)$  describe precisely the locus where the sets $\mathbb{T}S(\mu,\nu) $ remain constant under variation of $(\mu,\nu) $, i.e. the set of underlying directed graphs with variables as edge weights does not change.
		\end{lemma}
		\begin{proof}
			
			Fix a point $(\mu',\nu')\in \mathcal{H}(k,l)$ and let  $\pi: \Gamma \rightarrow \mathbb{TP}^1\in \mathbb{T}S(\mu',\nu')$ be a cover. \cref{lemma_edgeweightsarepolynomials} gives edge weights of $\Gamma$ in terms of $(\mu',\nu')$. \cref{corollary_propertyedgeweightsoddindependentfromoddendweightsg=0} shows that we may vary $(\mu',\nu')$ without losing any of the directed graphs that correspond to a feasible cover $\pi$ as long as all edge weights stay positive (see \cref{remark_twoperspectives} for correspondence between covers of $\mathbb{TP}^1$ and directed graphs). Hence, the inequalities
			
			\begin{align*}
				\omega(e)(\mu,\nu)> 0 \text{ for all } e\in E(\Gamma)  \text{ and } \Gamma \in \mathbb{T}S(\mu',\nu')
			\end{align*} describe the maximal subset of $\mathcal{H}(k,l) $ such that $\mathbb{T}S(\mu',\nu')$ remains unchanged.
		\end{proof}
		
		\begin{remark}\label{remark_twoperspectives}Note that by \cref{lemma_relevantsetsareconstantinsidechamber} it makes sense to associate a set of relevant covers $\mathbb{T}S$ to a chamber $c$, i.e. a connected component of the complement $\mathcal{H}(k,l)\setminus \mathcal{W}(l,k)$, instead of viewing all sets $\mathbb{T}S(\mu,\nu)$ for $(\mu,\nu)\in c$ separately. Depending on the situation at hand we identify an element of $\mathbb{T}S$ with
			\begin{enumerate}
				\item a family of tropical covers $\mathcal{\pi}$ over the discrete base $c\subset \mathcal{H}(k,l)$ whose total space is a cover of $\mathbb{TP}^1$ by a tree $\Gamma$ with polynomial edge weights that satisfies the balancing condition. More precisely, a fibre over $(\mu,\nu)$ is the cover in $\mathbb{T}S(\mu,\nu)$ given by evaluating edge weights at $(\mu,\nu)$.  
				\item a directed tree $\Gamma$ with polynomial edge weights that satisfies the balancing condition, together with an ordering $<$ on the vertices $V(\Gamma)$ that is compatible with the edge directions, i.e. source $v_1$ and target vertex $v_2$ satisfy $v_1< v_2$.
			\end{enumerate}
			The first emphasizes the dynamic nature of the situation, while the second offers an elegant way of analyzing wall-crossing (see following paragraph). 
		\end{remark}
		
		\begin{remark}
			
			The hyperplane arrangement $\mathcal{W}(l,k)$ is natural from another point of view: The connected components of the complement $\mathcal{H}(k,l)\setminus \mathcal{W}(l,k)$ describe precisely the locus, where connected and disconnected counts agree. A disconnected cover of degree $d$ can be expressed as a union of covers of smaller degrees, say $d_1$ and $d_2$ with $d=d_1+d_2$. 
			Therefore, such a cover exists only if the underlying discrete data $(d,\mu,\nu)$ admits a decomposition into smaller counting problems- that is, if there exists subsets $I\subset \{1,\dots,l\}$ and $J\subset \{1,\dots,k\}$ with  $d_1=\sum_{i\in I} \mu_i=\sum_{j\in J} \nu_i$, equivalently $d_2=\sum_{i\in I^c} \mu_i=\sum_{j\in J^c} \nu_i$.
		\end{remark}
		
		\begin{proposition}\label{proposition_piecewisepolynomialitygenus0}[\cite{arXiv:2104.05697}, Theorem 6.4]
			The function $\mathbb{T}h^{spin,r}_{0}$ is a homogeneous polynomial of degree $\frac{k+l-2}{r}-1$ within each chamber $c\subset \mathcal{H}(k,l) $ cut out by the hyperplanes $\mathcal{W}_{I,J}$ from \cref{lemma_relevantsetsareconstantinsidechamber}. 
		\end{proposition}
		\begin{proof}
		Restrict to a chamber $c$ and use \cref{lemma_relevantsetsareconstantinsidechamber} and \cref{remark_twoperspectives} to count families $\pi: \Gamma \rightarrow \mathbb{TP}^1 \in \mathbb{T}S$ of tropical covers with $(\mu,\nu)-$dependent weight.

		Since local vertex multiplicities are constants (\cref{lemma_vertexmultg=0}) and edge weights are homogeneous polynomials of degree 1 in $(\mu,\nu)$, $\mult(\Gamma)$ is a homogeneous polynomial of degree $|E(\Gamma)|=|V(\Gamma)|-1= \frac{l +k -2}{r}-1$ for any $\Gamma\in \mathbb{T}S$.
		As $\mathbb{T}S$ is finite, $\mathbb{T}h^{spin,r}_{0}$ is a finite sum of homogeneous polynomials of equal degree, and therefore homogeneous itself. This concludes the proof. 
	\end{proof}
	\begin{example}
		Consider the counting problem
		\begin{align}\label{equation_inexgenus0polyn}
			\mathbb{T}h^{spin,2}_{0}: \mathcal{H}(3,3) \rightarrow \mathbb{Q}, (\mu,\nu) \mapsto \mathbb{T}h^{spin,2}_{0}(\mu,\nu).
		\end{align} 
		\cref{figure_afamilycontributingtoTh2-0} shows a directed tree $\Gamma$ corresponding to a family of covers $\pi$ (see \cref{remark_twoperspectives}) with $\pi\in \mathbb{T}S$ for any chamber $c$ and whose contribution to $(\ref{equation_inexgenus0polyn})$ is given by:
		\begin{align*}
			\mult(\Gamma)= (2!)^2(\mu_1+\mu_2+\mu_3).
		\end{align*}
	\end{example}
	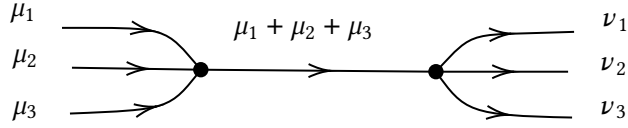
\begin{figure}[H]
		\centering

		\tikzset{every picture/.style={line width=0.75pt}} 
		
		\begin{tikzpicture}[x=0.75pt,y=0.75pt,yscale=-1,xscale=1]
			
			\draw    (264,107) -- (382,108) ;
			\draw [shift={(329,107.55)}, rotate = 180.49] [color={rgb, 255:red, 0; green, 0; blue, 0 }  ][line width=0.75]    (10.93,-3.29) .. controls (6.95,-1.4) and (3.31,-0.3) .. (0,0) .. controls (3.31,0.3) and (6.95,1.4) .. (10.93,3.29)   ;
			\draw    (382,108) ;
			\draw [shift={(382,108)}, rotate = 0] [color={rgb, 255:red, 0; green, 0; blue, 0 }  ][fill={rgb, 255:red, 0; green, 0; blue, 0 }  ][line width=0.75]      (0, 0) circle [x radius= 3.35, y radius= 3.35]   ;
			\draw    (382,108) .. controls (402.33,86) and (397.33,90) .. (451.33,88) ;
			\draw [shift={(419.7,88.91)}, rotate = 177.17] [color={rgb, 255:red, 0; green, 0; blue, 0 }  ][line width=0.75]    (10.93,-3.29) .. controls (6.95,-1.4) and (3.31,-0.3) .. (0,0) .. controls (3.31,0.3) and (6.95,1.4) .. (10.93,3.29)   ;
			\draw    (382,108) .. controls (396.33,128) and (398.33,130) .. (450.33,131) ;
			\draw [shift={(418.37,129.68)}, rotate = 186.58] [color={rgb, 255:red, 0; green, 0; blue, 0 }  ][line width=0.75]    (10.93,-3.29) .. controls (6.95,-1.4) and (3.31,-0.3) .. (0,0) .. controls (3.31,0.3) and (6.95,1.4) .. (10.93,3.29)   ;
			\draw    (264,107) ;
			\draw [shift={(264,107)}, rotate = 0] [color={rgb, 255:red, 0; green, 0; blue, 0 }  ][fill={rgb, 255:red, 0; green, 0; blue, 0 }  ][line width=0.75]      (0, 0) circle [x radius= 3.35, y radius= 3.35]   ;
			\draw    (264,107) .. controls (245.33,85) and (248.33,86) .. (194.33,86) ;
			\draw [shift={(239.77,87.58)}, rotate = 188.43] [color={rgb, 255:red, 0; green, 0; blue, 0 }  ][line width=0.75]    (10.93,-3.29) .. controls (6.95,-1.4) and (3.31,-0.3) .. (0,0) .. controls (3.31,0.3) and (6.95,1.4) .. (10.93,3.29)   ;
			\draw    (264,107) .. controls (248.33,127) and (248.33,127) .. (198.33,129) ;
			\draw [shift={(241.72,125.44)}, rotate = 168.85] [color={rgb, 255:red, 0; green, 0; blue, 0 }  ][line width=0.75]    (10.93,-3.29) .. controls (6.95,-1.4) and (3.31,-0.3) .. (0,0) .. controls (3.31,0.3) and (6.95,1.4) .. (10.93,3.29)   ;
			\draw    (382,108) -- (448.33,108) ;
			\draw [shift={(421.17,108)}, rotate = 180] [color={rgb, 255:red, 0; green, 0; blue, 0 }  ][line width=0.75]    (10.93,-3.29) .. controls (6.95,-1.4) and (3.31,-0.3) .. (0,0) .. controls (3.31,0.3) and (6.95,1.4) .. (10.93,3.29)   ;
			\draw    (199.33,106) -- (264,107) ;
			\draw [shift={(237.67,106.59)}, rotate = 180.89] [color={rgb, 255:red, 0; green, 0; blue, 0 }  ][line width=0.75]    (10.93,-3.29) .. controls (6.95,-1.4) and (3.31,-0.3) .. (0,0) .. controls (3.31,0.3) and (6.95,1.4) .. (10.93,3.29)   ;
			
			\draw (279,80.4) node [anchor=north west][inner sep=0.75pt]    {$\mu _{1} +\mu _{2} +\mu _{3}$};
			\draw (167,72.4) node [anchor=north west][inner sep=0.75pt]    {$\mu _{1}$};
			\draw (168,95.4) node [anchor=north west][inner sep=0.75pt]    {$\mu _{2}$};
			\draw (168,121.4) node [anchor=north west][inner sep=0.75pt]    {$\mu _{3}$};
			\draw (465,77.4) node [anchor=north west][inner sep=0.75pt]    {$\nu _{1}$};
			\draw (464,121.4) node [anchor=north west][inner sep=0.75pt]    {$\nu _{3}$};
			\draw (464,99.4) node [anchor=north west][inner sep=0.75pt]    {$\nu _{2}$};

		\end{tikzpicture}

		\caption{A directed tree that contributes to $\mathbb{T}h^{spin,2}_{0}(\mu,\nu)$.  }
		\label{figure_afamilycontributingtoTh2-0}
		
	\end{figure}
	\paragraph{ Wall crossing formulae.}
	A natural question to ask is what happens to ${\mathbb{T}h^{spin,r}_{0}}$ as we cross a wall. This information is encoded in so-called wall crossing formulae.
	\begin{definition}\label{definition_wallcrossingg=0}
		Let $c_1$ and $c_2$ be two neighboring chambers separated by a wall $\mathcal{W}_{I,J}$, i.e. $c_1$ and $c_2$ are defined by the same inequalities except for the one corresponding to  $\mathcal{W}_{I,J}$. Let $c_1$ be given by $\sum_{i\in I} \mu_i-\sum_{j\in J} \nu_i>0$ and $c_2$ by $\sum_{i\in I} \mu_i-\sum_{j\in J} \nu_i<0$. The polynomial $WC^{r}_{0,I,J}={\mathbb{T}h^{spin,r}_{0}}_{|c_1}-{\mathbb{T}h^{spin,r}_{0}}_{|c_2}$ is called a wall crossing formula.
	\end{definition}
	\begin{remark}
		Note: The restriction of $\mathbb{T}h^{spin,r}_{0}$ to a single chamber, together with all wall crossing formulae, is enough to reconstruct the full function $\mathbb{T}h^{spin,r}_{0}$. This encoding is efficient since wall crossing formulae depend only on smaller counting problems (see \cref{theorem_wall-crossinggenus0}).
	\end{remark}
	Let us denote the set of relevant covers in the chamber $c_1$ ($c_2$) by $\mathbb{T}S_1$ ($\mathbb{T}S_2$) and by $\mathbb{T}S^\omega_1$ ($\mathbb{T}S^{-\omega}_2$) the subset of all covers that have an edge $e$ of weight $\omega:=\omega(\mu,\nu):=\sum_{i\in I} \mu_i-\sum_{j\in J} \nu_i$ ($-\omega$). Note that by \cref{lemma_edgeweightsarepolynomials} this edge is unique. We define an equivalence relation on $\mathbb{T}S^\omega_1$, that "makes the structure of the set $\mathbb{T}S^\omega_1$ visible" by identifying common building blocks of its members:
	\begin{align*}
		(\Gamma,<) \sim_1 (\tilde{\Gamma},\tilde{<}) \text{ if and only if } &\text{the connected components of $\Gamma\setminus \{e\}$ and $\tilde{\Gamma}\setminus \{\tilde{e}\}$}\\
		& \text{together with their inherited orders agree}.
	\end{align*}
	The same construction yields an equivalence relation $\sim_2$ on  $\mathbb{T}S^{-\omega}_2$.
	\begin{example}\label{example_wallcrossinggenus0}
		Consider the counting problem
		\begin{align*}
			\mathbb{T}h^{spin,2}_{0}: \mathcal{H}(4,4) \rightarrow \mathbb{Q}, (\mu,\nu) \mapsto \mathbb{T}h^{spin,2}_{0}(\mu,\nu)
		\end{align*} 
		that asks for $4$-valent directed trees with $3$ vertices and ends of weight $(\mu,\nu)$. Let $c_1$ and $c_2$ be chambers separated by the wall $W=\{\omega:=\nu_4-\mu_3 +\mu_4=0\}$. \cref{figure_twoequivcoversandbuildingblocks} shows two equivalent covers in $\mathbb{T}S^{-\omega}_2$ and their common building blocks.
		\begin{figure}[H]
			\centering
			\tikzset{every picture/.style={line width=0.75pt}} 
			\scalebox{0.85}{ 
				\begin{tikzpicture}[x=0.75pt,y=0.75pt,yscale=-1,xscale=1]
					
					\draw [color={rgb, 255:red, 134; green, 244; blue, 17 }  ,draw opacity=1 ]   (332.33,235) -- (386.33,236) ;
					\draw [shift={(365.33,235.61)}, rotate = 181.06] [color={rgb, 255:red, 134; green, 244; blue, 17 }  ,draw opacity=1 ][line width=0.75]    (10.93,-3.29) .. controls (6.95,-1.4) and (3.31,-0.3) .. (0,0) .. controls (3.31,0.3) and (6.95,1.4) .. (10.93,3.29)   ;
					\draw    (386.33,236) ;
					\draw [shift={(386.33,236)}, rotate = 0] [color={rgb, 255:red, 0; green, 0; blue, 0 }  ][fill={rgb, 255:red, 0; green, 0; blue, 0 }  ][line width=0.75]      (0, 0) circle [x radius= 3.35, y radius= 3.35]   ;
					\draw [color={rgb, 255:red, 134; green, 244; blue, 17 }  ,draw opacity=1 ]   (162,249) .. controls (182.33,227) and (177.33,231) .. (231.33,229) ;
					\draw [shift={(199.7,229.91)}, rotate = 177.17] [color={rgb, 255:red, 134; green, 244; blue, 17 }  ,draw opacity=1 ][line width=0.75]    (10.93,-3.29) .. controls (6.95,-1.4) and (3.31,-0.3) .. (0,0) .. controls (3.31,0.3) and (6.95,1.4) .. (10.93,3.29)   ;
					\draw    (162,249) .. controls (176.33,269) and (178.33,271) .. (230.33,272) ;
					\draw [shift={(198.37,270.68)}, rotate = 186.58] [color={rgb, 255:red, 0; green, 0; blue, 0 }  ][line width=0.75]    (10.93,-3.29) .. controls (6.95,-1.4) and (3.31,-0.3) .. (0,0) .. controls (3.31,0.3) and (6.95,1.4) .. (10.93,3.29)   ;
					\draw    (162,249) ;
					\draw [shift={(162,249)}, rotate = 0] [color={rgb, 255:red, 0; green, 0; blue, 0 }  ][fill={rgb, 255:red, 0; green, 0; blue, 0 }  ][line width=0.75]      (0, 0) circle [x radius= 3.35, y radius= 3.35]   ;
					\draw    (162,249) .. controls (143.33,227) and (146.33,228) .. (92.33,228) ;
					\draw [shift={(137.77,229.58)}, rotate = 188.43] [color={rgb, 255:red, 0; green, 0; blue, 0 }  ][line width=0.75]    (10.93,-3.29) .. controls (6.95,-1.4) and (3.31,-0.3) .. (0,0) .. controls (3.31,0.3) and (6.95,1.4) .. (10.93,3.29)   ;
					\draw    (162,249) .. controls (146.33,269) and (146.33,269) .. (96.33,271) ;
					\draw [shift={(139.72,267.44)}, rotate = 168.85] [color={rgb, 255:red, 0; green, 0; blue, 0 }  ][line width=0.75]    (10.93,-3.29) .. controls (6.95,-1.4) and (3.31,-0.3) .. (0,0) .. controls (3.31,0.3) and (6.95,1.4) .. (10.93,3.29)   ;
					\draw    (386.33,236) -- (481.33,237) ;
					\draw [shift={(481.33,237)}, rotate = 0.6] [color={rgb, 255:red, 0; green, 0; blue, 0 }  ][fill={rgb, 255:red, 0; green, 0; blue, 0 }  ][line width=0.75]      (0, 0) circle [x radius= 3.35, y radius= 3.35]   ;
					\draw [shift={(439.83,236.56)}, rotate = 180.6] [color={rgb, 255:red, 0; green, 0; blue, 0 }  ][line width=0.75]    (10.93,-3.29) .. controls (6.95,-1.4) and (3.31,-0.3) .. (0,0) .. controls (3.31,0.3) and (6.95,1.4) .. (10.93,3.29)   ;
					\draw    (481.33,237) .. controls (501.67,215) and (496.67,219) .. (550.67,217) ;
					\draw [shift={(519.03,217.91)}, rotate = 177.17] [color={rgb, 255:red, 0; green, 0; blue, 0 }  ][line width=0.75]    (10.93,-3.29) .. controls (6.95,-1.4) and (3.31,-0.3) .. (0,0) .. controls (3.31,0.3) and (6.95,1.4) .. (10.93,3.29)   ;
					\draw    (481.33,237) .. controls (495.67,257) and (497.67,259) .. (549.67,260) ;
					\draw [shift={(517.7,258.68)}, rotate = 186.58] [color={rgb, 255:red, 0; green, 0; blue, 0 }  ][line width=0.75]    (10.93,-3.29) .. controls (6.95,-1.4) and (3.31,-0.3) .. (0,0) .. controls (3.31,0.3) and (6.95,1.4) .. (10.93,3.29)   ;
					\draw    (481.33,237) .. controls (462.67,215) and (465.67,216) .. (411.67,216) ;
					\draw [shift={(457.11,217.58)}, rotate = 188.43] [color={rgb, 255:red, 0; green, 0; blue, 0 }  ][line width=0.75]    (10.93,-3.29) .. controls (6.95,-1.4) and (3.31,-0.3) .. (0,0) .. controls (3.31,0.3) and (6.95,1.4) .. (10.93,3.29)   ;
					\draw    (386.33,236) .. controls (367.67,214) and (382.33,216) .. (328.33,216) ;
					\draw [shift={(368.62,217.1)}, rotate = 187.95] [color={rgb, 255:red, 0; green, 0; blue, 0 }  ][line width=0.75]    (10.93,-3.29) .. controls (6.95,-1.4) and (3.31,-0.3) .. (0,0) .. controls (3.31,0.3) and (6.95,1.4) .. (10.93,3.29)   ;
					\draw    (386.33,236) .. controls (400.67,256) and (402.67,258) .. (454.67,259) ;
					\draw [shift={(422.7,257.68)}, rotate = 186.58] [color={rgb, 255:red, 0; green, 0; blue, 0 }  ][line width=0.75]    (10.93,-3.29) .. controls (6.95,-1.4) and (3.31,-0.3) .. (0,0) .. controls (3.31,0.3) and (6.95,1.4) .. (10.93,3.29)   ;
					\draw [color={rgb, 255:red, 138; green, 240; blue, 19 }  ,draw opacity=1 ]   (102.45,117.02) .. controls (99.67,97.62) and (93.19,96.65) .. (85.16,76.29) ;
					\draw [shift={(98.49,102.31)}, rotate = 242.74] [color={rgb, 255:red, 138; green, 240; blue, 19 }  ,draw opacity=1 ][line width=0.75]    (10.93,-3.29) .. controls (6.95,-1.4) and (3.31,-0.3) .. (0,0) .. controls (3.31,0.3) and (6.95,1.4) .. (10.93,3.29)   ;
					\draw    (102.45,117.02) .. controls (115.73,136.42) and (188.6,135.45) .. (236.77,136.42) ;
					\draw [shift={(174.14,134.85)}, rotate = 183.26] [color={rgb, 255:red, 0; green, 0; blue, 0 }  ][line width=0.75]    (10.93,-3.29) .. controls (6.95,-1.4) and (3.31,-0.3) .. (0,0) .. controls (3.31,0.3) and (6.95,1.4) .. (10.93,3.29)   ;
					\draw    (102.45,117.02) ;
					\draw [shift={(102.45,117.02)}, rotate = 0] [color={rgb, 255:red, 0; green, 0; blue, 0 }  ][fill={rgb, 255:red, 0; green, 0; blue, 0 }  ][line width=0.75]      (0, 0) circle [x radius= 3.35, y radius= 3.35]   ;
					\draw    (102.45,117.02) .. controls (80.22,102.47) and (60.77,109.26) .. (34.83,108.29) ;
					\draw [shift={(76.52,107.94)}, rotate = 183.44] [color={rgb, 255:red, 0; green, 0; blue, 0 }  ][line width=0.75]    (10.93,-3.29) .. controls (6.95,-1.4) and (3.31,-0.3) .. (0,0) .. controls (3.31,0.3) and (6.95,1.4) .. (10.93,3.29)   ;
					\draw    (102.45,117.02) .. controls (87.94,136.42) and (85.78,135.45) .. (39.47,137.39) ;
					\draw [shift={(81.11,134.23)}, rotate = 169.64] [color={rgb, 255:red, 0; green, 0; blue, 0 }  ][line width=0.75]    (10.93,-3.29) .. controls (6.95,-1.4) and (3.31,-0.3) .. (0,0) .. controls (3.31,0.3) and (6.95,1.4) .. (10.93,3.29)   ;
					\draw    (85.16,76.29) ;
					\draw [shift={(85.16,76.29)}, rotate = 0] [color={rgb, 255:red, 0; green, 0; blue, 0 }  ][fill={rgb, 255:red, 0; green, 0; blue, 0 }  ][line width=0.75]      (0, 0) circle [x radius= 3.35, y radius= 3.35]   ;
					\draw    (85.16,76.29) -- (173.16,77.26) ;
					\draw [shift={(173.16,77.26)}, rotate = 0.63] [color={rgb, 255:red, 0; green, 0; blue, 0 }  ][fill={rgb, 255:red, 0; green, 0; blue, 0 }  ][line width=0.75]      (0, 0) circle [x radius= 3.35, y radius= 3.35]   ;
					\draw [shift={(135.16,76.84)}, rotate = 180.63] [color={rgb, 255:red, 0; green, 0; blue, 0 }  ][line width=0.75]    (10.93,-3.29) .. controls (6.95,-1.4) and (3.31,-0.3) .. (0,0) .. controls (3.31,0.3) and (6.95,1.4) .. (10.93,3.29)   ;
					\draw    (173.16,77.26) .. controls (192,55.92) and (187.36,59.8) .. (237.38,57.86) ;
					\draw [shift={(208.08,58.74)}, rotate = 177.04] [color={rgb, 255:red, 0; green, 0; blue, 0 }  ][line width=0.75]    (10.93,-3.29) .. controls (6.95,-1.4) and (3.31,-0.3) .. (0,0) .. controls (3.31,0.3) and (6.95,1.4) .. (10.93,3.29)   ;
					\draw    (173.16,77.26) .. controls (186.44,96.65) and (184.89,91.8) .. (233.06,92.77) ;
					\draw [shift={(206.08,92.67)}, rotate = 180.19] [color={rgb, 255:red, 0; green, 0; blue, 0 }  ][line width=0.75]    (10.93,-3.29) .. controls (6.95,-1.4) and (3.31,-0.3) .. (0,0) .. controls (3.31,0.3) and (6.95,1.4) .. (10.93,3.29)   ;
					\draw    (173.16,77.26) .. controls (172.85,45.25) and (101.53,36.52) .. (48.73,38.46) ;
					\draw [shift={(124.64,43.78)}, rotate = 190.93] [color={rgb, 255:red, 0; green, 0; blue, 0 }  ][line width=0.75]    (10.93,-3.29) .. controls (6.95,-1.4) and (3.31,-0.3) .. (0,0) .. controls (3.31,0.3) and (6.95,1.4) .. (10.93,3.29)   ;
					\draw    (85.16,76.29) .. controls (98.44,95.68) and (190.45,116.05) .. (238.62,117.02) ;
					\draw [shift={(165.71,107.45)}, rotate = 192.34] [color={rgb, 255:red, 0; green, 0; blue, 0 }  ][line width=0.75]    (10.93,-3.29) .. controls (6.95,-1.4) and (3.31,-0.3) .. (0,0) .. controls (3.31,0.3) and (6.95,1.4) .. (10.93,3.29)   ;
					\draw [color={rgb, 255:red, 138; green, 240; blue, 19 }  ,draw opacity=1 ]   (541.98,126.52) .. controls (524.12,123.92) and (416.02,103.98) .. (417.27,78.83) ;
					\draw [shift={(481.24,113.58)}, rotate = 195.42] [color={rgb, 255:red, 138; green, 240; blue, 19 }  ,draw opacity=1 ][line width=0.75]    (10.93,-3.29) .. controls (6.95,-1.4) and (3.31,-0.3) .. (0,0) .. controls (3.31,0.3) and (6.95,1.4) .. (10.93,3.29)   ;
					\draw    (541.98,126.52) ;
					\draw [shift={(541.98,126.52)}, rotate = 0] [color={rgb, 255:red, 0; green, 0; blue, 0 }  ][fill={rgb, 255:red, 0; green, 0; blue, 0 }  ][line width=0.75]      (0, 0) circle [x radius= 3.35, y radius= 3.35]   ;
					\draw    (541.98,126.52) .. controls (519.42,113.52) and (402.86,118.72) .. (376.53,117.85) ;
					\draw [shift={(465.95,117.71)}, rotate = 180.83] [color={rgb, 255:red, 0; green, 0; blue, 0 }  ][line width=0.75]    (10.93,-3.29) .. controls (6.95,-1.4) and (3.31,-0.3) .. (0,0) .. controls (3.31,0.3) and (6.95,1.4) .. (10.93,3.29)   ;
					\draw    (541.98,126.52) .. controls (527.26,143.87) and (422.6,136.06) .. (375.59,137.8) ;
					\draw [shift={(466.81,138.06)}, rotate = 179.69] [color={rgb, 255:red, 0; green, 0; blue, 0 }  ][line width=0.75]    (10.93,-3.29) .. controls (6.95,-1.4) and (3.31,-0.3) .. (0,0) .. controls (3.31,0.3) and (6.95,1.4) .. (10.93,3.29)   ;
					\draw    (417.27,78.83) ;
					\draw [shift={(417.27,78.83)}, rotate = 0] [color={rgb, 255:red, 0; green, 0; blue, 0 }  ][fill={rgb, 255:red, 0; green, 0; blue, 0 }  ][line width=0.75]      (0, 0) circle [x radius= 3.35, y radius= 3.35]   ;
					\draw    (417.27,78.83) -- (506.57,79.7) ;
					\draw [shift={(506.57,79.7)}, rotate = 0.56] [color={rgb, 255:red, 0; green, 0; blue, 0 }  ][fill={rgb, 255:red, 0; green, 0; blue, 0 }  ][line width=0.75]      (0, 0) circle [x radius= 3.35, y radius= 3.35]   ;
					\draw [shift={(467.92,79.32)}, rotate = 180.56] [color={rgb, 255:red, 0; green, 0; blue, 0 }  ][line width=0.75]    (10.93,-3.29) .. controls (6.95,-1.4) and (3.31,-0.3) .. (0,0) .. controls (3.31,0.3) and (6.95,1.4) .. (10.93,3.29)   ;
					\draw    (506.57,79.7) .. controls (525.69,60.62) and (545.74,64.09) .. (596.51,62.35) ;
					\draw [shift={(555.72,63.35)}, rotate = 177.3] [color={rgb, 255:red, 0; green, 0; blue, 0 }  ][line width=0.75]    (10.93,-3.29) .. controls (6.95,-1.4) and (3.31,-0.3) .. (0,0) .. controls (3.31,0.3) and (6.95,1.4) .. (10.93,3.29)   ;
					\draw    (506.57,79.7) .. controls (520.05,97.04) and (548.56,90.97) .. (597.45,91.84) ;
					\draw [shift={(556.6,92.14)}, rotate = 179.54] [color={rgb, 255:red, 0; green, 0; blue, 0 }  ][line width=0.75]    (10.93,-3.29) .. controls (6.95,-1.4) and (3.31,-0.3) .. (0,0) .. controls (3.31,0.3) and (6.95,1.4) .. (10.93,3.29)   ;
					\draw    (506.57,79.7) .. controls (506.26,51.08) and (426.36,45.01) .. (372.77,46.74) ;
					\draw [shift={(452.29,50.61)}, rotate = 188] [color={rgb, 255:red, 0; green, 0; blue, 0 }  ][line width=0.75]    (10.93,-3.29) .. controls (6.95,-1.4) and (3.31,-0.3) .. (0,0) .. controls (3.31,0.3) and (6.95,1.4) .. (10.93,3.29)   ;
					\draw    (417.27,78.83) .. controls (400.04,78.83) and (424.48,78.83) .. (373.71,78.83) ;
					\draw [shift={(402.49,78.83)}, rotate = 180] [color={rgb, 255:red, 0; green, 0; blue, 0 }  ][line width=0.75]    (10.93,-3.29) .. controls (6.95,-1.4) and (3.31,-0.3) .. (0,0) .. controls (3.31,0.3) and (6.95,1.4) .. (10.93,3.29)   ;
					\draw    (417.27,78.83) .. controls (430.74,96.17) and (544.8,113.52) .. (596.51,110.91) ;
					\draw [shift={(511.56,105.27)}, rotate = 188.22] [color={rgb, 255:red, 0; green, 0; blue, 0 }  ][line width=0.75]    (10.93,-3.29) .. controls (6.95,-1.4) and (3.31,-0.3) .. (0,0) .. controls (3.31,0.3) and (6.95,1.4) .. (10.93,3.29)   ;
					\draw    (541.98,126.52) -- (594.63,126.52) ;
					\draw [shift={(574.3,126.52)}, rotate = 180] [color={rgb, 255:red, 0; green, 0; blue, 0 }  ][line width=0.75]    (10.93,-3.29) .. controls (6.95,-1.4) and (3.31,-0.3) .. (0,0) .. controls (3.31,0.3) and (6.95,1.4) .. (10.93,3.29)   ;
					\draw    (36.69,76.29) -- (85.16,76.29) ;
					\draw [shift={(66.93,76.29)}, rotate = 180] [color={rgb, 255:red, 0; green, 0; blue, 0 }  ][line width=0.75]    (10.93,-3.29) .. controls (6.95,-1.4) and (3.31,-0.3) .. (0,0) .. controls (3.31,0.3) and (6.95,1.4) .. (10.93,3.29)   ;
					\draw   (248.33,20) -- (273.33,20) -- (273.33,159) ;
					\draw   (374.33,21) -- (335.33,21) -- (335.33,159) ;
					\draw   (596,20) -- (633.33,20) -- (633.33,160) ;
					\draw   (39.33,19) -- (7.33,19) -- (7.33,158) ;
					\draw    (7.33,158) -- (39.33,158) ;
					\draw    (248.33,159) -- (273.33,159) ;
					\draw    (335.33,159) -- (367.33,159) ;
					\draw    (601.33,160) -- (633.33,160) ;
					
					\draw (419,193.4) node [anchor=north west][inner sep=0.75pt]    {$\mu _{1}$};
					\draw (204,206.4) node [anchor=north west][inner sep=0.75pt]    {$\textcolor[rgb]{0.53,0.96,0.07}{\omega }$};
					\draw (334,193.4) node [anchor=north west][inner sep=0.75pt]    {$\mu _{2}$};
					\draw (98,205.4) node [anchor=north west][inner sep=0.75pt]    {$\mu _{3}$};
					\draw (565,204.4) node [anchor=north west][inner sep=0.75pt]    {$\nu _{1}$};
					\draw (438,263.4) node [anchor=north west][inner sep=0.75pt]    {$\nu _{3}$};
					\draw (564,250.4) node [anchor=north west][inner sep=0.75pt]    {$\nu _{2}$};
					\draw (334.33,238.4) node [anchor=north west][inner sep=0.75pt]    {$\textcolor[rgb]{0.53,0.96,0.07}{\omega }$};
					\draw (209,250.4) node [anchor=north west][inner sep=0.75pt]    {$\nu _{4}$};
					\draw (97,247.4) node [anchor=north west][inner sep=0.75pt]    {$\mu _{4}$};
					\draw (58.65,85.29) node [anchor=north west][inner sep=0.75pt]    {$\textcolor[rgb]{0.53,0.96,0.07}{-\omega }$};
					\draw (16.23,91.93) node [anchor=north west][inner sep=0.75pt]    {$\mu _{3}$};
					\draw (241.87,130.73) node [anchor=north west][inner sep=0.75pt]    {$\nu _{4}$};
					\draw (18.08,128.79) node [anchor=north west][inner sep=0.75pt]    {$\mu _{4}$};
					\draw (18.39,23.07) node [anchor=north west][inner sep=0.75pt]    {$\mu _{1}$};
					\draw (17.46,52.17) node [anchor=north west][inner sep=0.75pt]    {$\mu _{2}$};
					\draw (244.03,47.32) node [anchor=north west][inner sep=0.75pt]    {$\nu _{1}$};
					\draw (242.18,107.45) node [anchor=north west][inner sep=0.75pt]    {$\nu _{3}$};
					\draw (240.33,85.14) node [anchor=north west][inner sep=0.75pt]    {$\nu _{2}$};
					\draw (403.7,95) node [anchor=north west][inner sep=0.75pt]    {$\textcolor[rgb]{0.53,0.96,0.07}{-\omega }$};
					\draw (349.33,103.18) node [anchor=north west][inner sep=0.75pt]    {$\mu _{3}$};
					\draw (607.9,122.26) node [anchor=north west][inner sep=0.75pt]    {$\nu _{4}$};
					\draw (352.15,124.86) node [anchor=north west][inner sep=0.75pt]    {$\mu _{4}$};
					\draw (348.7,32.94) node [anchor=north west][inner sep=0.75pt]    {$\mu _{1}$};
					\draw (351.52,64.16) node [anchor=north west][inner sep=0.75pt]    {$\mu _{2}$};
					\draw (607.28,51.15) node [anchor=north west][inner sep=0.75pt]    {$\nu _{1}$};
					\draw (607.28,98.85) node [anchor=north west][inner sep=0.75pt]    {$\nu _{3}$};
					\draw (605.4,78.9) node [anchor=north west][inner sep=0.75pt]    {$\nu _{2}$};
					\draw (289,84.4) node [anchor=north west][inner sep=0.75pt]    {$\sim _{2}$};
					\draw (19,177.4) node [anchor=north west][inner sep=0.75pt]    {$Building\ Blocks:$};

			\end{tikzpicture}}
			\caption{An example for the equivalence relation $\sim_2$.}
			\label{figure_twoequivcoversandbuildingblocks}
		\end{figure}
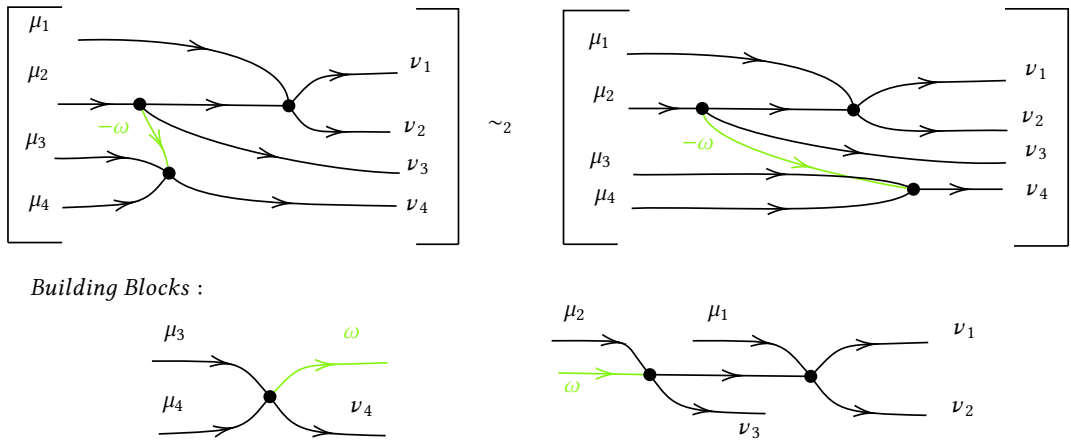
	\end{example}
	\begin{remark}\label{remark_a-cross-chamber-perspective}
		It is worth looking at the following from a "cross-chamber-perspective": 
		We know that edge weights of a balanced graph are polynomials in the end weights, talking about their sign does not make sense. However, should we have two graphs in different chambers that have edges $e$ and $\tilde{e}$ such that $|\omega(e)|=|\omega(\tilde{e})|$ holds, we can. If $sgn(\omega(e))=sgn(\omega(\tilde{e}))$ is the case, we say $e$ and $\tilde{e}$ point in the same direction, and in opposite direction otherwise. When viewed as a cover, this "orientation-reversal" corresponds to switching the order in which the vertices adjacent to $e$ are mapped to $\mathbb{TP}^1$. Indeed, this is just a consequence of the balancing condition.
	\end{remark}
	\begin{proposition}\label{proposition_adjacentchamberstructure}
		We have a bijection $\phi: \mathbb{T}S^{\omega}_1/\sim_1 \rightarrow \mathbb{T}S^{-\omega}_2/\sim_2$ that behaves sign-reversing with respect to the multiplicity, i.e. $\mult(\phi([\Gamma]_1))=-\mult([\Gamma]_1)$ holds, and satisfies $|[\Gamma]_1| + |\phi([\Gamma]_1)|= {{a}\choose{a_1,a_2}}$, where $a=\frac{k+l-2}{r}$, $a_1=\frac{|I|+|J|-1}{r}$ and $a_2=a-a_1$.
	\end{proposition}
	\begin{proof}
		We construct $\phi$ as follows: Given $(\Gamma,<)\in \mathbb{T}S^{\omega}_1 $ with edge of weight $\omega$ denoted by $e$, consider $\Gamma\setminus \{e\}=\Gamma_1 \cup \Gamma_2$. We claim that there exists $(\tilde{\Gamma},\tilde{<})\in \mathbb{T}S^{-\omega}_2$  such that $\tilde{\Gamma}\setminus \{\tilde{e}\}=\Gamma_1 \cup \Gamma_2$, where $\tilde{e}$  is the edge of weight $-\omega$, and define $\tilde{\phi}: \mathbb{T}S^{\omega}_1\rightarrow \mathbb{T}S^{-\omega}_2/\sim_2$ by $\tilde{\phi}(\Gamma):=[\tilde{\Gamma}]_2.$ By construction $\tilde{\phi}$ factors through the quotient (i.e. well-defined on equivalence classes) and yields an injective map $\phi: \mathbb{T}S^{\omega}_1/\sim_1 \rightarrow \mathbb{T}S^{-\omega}_2/\sim_2$. To see that $\phi$ is surjective as well requires an argument similar to the one needed to prove the existence of $\tilde{\Gamma}$. We address the existence first.\\
		Note, we already know $\tilde{\Gamma}$ as a graph: it consists of the two components $\Gamma_1$ and $\Gamma_2$ connected by an edge of weight $|\omega|$. It remains to be shown that $\tilde{\Gamma}$ can be turned into a cover, i.e. that we find $\tilde{<}$ such that $(\tilde{\Gamma},\tilde{<}) \in \mathbb{T}S^{-\omega}_2$ holds. The inequality $\omega <0 $ in $c_2$ forces $\tilde{e}$ to have polynomial weight $-\omega$ since relevant covers have positive edge weights when evaluated at partitions. Recall that by \cref{remark_a-cross-chamber-perspective} we interpret this as "switching the direction" of $e$ in $\Gamma$. Why is this always possible? We can only produce a sink or a source, if one of the vertices adjacent to $e$, say $v$, has $e$ as unique in-going or out-going edge. Without loss of generality suppose $v$ has in-going edges of weight $\omega_1,\dots,\omega_n$ and outgoing edge $e$ such that $\sum^n_{i=1}\omega_i=\omega$ (as polynomials). Reversing the orientation of $e$ creates a sink. We argued previously that all inequalities defining $c_1$ and $c_2$ are the same except for the one given by the wall. Hence, $\omega_i>0$ holds in both chambers and $\omega$ can never be zero, a contradiction.\\
		We can now define $\tilde{<}$ on $V(\tilde{\Gamma})=V(\Gamma)$ by modifying $<$ as follows: interchange/switch the positions of the vertices adjacent to $e$ in $<$ . This proves the claim. By the same type of argument (starting from $[\Gamma]_2\in \mathbb{T}S^{-\omega}_2/\sim_2$ instead of $[\Gamma]_1\in \mathbb{T}S^{\omega}_1/\sim_1$) we see that $\phi$ is surjective as well. \\ 
		The previous construction shows that image and preimage of $\phi$ are made of the same building blocks:
		\begin{enumerate}
			\item a directed weighted tree $\Gamma_1$ with $a_1:=|V(\Gamma_1)|=\frac{|I|+|J|-1}{r}$ (by Riemann-Hurwitz) ordered vertices from which precisely one is not balanced and end weights given by $\mu_I$ and $\nu_J$.
			\item a directed weighted tree $\Gamma_2$ with $a_2:=|V(\Gamma_2)|=\frac{|I^c|+|J^c|-1}{r}$ ordered vertices from which precisely one is not balanced and end weights given by $\mu_{I^c}$ and $\nu_{J^c}$.
		\end{enumerate}
		In order to obtain a graph in $[\Gamma]_1$ respectively $\phi([\Gamma]_1)$, connect the two unbalanced vertices by an edge of weight $|\omega|$ and pick an ordering $<$ that refines the ordering on $V(\Gamma_1)$ and $V(\Gamma_2)$. If $v_1<v_2$ holds, we get a graph in $[\Gamma]_1 $, and one in $\phi([\Gamma]_1)$ otherwise. We see that $|[\Gamma]_1| + |\phi([\Gamma]_1)|$ is given by the number of orderings on $a:=a_1+a_2=\frac{k+l-2}{r}$ vertices that refine two given orderings on $a_1$ and $a_2$ vertices, which is ${{a}\choose{a_1,a_2}}$: such an order corresponds to a vector $v$ of length $r_1+r_2$. We can position the $r_1$ vertices freely in $v$ with the only condition being that we respect the order dictated by $\Gamma_1$. The positions of the remaining $r_2$ vertices are then fixed.\\
		To finish let us address how $\phi$ behaves with respect to the multiplicity. First, note that this is  well-defined (on equivalence classes) since all members of an equivalence class share the same underlying graph. We have just seen that $\phi([\Gamma]_1)$ for $[\Gamma]_1\in \mathbb{T}S^{\omega}_1/\sim_1 $ is given by the equivalence class of covers obtained by switching the direction of $e$ in $\Gamma$. Thus, each member is counted with the same multiplicity up to a change of sign:
		Recall that $\mult(\Gamma)$ is a product over edge weights and local completed cycles spin Hurwitz numbers. The only terms affected by the wall-crossing are:
		\begin{enumerate}
			\item The edge weight $\omega$. It changes by a minus sign.
			\item The local spin completed cycles Hurwitz numbers attached to the vertices adjacent to $e$. These change by moving ramification from $0$ to $\infty$ or backwards, which has no effect since these are even in the entries of the partition corresponding to the respective ramification profiles.
		\end{enumerate}
		
	\end{proof}
	\begin{example}
		Refer to \cref{example_wallcrossinggenus0} and consider $\Gamma\in \mathbb{T}S^\omega_1$.  Then \cref{figure_themapphiinwallcross} shows $[\Gamma]_1$, $\phi([\Gamma]_1)$  and verifies  $|[\Gamma]_1| + |\phi([\Gamma]_1)|= 1 + 2= {{3}\choose{1,2}}$.
		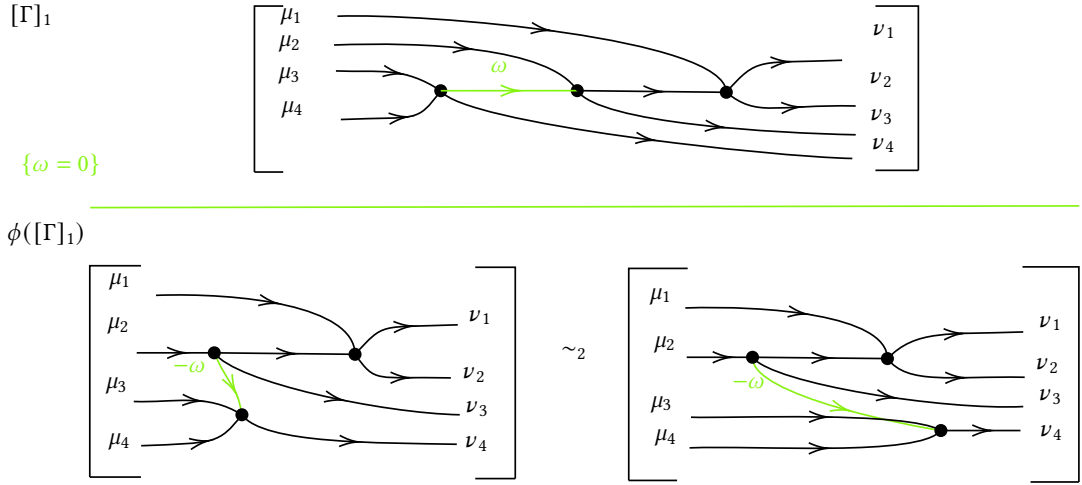
\begin{figure}[H]
			
			\centering
			\tikzset{every picture/.style={line width=0.75pt}} 
			\scalebox{0.85}{

				\begin{tikzpicture}[x=0.75pt,y=0.75pt,yscale=-1,xscale=1]
					
					\draw [color={rgb, 255:red, 138; green, 240; blue, 19 }  ,draw opacity=1 ]   (145.94,253.32) .. controls (143.32,235.91) and (137.2,235.04) .. (129.61,216.75) ;
					\draw [shift={(142.42,240.57)}, rotate = 241.93] [color={rgb, 255:red, 138; green, 240; blue, 19 }  ,draw opacity=1 ][line width=0.75]    (10.93,-3.29) .. controls (6.95,-1.4) and (3.31,-0.3) .. (0,0) .. controls (3.31,0.3) and (6.95,1.4) .. (10.93,3.29)   ;
					\draw    (145.94,253.32) .. controls (158.48,270.73) and (227.3,269.86) .. (272.8,270.73) ;
					\draw [shift={(214.42,269.36)}, rotate = 183.06] [color={rgb, 255:red, 0; green, 0; blue, 0 }  ][line width=0.75]    (10.93,-3.29) .. controls (6.95,-1.4) and (3.31,-0.3) .. (0,0) .. controls (3.31,0.3) and (6.95,1.4) .. (10.93,3.29)   ;
					\draw    (145.94,253.32) ;
					\draw [shift={(145.94,253.32)}, rotate = 0] [color={rgb, 255:red, 0; green, 0; blue, 0 }  ][fill={rgb, 255:red, 0; green, 0; blue, 0 }  ][line width=0.75]      (0, 0) circle [x radius= 3.35, y radius= 3.35]   ;
					\draw    (145.94,253.32) .. controls (124.95,240.26) and (106.58,246.35) .. (82.08,245.48) ;
					\draw [shift={(121.75,245.2)}, rotate = 183.38] [color={rgb, 255:red, 0; green, 0; blue, 0 }  ][line width=0.75]    (10.93,-3.29) .. controls (6.95,-1.4) and (3.31,-0.3) .. (0,0) .. controls (3.31,0.3) and (6.95,1.4) .. (10.93,3.29)   ;
					\draw    (145.94,253.32) .. controls (132.24,270.73) and (130.2,269.86) .. (86.46,271.6) ;
					\draw [shift={(126.04,268.71)}, rotate = 170.14] [color={rgb, 255:red, 0; green, 0; blue, 0 }  ][line width=0.75]    (10.93,-3.29) .. controls (6.95,-1.4) and (3.31,-0.3) .. (0,0) .. controls (3.31,0.3) and (6.95,1.4) .. (10.93,3.29)   ;
					\draw    (129.61,216.75) ;
					\draw [shift={(129.61,216.75)}, rotate = 0] [color={rgb, 255:red, 0; green, 0; blue, 0 }  ][fill={rgb, 255:red, 0; green, 0; blue, 0 }  ][line width=0.75]      (0, 0) circle [x radius= 3.35, y radius= 3.35]   ;
					\draw    (129.61,216.75) -- (212.72,217.62) ;
					\draw [shift={(212.72,217.62)}, rotate = 0.6] [color={rgb, 255:red, 0; green, 0; blue, 0 }  ][fill={rgb, 255:red, 0; green, 0; blue, 0 }  ][line width=0.75]      (0, 0) circle [x radius= 3.35, y radius= 3.35]   ;
					\draw [shift={(177.17,217.25)}, rotate = 180.6] [color={rgb, 255:red, 0; green, 0; blue, 0 }  ][line width=0.75]    (10.93,-3.29) .. controls (6.95,-1.4) and (3.31,-0.3) .. (0,0) .. controls (3.31,0.3) and (6.95,1.4) .. (10.93,3.29)   ;
					\draw    (212.72,217.62) .. controls (230.51,198.47) and (226.14,201.95) .. (273.38,200.21) ;
					\draw [shift={(246.47,200.97)}, rotate = 177.29] [color={rgb, 255:red, 0; green, 0; blue, 0 }  ][line width=0.75]    (10.93,-3.29) .. controls (6.95,-1.4) and (3.31,-0.3) .. (0,0) .. controls (3.31,0.3) and (6.95,1.4) .. (10.93,3.29)   ;
					\draw    (212.72,217.62) .. controls (225.26,235.04) and (223.81,230.68) .. (269.3,231.55) ;
					\draw [shift={(244.62,231.45)}, rotate = 180.1] [color={rgb, 255:red, 0; green, 0; blue, 0 }  ][line width=0.75]    (10.93,-3.29) .. controls (6.95,-1.4) and (3.31,-0.3) .. (0,0) .. controls (3.31,0.3) and (6.95,1.4) .. (10.93,3.29)   ;
					\draw    (212.72,217.62) .. controls (212.43,188.89) and (145.07,181.06) .. (95.2,182.8) ;
					\draw [shift={(166.89,187.58)}, rotate = 190.3] [color={rgb, 255:red, 0; green, 0; blue, 0 }  ][line width=0.75]    (10.93,-3.29) .. controls (6.95,-1.4) and (3.31,-0.3) .. (0,0) .. controls (3.31,0.3) and (6.95,1.4) .. (10.93,3.29)   ;
					\draw    (129.61,216.75) .. controls (142.15,234.17) and (229.05,252.45) .. (274.55,253.32) ;
					\draw [shift={(205.69,244.73)}, rotate = 191.69] [color={rgb, 255:red, 0; green, 0; blue, 0 }  ][line width=0.75]    (10.93,-3.29) .. controls (6.95,-1.4) and (3.31,-0.3) .. (0,0) .. controls (3.31,0.3) and (6.95,1.4) .. (10.93,3.29)   ;
					\draw [color={rgb, 255:red, 138; green, 240; blue, 19 }  ,draw opacity=1 ]   (558.64,262.63) .. controls (542.71,260.27) and (446.29,242.18) .. (447.41,219.37) ;
					\draw [shift={(504.46,250.89)}, rotate = 195.67] [color={rgb, 255:red, 138; green, 240; blue, 19 }  ,draw opacity=1 ][line width=0.75]    (10.93,-3.29) .. controls (6.95,-1.4) and (3.31,-0.3) .. (0,0) .. controls (3.31,0.3) and (6.95,1.4) .. (10.93,3.29)   ;
					\draw    (558.64,262.63) ;
					\draw [shift={(558.64,262.63)}, rotate = 0] [color={rgb, 255:red, 0; green, 0; blue, 0 }  ][fill={rgb, 255:red, 0; green, 0; blue, 0 }  ][line width=0.75]      (0, 0) circle [x radius= 3.35, y radius= 3.35]   ;
					\draw    (558.64,262.63) .. controls (538.51,250.83) and (434.55,255.55) .. (411.08,254.77) ;
					\draw [shift={(491.76,254.65)}, rotate = 180.87] [color={rgb, 255:red, 0; green, 0; blue, 0 }  ][line width=0.75]    (10.93,-3.29) .. controls (6.95,-1.4) and (3.31,-0.3) .. (0,0) .. controls (3.31,0.3) and (6.95,1.4) .. (10.93,3.29)   ;
					\draw    (558.64,262.63) .. controls (545.5,278.37) and (452.16,271.29) .. (410.24,272.86) ;
					\draw [shift={(491.59,273.09)}, rotate = 179.69] [color={rgb, 255:red, 0; green, 0; blue, 0 }  ][line width=0.75]    (10.93,-3.29) .. controls (6.95,-1.4) and (3.31,-0.3) .. (0,0) .. controls (3.31,0.3) and (6.95,1.4) .. (10.93,3.29)   ;
					\draw    (447.41,219.37) ;
					\draw [shift={(447.41,219.37)}, rotate = 0] [color={rgb, 255:red, 0; green, 0; blue, 0 }  ][fill={rgb, 255:red, 0; green, 0; blue, 0 }  ][line width=0.75]      (0, 0) circle [x radius= 3.35, y radius= 3.35]   ;
					\draw    (447.41,219.37) -- (527.06,220.15) ;
					\draw [shift={(527.06,220.15)}, rotate = 0.57] [color={rgb, 255:red, 0; green, 0; blue, 0 }  ][fill={rgb, 255:red, 0; green, 0; blue, 0 }  ][line width=0.75]      (0, 0) circle [x radius= 3.35, y radius= 3.35]   ;
					\draw [shift={(493.23,219.82)}, rotate = 180.57] [color={rgb, 255:red, 0; green, 0; blue, 0 }  ][line width=0.75]    (10.93,-3.29) .. controls (6.95,-1.4) and (3.31,-0.3) .. (0,0) .. controls (3.31,0.3) and (6.95,1.4) .. (10.93,3.29)   ;
					\draw    (527.06,220.15) .. controls (544.1,202.85) and (561.99,205.99) .. (607.26,204.42) ;
					\draw [shift={(571.34,205.3)}, rotate = 177.3] [color={rgb, 255:red, 0; green, 0; blue, 0 }  ][line width=0.75]    (10.93,-3.29) .. controls (6.95,-1.4) and (3.31,-0.3) .. (0,0) .. controls (3.31,0.3) and (6.95,1.4) .. (10.93,3.29)   ;
					\draw    (527.06,220.15) .. controls (539.07,235.89) and (564.5,230.38) .. (608.1,231.17) ;
					\draw [shift={(572.16,231.43)}, rotate = 179.54] [color={rgb, 255:red, 0; green, 0; blue, 0 }  ][line width=0.75]    (10.93,-3.29) .. controls (6.95,-1.4) and (3.31,-0.3) .. (0,0) .. controls (3.31,0.3) and (6.95,1.4) .. (10.93,3.29)   ;
					\draw    (527.06,220.15) .. controls (526.78,194.19) and (455.51,188.69) .. (407.73,190.26) ;
					\draw [shift={(479.38,193.88)}, rotate = 188.22] [color={rgb, 255:red, 0; green, 0; blue, 0 }  ][line width=0.75]    (10.93,-3.29) .. controls (6.95,-1.4) and (3.31,-0.3) .. (0,0) .. controls (3.31,0.3) and (6.95,1.4) .. (10.93,3.29)   ;
					\draw    (447.41,219.37) .. controls (432.04,219.37) and (453.84,219.37) .. (408.56,219.37) ;
					\draw [shift={(434.84,219.37)}, rotate = 180] [color={rgb, 255:red, 0; green, 0; blue, 0 }  ][line width=0.75]    (10.93,-3.29) .. controls (6.95,-1.4) and (3.31,-0.3) .. (0,0) .. controls (3.31,0.3) and (6.95,1.4) .. (10.93,3.29)   ;
					\draw    (447.41,219.37) .. controls (459.43,235.1) and (561.15,250.83) .. (607.26,248.47) ;
					\draw [shift={(531.51,243.36)}, rotate = 188.36] [color={rgb, 255:red, 0; green, 0; blue, 0 }  ][line width=0.75]    (10.93,-3.29) .. controls (6.95,-1.4) and (3.31,-0.3) .. (0,0) .. controls (3.31,0.3) and (6.95,1.4) .. (10.93,3.29)   ;
					\draw    (558.64,262.63) -- (605.59,262.63) ;
					\draw [shift={(588.11,262.63)}, rotate = 180] [color={rgb, 255:red, 0; green, 0; blue, 0 }  ][line width=0.75]    (10.93,-3.29) .. controls (6.95,-1.4) and (3.31,-0.3) .. (0,0) .. controls (3.31,0.3) and (6.95,1.4) .. (10.93,3.29)   ;
					\draw    (83.83,216.75) -- (129.61,216.75) ;
					\draw [shift={(112.72,216.75)}, rotate = 180] [color={rgb, 255:red, 0; green, 0; blue, 0 }  ][line width=0.75]    (10.93,-3.29) .. controls (6.95,-1.4) and (3.31,-0.3) .. (0,0) .. controls (3.31,0.3) and (6.95,1.4) .. (10.93,3.29)   ;
					\draw   (283.72,166.23) -- (307.33,166.23) -- (307.33,291) ;
					\draw   (409.12,166.91) -- (374.33,166.91) -- (374.33,292.09) ;
					\draw   (606.81,166) -- (640.11,166) -- (640.11,293) ;
					\draw   (86.33,165.33) -- (56.11,165.33) -- (56.11,290.1) ;
					\draw    (56.11,290.1) -- (86.33,290.1) ;
					\draw    (283.72,291) -- (307.33,291) ;
					\draw    (374.33,292.09) -- (402.87,292.09) ;
					\draw    (611.57,293) -- (640.11,293) ;
					\draw [color={rgb, 255:red, 134; green, 244; blue, 17 }  ,draw opacity=1 ]   (56.33,131) -- (640.33,130) ;
					\draw    (263.33,62) .. controls (276.61,81.4) and (458.17,101.03) .. (506.33,102) ;
					\draw [shift={(389.37,91.28)}, rotate = 187.47] [color={rgb, 255:red, 0; green, 0; blue, 0 }  ][line width=0.75]    (10.93,-3.29) .. controls (6.95,-1.4) and (3.31,-0.3) .. (0,0) .. controls (3.31,0.3) and (6.95,1.4) .. (10.93,3.29)   ;
					\draw    (263.33,62) ;
					\draw [shift={(263.33,62)}, rotate = 0] [color={rgb, 255:red, 0; green, 0; blue, 0 }  ][fill={rgb, 255:red, 0; green, 0; blue, 0 }  ][line width=0.75]      (0, 0) circle [x radius= 3.35, y radius= 3.35]   ;
					\draw    (263.33,62) .. controls (241.1,47.45) and (227.38,51.28) .. (201.44,50.31) ;
					\draw [shift={(240.26,52.04)}, rotate = 188.86] [color={rgb, 255:red, 0; green, 0; blue, 0 }  ][line width=0.75]    (10.93,-3.29) .. controls (6.95,-1.4) and (3.31,-0.3) .. (0,0) .. controls (3.31,0.3) and (6.95,1.4) .. (10.93,3.29)   ;
					\draw    (263.33,62) .. controls (248.82,81.4) and (250.65,77.06) .. (204.33,79) ;
					\draw [shift={(243.73,77.42)}, rotate = 174.5] [color={rgb, 255:red, 0; green, 0; blue, 0 }  ][line width=0.75]    (10.93,-3.29) .. controls (6.95,-1.4) and (3.31,-0.3) .. (0,0) .. controls (3.31,0.3) and (6.95,1.4) .. (10.93,3.29)   ;
					\draw    (343.94,61.95) ;
					\draw [shift={(343.94,61.95)}, rotate = 0] [color={rgb, 255:red, 0; green, 0; blue, 0 }  ][fill={rgb, 255:red, 0; green, 0; blue, 0 }  ][line width=0.75]      (0, 0) circle [x radius= 3.35, y radius= 3.35]   ;
					\draw    (343.94,61.95) -- (431.93,62.92) ;
					\draw [shift={(431.93,62.92)}, rotate = 0.63] [color={rgb, 255:red, 0; green, 0; blue, 0 }  ][fill={rgb, 255:red, 0; green, 0; blue, 0 }  ][line width=0.75]      (0, 0) circle [x radius= 3.35, y radius= 3.35]   ;
					\draw [shift={(393.94,62.51)}, rotate = 180.63] [color={rgb, 255:red, 0; green, 0; blue, 0 }  ][line width=0.75]    (10.93,-3.29) .. controls (6.95,-1.4) and (3.31,-0.3) .. (0,0) .. controls (3.31,0.3) and (6.95,1.4) .. (10.93,3.29)   ;
					\draw    (431.93,62.92) .. controls (450.77,41.59) and (450.31,45.94) .. (500.33,44) ;
					\draw [shift={(469.07,44.84)}, rotate = 177.41] [color={rgb, 255:red, 0; green, 0; blue, 0 }  ][line width=0.75]    (10.93,-3.29) .. controls (6.95,-1.4) and (3.31,-0.3) .. (0,0) .. controls (3.31,0.3) and (6.95,1.4) .. (10.93,3.29)   ;
					\draw    (431.93,62.92) .. controls (450.33,78) and (460.17,70.03) .. (508.33,71) ;
					\draw [shift={(475.41,71.61)}, rotate = 177.06] [color={rgb, 255:red, 0; green, 0; blue, 0 }  ][line width=0.75]    (10.93,-3.29) .. controls (6.95,-1.4) and (3.31,-0.3) .. (0,0) .. controls (3.31,0.3) and (6.95,1.4) .. (10.93,3.29)   ;
					\draw    (431.93,62.92) .. controls (431.63,30.92) and (255.13,16.06) .. (202.33,18) ;
					\draw [shift={(327.62,26.43)}, rotate = 187.56] [color={rgb, 255:red, 0; green, 0; blue, 0 }  ][line width=0.75]    (10.93,-3.29) .. controls (6.95,-1.4) and (3.31,-0.3) .. (0,0) .. controls (3.31,0.3) and (6.95,1.4) .. (10.93,3.29)   ;
					\draw    (343.94,61.95) .. controls (357.21,81.35) and (460.17,87.03) .. (508.33,88) ;
					\draw [shift={(430.38,83.72)}, rotate = 185.57] [color={rgb, 255:red, 0; green, 0; blue, 0 }  ][line width=0.75]    (10.93,-3.29) .. controls (6.95,-1.4) and (3.31,-0.3) .. (0,0) .. controls (3.31,0.3) and (6.95,1.4) .. (10.93,3.29)   ;
					\draw [color={rgb, 255:red, 134; green, 244; blue, 17 }  ,draw opacity=1 ]   (263.33,62) -- (343.94,61.95) ;
					\draw [shift={(309.64,61.97)}, rotate = 179.97] [color={rgb, 255:red, 134; green, 244; blue, 17 }  ,draw opacity=1 ][line width=0.75]    (10.93,-3.29) .. controls (6.95,-1.4) and (3.31,-0.3) .. (0,0) .. controls (3.31,0.3) and (6.95,1.4) .. (10.93,3.29)   ;
					\draw    (343.94,61.95) .. controls (327.33,36) and (253.13,33.06) .. (200.33,35) ;
					\draw [shift={(281.59,37.48)}, rotate = 186.46] [color={rgb, 255:red, 0; green, 0; blue, 0 }  ][line width=0.75]    (10.93,-3.29) .. controls (6.95,-1.4) and (3.31,-0.3) .. (0,0) .. controls (3.31,0.3) and (6.95,1.4) .. (10.93,3.29)   ;
					\draw   (170.33,12) -- (153.33,12) -- (153.33,110) ;
					\draw   (520,12) -- (545.33,12) -- (545.33,109) ;
					\draw    (153.33,110) -- (170.33,110) ;
					\draw    (520,109) -- (545.33,109) ;
					
					\draw (103.91,220.47) node [anchor=north west][inner sep=0.75pt]    {$\textcolor[rgb]{0.53,0.96,0.07}{-\omega }$};
					\draw (63.98,229.92) node [anchor=north west][inner sep=0.75pt]    {$\mu _{3}$};
					\draw (277.15,264.74) node [anchor=north west][inner sep=0.75pt]    {$\nu _{4}$};
					\draw (65.73,263) node [anchor=north west][inner sep=0.75pt]    {$\mu _{4}$};
					\draw (66.02,168.1) node [anchor=north west][inner sep=0.75pt]    {$\mu _{1}$};
					\draw (65.14,194.22) node [anchor=north west][inner sep=0.75pt]    {$\mu _{2}$};
					\draw (279.19,189.87) node [anchor=north west][inner sep=0.75pt]    {$\nu _{1}$};
					\draw (277.44,243.85) node [anchor=north west][inner sep=0.75pt]    {$\nu _{3}$};
					\draw (275.69,223.82) node [anchor=north west][inner sep=0.75pt]    {$\nu _{2}$};
					\draw (434.01,229.86) node [anchor=north west][inner sep=0.75pt]    {$\textcolor[rgb]{0.53,0.96,0.07}{-\omega }$};
					\draw (385.79,240.66) node [anchor=north west][inner sep=0.75pt]    {$\mu _{3}$};
					\draw (616.51,257.97) node [anchor=north west][inner sep=0.75pt]    {$\nu _{4}$};
					\draw (388.3,260.33) node [anchor=north west][inner sep=0.75pt]    {$\mu _{4}$};
					\draw (385.23,176.94) node [anchor=north west][inner sep=0.75pt]    {$\mu _{1}$};
					\draw (387.75,205.26) node [anchor=north west][inner sep=0.75pt]    {$\mu _{2}$};
					\draw (615.95,193.46) node [anchor=north west][inner sep=0.75pt]    {$\nu _{1}$};
					\draw (615.95,236.73) node [anchor=north west][inner sep=0.75pt]    {$\nu _{3}$};
					\draw (614.27,218.63) node [anchor=north west][inner sep=0.75pt]    {$\nu _{2}$};
					\draw (291.43,42.96) node [anchor=north west][inner sep=0.75pt]    {$\textcolor[rgb]{0.53,0.96,0.07}{\omega }$};
					\draw (166,46.6) node [anchor=north west][inner sep=0.75pt]    {$\mu _{3}$};
					\draw (517.33,88.4) node [anchor=north west][inner sep=0.75pt]    {$\nu _{4}$};
					\draw (167.85,66.46) node [anchor=north west][inner sep=0.75pt]    {$\mu _{4}$};
					\draw (167.16,11.74) node [anchor=north west][inner sep=0.75pt]    {$\mu _{1}$};
					\draw (166.24,27.84) node [anchor=north west][inner sep=0.75pt]    {$\mu _{2}$};
					\draw (517.81,20.99) node [anchor=north west][inner sep=0.75pt]    {$\nu _{1}$};
					\draw (516.33,71.4) node [anchor=north west][inner sep=0.75pt]    {$\nu _{3}$};
					\draw (516.1,49.82) node [anchor=north west][inner sep=0.75pt]  [rotate=-359.96]  {$\nu _{2}$};
					\draw (7,9.4) node [anchor=north west][inner sep=0.75pt]    {$[ \Gamma ]_{1}{}$};
					\draw (333,211.4) node [anchor=north west][inner sep=0.75pt]    {$\sim _{2}$};
					\draw (6,138.4) node [anchor=north west][inner sep=0.75pt]    {$\phi ([ \Gamma ]_{1})$};
					\draw (14,97.4) node [anchor=north west][inner sep=0.75pt]    {$\textcolor[rgb]{0.53,0.96,0.07}{\{}\textcolor[rgb]{0.53,0.96,0.07}{\omega =0}\textcolor[rgb]{0.53,0.96,0.07}{\}}$};

			\end{tikzpicture}}
			\caption{The map $\phi$ from \cref{proposition_adjacentchamberstructure}.}
			\label{figure_themapphiinwallcross}
		\end{figure}
	\end{example}
	\begin{theorem}\label{theorem_wall-crossinggenus0}
		For the wall $\mathcal{W}_{I,J}$ with adjacent chambers $c_1$ and $c_2$ the wall crossing formula is given by
		\begin{align*}
			WC^{r}_{0,I,J}={{a}\choose{a_1, a_2}}\mathbb{T}h^{spin,r}_{0}(\mu_I,(\nu_J,\omega))\mathbb{T}h^{spin,r}_{0}((\mu_{I^c},\omega),\nu_{J^c})\omega.
		\end{align*} 
		
	\end{theorem}
	The whole function $\mathbb{T}h^{spin,r}_{0}$ can be recovered from its value inside one chamber and all its wall-crossing formulas. \cref{theorem_wall-crossinggenus0} states that this encoding is efficient, since wall-crossing formulas are determined by spin Hurwitz numbers (with completed cycles) with less input data. 
	\begin{proof}
		Let $(\Gamma, <)\in \mathbb{T}S_1 $. We differentiate two cases, either $(\Gamma, <)\in \mathbb{T}S^\omega_1 $ or not. If not, then $(\Gamma, <)\in \mathbb{T}S_2 $ as well and their contributions cancel in $WC^{r}_{0,I,J}$ since they are counted with the same multiplicity on both sides of the wall. This yields
		\begin{align*}
			WC^{r}_{0,I,J}={\mathbb{T}h^{spin,r}_{0}}_{|c_1}-{\mathbb{T}h^{spin,r}_{0}}_{|c_2}=\sum_{\Gamma \in \mathbb{T}S^\omega_1 }\mult(\Gamma)-\sum_{\Gamma' \in \mathbb{T}S^{-\omega}_2 }\mult(\Gamma').
		\end{align*}
		Since multiplicities are well-defined on equivalence classes, we get by \cref{proposition_adjacentchamberstructure}:
		\begin{align*}
			WC^{r}_{0,I,J}&=\sum_{[\Gamma]_1 \in \mathbb{T}S^\omega_1/\sim_1 }|[\Gamma]_1|\cdot \mult([\Gamma]_1)-\sum_{[\Gamma']_2 \in \mathbb{T}S^\omega_2/\sim_2 }|[\Gamma']_2| \cdot \mult([\Gamma']_2)\\
			&=\sum_{[\Gamma]_1 \in \mathbb{T}S^\omega_1/\sim_1 }(|[\Gamma]_1|+|\phi([\Gamma]_1)|) \cdot \mult([\Gamma]_1)=
			\sum_{[\Gamma]_1 \in \mathbb{T}S^\omega_1/\sim_1 }{{a}\choose{a_1,a_2}} \cdot \mult([\Gamma]_1).
		\end{align*}
		By definition $\mathbb{T}S^\omega_1/\sim_1 $ is in bijection with the set of tuples of graphs $(\Gamma_1,\Gamma_2)$ as described in the proof of \cref{proposition_adjacentchamberstructure}. We can turn $\Gamma_1/\Gamma_2$ into a cover that counts towards the spin Hurwitz number\\ $\mathbb{T}h^{spin,r}_{0}(\mu_I,(\nu_J,\omega))/\mathbb{T}h^{spin,r}_{0}((\mu_{I^c},\omega),\nu_{J^c})$ by attaching an outward/inward directed end of weight $\omega$ to the unique unbalanced vertex of $\Gamma_1$/$\Gamma_2$. This modification yields a bijection to the set of tuples of covers that contribute to $\mathbb{T}h^{spin,r}_{0}(\mu_I,(\nu_J,\omega))$ and $\mathbb{T}h^{spin,r}_{0}((\mu_{I^c},\omega),\omega),\nu_{J^c})$. Thus
		\begin{align*}
			WC^{r}_{0,I,J}&={{a}\choose{a_1,a_2}}\omega\sum_{(\Gamma_1,\Gamma_2) }\mult(\Gamma_1)\mult(\Gamma_2)\\
			&={{a}\choose{a_1, a_2}}\mathbb{T}h^{spin,r}_{0}(\mu_I,(\nu_J,\omega))\mathbb{T}h^{spin,r}_{0}((\mu_{I^c},\omega),\nu_{J^c})\omega
		\end{align*}
		proves the claim.
		
	\end{proof}
	\subsection{Higher genus}\label{subsubsection_highergenus}
	Following \cite{zbMATH06806783} we do not fix the lengths of ramification profiles over $0$ and $\infty$ separately: We investigate properties of the map
	\begin{align*}
		\mathbb{T}h^{spin,r}_{g}: \mathcal{H} \rightarrow \mathbb{Q}, x \mapsto \mathbb{T}h^{spin,r}_{g}(x^-,x^+),
	\end{align*}
	where  $\mathcal{H}:=\{x\in \mathbb{Z}^{n}: \sum^{n}_{i=1}x_i=0 \text{ and } x_i \text{ odd} \}$.
	\begin{remark}
		The domain of the map $\mathbb{T}h^{spin,r}_{g}$ is discrete. We can extend $\mathbb{T}h^{spin,r}_{g}$ to the space $\mathcal{H}_\mathbb{R}:=\{x\in \mathbb{R}^{n}: \sum^{n}_{i=1}x_i=0\}$ of real valued vectors as follows: For $x\in \mathcal{H}$, $\mathbb{T}h^{spin,r}_{g}(x)$ is a weighted sum of graphs. As for the genus $0$ case, the set of these graphs remains unchanged "near $x$" (see \cref{proposition_multGammaispolynomial} for a precise statement). By equipping these graphs with real valued end and edge weights, we can extend  $\mathbb{T}h^{spin,r}_{g}$ locally to a real-valued function. However, for $x\in \mathcal{H}_\mathbb{R}\setminus \mathcal{H}$, $\mathbb{T}h^{spin,r}_{g}(x)$ is not a completed cycles spin Hurwitz number and graphs are not tropical completed cycle spin Hurwitz covers. These graphs arise from discrete counting problems but, with real weights, the parity notion is lost.
	\end{remark} 
	\emph{Added complexity.} \cref{lemma_edgeweightsarepolynomials} is what makes the genus $0$ case clean. Together with \cref{corollary_propertyedgeweightsoddindependentfromoddendweightsg=0} it allows us to associate to each chamber $c$ a set $\mathbb{T}S$ whose elements are families of tropical covers over $c$ - an association that fails in higher genus. The reason is as follows: For each independent cycle we get a free parameter, that is we have a certain amount of choice when equipping the graph with a balanced weighting. Consequently, the subset of $\mathcal{H}$ where all edge weights remain positive now depends on these extra parameters. The chambers of polynomiality (see \cref{lemma_hyperplanearrangementinducedbychambercounting}), however, do not.
	
	To handle this added complexity, we broaden our notion of "smallest counting unit" from a directed weighted graph with vertex ordering to an \emph{$x$-graph.}
	
	\begin{definition}\label{definition_xgraph}
		Given a completed cycle spin Hurwitz cover $\pi: \Gamma \rightarrow \mathbb{TP}^1$ arising from a counting problem $\mathbb{T}h_g^{spin,r}(x)$ fixed by $g>0$ and $x=(x_1,\dots,x_n)\in \mathcal{H}$ with $x_i\neq 0$, we build a graph with variables as edge weights:
		The \emph{$x$-graph} $\Gamma(x)$ associated to $\pi$ is obtained by
		\begin{itemize}
			\item compactifying $\Gamma$, i.e. adding a $1$-valent vertex to each end (thereby capping all infinite rays).
			\item  forgetting edge directions and vertex ordering.
		\end{itemize}
		An $x$-graph is equipped with a fixed orientation, the \emph{reference orientation}, that allows us to endow $\Gamma(x)$ with a \emph{reference weighting}: Let $(B_1,\dots,B_{b_1(\Gamma)})$  be a canonical homology basis of $\Gamma(x)$ and $e_i$ a fixed edge in $B_i$. Assign to $e_i$ a parameter $j_i$, its weight. The remaining edge weights are fixed by the balancing condition, i.e. can be expressed uniquely in terms of $x$ and $(j_1,\dots,j_{b_1(\Gamma)})$. Note, $\Gamma(x)$ is geometrically meaningful for odd weight and end vectors! 
	\end{definition}
	To view $\mathbb{T}h^{spin,r}_g$ as a sum over $x$-graphs, we must be able to reconstruct from $\Gamma(x)$ all covers whose $x$-graph is $\Gamma(x)$. This reconstruction problem reduces to finding all compatible pairs $((j_1,\dots,j_{b_1(\Gamma)}),<)$, where $(j_1,\dots,j_{b_1(\Gamma)})$ is a weight vector and $<$ a vertex ordering. The challenge is that we do not a priori know how these pieces of data interact (The relationship between weight vectors and vertex orderings, as well as the relationship between different weight vectors, is unclear.):
	To resolve the situation, we construct a suitable "container" space that naturally organizes this data.

	\begin{construction}\label{construction_xgraphbundle}
		As $\Gamma:=\Gamma(x)$ is a simplicial object, we can consider the associated chain complex
		\begin{align*}
			0 \rightarrow C^\Gamma_1 \xrightarrow{\partial} C^\Gamma_0 \rightarrow 0,
		\end{align*}
		where
		\begin{itemize}
			\item $C^\Gamma_0$ is the free abelian group of $0-$chains with coefficients in $\mathbb{R}$.
			\item $C^\Gamma_1$ is the free abelian group of $1-$chains with coefficients in $\mathbb{R}$.
			\item  $\partial: C^\Gamma_1 \rightarrow C^\Gamma_0$ is the boundary operator.
		\end{itemize}
		Recall that the groups $C^\Gamma_0$ and $C^\Gamma_1$ both have natural bases given by the set of vertices $V(\Gamma)$ (interior and $1$-valent vertices), of edges $E(\Gamma)$ (interior edges and ends) respectively. Hence, we may identify the coefficient-vector of a $1-$chain $c$ (represented with respect to the natural basis of $C^\Gamma_1$) with an assignment of $\mathbb{R}$-valued weights to the edges of $\Gamma$ ($\mathbb{R}$-weighting on $\Gamma$). The $\mathbb{R}$-weighting satisfies the balancing condition at every interior vertex if and only if the image of $c$ under $\partial$ lies in the subgroup $U$ generated by all $1-$valent vertices with coefficient vector $x\in \mathcal{H}_\mathbb{R}$, i.e.
		\begin{align*}
			\partial(c)=\sum_i x_i v_i, \text{ where } x_i=0, \text{ if } \val(v_i)>1, \text{ and } \sum_i x_i =0.
		\end{align*}
	\end{construction}
	The spaces $\partial^{-1}(\sum_i x_i v_i)$ will therefore parameterize all possible ways to endow $\Gamma$ with a balanced $\mathbb{R}$-weighting. As they also emerge naturally from $\Gamma$, they qualify as the container spaces we are looking for. 
	\begin{definition}\label{definition_bundleofcontainer}
		To an $x$-graph $\Gamma$ we attach the \emph{bundle of container spaces} $\mathcal{C}_\Gamma \xrightarrow{\pi_\Gamma}\mathcal{H}_\mathbb{R}$, that records all balanced $\mathbb{R}$-weightings as end weights of $\Gamma$ vary.
		We define $\mathcal{C}_\Gamma \rightarrow \mathcal{H}_\mathbb{R}$ to be the union of the sets $\{x\} \times \partial^{-1}(x) $ for all $x\in \mathcal{H}_\mathbb{R}$, where $x$ is identified with an element of $U$. The fibers $\partial^{-1}(x)$ are translates of the $b_1(\Gamma)-$dimensional vector space $\ker(\partial)=\text{H}_1(\Gamma,\mathbb{R})$, and  $\mathcal{C}_\Gamma$ is therefore a subset of the trivial bundle $\mathcal{H}_\mathbb{R} \times C^\Gamma_1 $.

	\end{definition}
	The following Lemma provides a justification for why \cref{definition_bundleofcontainer} is suitable for our combinatorial problem. 
	\begin{lemma}\label{lemma_alloddedgeweightspropertyindependentofconcretecyclealloddandendallod}
		Let $\Gamma$ be an $x$-graph of genus $g$. The property "$\omega(e) \text{ odd } \forall e\in \text{E}(\Gamma)$" does not depend on any specific choice of end weights, $x\in \mathcal{H}$, or odd edge weights, $(j_1,\dots,j_{b_1(\Gamma)})$, (\cref{definition_xgraph}).
	\end{lemma}
	\begin{proof}
		Cut $\Gamma$ at the edges labelled by $(j_1,\dots,j_{b_1(\Gamma)})$ and use \cref{corollary_propertyedgeweightsoddindependentfromoddendweightsg=0}.
	\end{proof}
	Note that although any point $\omega$ in a fiber of $\mathcal{C}_\Gamma$ will give rise to a balanced $\mathbb{R}$-weighting, it does not automatically correspond to a tropical cover: The orientation $\omega$ induces on $\Gamma$ might not allow for a total order on the vertices, thus $\omega$ is not \emph{adequate}. This singles out the induced orientation as \emph{the} property to keep track of when it comes to deciding whether a weight vector $\omega$ is adequate. In fact "inducing the same orientation" is an equivalence relation on the fibre and partitions ${\pi_\Gamma}^{-1}(x)$ in a natural way.
	
	\begin{definition}[\cite{MR2836109}] \label{definition_familyofhyperplanearrangements}\label{definition_bundleofcontainerhyperplanearrangement}
	
		Let $\mathcal{C}_\Gamma \xrightarrow{\pi_\Gamma} \mathcal{H}_\mathbb{R}$ be the bundle of container spaces for a given $x$-graph $\Gamma$ and view ${\pi_\Gamma}^{-1}(x)$ as a subspace of $\mathbb{R}^{|E(\Gamma)|}$ (identifying $C^\Gamma_1$ with $\mathbb{R}^{|E(\Gamma)|}$ via the isomorphism that takes the basis consisting of edges to the standard basis).
		The set of coordinate hyperplanes $\{ e_i=0 \thinspace \}$, where $e_i$ denotes the coordinate of the $i$-th standard basis vector associated to an \emph{internal edge}, restricts to a hyperplane arrangement $\mathcal{A}(x)$ inside each fiber ${\pi_\Gamma}^{-1}(x)$. The partition induced by $\mathcal{A}(x)$ is exactly the partition corresponding to the equivalence relation from above, i.e. where the chambers of $\mathcal{A}(x)$ are equivalence classes indexed by fixed orientations on $\Gamma$.
	\end{definition}
	\begin{example}\label{example_bundleofcontainerspaces}
		We illustrate Construction \ref{construction_xgraphbundle},  \cref{definition_bundleofcontainer,definition_familyofhyperplanearrangements} using the $x$-graph $\Gamma$ shown in \cref{figure_xgraphinexamplebundleofcontainer}. As this example will be revisited, we take the time to embed it in a broader context here: 
		Let the data $g=1,n=6,r=2$ define the following counting problem
		\begin{align}\label{equation_examplebundleofcontainerspaces}
			\mathbb{T}h^{spin,2,+}_{1}: \mathcal{H} \rightarrow \mathbb{Q}, (x_1,\dots,x_6) \mapsto \mathbb{T}h^{spin,2,+}_{1}(x^-,x^+).
		\end{align}
		Then (\ref{equation_examplebundleofcontainerspaces}) asks for graphs of genus $1$ with $6$ ends, $3$ inner vertices of valency $4$ or $2$ depending on the genus at that vertex. Note that $\Gamma$ is such a graph.
		
		One computes $\ker(\partial)=\langle e_1 + e_3 - e_2 \rangle\subset \mathbb{R}^9$. Hence, $\pi_\Gamma$ is contained in the trivial bundle $ \mathcal{H}_\mathbb{R} \times \mathbb{R}^9$ via the embedding $i$ (see \cref{figure_xgraphinexamplebundleofcontainer} for a sketch in the quotient $\mathbb{R}^9/\mathbb{R}^6\cong \mathbb{R}^3$ corresponding to $\langle e_1,e_2,e_3 \rangle$). 
	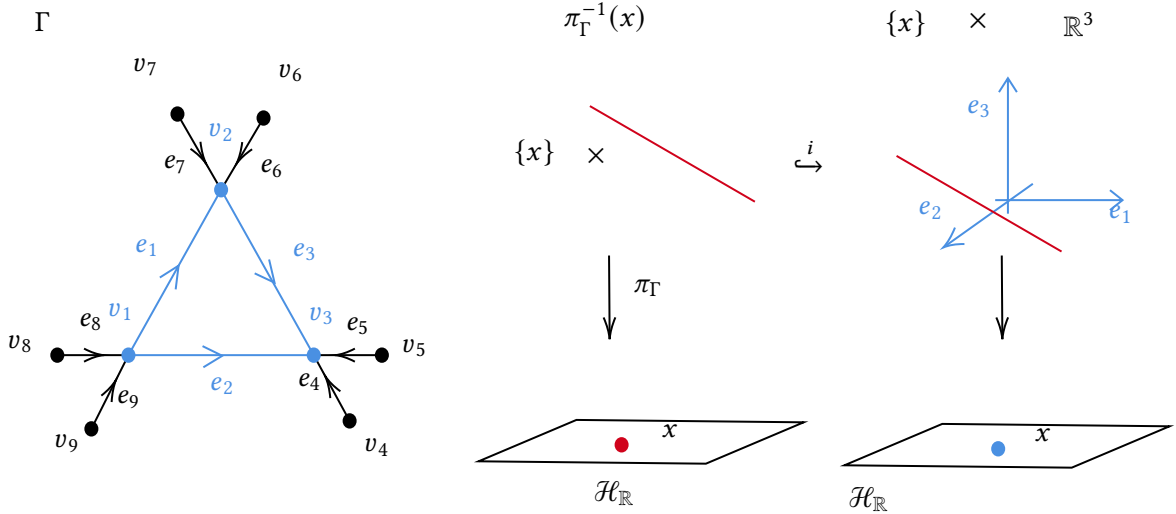
\begin{figure}[H]
		\tikzset{every picture/.style={line width=0.75pt}} 
		
		\begin{tikzpicture}[x=0.75pt,y=0.75pt,yscale=-1,xscale=0.9]
			
			\draw  [color={rgb, 255:red, 74; green, 144; blue, 226 }  ,draw opacity=1 ] (121.67,107) -- (173.33,190) -- (70,190) -- cycle ;
			\draw    (173.33,190) -- (193.33,223) ;
			\draw [shift={(179.71,200.51)}, rotate = 58.78] [color={rgb, 255:red, 0; green, 0; blue, 0 }  ][line width=0.75]    (10.93,-3.29) .. controls (6.95,-1.4) and (3.31,-0.3) .. (0,0) .. controls (3.31,0.3) and (6.95,1.4) .. (10.93,3.29)   ;
			\draw    (173.33,190) -- (211.33,190) ;
			\draw [shift={(185.33,190)}, rotate = 0] [color={rgb, 255:red, 0; green, 0; blue, 0 }  ][line width=0.75]    (10.93,-3.29) .. controls (6.95,-1.4) and (3.31,-0.3) .. (0,0) .. controls (3.31,0.3) and (6.95,1.4) .. (10.93,3.29)   ;
			\draw    (70,190) -- (49.33,227) ;
			\draw [shift={(63.08,202.39)}, rotate = 119.19] [color={rgb, 255:red, 0; green, 0; blue, 0 }  ][line width=0.75]    (10.93,-3.29) .. controls (6.95,-1.4) and (3.31,-0.3) .. (0,0) .. controls (3.31,0.3) and (6.95,1.4) .. (10.93,3.29)   ;
			\draw    (30.33,190) -- (70,190) ;
			\draw [shift={(56.17,190)}, rotate = 180] [color={rgb, 255:red, 0; green, 0; blue, 0 }  ][line width=0.75]    (10.93,-3.29) .. controls (6.95,-1.4) and (3.31,-0.3) .. (0,0) .. controls (3.31,0.3) and (6.95,1.4) .. (10.93,3.29)   ;
			\draw    (121.67,107) -- (97.33,69) ;
			\draw [shift={(113.27,93.89)}, rotate = 237.37] [color={rgb, 255:red, 0; green, 0; blue, 0 }  ][line width=0.75]    (10.93,-3.29) .. controls (6.95,-1.4) and (3.31,-0.3) .. (0,0) .. controls (3.31,0.3) and (6.95,1.4) .. (10.93,3.29)   ;
			\draw    (145.33,71) -- (121.67,107) ;
			\draw [shift={(130.2,94.01)}, rotate = 303.32] [color={rgb, 255:red, 0; green, 0; blue, 0 }  ][line width=0.75]    (10.93,-3.29) .. controls (6.95,-1.4) and (3.31,-0.3) .. (0,0) .. controls (3.31,0.3) and (6.95,1.4) .. (10.93,3.29)   ;
			\draw  [color={rgb, 255:red, 74; green, 144; blue, 226 }  ,draw opacity=1 ] (87.51,151.05) -- (97.96,145.08) -- (96.33,157) ;
			\draw  [color={rgb, 255:red, 74; green, 144; blue, 226 }  ,draw opacity=1 ] (148.43,141.99) -- (151.47,153.63) -- (140.39,148.95) ;
			\draw  [color={rgb, 255:red, 74; green, 144; blue, 226 }  ,draw opacity=1 ] (111.32,185.47) -- (122.33,190.32) -- (111.78,196.1) ;
			\draw [color={rgb, 255:red, 74; green, 144; blue, 226 }  ,draw opacity=1 ]   (70,190) ;
			\draw [shift={(70,190)}, rotate = 0] [color={rgb, 255:red, 74; green, 144; blue, 226 }  ,draw opacity=1 ][fill={rgb, 255:red, 74; green, 144; blue, 226 }  ,fill opacity=1 ][line width=0.75]      (0, 0) circle [x radius= 3.35, y radius= 3.35]   ;
			\draw    (49.33,227) ;
			\draw [shift={(49.33,227)}, rotate = 0] [color={rgb, 255:red, 0; green, 0; blue, 0 }  ][fill={rgb, 255:red, 0; green, 0; blue, 0 }  ][line width=0.75]      (0, 0) circle [x radius= 3.35, y radius= 3.35]   ;
			\draw [color={rgb, 255:red, 0; green, 0; blue, 0 }  ,draw opacity=1 ]   (30.33,190) ;
			\draw [shift={(30.33,190)}, rotate = 0] [color={rgb, 255:red, 0; green, 0; blue, 0 }  ,draw opacity=1 ][fill={rgb, 255:red, 0; green, 0; blue, 0 }  ,fill opacity=1 ][line width=0.75]      (0, 0) circle [x radius= 3.35, y radius= 3.35]   ;
			\draw [color={rgb, 255:red, 0; green, 0; blue, 0 }  ,draw opacity=1 ]   (97.33,69) ;
			\draw [shift={(97.33,69)}, rotate = 0] [color={rgb, 255:red, 0; green, 0; blue, 0 }  ,draw opacity=1 ][fill={rgb, 255:red, 0; green, 0; blue, 0 }  ,fill opacity=1 ][line width=0.75]      (0, 0) circle [x radius= 3.35, y radius= 3.35]   ;
			\draw    (145.33,71) ;
			\draw [shift={(145.33,71)}, rotate = 0] [color={rgb, 255:red, 0; green, 0; blue, 0 }  ][fill={rgb, 255:red, 0; green, 0; blue, 0 }  ][line width=0.75]      (0, 0) circle [x radius= 3.35, y radius= 3.35]   ;
			\draw [color={rgb, 255:red, 74; green, 144; blue, 226 }  ,draw opacity=1 ]   (121.67,107) ;
			\draw [shift={(121.67,107)}, rotate = 0] [color={rgb, 255:red, 74; green, 144; blue, 226 }  ,draw opacity=1 ][fill={rgb, 255:red, 74; green, 144; blue, 226 }  ,fill opacity=1 ][line width=0.75]      (0, 0) circle [x radius= 3.35, y radius= 3.35]   ;
			\draw [color={rgb, 255:red, 74; green, 144; blue, 226 }  ,draw opacity=1 ]   (173.33,190) ;
			\draw [shift={(173.33,190)}, rotate = 0] [color={rgb, 255:red, 74; green, 144; blue, 226 }  ,draw opacity=1 ][fill={rgb, 255:red, 74; green, 144; blue, 226 }  ,fill opacity=1 ][line width=0.75]      (0, 0) circle [x radius= 3.35, y radius= 3.35]   ;
			\draw    (211.33,190) ;
			\draw [shift={(211.33,190)}, rotate = 0] [color={rgb, 255:red, 0; green, 0; blue, 0 }  ][fill={rgb, 255:red, 0; green, 0; blue, 0 }  ][line width=0.75]      (0, 0) circle [x radius= 3.35, y radius= 3.35]   ;
			\draw    (193.33,223) ;
			\draw [shift={(193.33,223)}, rotate = 0] [color={rgb, 255:red, 0; green, 0; blue, 0 }  ][fill={rgb, 255:red, 0; green, 0; blue, 0 }  ][line width=0.75]      (0, 0) circle [x radius= 3.35, y radius= 3.35]   ;
			\draw [color={rgb, 255:red, 74; green, 144; blue, 226 }  ,draw opacity=1 ] (553.33,112.2) -- (624.33,112.2)(560.43,51) -- (560.43,119) (617.33,107.2) -- (624.33,112.2) -- (617.33,117.2) (555.43,58) -- (560.43,51) -- (565.43,58)  ;
			\draw [color={rgb, 255:red, 74; green, 144; blue, 226 }  ,draw opacity=1 ]   (574.33,104) -- (526.02,134.92) ;
			\draw [shift={(524.33,136)}, rotate = 327.38] [color={rgb, 255:red, 74; green, 144; blue, 226 }  ,draw opacity=1 ][line width=0.75]    (10.93,-4.9) .. controls (6.95,-2.3) and (3.31,-0.67) .. (0,0) .. controls (3.31,0.67) and (6.95,2.3) .. (10.93,4.9)   ;
			\draw [color={rgb, 255:red, 208; green, 2; blue, 27 }  ,draw opacity=1 ]   (498.33,90) -- (590.33,138) ;
			\draw [color={rgb, 255:red, 208; green, 2; blue, 27 }  ,draw opacity=1 ]   (327.33,65) -- (419.33,113) ;
			\draw   (319.8,222.51) -- (446.53,223.71) -- (392.02,244.47) -- (265.28,243.28) -- cycle ;
			\draw [color={rgb, 255:red, 208; green, 2; blue, 27 }  ,draw opacity=1 ]   (345,235) ;
			\draw [shift={(345,235)}, rotate = 0] [color={rgb, 255:red, 208; green, 2; blue, 27 }  ,draw opacity=1 ][fill={rgb, 255:red, 208; green, 2; blue, 27 }  ,fill opacity=1 ][line width=0.75]      (0, 0) circle [x radius= 3.35, y radius= 3.35]   ;
			\draw   (523.8,224.51) -- (650.53,225.71) -- (596.02,246.47) -- (469.28,245.28) -- cycle ;
			\draw [color={rgb, 255:red, 74; green, 144; blue, 226 }  ,draw opacity=1 ]   (555,237) ;
			\draw [shift={(555,237)}, rotate = 0] [color={rgb, 255:red, 74; green, 144; blue, 226 }  ,draw opacity=1 ][fill={rgb, 255:red, 74; green, 144; blue, 226 }  ,fill opacity=1 ][line width=0.75]      (0, 0) circle [x radius= 3.35, y radius= 3.35]   ;
			\draw    (338,140) -- (338.32,180) ;
			\draw [shift={(338.33,182)}, rotate = 269.55] [color={rgb, 255:red, 0; green, 0; blue, 0 }  ][line width=0.75]    (10.93,-3.29) .. controls (6.95,-1.4) and (3.31,-0.3) .. (0,0) .. controls (3.31,0.3) and (6.95,1.4) .. (10.93,3.29)   ;
			\draw    (557,140) -- (557.32,180) ;
			\draw [shift={(557.33,182)}, rotate = 269.55] [color={rgb, 255:red, 0; green, 0; blue, 0 }  ][line width=0.75]    (10.93,-3.29) .. controls (6.95,-1.4) and (3.31,-0.3) .. (0,0) .. controls (3.31,0.3) and (6.95,1.4) .. (10.93,3.29)   ;
			
			\draw (72,39.4) node [anchor=north west][inner sep=0.75pt]    {$v_{7}$};
			\draw (2,178.4) node [anchor=north west][inner sep=0.75pt]    {$v_{8}$};
			\draw (28,228.4) node [anchor=north west][inner sep=0.75pt]    {$v_{9}$};
			\draw (201,230.4) node [anchor=north west][inner sep=0.75pt]    {$v_{4}$};
			\draw (222,180.4) node [anchor=north west][inner sep=0.75pt]    {$v_{5}$};
			\draw (153,42.4) node [anchor=north west][inner sep=0.75pt]    {$v_{6}$};
			\draw (142,90.4) node [anchor=north west][inner sep=0.75pt]    {$e_{6}$};
			\draw (72,130.4) node [anchor=north west][inner sep=0.75pt]    {$\textcolor[rgb]{0.29,0.56,0.89}{e}\textcolor[rgb]{0.29,0.56,0.89}{_{1}}$};
			\draw (114,199.4) node [anchor=north west][inner sep=0.75pt]    {$\textcolor[rgb]{0.29,0.56,0.89}{e}\textcolor[rgb]{0.29,0.56,0.89}{_{2}}$};
			\draw (160,132.4) node [anchor=north west][inner sep=0.75pt]    {$\textcolor[rgb]{0.29,0.56,0.89}{e}\textcolor[rgb]{0.29,0.56,0.89}{_{3}}$};
			\draw (190,169.4) node [anchor=north west][inner sep=0.75pt]    {$e_{5}$};
			\draw (163,197.4) node [anchor=north west][inner sep=0.75pt]    {$e_{4}$};
			\draw (88,88.4) node [anchor=north west][inner sep=0.75pt]    {$e_{7}$};
			\draw (41,167.4) node [anchor=north west][inner sep=0.75pt]    {$e_{8}$};
			\draw (62,206.4) node [anchor=north west][inner sep=0.75pt]    {$e_{9}$};
			\draw (58,162.4) node [anchor=north west][inner sep=0.75pt]    {$\textcolor[rgb]{0.29,0.56,0.89}{v}\textcolor[rgb]{0.29,0.56,0.89}{_{1}}$};
			\draw (114,72.4) node [anchor=north west][inner sep=0.75pt]    {$\textcolor[rgb]{0.29,0.56,0.89}{v}\textcolor[rgb]{0.29,0.56,0.89}{_{2}}$};
			\draw (170,164.4) node [anchor=north west][inner sep=0.75pt]    {$\textcolor[rgb]{0.29,0.56,0.89}{v}\textcolor[rgb]{0.29,0.56,0.89}{_{3}}$};
			\draw (283,80.4) node [anchor=north west][inner sep=0.75pt]    {$\{x\}$};
			\draw (490,15.4) node [anchor=north west][inner sep=0.75pt]    {$\{x\}$};
			\draw (366,223.4) node [anchor=north west][inner sep=0.75pt]    {$x$};
			\draw (574,225.4) node [anchor=north west][inner sep=0.75pt]  [color={rgb, 255:red, 0; green, 0; blue, 0 }  ,opacity=1 ]  {$x$};
			\draw (324,84.4) node [anchor=north west][inner sep=0.75pt]    {$\times $};
			\draw (537,16.4) node [anchor=north west][inner sep=0.75pt]    {$\times $};
			\draw (328,251.4) node [anchor=north west][inner sep=0.75pt]    {$\mathcal{H}_{\mathbb{R}}$};
			\draw (470,254.4) node [anchor=north west][inner sep=0.75pt]    {$\mathcal{H}_{\mathbb{R}}$};
			\draw (350,150.4) node [anchor=north west][inner sep=0.75pt]    {$\pi _{\Gamma }$};
			\draw (311,11.4) node [anchor=north west][inner sep=0.75pt]    {$\pi _{\Gamma }^{-1}( x)$};
			\draw (16,16.4) node [anchor=north west][inner sep=0.75pt]    {$\Gamma $};
			\draw (615,113.4) node [anchor=north west][inner sep=0.75pt]    {$\textcolor[rgb]{0.29,0.56,0.89}{e}\textcolor[rgb]{0.29,0.56,0.89}{_{1}}$};
			\draw (509,111.4) node [anchor=north west][inner sep=0.75pt]    {$\textcolor[rgb]{0.29,0.56,0.89}{e}\textcolor[rgb]{0.29,0.56,0.89}{_{2}}$};
			\draw (536,60.4) node [anchor=north west][inner sep=0.75pt]    {$\textcolor[rgb]{0.29,0.56,0.89}{e}\textcolor[rgb]{0.29,0.56,0.89}{_{3}}$};
			\draw (439,79.4) node [anchor=north west][inner sep=0.75pt]    {$\xhookrightarrow{i}$};
			\draw (590,16.4) node [anchor=north west][inner sep=0.75pt]    {$\mathbb{R}^{3}$};

		\end{tikzpicture}

		\caption{On the left the $x$-graph $\Gamma$, on the right a sketch of $\pi_\Gamma$.}
		\label{figure_xgraphinexamplebundleofcontainer}
	\end{figure}
	Inside each fibre ${\pi_\Gamma}^{-1}(x)\cong \mathbb{R}$, we have the hyperplane arrangement $\mathcal{A}(x)$ as sketched in \cref{figure_inexamplepolyng>0}: Here, $\mathcal{A}(x)$ consists of three walls, $j=0, j=x_2 + x_3$ and $j=-x_4-x_5$, corresponding to interior edges, which together cut out two bounded chambers.
\end{example}
Building on \cref{example_bundleofcontainerspaces}, we generalize the observations concerning the "shape" of $\mathcal{A}(x)$:

\begin{remark}\label{remark_theshapeofA(x)}
	The arrangement $\mathcal{A}(x)$ exhibits structural properties that are independent of $x$, e.g. it comes with a minimal number of non-transversalities that are forced by the structure of the graph: Each vertex induces a non-transversality given by the intersection of the hyperplanes corresponding to its adjacent edges. These cannot be resolved by variations of $x$, which can only create additional non-transversalities. This means that a change in topology of $\mathcal{A}(x)$ is always accompanied by a vanishing chamber.
\end{remark}
\begin{lemma}[ \cite{MR2836109}, Lemma 2.12]\label{lemma_adequateweightvectors}
	Let $c$ be a chamber of $\mathcal{A}(x)$. The corresponding orientation allows for the set of internal vertices to be totally ordered if and only if $c$ is bounded.
\end{lemma}

\paragraph{$\mathbb{T}h^{spin,r}_g$ as a sum over $x$-graphs.} Let $\tilde{x}\in \mathcal{H}$. \cref{lemma_adequateweightvectors} identifies the \emph{locus of adequate weight vectors} in ${\pi_\Gamma}^{-1}(\tilde{x})$ as the union of bounded chambers of $\mathcal{A}(\tilde{x})$. The contribution of $\Gamma$ to $\mathbb{T}h^{spin,r}_g(\tilde{x})$ is obtained by summing the contributions from all covers $\pi$ whose underlying $x-$graph is $\Gamma$:
\begin{align*}
	\mult(\Gamma):= \sum_{\pi} \mult(\pi).
\end{align*}
Then, $\mult(\Gamma)$ is given by the weighted integral of the polynomial
\begin{align*}
	P_\Gamma(\underline{e}):= \prod^{|E^{\text{in}}(\Gamma)|}_{i=1} e_i \prod^{|V{\text{in}}(\Gamma)|}_{i=1} h_{i}(\underline{e}), 
	\thinspace	\underline{e}:=(e_1,\dots,e_{|E(\Gamma)|}) \in \mathbb{R}^{|E^{\text{in}}(\Gamma)|} \oplus \mathbb{R}^{n},
\end{align*}
over the \emph{odd} lattice points of the bounded chambers $c$ of  $\mathcal{A}(\tilde{x})$:
\begin{align*}
	\mult(\Gamma)= \frac{1}{|\Aut(\Gamma)|} \sum_{c \text{ bounded}} \text{sgn}(P_\Gamma(\underline{e}))m(c) \sum_{\underline{e} \in c \cap (2\mathbb{Z} + 1)^{|E(\Gamma)|} } P_\Gamma(\underline{e}),
\end{align*}
where
\begin{itemize}
	\item $E^{\text{in}}(\Gamma)$ ($V^{\text{in}}(\Gamma)$) is the set of internal edges (vertices).

	\item $h_{i}(\underline{e})$ is the local completed cycle spin Hurwitz number at vertex $i$ (\cref{figure_examplenotationlocalcompletedcyclehurwitznumber}).
	
	\item $ \text{sgn}(P_\Gamma(\underline{e}))$ is notation for the sign of $\prod^{|E^{\text{in}}(\Gamma)|}_{i=1} e_i $ on $c$ (This sign is constant on $c$ (see \cref{definition_bundleofcontainerhyperplanearrangement}). The sign of $h_{i}(\underline{e})$ is even constant on $\mathcal{H}_{\mathbb{R}}$ as it is an even polynomial in the edge weights (see \cref{lemma_1point})).
	\item $\Aut(\Gamma)$ is the automorphism group of the $x$-graph $\Gamma$ (disregarding edge weights and orientation).
\end{itemize}

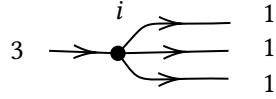
\begin{figure}[H]
	\centering
	\tikzset{every picture/.style={line width=0.75pt}} 
	\scalebox{1}{
		\begin{tikzpicture}[x=0.75pt,y=0.75pt,yscale=-1,xscale=1]
			
			\draw [color={rgb, 255:red, 0; green, 0; blue, 0 }  ,draw opacity=1 ]   (120.87,126.25) .. controls (140.14,106.06) and (126.15,112.84) .. (177.33,111) ;
			\draw [shift={(152.38,111.34)}, rotate = 180.52] [color={rgb, 255:red, 0; green, 0; blue, 0 }  ,draw opacity=1 ][line width=0.75]    (10.93,-3.29) .. controls (6.95,-1.4) and (3.31,-0.3) .. (0,0) .. controls (3.31,0.3) and (6.95,1.4) .. (10.93,3.29)   ;
			\draw    (120.87,126.25) .. controls (134.46,144.61) and (127.05,138.08) .. (176.33,139) ;
			\draw [shift={(152.07,139.11)}, rotate = 178.3] [color={rgb, 255:red, 0; green, 0; blue, 0 }  ][line width=0.75]    (10.93,-3.29) .. controls (6.95,-1.4) and (3.31,-0.3) .. (0,0) .. controls (3.31,0.3) and (6.95,1.4) .. (10.93,3.29)   ;
			\draw    (120.87,126.25) ;
			\draw [shift={(120.87,126.25)}, rotate = 0] [color={rgb, 255:red, 0; green, 0; blue, 0 }  ][fill={rgb, 255:red, 0; green, 0; blue, 0 }  ][line width=0.75]      (0, 0) circle [x radius= 3.35, y radius= 3.35]   ;
			\draw    (120.87,126.25) -- (176.33,125) ;
			\draw [shift={(154.6,125.49)}, rotate = 178.71] [color={rgb, 255:red, 0; green, 0; blue, 0 }  ][line width=0.75]    (10.93,-3.29) .. controls (6.95,-1.4) and (3.31,-0.3) .. (0,0) .. controls (3.31,0.3) and (6.95,1.4) .. (10.93,3.29)   ;
			\draw    (86.33,125) -- (120.87,126.25) ;
			\draw [shift={(109.6,125.84)}, rotate = 182.08] [color={rgb, 255:red, 0; green, 0; blue, 0 }  ][line width=0.75]    (10.93,-3.29) .. controls (6.95,-1.4) and (3.31,-0.3) .. (0,0) .. controls (3.31,0.3) and (6.95,1.4) .. (10.93,3.29)   ;
			
			\draw (191.95,117.16) node [anchor=north west][inner sep=0.75pt]    {$1$};
			\draw (65.47,118.58) node [anchor=north west][inner sep=0.75pt]    {$3$};
			\draw (55,160.4) node [anchor=north west][inner sep=0.75pt]    {$h_{i}(\underline{e}) =h_{0}^{0,2}(( 3) ,( 1,1,1))$};
			\draw (191.95,135.16) node [anchor=north west][inner sep=0.75pt]    {$1$};
			\draw (191.95,100.16) node [anchor=north west][inner sep=0.75pt]    {$1$};
			\draw (118,98.4) node [anchor=north west][inner sep=0.75pt]    {$i$};

	\end{tikzpicture}}
	
	\caption{The local completed cycle Hurwitz number at vertex $i$. The non-zero entries of $\underline{e}$ are $3,1,1,1$. }
	\label{figure_examplenotationlocalcompletedcyclehurwitznumber}
\end{figure}
Finally, we update the counting unit from $\pi$ to $\Gamma$:
\begin{align*}
	\mathbb{T}h^{spin,r}_g(\tilde{x})=\sum_{\Gamma\in \mathbb{T}S_x} \mult(\Gamma),
\end{align*}
where $\mathbb{T}S_x$ denotes the set of $x-$graphs relevant for $\mathbb{T}h^{spin,r}_g(\tilde{x})$ and extend $\mathbb{T}h^{spin,r}_g$ to $\mathcal{H}_\mathbb{R}$ in the obvious way (see \cref{proposition_multGammaispolynomial}).
\begin{remark}\label{remark_transitiontojspace}
	The decomposition 
	${\pi_\Gamma}^{-1}(\tilde{x})= V \oplus \{\tilde{x}\}\subset \mathbb{R}^{|E^{\text{in}}(\Gamma)|} \oplus \mathbb{R}^{n}$ allows us to view  ${\pi_\Gamma}^{-1}(\tilde{x})$ and $\mathcal{A}(\tilde{x})$ as subsets of $\mathbb{R}^{{b_1(\Gamma)}}$ via the isomorphism $\mathbb{R}^{|E^{\text{in}}(\Gamma)|}\cong \mathbb{R}^{{b_1(\Gamma)}}$ induced by the reference weighting (\cref{definition_xgraph}). Accordingly, we write  
	\begin{align*}
		\mult(\Gamma)(\tilde{x})= \frac{1}{|\Aut(\Gamma)|} \sum_{c(\tilde{x}) \text{ bounded}} m(c(x)) \sum_{\underline{j} \in c(\tilde{x}) \cap (2\mathbb{Z} + 1)^{{b_1(\Gamma)}} } \text{sgn}(P_\Gamma(\tilde{x},\underline{j})P_\Gamma(\tilde{x},\underline{j})
	\end{align*}
	and observe that the new counting unit $\Gamma$ behaves similarly to the genus $0$ case: It is independent of the extra parameters $\underline{j}$. However, unlike the genus $0$ case, the set $\mathbb{T}S_x$ is now even independent of specific chambers: Whether $\Gamma$ contributes to $\mathbb{T}h^{spin,r}_g(\tilde{x})$ is encoded in $\mathcal{A}(\tilde{x})$.
\end{remark}

\begin{lemma}\label{lemma_hyperplanearrangementinducedbychambercounting}
	To each $\Gamma\in \mathbb{T}S_x$ we associate an assignment $\phi_\Gamma: \mathcal{H}_\mathbb{R} \rightarrow \mathbb{Q}: x\mapsto b(\mathcal{A}(x))$ that counts the number of bounded chambers $b(\mathcal{A}(x))$ in $\mathcal{A}(x)$. Individually, each $\phi_\Gamma$ is locally constant except along certain hyperplanes. Collectively, they naturally give rise to a hyperplane arrangement $\mathcal{W}$ in $\mathcal{H}_\mathbb{R}$ consisting of hyperplanes of the form $\mathcal{W}_{I}=\{x: \sum_{i\in I} x_i=0 \}$, where $I\subsetneq \{1,\dots,n\}$ is a non-empty subset. The connected components of the complement $\mathcal{H}_\mathbb{R}\setminus \mathcal{W}$ describe precisely the locus where all functions $\phi_\Gamma$ are constant.
\end{lemma}
The proof of \cref{lemma_hyperplanearrangementinducedbychambercounting} requires the following technical ingredient:
\begin{lemma}[Variant of Farkas' Lemma \cite{zbMATH05076526}, Chapter 6]
	Let $A\in \mathbb{R}^{m\times n}$ and  $b\in \mathbb{R}^{m}$. Then either $Ax\leqslant b$ has a solution $x\in \mathbb{R}^n$ or $A^ty=0$ has a solution $y\geqslant 0$  with $b^ty<0$.
\end{lemma}
\begin{proof}
	From \cref{remark_theshapeofA(x)} we know that the number $b(\mathcal{A}(x))$ changes whenever the system of inequalities defining a bounded chamber $c$ in $\mathcal{A}(x)$
	\begin{align}\label{equation_inproofnumberofboundedchambers}
		A\cdot\omega \leqslant b, \text{ where we write }  b_i:=b(x)_i=\sum_{i \in I_i} x_i \text{ for } I_i \subset \{1,\dots,n\}
	\end{align}
	(renaming $x_i:=-x_i$ if necessary) and $A$ is a matrix with integer entries, becomes unsolvable. Consider $x'$ and $x''$ such that system (\ref{equation_inproofnumberofboundedchambers}) is solvable for $x'$ and not for $x''$. By Farkas' Lemma there exists a vector $y \geqslant 0$ that satisfies $y^tA=0$ such that  ${b(x')}^ty \geqslant 0$, but $b(x'')^ty < 0$. This can only happen, if a component of $b(x)$ changes sign, i.e. there exists $x'''$ with  $0=b(x''')_i=\sum_{i \in I_i} {x'''}_i$, thus $x''' \in \mathcal{W}_{I_i}=:\mathcal{W}_{I} $. Conversely, suppose $x''' \in \mathcal{W}_{I} $. We construct an $x-$graph $\Gamma$ such that $\phi_\Gamma$ drops at $x'''$: 
	\begin{itemize}
		\item If $|I|$ is odd: Choose $x-$graphs $\Gamma_1$ and $\Gamma_2$ that contribute to the counts $\mathbb{T}h^{spin,r}_{g_1}(x_I,-\sum_{i \in I} x_i)$ and $\mathbb{T}h^{spin,r}_{g-g_1}(\sum_{i \in I} x_i,x_{I^c})$, respectively, and glue them along their end of weight $\sum_{i \in I} x_i$.
		\item  If $|I|$ is even: Choose $x-$graphs $\Gamma_1$ and $\Gamma_2$ that contribute to $\mathbb{T}h^{spin,r}_{g_1}(x_I,-\sum_{i \in I} x_i-y,-y)$ and $\mathbb{T}h^{spin,r}_{g-g_1-1}(\sum_{i \in I} x_i+y,y,x_{I^c})$ (note: since $|I|$ is even, $g>0$) and identify the ends of weight $\sum_{i \in I} x_i+y$ and $y$, thus forming a new cycle in the glued graph.
	\end{itemize} 
	
	It follows immediately that for $x$ sufficiently generic (i.e. $x \in \mathcal{H}_\mathbb{R}\setminus \mathcal{W}$) $\phi_\Gamma$ is constant in a neighborhood of $x$ (since the topology of $\mathcal{A}(x)$ is preserved) and that the locus where all functions $\phi_\Gamma$ are constant corresponds exactly to the connected components of the complement $\mathcal{H}_\mathbb{R}\setminus \mathcal{W}$.
\end{proof}

\begin{proposition}\label{proposition_multGammaispolynomial}
	
	Let $\Gamma\in \mathbb{T}S_x$ contribute to $\mathbb{T}h^{spin,r}_{g}$. For $\tilde{x} \in \mathcal{H} \cap(\mathcal{H}_\mathbb{R}\setminus \mathcal{W})$, $\mult(\Gamma)$ is a polynomial in the $\tilde{x}_i$ of degree $2g(\Gamma) + |V^\text{in}(\Gamma)| - 1 = 2g+b-1$.
\end{proposition}
\cref{proposition_multGammaispolynomial} falls naturally within the scope of Ehrhart theory. We present a proof here while deferring the proofs of certain required Ehrhart-theoretic results until later.

\begin{proof}
	A bounded chamber $c(\tilde{x})$ in $\mathcal{A}(\tilde{x})\subset \mathbb{R}^{b_1(\Gamma)}$ induces a family $\mathbb{Z}^{n} \rightarrow \mathbb{R}^{b_1(\Gamma)}, x \mapsto c(x)$ of integer polytopes such that:
	The topology is fixed across the family (i.e. the set of facets including their \emph{outward} normal vector is fixed up to translation and scaling), as long as $x$ varies within the connected component $c_H$ of $ \mathcal{H}_\mathbb{R}\setminus \mathcal{W}$ that contains $\tilde{x}$ (\cref{lemma_hyperplanearrangementinducedbychambercounting}).
	More precisely, $c(x)$ is the convex hull of a \emph{fixed} set $V(c(x))$ of $x$-dependent vertices, whose coordinates are affine integer linear functions in $x$.
	
	Applying the affine map $\phi:  \mathbb{R}^{b_1(\Gamma)}\rightarrow \mathbb{R}^{b_1(\Gamma)}, j \mapsto \frac{1}{2}\text{Id}j - \frac{1}{2} \begin{psmallmatrix}
		1\\
		\vdots\\
		1
	\end{psmallmatrix}$ yields another family of polytopes $\mathbb{Z}^{n} \cap c_H \rightarrow \mathbb{R}^{b_1(\Gamma)}, x \mapsto \tilde{c}(x):=\phi(c(x))$ with vertex set $\{\phi(v(x))| v(x)\in V(c(x)) \}$, where vertex coordinates are now affine linear functions in $x$ with coefficients in $\frac{1}{2}\mathbb{Z}$.

	Then the contribution 
	\begin{align}\label{equation_inproofpolynofmult}
		\sum_{j \in c(x) \cap (2\mathbb{Z} + 1)^{b_1(\Gamma)} } \text{sgn}(P_\Gamma(x,j))P_\Gamma(x,j)
	\end{align}
	of $c(x)$ to $\mult(\Gamma)$ may be rewritten in terms of $\tilde{c}(x)$ as
	\begin{align*}
		\sum_{k \in \tilde{c}(x) \cap \mathbb{Z}^{b_1(\Gamma)} } \text{sgn}(\tilde{P}_\Gamma(x,k))\tilde{P}_\Gamma(x,k)=\sum_{\alpha} c_\alpha(x)\sum_{k \in \tilde{c}(x) \cap \mathbb{Z}^{b_1(\Gamma)} }k^\alpha, \text{ where } 
	\end{align*}
	$\tilde{P}_\Gamma(x,k):=P_\Gamma(x,2k+1):=\sum_{\alpha\in \mathbb{N}^{b_1(\Gamma)}} c_\alpha(x)k^\alpha$ is its decomposition as polynomial in $k$ with coefficients $c_\alpha(x)\in \mathbb{Q}[x]$. 
	By \cref{lemma_qpofperiodatmost2} (for the family $\tilde{c}(x)$ and polynomial weight $k^\alpha$)
	\begin{align*}
		F:	\mathbb{Z}^{n} \cap c_H \rightarrow \mathbb{R}, x \mapsto \sum_{ k \in \tilde{c}(x) \cap \mathbb{Z}^{b_1(\Gamma)}} k^\alpha
	\end{align*}
	is a quasi-polynomial of degree $\dim(\tilde{c}(x)) + |\alpha| = b_1(\Gamma) + |\alpha|$ and period "at most" $(2,\dots,2)^T\in \mathbb{Z}^{n}$, i.e. there exists at most $2^n$ polynomials $F_i$ such that
	\begin{align*}
		F(x)=\begin{cases}
			F_1(x), x_i \equiv 1 \text{ mod } 2 \text{ for } i=1,\dots,n.\\
			F_2(x), x_1 \equiv 0  \text{ mod } 2, x_i \equiv 1 \text{ mod } 2 \text{ for } i=2,\dots,n.\\
			\vdots\\
			F_{2^n}, x_i \equiv 0 \text{ mod } 2 \text{ for } i=1,\dots,n.\\
		\end{cases}
	\end{align*} 
	Restricting to $\mathcal{H} \cap c_H$, one obtains polynomials
	\begin{align*}
		\mathcal{H} \cap c_H \ni x \mapsto \sum_{ k \in \tilde{c}(x) \cap \mathbb{Z}^{b_1(\Gamma)}} k^\alpha \text{ and } \mathcal{H} \cap c_H \ni x \mapsto \sum_{\alpha} c_\alpha(x)\sum_{k \in \tilde{c}(x) \cap \mathbb{Z}^{b_1(\Gamma)} }k^\alpha
	\end{align*}
	of degree $b_1(\Gamma) + |\alpha|$ and $b_1(\Gamma) + \deg(P_\Gamma(x,j))$. Combining the Euler characteristic formula for $\Gamma$ with the fact that the local completed-cycles spin Hurwitz number $h_i(x,j)$ at vertex $v_i$ has degree $2g(v_i)$ (\cite{arXiv:2104.05697}, Theorem 6.4), shows
	\begin{align*}
		&	\deg(P_\Gamma(x,j)) = |E^\text{in}(\Gamma)| + \sum^{|V^{in}(\Gamma)|}_{i=1} \deg(h_i(x,j)) =
		|V^{in}(\Gamma)| + b_1(\Gamma) - 1 + \sum^{|V^{in}(\Gamma)|}_{i=1} 2g(v_i) \text{ and }\\
		&	b_1(\Gamma) + \deg(P_\Gamma(x,j)) = 2g(\Gamma) + |V^{in}(\Gamma)| - 1.		
	\end{align*}
	It follows that $\mathcal{H} \cap c_H \ni x\mapsto \mult(\Gamma)(x)$ is a polynomial of degree $2g(\Gamma) + |V^{in}(\Gamma)| - 1= 2g + b -1$, since $\Gamma$ contributes to $\mathbb{T}h^{spin,r}_{g}$.
\end{proof} 
\begin{example}\label{example_polynomialityg>0}
	We return to \cref{example_bundleofcontainerspaces}: Let $\Gamma$ be the $x$-graph from \cref{figure_xgraphinexamplebundleofcontainer} and consider the connected component $c_H$ of $\mathcal{H}_\mathbb{R}\setminus \mathcal{W}$ defined by: $x_1,x_3,x_3 \geqslant 0$, $x_4,x_5,x_6 \leqslant 0$, $x_2 + x_3 \leqslant -x_4 -x _5$. \cref{figure_inexamplepolyng>0} illustrates for $\tilde{x}=(5,3,3,-5,-5,-1)\in c_H$  the hyperplane arrangement $\mathcal{A}(\tilde{x})$ with its two bounded chambers, $c_1(\tilde{x})$ and $c_2(\tilde{x})$, and the corresponding induced orientations.
	\begin{figure}[H]
		\tikzset{every picture/.style={line width=0.75pt}} 
		\centering       
		\scalebox{0.78}{
			\begin{tikzpicture}[x=0.75pt,y=0.75pt,yscale=-1,xscale=1]
				
				\draw    (333.4,134.56) .. controls (355.59,130.97) and (367.58,120.62) .. (378.4,106.86) ;
				\draw [shift={(363.88,121.84)}, rotate = 144.84] [color={rgb, 255:red, 0; green, 0; blue, 0 }  ][line width=0.75]    (10.93,-3.29) .. controls (6.95,-1.4) and (3.31,-0.3) .. (0,0) .. controls (3.31,0.3) and (6.95,1.4) .. (10.93,3.29)   ;
				\draw    (333.4,134.56) .. controls (356.01,142.74) and (402.29,146.56) .. (421.53,132.01) ;
				\draw [shift={(384.25,141.76)}, rotate = 179.55] [color={rgb, 255:red, 0; green, 0; blue, 0 }  ][line width=0.75]    (10.93,-3.29) .. controls (6.95,-1.4) and (3.31,-0.3) .. (0,0) .. controls (3.31,0.3) and (6.95,1.4) .. (10.93,3.29)   ;
				\draw [shift={(333.4,134.56)}, rotate = 19.89] [color={rgb, 255:red, 0; green, 0; blue, 0 }  ][fill={rgb, 255:red, 0; green, 0; blue, 0 }  ][line width=0.75]      (0, 0) circle [x radius= 3.35, y radius= 3.35]   ;
				\draw    (421.53,132.01) .. controls (406.37,132.06) and (387.78,121.56) .. (378.4,106.86) ;
				\draw [shift={(403.22,127.74)}, rotate = 207.1] [color={rgb, 255:red, 0; green, 0; blue, 0 }  ][line width=0.75]    (10.93,-3.29) .. controls (6.95,-1.4) and (3.31,-0.3) .. (0,0) .. controls (3.31,0.3) and (6.95,1.4) .. (10.93,3.29)   ;
				\draw    (307.01,140.39) .. controls (315.21,144.01) and (328.08,141.11) .. (333.4,134.56) ;
				\draw [shift={(326.77,139.45)}, rotate = 161.67] [color={rgb, 255:red, 0; green, 0; blue, 0 }  ][line width=0.75]    (10.93,-3.29) .. controls (6.95,-1.4) and (3.31,-0.3) .. (0,0) .. controls (3.31,0.3) and (6.95,1.4) .. (10.93,3.29)   ;
				\draw    (306.33,92.4) .. controls (310.43,88.06) and (355.01,87.53) .. (378.4,106.86) ;
				\draw [shift={(350.08,93.52)}, rotate = 192.72] [color={rgb, 255:red, 0; green, 0; blue, 0 }  ][line width=0.75]    (10.93,-3.29) .. controls (6.95,-1.4) and (3.31,-0.3) .. (0,0) .. controls (3.31,0.3) and (6.95,1.4) .. (10.93,3.29)   ;
				\draw    (307.6,130.25) .. controls (311.7,127.35) and (326.33,125.91) .. (333.4,134.56) ;
				\draw [shift={(327.11,129.89)}, rotate = 195.6] [color={rgb, 255:red, 0; green, 0; blue, 0 }  ][line width=0.75]    (10.93,-4.9) .. controls (6.95,-2.3) and (3.31,-0.67) .. (0,0) .. controls (3.31,0.67) and (6.95,2.3) .. (10.93,4.9)   ;
				\draw    (378.4,106.86) .. controls (385.44,97.67) and (428.75,85.36) .. (452.75,88.98) ;
				\draw [shift={(420.08,90.71)}, rotate = 168.36] [color={rgb, 255:red, 0; green, 0; blue, 0 }  ][line width=0.75]    (10.93,-3.29) .. controls (6.95,-1.4) and (3.31,-0.3) .. (0,0) .. controls (3.31,0.3) and (6.95,1.4) .. (10.93,3.29)   ;
				\draw [shift={(378.4,106.86)}, rotate = 307.44] [color={rgb, 255:red, 0; green, 0; blue, 0 }  ][fill={rgb, 255:red, 0; green, 0; blue, 0 }  ][line width=0.75]      (0, 0) circle [x radius= 3.35, y radius= 3.35]   ;
				\draw   (132.29,121.25) -- (182.08,203.08) -- (82.5,203.08) -- cycle ;
				\draw    (182.08,203.08) -- (201.36,235.62) ;
				\draw [shift={(188.15,213.33)}, rotate = 59.36] [color={rgb, 255:red, 0; green, 0; blue, 0 }  ][line width=0.75]    (10.93,-3.29) .. controls (6.95,-1.4) and (3.31,-0.3) .. (0,0) .. controls (3.31,0.3) and (6.95,1.4) .. (10.93,3.29)   ;
				\draw    (182.08,203.08) -- (218.71,203.08) ;
				\draw [shift={(193.4,203.08)}, rotate = 0] [color={rgb, 255:red, 0; green, 0; blue, 0 }  ][line width=0.75]    (10.93,-3.29) .. controls (6.95,-1.4) and (3.31,-0.3) .. (0,0) .. controls (3.31,0.3) and (6.95,1.4) .. (10.93,3.29)   ;
				\draw    (82.5,203.08) -- (62.58,239.56) ;
				\draw [shift={(75.89,215.18)}, rotate = 118.63] [color={rgb, 255:red, 0; green, 0; blue, 0 }  ][line width=0.75]    (10.93,-3.29) .. controls (6.95,-1.4) and (3.31,-0.3) .. (0,0) .. controls (3.31,0.3) and (6.95,1.4) .. (10.93,3.29)   ;
				\draw    (44.27,203.08) -- (82.5,203.08) ;
				\draw [shift={(69.38,203.08)}, rotate = 180] [color={rgb, 255:red, 0; green, 0; blue, 0 }  ][line width=0.75]    (10.93,-3.29) .. controls (6.95,-1.4) and (3.31,-0.3) .. (0,0) .. controls (3.31,0.3) and (6.95,1.4) .. (10.93,3.29)   ;
				\draw    (132.29,121.25) -- (108.84,83.79) ;
				\draw [shift={(124.28,108.45)}, rotate = 237.95] [color={rgb, 255:red, 0; green, 0; blue, 0 }  ][line width=0.75]    (10.93,-3.29) .. controls (6.95,-1.4) and (3.31,-0.3) .. (0,0) .. controls (3.31,0.3) and (6.95,1.4) .. (10.93,3.29)   ;
				\draw    (155.1,85.76) -- (132.29,121.25) ;
				\draw [shift={(140.45,108.55)}, rotate = 302.73] [color={rgb, 255:red, 0; green, 0; blue, 0 }  ][line width=0.75]    (10.93,-3.29) .. controls (6.95,-1.4) and (3.31,-0.3) .. (0,0) .. controls (3.31,0.3) and (6.95,1.4) .. (10.93,3.29)   ;
				\draw   (99.37,164.68) -- (109.44,158.79) -- (107.88,170.55) ;
				\draw   (158.09,155.75) -- (161.01,167.23) -- (150.33,162.61) ;
				\draw   (122.32,198.61) -- (132.93,203.39) -- (122.76,209.09) ;
				\draw    (82.5,203.08) ;
				\draw [shift={(82.5,203.08)}, rotate = 0] [color={rgb, 255:red, 0; green, 0; blue, 0 }  ][fill={rgb, 255:red, 0; green, 0; blue, 0 }  ][line width=0.75]      (0, 0) circle [x radius= 3.35, y radius= 3.35]   ;
				\draw    (62.58,239.56) ;
				\draw [shift={(62.58,239.56)}, rotate = 0] [color={rgb, 255:red, 0; green, 0; blue, 0 }  ][fill={rgb, 255:red, 0; green, 0; blue, 0 }  ][line width=0.75]      (0, 0) circle [x radius= 3.35, y radius= 3.35]   ;
				\draw    (44.27,203.08) ;
				\draw [shift={(44.27,203.08)}, rotate = 0] [color={rgb, 255:red, 0; green, 0; blue, 0 }  ][fill={rgb, 255:red, 0; green, 0; blue, 0 }  ][line width=0.75]      (0, 0) circle [x radius= 3.35, y radius= 3.35]   ;
				\draw    (108.84,83.79) ;
				\draw [shift={(108.84,83.79)}, rotate = 0] [color={rgb, 255:red, 0; green, 0; blue, 0 }  ][fill={rgb, 255:red, 0; green, 0; blue, 0 }  ][line width=0.75]      (0, 0) circle [x radius= 3.35, y radius= 3.35]   ;
				\draw    (155.1,85.76) ;
				\draw [shift={(155.1,85.76)}, rotate = 0] [color={rgb, 255:red, 0; green, 0; blue, 0 }  ][fill={rgb, 255:red, 0; green, 0; blue, 0 }  ][line width=0.75]      (0, 0) circle [x radius= 3.35, y radius= 3.35]   ;
				\draw    (132.29,121.25) ;
				\draw [shift={(132.29,121.25)}, rotate = 0] [color={rgb, 255:red, 0; green, 0; blue, 0 }  ][fill={rgb, 255:red, 0; green, 0; blue, 0 }  ][line width=0.75]      (0, 0) circle [x radius= 3.35, y radius= 3.35]   ;
				\draw    (182.08,203.08) ;
				\draw [shift={(182.08,203.08)}, rotate = 0] [color={rgb, 255:red, 0; green, 0; blue, 0 }  ][fill={rgb, 255:red, 0; green, 0; blue, 0 }  ][line width=0.75]      (0, 0) circle [x radius= 3.35, y radius= 3.35]   ;
				\draw    (218.71,203.08) ;
				\draw [shift={(218.71,203.08)}, rotate = 0] [color={rgb, 255:red, 0; green, 0; blue, 0 }  ][fill={rgb, 255:red, 0; green, 0; blue, 0 }  ][line width=0.75]      (0, 0) circle [x radius= 3.35, y radius= 3.35]   ;
				\draw    (201.36,235.62) ;
				\draw [shift={(201.36,235.62)}, rotate = 0] [color={rgb, 255:red, 0; green, 0; blue, 0 }  ][fill={rgb, 255:red, 0; green, 0; blue, 0 }  ][line width=0.75]      (0, 0) circle [x radius= 3.35, y radius= 3.35]   ;
				\draw    (294.33,209.33) -- (325.79,209.3) -- (357.24,209.27) -- (388.7,209.24) -- (420.15,209.21) -- (451.61,209.18) -- (483.06,209.15) -- (514.52,209.12) -- (550.33,209.67) -- (593.15,209.05) -- (639.98,208.68) ;
				\draw [shift={(642.33,208.67)}, rotate = 359.56] [color={rgb, 255:red, 0; green, 0; blue, 0 }  ][line width=0.75]      (0, 0) circle [x radius= 3.35, y radius= 3.35]   ;
				\draw [shift={(310.06,209.32)}, rotate = 359.94] [color={rgb, 255:red, 0; green, 0; blue, 0 }  ][line width=0.75]      (0, 0) circle [x radius= 3.35, y radius= 3.35]   ;
				\draw [shift={(341.52,209.29)}, rotate = 359.94] [color={rgb, 255:red, 0; green, 0; blue, 0 }  ][line width=0.75]      (0, 0) circle [x radius= 3.35, y radius= 3.35]   ;
				\draw [shift={(372.97,209.26)}, rotate = 359.94] [color={rgb, 255:red, 0; green, 0; blue, 0 }  ][line width=0.75]      (0, 0) circle [x radius= 3.35, y radius= 3.35]   ;
				\draw [shift={(404.42,209.23)}, rotate = 359.94] [color={rgb, 255:red, 0; green, 0; blue, 0 }  ][line width=0.75]      (0, 0) circle [x radius= 3.35, y radius= 3.35]   ;
				\draw [shift={(435.88,209.2)}, rotate = 359.94] [color={rgb, 255:red, 0; green, 0; blue, 0 }  ][line width=0.75]      (0, 0) circle [x radius= 3.35, y radius= 3.35]   ;
				\draw [shift={(467.33,209.17)}, rotate = 359.94] [color={rgb, 255:red, 0; green, 0; blue, 0 }  ][line width=0.75]      (0, 0) circle [x radius= 3.35, y radius= 3.35]   ;
				\draw [shift={(498.79,209.14)}, rotate = 359.94] [color={rgb, 255:red, 0; green, 0; blue, 0 }  ][line width=0.75]      (0, 0) circle [x radius= 3.35, y radius= 3.35]   ;
				\draw [shift={(532.42,209.39)}, rotate = 0.87] [color={rgb, 255:red, 0; green, 0; blue, 0 }  ][line width=0.75]      (0, 0) circle [x radius= 3.35, y radius= 3.35]   ;
				\draw [shift={(571.74,209.36)}, rotate = 359.17] [color={rgb, 255:red, 0; green, 0; blue, 0 }  ][line width=0.75]      (0, 0) circle [x radius= 3.35, y radius= 3.35]   ;
				\draw [shift={(617.74,208.86)}, rotate = 359.56] [color={rgb, 255:red, 0; green, 0; blue, 0 }  ][line width=0.75]      (0, 0) circle [x radius= 3.35, y radius= 3.35]   ;
				\draw [color={rgb, 255:red, 208; green, 2; blue, 27 }  ,draw opacity=1 ]   (404.42,209.23) ;
				\draw [shift={(404.42,209.23)}, rotate = 0] [color={rgb, 255:red, 208; green, 2; blue, 27 }  ,draw opacity=1 ][fill={rgb, 255:red, 208; green, 2; blue, 27 }  ,fill opacity=1 ][line width=0.75]      (0, 0) circle [x radius= 3.35, y radius= 3.35]   ;
				\draw    (453.33,119.39) .. controls (449.24,119.39) and (440.46,119.39) .. (421.53,132.01) ;
				\draw [shift={(421.53,132.01)}, rotate = 146.31] [color={rgb, 255:red, 0; green, 0; blue, 0 }  ][fill={rgb, 255:red, 0; green, 0; blue, 0 }  ][line width=0.75]      (0, 0) circle [x radius= 3.35, y radius= 3.35]   ;
				\draw [shift={(442.99,120.97)}, rotate = 159.32] [color={rgb, 255:red, 0; green, 0; blue, 0 }  ][line width=0.75]    (10.93,-3.29) .. controls (6.95,-1.4) and (3.31,-0.3) .. (0,0) .. controls (3.31,0.3) and (6.95,1.4) .. (10.93,3.29)   ;
				\draw    (421.53,132.01) .. controls (430.87,144.27) and (445.17,144.27) .. (453.33,138.94) ;
				\draw [shift={(442.15,142.32)}, rotate = 185.5] [color={rgb, 255:red, 0; green, 0; blue, 0 }  ][line width=0.75]    (10.93,-3.29) .. controls (6.95,-1.4) and (3.31,-0.3) .. (0,0) .. controls (3.31,0.3) and (6.95,1.4) .. (10.93,3.29)   ;
				\draw    (545.33,135.67) .. controls (526.33,131.67) and (513.33,127.67) .. (511.33,108.67) ;
				\draw [shift={(529.1,131.34)}, rotate = 204.78] [color={rgb, 255:red, 0; green, 0; blue, 0 }  ][line width=0.75]    (10.93,-3.29) .. controls (6.95,-1.4) and (3.31,-0.3) .. (0,0) .. controls (3.31,0.3) and (6.95,1.4) .. (10.93,3.29)   ;
				\draw    (545.33,135.67) .. controls (570.33,139.67) and (604.33,129.67) .. (607.15,125.5) ;
				\draw [shift={(583.11,133.94)}, rotate = 169.91] [color={rgb, 255:red, 0; green, 0; blue, 0 }  ][line width=0.75]    (10.93,-3.29) .. controls (6.95,-1.4) and (3.31,-0.3) .. (0,0) .. controls (3.31,0.3) and (6.95,1.4) .. (10.93,3.29)   ;
				\draw [shift={(545.33,135.67)}, rotate = 9.09] [color={rgb, 255:red, 0; green, 0; blue, 0 }  ][fill={rgb, 255:red, 0; green, 0; blue, 0 }  ][line width=0.75]      (0, 0) circle [x radius= 3.35, y radius= 3.35]   ;
				\draw    (607.15,125.5) .. controls (596.33,100.67) and (539.33,105.67) .. (511.33,108.67) ;
				\draw [shift={(569.1,106.95)}, rotate = 184.55] [color={rgb, 255:red, 0; green, 0; blue, 0 }  ][line width=0.75]    (10.93,-3.29) .. controls (6.95,-1.4) and (3.31,-0.3) .. (0,0) .. controls (3.31,0.3) and (6.95,1.4) .. (10.93,3.29)   ;
				\draw    (468.33,151.67) .. controls (477.67,155.27) and (560.33,140.67) .. (545.33,135.67) ;
				\draw [shift={(515.96,146.89)}, rotate = 169.6] [color={rgb, 255:red, 0; green, 0; blue, 0 }  ][line width=0.75]    (10.93,-3.29) .. controls (6.95,-1.4) and (3.31,-0.3) .. (0,0) .. controls (3.31,0.3) and (6.95,1.4) .. (10.93,3.29)   ;
				\draw    (465.33,103.67) .. controls (469.33,103.67) and (489.33,100.67) .. (511.33,108.67) ;
				\draw [shift={(494.82,104.41)}, rotate = 187.37] [color={rgb, 255:red, 0; green, 0; blue, 0 }  ][line width=0.75]    (10.93,-3.29) .. controls (6.95,-1.4) and (3.31,-0.3) .. (0,0) .. controls (3.31,0.3) and (6.95,1.4) .. (10.93,3.29)   ;
				\draw    (468.33,133.67) .. controls (467.33,138.67) and (537.33,135.67) .. (545.33,135.67) ;
				\draw [shift={(512.85,136.43)}, rotate = 179.22] [color={rgb, 255:red, 0; green, 0; blue, 0 }  ][line width=0.75]    (10.93,-4.9) .. controls (6.95,-2.3) and (3.31,-0.67) .. (0,0) .. controls (3.31,0.67) and (6.95,2.3) .. (10.93,4.9)   ;
				\draw    (511.33,108.67) .. controls (518.33,66.67) and (620,66.06) .. (647.33,69.67) ;
				\draw [shift={(578.87,70.98)}, rotate = 171.53] [color={rgb, 255:red, 0; green, 0; blue, 0 }  ][line width=0.75]    (10.93,-3.29) .. controls (6.95,-1.4) and (3.31,-0.3) .. (0,0) .. controls (3.31,0.3) and (6.95,1.4) .. (10.93,3.29)   ;
				\draw [shift={(511.33,108.67)}, rotate = 279.46] [color={rgb, 255:red, 0; green, 0; blue, 0 }  ][fill={rgb, 255:red, 0; green, 0; blue, 0 }  ][line width=0.75]      (0, 0) circle [x radius= 3.35, y radius= 3.35]   ;
				\draw    (650.33,118.67) .. controls (645.67,118.67) and (628.71,112.94) .. (607.15,125.5) ;
				\draw [shift={(607.15,125.5)}, rotate = 149.77] [color={rgb, 255:red, 0; green, 0; blue, 0 }  ][fill={rgb, 255:red, 0; green, 0; blue, 0 }  ][line width=0.75]      (0, 0) circle [x radius= 3.35, y radius= 3.35]   ;
				\draw [shift={(634.85,117.38)}, rotate = 175.51] [color={rgb, 255:red, 0; green, 0; blue, 0 }  ][line width=0.75]    (10.93,-3.29) .. controls (6.95,-1.4) and (3.31,-0.3) .. (0,0) .. controls (3.31,0.3) and (6.95,1.4) .. (10.93,3.29)   ;
				\draw    (607.15,125.5) .. controls (618,136.14) and (645.13,137.46) .. (651.33,133.67) ;
				\draw [shift={(634.42,135.37)}, rotate = 187.17] [color={rgb, 255:red, 0; green, 0; blue, 0 }  ][line width=0.75]    (10.93,-3.29) .. controls (6.95,-1.4) and (3.31,-0.3) .. (0,0) .. controls (3.31,0.3) and (6.95,1.4) .. (10.93,3.29)   ;
				\draw [color={rgb, 255:red, 208; green, 2; blue, 27 }  ,draw opacity=1 ]   (341.52,209.29) ;
				\draw [shift={(341.52,209.29)}, rotate = 0] [color={rgb, 255:red, 208; green, 2; blue, 27 }  ,draw opacity=1 ][fill={rgb, 255:red, 208; green, 2; blue, 27 }  ,fill opacity=1 ][line width=0.75]      (0, 0) circle [x radius= 3.35, y radius= 3.35]   ;
				\draw [color={rgb, 255:red, 208; green, 2; blue, 27 }  ,draw opacity=1 ]   (467.33,209.17) ;
				\draw [shift={(467.33,209.17)}, rotate = 0] [color={rgb, 255:red, 208; green, 2; blue, 27 }  ,draw opacity=1 ][fill={rgb, 255:red, 208; green, 2; blue, 27 }  ,fill opacity=1 ][line width=0.75]      (0, 0) circle [x radius= 3.35, y radius= 3.35]   ;
				\draw [color={rgb, 255:red, 208; green, 2; blue, 27 }  ,draw opacity=1 ]   (532.42,209.39) ;
				\draw [shift={(532.42,209.39)}, rotate = 0] [color={rgb, 255:red, 208; green, 2; blue, 27 }  ,draw opacity=1 ][fill={rgb, 255:red, 208; green, 2; blue, 27 }  ,fill opacity=1 ][line width=0.75]      (0, 0) circle [x radius= 3.35, y radius= 3.35]   ;
				\draw [color={rgb, 255:red, 208; green, 2; blue, 27 }  ,draw opacity=1 ]   (593.15,209.05) ;
				\draw [shift={(593.15,209.05)}, rotate = 0] [color={rgb, 255:red, 208; green, 2; blue, 27 }  ,draw opacity=1 ][fill={rgb, 255:red, 208; green, 2; blue, 27 }  ,fill opacity=1 ][line width=0.75]      (0, 0) circle [x radius= 3.35, y radius= 3.35]   ;
				\draw [color={rgb, 255:red, 208; green, 2; blue, 27 }  ,draw opacity=1 ]   (642.33,208.67) ;
				\draw [shift={(642.33,208.67)}, rotate = 0] [color={rgb, 255:red, 208; green, 2; blue, 27 }  ,draw opacity=1 ][fill={rgb, 255:red, 208; green, 2; blue, 27 }  ,fill opacity=1 ][line width=0.75]      (0, 0) circle [x radius= 3.35, y radius= 3.35]   ;
				
				\draw (77.37,71.24) node [anchor=north west][inner sep=0.75pt]    {$x_{1}$};
				\draw (16.66,196.45) node [anchor=north west][inner sep=0.75pt]    {$x_{2}$};
				\draw (33.04,233.92) node [anchor=north west][inner sep=0.75pt]    {$x_{3}$};
				\draw (219.04,235.89) node [anchor=north west][inner sep=0.75pt]    {$x_{4}$};
				\draw (228.68,193.5) node [anchor=north west][inner sep=0.75pt]    {$x_{5}$};
				\draw (169.89,72.23) node [anchor=north west][inner sep=0.75pt]    {$x_{6}$};
				\draw (133.36,211.26) node [anchor=north west][inner sep=0.75pt]    {$j$};
				\draw (31.04,136.31) node [anchor=north west][inner sep=0.75pt]    {$x_{2} +x_{3} -j$};
				\draw (161.95,136.31) node [anchor=north west][inner sep=0.75pt]    {$-x_{4} -x_{5} -j$};
				\draw (305,220.4) node [anchor=north west][inner sep=0.75pt]    {$0$};
				\draw (475,221.4) node [anchor=north west][inner sep=0.75pt]    {$x_{2} +x_{3}$};
				\draw (586,218.4) node [anchor=north west][inner sep=0.75pt]    {$-x_{4} -x_{5}$};
				\draw (636,171.4) node [anchor=north west][inner sep=0.75pt]    {$j$};
				\draw (356,178.4) node [anchor=north west][inner sep=0.75pt]    {$c_{2}\left(\tilde{x}\right)$};
				\draw (527,179.4) node [anchor=north west][inner sep=0.75pt]    {$c_{1}\left(\tilde{x}\right)$};
				\draw (30,34.4) node [anchor=north west][inner sep=0.75pt]    {$\Gamma $};
				\draw (304,25.4) node [anchor=north west][inner sep=0.75pt]    {$\mathcal{A}\left(\tilde{x}\right)$};
				\draw (508.25,238.6) node [anchor=north west][inner sep=0.75pt]  [rotate=-90.75]  {$=$};
				\draw (627.25,234.6) node [anchor=north west][inner sep=0.75pt]  [rotate=-90.75]  {$=$};
				\draw (496,259.4) node [anchor=north west][inner sep=0.75pt]    {$6$};
				\draw (611,259.4) node [anchor=north west][inner sep=0.75pt]    {$10$};

		\end{tikzpicture}}

		\caption{The $x$-graph from \cref{example_polynomialityg>0} with reference weighting and hyperplane arrangement $\mathcal{A}(\tilde{x})$ for $\tilde{x}\in c_H$.}
		\label{figure_inexamplepolyng>0}
	\end{figure}
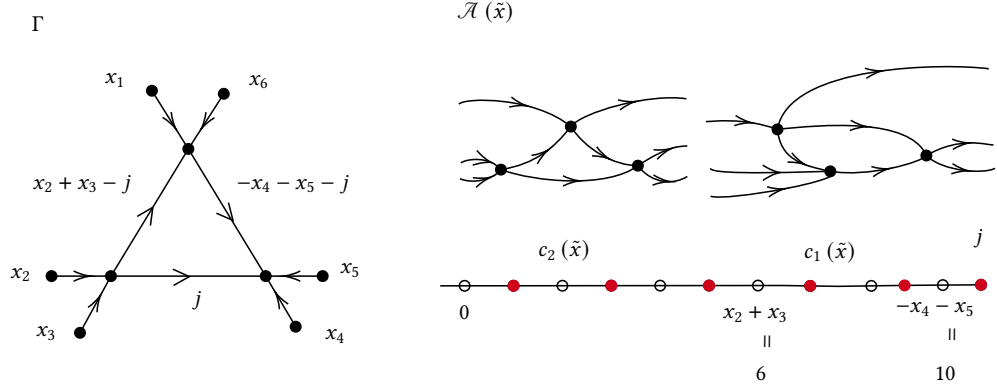
	We follow the strategy outlined in the proof of \cref{proposition_multGammaispolynomial} to determine $\Gamma$'s contribution to $\mathbb{T}h^{spin,2,+}_{1}$. As $\text{mult}(\Gamma)$ splits additively across the chambers, $c_1(\tilde{x})$ and $c_2(\tilde{x})$, we can handle each chamber individually. Below, we detail the computation for $c_2(\tilde{x})$, $c_1(\tilde{x})$ is treated similarly.
	We can express the contribution of $c_H\cap (2\mathbb{Z} + 1) \ni x \mapsto c_2(x)$ to $\mult(\Gamma)$ as
	\begin{align*}
		&	m(c_2(x))\sum_{j \in c_2(x) \cap (2\mathbb{Z} + 1)} \text{sgn}(P_\Gamma(x,j))P_\Gamma(x,j)= \sum_{0 \leqslant j \leqslant x_2 + x_3 \text{ odd}}  P_\Gamma(x,j), \text{where}\\
	\end{align*}
	($2!$ accounts for the local vertex multiplicities), in terms of the family $\tilde{c}_2(x)$ obtained via the affine coordinate change "$j=2k+1$" (see \cref{figure_inexampleforaffinecoordchange}):
	\begin{align*}
		\sum_{k \in \tilde{c}_2(x) \cap \mathbb{Z} } \tilde{P}_\Gamma(x,k)=\sum^{\lfloor\frac{x_2+x_3}{2} -\frac{1}{2} \rfloor}_{k=0}(2!)^3(x_2 + x_3 -2k-1)(-x_4-x_5 -2k-1)(2k +1). 
	\end{align*}
	Since all $x_i$ are odd, the upper summation bound  $\lfloor\frac{x_2+x_3}{2} \rfloor$ is equal to $\frac{x_2+x_3}{2} -1$, a polynomial (as opposed to a quasi-polynomial) expression in the $x_i$. This shows 
	\begin{align*}
		\sum_{k \in \tilde{c}_2(x) \cap \mathbb{Z} } \tilde{P}_\Gamma(x,k)=\sum^{\frac{x_2+x_3}{2} -1 }_{k=0}(2!)^3(x_2 + x_3 -2k-1)(-x_4-x_5 -2k-1)(2k +1)\\= \frac{-8}{24}(x_2 + x_3)(x^2_2 + 2x_2x_3+ x^2_3 +2)(2x_4+x_2+x_3+2x_5) ,
	\end{align*}
	which is indeed a polynomial of degree $2g + 3 - 1= 4$.
	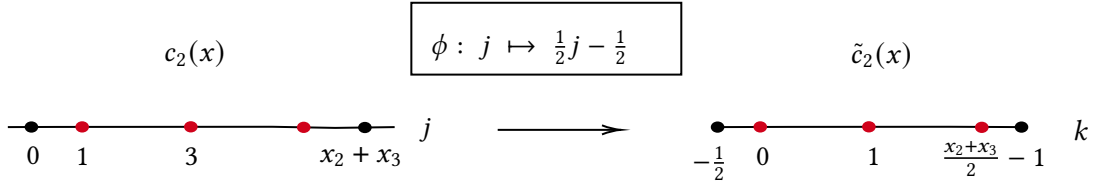
\begin{figure}
		\tikzset{every picture/.style={line width=0.75pt}} 
		
		\begin{tikzpicture}[x=0.65pt,y=0.65pt,yscale=-0.8,xscale=1]
			
			\draw    (14.33,149) -- (41.65,148.88) -- (73.11,148.85) -- (104.56,148.82) -- (136.02,148.79) -- (167.47,148.76) -- (203.86,149.58) -- (238.33,149) ;
			\draw [color={rgb, 255:red, 208; green, 2; blue, 27 }  ,draw opacity=1 ]   (120.29,148.8) ;
			\draw [shift={(120.29,148.8)}, rotate = 0] [color={rgb, 255:red, 208; green, 2; blue, 27 }  ,draw opacity=1 ][fill={rgb, 255:red, 208; green, 2; blue, 27 }  ,fill opacity=1 ][line width=0.75]      (0, 0) circle [x radius= 3.35, y radius= 3.35]   ;
			\draw [color={rgb, 255:red, 208; green, 2; blue, 27 }  ,draw opacity=1 ]   (57.38,148.86) ;
			\draw [shift={(57.38,148.86)}, rotate = 0] [color={rgb, 255:red, 208; green, 2; blue, 27 }  ,draw opacity=1 ][fill={rgb, 255:red, 208; green, 2; blue, 27 }  ,fill opacity=1 ][line width=0.75]      (0, 0) circle [x radius= 3.35, y radius= 3.35]   ;
			\draw [color={rgb, 255:red, 208; green, 2; blue, 27 }  ,draw opacity=1 ]   (185.67,149.17) ;
			\draw [shift={(185.67,149.17)}, rotate = 0] [color={rgb, 255:red, 208; green, 2; blue, 27 }  ,draw opacity=1 ][fill={rgb, 255:red, 208; green, 2; blue, 27 }  ,fill opacity=1 ][line width=0.75]      (0, 0) circle [x radius= 3.35, y radius= 3.35]   ;
			\draw    (425.33,149) -- (465.77,148.85) -- (497.23,148.82) -- (528.68,148.79) -- (560.14,148.76) -- (601.33,149) ;
			\draw [shift={(601.33,149)}, rotate = 0.34] [color={rgb, 255:red, 0; green, 0; blue, 0 }  ][fill={rgb, 255:red, 0; green, 0; blue, 0 }  ][line width=0.75]      (0, 0) circle [x radius= 3.35, y radius= 3.35]   ;
			\draw [shift={(425.33,149)}, rotate = 359.79] [color={rgb, 255:red, 0; green, 0; blue, 0 }  ][fill={rgb, 255:red, 0; green, 0; blue, 0 }  ][line width=0.75]      (0, 0) circle [x radius= 3.35, y radius= 3.35]   ;
			\draw [color={rgb, 255:red, 208; green, 2; blue, 27 }  ,draw opacity=1 ]   (512.95,148.8) ;
			\draw [shift={(512.95,148.8)}, rotate = 0] [color={rgb, 255:red, 208; green, 2; blue, 27 }  ,draw opacity=1 ][fill={rgb, 255:red, 208; green, 2; blue, 27 }  ,fill opacity=1 ][line width=0.75]      (0, 0) circle [x radius= 3.35, y radius= 3.35]   ;
			\draw [color={rgb, 255:red, 208; green, 2; blue, 27 }  ,draw opacity=1 ]   (450.05,148.86) ;
			\draw [shift={(450.05,148.86)}, rotate = 0] [color={rgb, 255:red, 208; green, 2; blue, 27 }  ,draw opacity=1 ][fill={rgb, 255:red, 208; green, 2; blue, 27 }  ,fill opacity=1 ][line width=0.75]      (0, 0) circle [x radius= 3.35, y radius= 3.35]   ;
			\draw [color={rgb, 255:red, 208; green, 2; blue, 27 }  ,draw opacity=1 ]   (578.33,149.17) ;
			\draw [shift={(578.33,149.17)}, rotate = 0] [color={rgb, 255:red, 208; green, 2; blue, 27 }  ,draw opacity=1 ][fill={rgb, 255:red, 208; green, 2; blue, 27 }  ,fill opacity=1 ][line width=0.75]      (0, 0) circle [x radius= 3.35, y radius= 3.35]   ;
			\draw    (221.1,149.29) ;
			\draw [shift={(221.1,149.29)}, rotate = 0] [color={rgb, 255:red, 0; green, 0; blue, 0 }  ][fill={rgb, 255:red, 0; green, 0; blue, 0 }  ][line width=0.75]      (0, 0) circle [x radius= 3.35, y radius= 3.35]   ;
			\draw    (27.99,148.94) ;
			\draw [shift={(27.99,148.94)}, rotate = 0] [color={rgb, 255:red, 0; green, 0; blue, 0 }  ][fill={rgb, 255:red, 0; green, 0; blue, 0 }  ][line width=0.75]      (0, 0) circle [x radius= 3.35, y radius= 3.35]   ;
			\draw    (298,151) -- (367.33,151) ;
			\draw [shift={(369.33,151)}, rotate = 180] [color={rgb, 255:red, 0; green, 0; blue, 0 }  ][line width=0.75]    (10.93,-3.29) .. controls (6.95,-1.4) and (3.31,-0.3) .. (0,0) .. controls (3.31,0.3) and (6.95,1.4) .. (10.93,3.29)   ;
			\draw   (248,59) -- (404.33,59) -- (404.33,111) -- (248,111) -- cycle ;
			
			\draw (23.33,160.4) node [anchor=north west][inner sep=0.75pt]    {$0$};
			\draw (193.33,161.4) node [anchor=north west][inner sep=0.75pt]    {$x_{2} +x_{3}$};
			\draw (103.33,84.4) node [anchor=north west][inner sep=0.75pt]    {$c_{2}( x)$};
			\draw (446,161.4) node [anchor=north west][inner sep=0.75pt]    {$0$};
			\draw (553,157.4) node [anchor=north west][inner sep=0.75pt]    {$\frac{x_{2} +x_{3}}{2} -1$};
			\draw (501,85.4) node [anchor=north west][inner sep=0.75pt]    {$\tilde{c}_{2}( x)$};
			\draw (52,161.4) node [anchor=north west][inner sep=0.75pt]    {$1$};
			\draw (251,140.4) node [anchor=north west][inner sep=0.75pt]    {$j$};
			\draw (630,141.4) node [anchor=north west][inner sep=0.75pt]    {$k$};
			\draw (258,76.4) node [anchor=north west][inner sep=0.75pt]    {$\phi :\ j\ \mapsto \ \frac{1}{2} j-\frac{1}{2}$};
			\draw (115,161.4) node [anchor=north west][inner sep=0.75pt]    {$3$};
			\draw (510,161.4) node [anchor=north west][inner sep=0.75pt]    {$1$};
			\draw (409,160.4) node [anchor=north west][inner sep=0.75pt]    {$-\frac{1}{2}$};

		\end{tikzpicture}
		\caption{Sketch of affine coordinate change $\phi$. The red points correspond to the lattice points that contribute to the count.}
		\label{figure_inexampleforaffinecoordchange}
	\end{figure}
\end{example}

\begin{lemma}\label{lemma_qpofperiodatmost2}
	
	Let $\mathbb{Z}^{n'} \rightarrow \mathbb{R}^n, x \mapsto P(x)$ be a family of $n$-dimensional rational polytopes and $v_i(x)= G_ix+ h_i \text{ with } G_i\in (\frac{1}{2}\mathbb{Z})^{n' \times n} \text{ and } h_i \in (\frac{1}{2}\mathbb{Z})^n, i=1,\dots,l$, such that, for all $x$ in a region $c\subset \mathbb{Q}^{n'}$, the vertex set $V(P(\underline{x}))$ of $P(\underline{x})$ is given by:
	\begin{align*}
		V(P(x))=\{v_1(x),\dots,v_l(x)\}.
	\end{align*} 
	Then, for any $f\in \mathbb{Q}[j_1,\dots,j_n]$
	\begin{align*}
		\mathbb{Z}^{n'}\cap c \rightarrow \mathbb{R}, x \mapsto \sum_{ j \in P(x) \cap \mathbb{Z}^n} f(j)
	\end{align*}
	is a quasi-polynomial in $x$ of degree $n + \deg(f)$ whose period $(Q_1,\dots,Q_{n'})$ satisfies $Q_i\in \{1,2\} \forall i$.
\end{lemma} 
It is known that for a family of polytopes as in \cref{lemma_qpofperiodatmost2}, the lattice-point enumerator ($f=1$) is a quasi-polynomial of period at most $2$. To obtain the corresponding result for the weighted enumeration we use the methodology of \cite{arXiv:2402.11328}, converting the weighted sum into a lattice point count in a higher dimensional polytope.
\begin{proof}
	Without loss of generality, assume $P(x)\subset \mathbb{R}_{\geqslant 0}^n$ (if not: see \cref{remark_onwlogforPhascoord>=0}) and  $f(j)=j^{\alpha_1}_1\cdot \dots \cdot j^{\alpha_n}_n$. The statement of \cref{lemma_qpofperiodatmost2} for general $f' \in \mathbb{Q}[j_1,\dots,j_n]$ follows by decomposing $f'$ into a sum of monomials.
	
	\emph{Step 1: Lift to the unweighted case.} Convert the $V$-description of $P(x)$ on $c$ into an $H$-description and write $P(x)$ in the form $P(x)=\{y| A(x)y = b(x), y\geqslant 0\}$ with $A(x)\in \mathbb{Z}[x]^{s\times \tilde{n}}, b(x)\in \mathbb{Z}[x]^{s}$ by introducing additional variables (i.e. $n\leqslant \tilde{n}$).

	Next, consider the weight-lifting polytope $P^*(x)$ for $P(x)$ and weight $f$ constructed as in \cite{arXiv:2402.11328} (Corollary 2.5 and 2.6):
	\begin{align*}
		&	P^*=\bigg{\{} \begin{pmatrix}
			y\\y'
		\end{pmatrix}| \begin{pmatrix}
			A(x)   & 0 \\
			\underbrace{D}_{n} \thickspace \underbrace{0}_{\tilde{n}-n} & C
		\end{pmatrix}\begin{pmatrix}
			y\\y'
		\end{pmatrix}= \begin{pmatrix}
			b(x)\\- e
		\end{pmatrix}, y,y' \geqslant 0 \bigg{\}} \subset \mathbb{R}^{\tilde{n}+m}, \text{ where }\\
		&
			e =\begin{pmatrix}
				-1\\
				\vdots \\ -1
			\end{pmatrix}  
			, C= \begin{pmatrix}
				C_1 & 0 & 0& \dots & 0\\
				0   & C_2& 0 &\dots&0\\
				\vdots& \dots & \ddots & \dots & \vdots \\
				0 & \dots& \dots & \dots & C_n
			\end{pmatrix} \text{ with } C_i= \begin{pmatrix}
				I_{\alpha_i} & I_{\alpha_i}
			\end{pmatrix} \text{ and }\\
			&D= \begin{pmatrix}
				d_1 & 0 & 0& \dots & 0\\
				0   & d_2& 0 &\dots&0\\
				\vdots& \vdots & \ddots &\vdots & \vdots \\
				0 & \dots& \dots & \dots & d_n
			\end{pmatrix} \text{ with } \left.\begin{aligned}
				d_i =\begin{pmatrix}
					1\\
					\vdots \\ 1
				\end{pmatrix}  
			\end{aligned}\right\}
			\text{$\alpha_i$}, 
		\end{align*} 
		and note: For $y\in P(x)$ with $y_i>1$, the set $\big{\{}y' | \begin{psmallmatrix}
			D & 0
		\end{psmallmatrix} y + e = Cy', y'\geqslant 0 \big{\}}$ is equal to the product $(y_1-1)C(\alpha_1)\times \dots \times (y_n-1)C(\alpha_n)$ of dilated hypercubes $(y_i-1)C(\alpha_i)$ of size $\alpha_i$ associated to $y$.

		For fixed $x\in \mathbb{Z}^{n'}$, Theorem 1.1 (\cite{arXiv:2402.11328}) gives: 
		\begin{align}\label{equation_inproofQPolyofperiod<=2_1}
			\sum_{ y \in P(x) \cap \mathbb{Z}^{\tilde{n}}} f(y)= \sum_{ \begin{psmallmatrix}
					y \\y'
				\end{psmallmatrix}\in P^*(x) \cap \mathbb{Z}^{\tilde{n}+m}} 1.
		\end{align}
		\emph{Step 2: Proof of the unweighted case.} Without loss of generality, suppose that the vertices of $P(x)$ have coordinates strictly larger than $1$. If not, translate $P$ by $v\in \mathbb{Z}^{\tilde{n}}$, $P^*(x)$ by $(v , 0)^T$ and consider 
		\begin{align}\label{equation_inproofQPolyofperiod<=2_2}
			\sum_{ y \in \big{(}P(x) + v\big{)} \cap \mathbb{Z}^{\tilde{n}}} f(y-v)= \sum_{ \begin{psmallmatrix}
					y \\y'
				\end{psmallmatrix}\in \big{(}P^*(x) + \begin{psmallmatrix}
					v \\0
				\end{psmallmatrix}\big{)}  \cap \mathbb{Z}^{\tilde{n}+m}} 1.
		\end{align}
		If (\ref{equation_inproofQPolyofperiod<=2_2}) gives rise to a quasiploynomial in $x$ of period at most $2$, so does (\ref{equation_inproofQPolyofperiod<=2_1}).
		
		By \cref{lemma_verticesoflift}, the vertices of $P^*(x)$ are of the form $v^*(x)=(y(x),y'(x))$, where
		\begin{itemize}
			\item $y(x):=(y_1(x),\dots,y_{\tilde{n}}(x))$ is a vertex of $P(x)$.
			\item $y'(x)$ is a vertex of the product of hypercubes $\prod^n_{i=1}(y_i(x)-1)C(\alpha_i)$ associated to $y(x)$.
		\end{itemize}
		In particular, $v^*(x)$ has only coordinates $y_i(x)$, $y_i(x)-1$ or $0$ for $y(x)$ a vertex of $P(x)$.
		
		This means: The weight-lifting process (Step $1$) gives rise to an $|\alpha| + n$-dimensional family $c\ni x \mapsto P^*(x)$ with the property that for every $x \in c$, each $v^*(x)\in V(P^*(x))$ is an affine linear function in $x$ with coefficients in $\frac{1}{2}\mathbb{Z}$ (since this holds for the vertices of $P(x)$ by assumption). We can apply Proposition 1 (\cite{zbMATH05172599}) to conclude that
		\begin{align*}
			\mathbb{Z}^{n'}\cap c \rightarrow \mathbb{R}, x \mapsto \sum_{ j \in P^*(x) \cap \mathbb{Z}^{\tilde{n}+m}} 1
		\end{align*}
		is a quasi-polynomial in $x$ of degree $n + |\alpha|= n + \deg(f)$ whose period $(Q_1,\dots,Q_{n'})$ satisfies $Q_i\in \{1,2\} \forall i$.
		
	\end{proof}
	\begin{remark}\label{remark_onwlogforPhascoord>=0}
		If $P(x)\subset \mathbb{R}_{\geqslant 0}^n$ does not hold, then writing $P(x)$ in the form $A(x)y=b(x), y\geqslant 0$ necessitates introducing for each original variable $j_i$ in $P$ two new variables $j^+_i$ and $j^-_i$ (representing positive and negative parts, respectively). Under this transformation: A monomial $j^\alpha$ becomes a \emph{sum} of monomials in the variables $j^+_i$ and $j^-_i$ (e.g. $j_1j_2$ turns into $j^+_1j^+_2 + j^-_1j^-_2 - j^+_1j^-_2 - j^-_1j^+_2$). This requires the following adjustments in the proof of \cref{lemma_qpofperiodatmost2}: 
		\begin{itemize}
			\item Prove Step $1$ for a monomial $f$ in $2n$ variables (replace index $n$ by $2n$ from Step 1 onward).
			\item Deduce the statement of \cref{lemma_qpofperiodatmost2} for polytopes $\tilde{P}(x)$ in the form $A(x)y=b(x), y\geqslant 0$ and general polynomial weight $f$.
			\item Recover \cref{lemma_qpofperiodatmost2} for the original $P(x)$.
		\end{itemize}
		
	\end{remark}
	
	\begin{lemma}\label{lemma_verticesoflift}
		Let 
		\begin{itemize}
			\item $P \subset \mathbb{R}^n$ be a rational polytope in the form $\{x | Ax=b, x\geqslant 0 \}$, where $A\in \mathbb{Z}^{s \times n}$ and $b\in \mathbb{Z}^s$, whose vertices have coordinates strictly greater than $1$.
			\item $P^* \subset \mathbb{R}^{n+m}$ be the weight-lifting polytope for the weight function $\omega(\underline{x}):=x^{\alpha_1}_1\cdot \dots \cdot x^{\alpha_n}_n$ constructed in \cite{arXiv:2402.11328}, Corollary 2.5 and 2.6.
		\end{itemize}
		Then, the vertex set $V(P^*)$ satisfies
		\begin{align*}
			V(P^*) \subset \{(x,x')| x\in V(P) \text{and } x'\in V(\prod^n_{i=1}(x^i-1)C(\alpha_i))\}, \text{ where}
		\end{align*}
		$(x^1-1)C(\alpha_1)\times \dots \times (x^n-1)C(\alpha_n)$ denotes the product of dilated hypercubes $(x^i-1)C(\alpha_i)$ of size $\alpha_i$ associated to $x$.
	\end{lemma}
	\begin{proof}
		Without loss of generality, assume $\alpha_i>0$ for $i=1,\dots,n$. Otherwise, permute the coordinates of $P$ such that the first $\tilde{n}$ coordinates correspond to $\alpha_i>0$, and construct $P^*$ as in the proof of \cref{lemma_qpofperiodatmost2}. The following then applies, noting that the product of hypercubes is only taken up to $\tilde{n}$.
		
		Let $M:=\{(x,x')| x\in V(P) \text{ and } x'\in V(\prod^n_{i=1}(x^i-1)C(\alpha_i))\}$. We show $P^*=\text{conv}(M)$.\\
		Observations:
		\begin{enumerate}
			\item A vertex of the dilated hypercube $(x^i-1)C(\alpha_i)$ has coordinates in $\{0,x^i-1\}$.
			\item A vertex of the product $\prod^n_{i=1}(x^i-1)C(\alpha_i)$ is a product of vertices from each of the factors.
		\end{enumerate}
		Consider a point $(\tilde{x},\tilde{x}')\in P^*$. Then, $\tilde{x}\in P$  and  $\tilde{x}'\in\prod^n_{i=1}(\tilde{x}^i-1)C(\alpha_i)$, where $\tilde{x}'\in\prod^n_{i=1}(\tilde{x}^i-1)C(\alpha_i)$ is the product of positively dilated hypercubes associated to $\tilde{x}$ (Recall: $\tilde{x}^i-1 >1$ holds by assumption on the vertex coordinates of $P$). Hence, we can write
		\begin{itemize}
			\item $\tilde{x}=\sum_{k}\lambda_kx_k$ for some $\lambda_k\geqslant 0$ with $\sum_{k}\lambda_k=1$ and $x_k\in V(P)$ .
			\item $\tilde{x}'=\sum_{j}\lambda'_jy_j$ for some $\lambda'_j\geqslant 0$ with $\sum_{j}\lambda'_j=1$ and $y_j\in V(\prod^n_{i=1}(\tilde{x}^i-1)C(\alpha_i)).$
		\end{itemize}
		Combining Observations 1 and 2, we use the coefficients $\lambda_k$ to express each vertex $y_j$ as a convex combination of vertices $y^{j}_k$ of the product $\prod^n_{i=1}(x_k^i-1)C(\alpha_i)$ associated to $x_k\in V(P)$: $y_j=\sum_k \lambda_ky^{j}_k$. Here, the superscript $j$ does not denote a component, but indicates that the choice of vertex $y^{j}_k$ of $\prod^n_{i=1}(x_k^i-1)C(\alpha_i)$ depends on $y_j$. Then, $\tilde{x}'=\sum_{j}\lambda'_jy_j=\sum_{j}\sum_k\lambda'_j \lambda_ky^{j}_k$ yields
		\begin{align*}
			(\tilde{x},\tilde{x}')=\sum_{j}\sum_k \lambda'_j \lambda_k (x_k,y^{j}_k), \text{ where } (x_k,y^{j}_k)\in M,
		\end{align*}
		as desired.
	\end{proof}
	The following example was computed using \textsc{Polymake} \cite{polymake:2000}.
	\begin{example}
		Consider a rectangle $P\subset \mathbb{R}^2$ with vertices $(\frac{7}{2},2),(\frac{7}{2},\frac{5}{2}),(2,\frac{5}{2}),(2,2)$. Write $P$ in the form $\{x | Ax=b, x\geqslant 0 \}$:\\
		\begin{small}
			\begin{verbatim}
				$P = new Polytope<Rational>( EQUATIONS => [[-11,2, 0,2,0,0,0,0,0],
				[-9,0, 2,0,2,0,0,0,0],[-2,1,0,0,0,-1,0,0,0],[-2,0,1,0,0,0,-1,0,0],
				[-2,0,0,1,0,0,0,-1,0],[-2,0,0,0,1,0,0,0,-1]], INEQUALITIES => 
				[[0, 1, 0,0,0,0,0,0,0],[0, 0, 1,0,0,0,0,0,0],[0, 0, 0, 1,0,0,0,0,0],
				[0,0,0,0,1,0,0,0,0],[0,0,0,0,0,1,0,0,0],[0,0,0,0,0,0,1,0,0],
				[0,0,0,0,0,0,0,1,0],[0,0,0,0,0,0,0,0,1]] );
				
				polytope > print $P->VERTICES;
				polymake: used package ppl                                                                                                                  
				The Parma Polyhedra Library ([[wiki:external_software#PPL]]):
				A C++ library for convex polyhedra                                          
				and other numerical abstractions.                                                                                                         
				http://www.cs.unipr.it/ppl/                                                                                                               
				
				1 7/2 2 2 5/2 3/2 0 0 1/2                                                                                                                   
				1 7/2 5/2 2 2 3/2 1/2 0 0                                                                                                                   
				1 2 5/2 7/2 2 0 1/2 3/2 0                                                                                                                   
				1 2 2 7/2 5/2 0 0 3/2 1/2        
			\end{verbatim}
		\end{small}
		\ \\
		
		Let $P^*$ be the weight-lifting polytope for the weight function $\omega(\underline{x}):=x^{1}_1 x^{2}_2$ :\\
		\begin{small}
			\begin{verbatim}
				polytope > $P = new Polytope<Rational>( EQUATIONS => 
				[[-11,2, 0,2,0,0,0,0,0,0,0,0,0,0,0],[-9,0, 2,0,2,0,0,0,0,0,0,0,0,0,0],
				[-2,1,0,0,0,-1,0,0,0,0,0,0,0,0,0],[-2,0,1,0,0,0,-1,0,0,0,0,0,0,0,0],
				[-2,0,0,1,0,0,0,-1,0,0,0,0,0,0,0],[-2,0,0,0,1,0,0,0,-1,0,0,0,0,0,0],
				[-1,1,0,0,0,0,0,0,0,-1,-1,0,0,0,0],[-1,0,1,0,0,0,0,0,0,0,0,-1,0,-1,0],
				[-1,0,1,0,0,0,0,0,0,0,0,0,-1,0,-1]],INEQUALITIES => 
				[[0, 1, 0,0,0,0,0,0,0,0,0,0,0,0,0],[0,0, 1, 0,0,0,0,0,0,0,0,0,0,0,0],
				[0,0,0,1,0,0,0,0,0,0,0,0,0,0,0],[0,0,0,0,1,0,0,0,0,0,0,0,0,0,0], 
				[0,0,0,0,0,1,0,0,0,0,0,0,0,0,0],[0,0,0,0,0,0,1,0,0,0,0,0,0,0,0],
				[0,0,0,0,0,0,0,1,0,0,0,0,0,0,0],[0,0,0,0,0,0,0,0,1,0,0,0,0,0,0],
				[0,0,0,0,0,0,0,0,0,1,0,0,0,0,0],[0,0,0,0,0,0,0,0,0,0,1,0,0,0,0],
				[0,0,0,0,0,0,0,0,0,0,0,1,0,0,0],[0,0,0,0,0,0,0,0,0,0,0,0,1,0,0],
				[0,0,0,0,0,0,0,0,0,0,0,0,0,1,0],[0,0,0,0,0,0,0,0,0,0,0,0,0,0,1]]);
				
			\end{verbatim}
			
		\end{small}
		
		Then, $P^*$ has exactly $|V(P)|\cdot 2^{1}\cdot 2^{2}= 32$ vertices:\\
		\begin{small}
			\begin{verbatim}
				polytope > print $P->VERTICES;                                                                                                              
				1 7/2 2 2 5/2 3/2 0 0 1/2 5/2 0 1 0 0 1
				1 7/2 2 2 5/2 3/2 0 0 1/2 0 5/2 1 0 0 1
				1 7/2 5/2 2 2 3/2 1/2 0 0 5/2 0 3/2 3/2 0 0
				1 7/2 2 2 5/2 3/2 0 0 1/2 0 5/2 0 1 1 0
				1 7/2 2 2 5/2 3/2 0 0 1/2 5/2 0 0 1 1 0
				1 7/2 5/2 2 2 3/2 1/2 0 0 5/2 0 0 3/2 3/2 0
				1 7/2 5/2 2 2 3/2 1/2 0 0 0 5/2 0 3/2 3/2 0                                                                                                 
				1 7/2 5/2 2 2 3/2 1/2 0 0 0 5/2 3/2 0 0 3/2                                                                                                 
				1 7/2 5/2 2 2 3/2 1/2 0 0 5/2 0 3/2 0 0 3/2                                                                                                 
				1 7/2 5/2 2 2 3/2 1/2 0 0 0 5/2 0 0 3/2 3/2                                                                                                 
				1 7/2 5/2 2 2 3/2 1/2 0 0 5/2 0 0 0 3/2 3/2                                                                                                 
				1 7/2 5/2 2 2 3/2 1/2 0 0 0 5/2 3/2 3/2 0 0                                                                                                 
				1 7/2 2 2 5/2 3/2 0 0 1/2 0 5/2 1 1 0 0                                                                                                     
				1 7/2 2 2 5/2 3/2 0 0 1/2 5/2 0 0 0 1 1                                                                                                     
				1 7/2 2 2 5/2 3/2 0 0 1/2 0 5/2 0 0 1 1                                                                                                     
				1 7/2 2 2 5/2 3/2 0 0 1/2 5/2 0 1 1 0 0                                                                                                     
				1 2 2 7/2 5/2 0 0 3/2 1/2 1 0 1 1 0 0                                                                                                       
				1 2 2 7/2 5/2 0 0 3/2 1/2 0 1 0 0 1 1                                                                                                       
				1 2 2 7/2 5/2 0 0 3/2 1/2 1 0 0 0 1 1                                                                                                       
				1 2 2 7/2 5/2 0 0 3/2 1/2 0 1 1 1 0 0                                                                                                       
				1 2 5/2 7/2 2 0 1/2 3/2 0 0 1 3/2 3/2 0 0                                                                                                   
				1 2 5/2 7/2 2 0 1/2 3/2 0 1 0 0 0 3/2 3/2                                                                                                   
				1 2 5/2 7/2 2 0 1/2 3/2 0 0 1 0 0 3/2 3/2                                                                                                   
				1 2 5/2 7/2 2 0 1/2 3/2 0 1 0 3/2 0 0 3/2                                                                                                   
				1 2 2 7/2 5/2 0 0 3/2 1/2 1 0 1 0 0 1                                                                                                       
				1 2 2 7/2 5/2 0 0 3/2 1/2 0 1 1 0 0 1                                                                                                       
				1 2 5/2 7/2 2 0 1/2 3/2 0 1 0 3/2 3/2 0 0                                                                                                   
				1 2 2 7/2 5/2 0 0 3/2 1/2 0 1 0 1 1 0                                                                                                       
				1 2 2 7/2 5/2 0 0 3/2 1/2 1 0 0 1 1 0                                                                                                       
				1 2 5/2 7/2 2 0 1/2 3/2 0 1 0 0 3/2 3/2 0                                                                                                   
				1 2 5/2 7/2 2 0 1/2 3/2 0 0 1 0 3/2 3/2 0                                                                                                   
				1 2 5/2 7/2 2 0 1/2 3/2 0 0 1 3/2 0 0 3/2                                                                                                   
				
				polytope > print $P->N_VERTICES;
				32   
			\end{verbatim}
			
		\end{small}

	\end{example}
	\paragraph{Wall crossing formulae.} 
	Let $C_1$ and $C_2$ be two chambers separated by a wall $\mathcal{W}_{I}$ and assume that $\sum_{i\in I} x_i <0$  holds in $C_1$, but not in $C_2$. Once again, we seek a more intrinsic expression for 
	\begin{align}\label{equation_wallcrossingdefg>0}
		WC^{r}_{g,I}: C_2 \rightarrow \mathbb{Q},\thinspace x \mapsto {\mathbb{T}h_g^{spin,r}}_{|C_2}(x)-{\mathbb{T}h_g^{spin,r}}_{|C_1}(x),
	\end{align} 
	in terms of Hurwitz numbers. Following the strategy described in \cite{zbMATH06806783}, we proceed in two steps.

	\emph{Step 1: Explicit Wall-Crossing formula.} As in the genus $0$ case, the analysis of $WC^{r}_{g,I}$ reduces to studying the wall-crossing behavior of individual $x-$graphs. By \cref{remark_theshapeofA(x)} and \cref{proposition_multGammaispolynomial}, it is enough to consider graphs $\Gamma$ for which $\phi_\Gamma$ jumps when crossing $\mathcal{W}_{I}$, otherwise the contribution to ${\mathbb{T}h_g^{spin,r}}_{|C_2}(x)$ cancels with the one to ${\mathbb{T}h_g^{spin,r}}_{|C_1}(x)$. Doing so we note: The condition "$\phi_\Gamma$ drops for $x\in W_I$" means that $\mathcal{A}_\Gamma(x_2)$ for $x_2\in C_2$ contains bounded chambers that do not exist in $\mathcal{A}_\Gamma(x_1)$ for $x_1\in C_1$ or, equivalently, the inequality $\sum_{i\in I}x_i <0$ (that holds in $C_1$) excludes certain orientations adequate in $C_2$ and vice versa.
	
	We observed in genus $0$ that we could relate these by cutting and reversing the orientation of the \emph{unique} edge of weight $\sum_{i\in I}x_i$. The situation is similar for higher genus, except that we usually have to cut and reorient \emph{more than one} edge to obtain a particular orientation in $\mathcal{A}_\Gamma(x_1)$ from one in $\mathcal{A}_\Gamma(x_2)$. The \emph{set of $I-$cuts} comprises such cuts.
	
	\begin{definition}[\cite{MR2836109}, Definition 6.4]
		For a bounded chamber $c$ in $\mathcal{A}_\Gamma$, we denote by $\Gamma_c$ the corresponding directed graph and associate to each subset of ends indexed by $I\subset \{1,\dots,n\}$ a poset $Cuts_{\Gamma_c}(I)\subset E(\Gamma_c)$, where $E\in Cuts_{\Gamma_c}(I)$ holds, if  $E=\emptyset$ or 
		\begin{enumerate}
			\item Cutting $\Gamma_c$ along $E$ distributes the ends of $\Gamma$ on precisely two connected components, one containing all ends in $I$ (\emph{the initial component}), the other $I^c$ (\emph{the final component}). In particular, $\Gamma_c\setminus E$ is disconnected. 
			\item The directed graph $\Gamma_c/E^c$ is acyclic with initial vertex dual to the initial component and final vertex dual to the final component. 
		\end{enumerate}
		Each $E\in Cuts_{\Gamma_c}(I)$ is ranked by $\rk(E):=\nu(\Gamma_c\setminus E)-1$, where $\nu(\Gamma_c\setminus E)$ denotes the number of components of $\Gamma_c\setminus E$.
	\end{definition}
	Despite the added combinatorial complexity, we encounter a familiar phenomenon: The complements of $I-$cuts (i.e. the connected components of $\Gamma_c\setminus E$ for $E\in Cuts_{\Gamma_c}(I)$) form the common building blocks of directed graphs $\Gamma_c$ that occur on both sides of $W_I$ (compare to \cref{proposition_adjacentchamberstructure}). This allows us to express $WC^r_{g,I}(x)$ in terms of these building blocks:
	\begin{theorem}\label{theorem_wallcrossinggenusg>0explicit}[Explicit Wall crossing formula]
		For the wall $W_I$ the wall crossing formula is given by:
		\begin{align*}
			WC^r_{g,I}(x_2)=& \sum_\Gamma \sum_{c\in \mathcal{BC}(\mathcal{A}_\Gamma(x_2))} \bigg{(}\sum_{E\in Cuts_{\Gamma_c}(I)}{(-1)}^{rk(E)-1} \binom{a}{a_0,a_1,\dots,a_N,a_{N+1}}\bigg{)} \\
			&	\bigg{(}\sum_{j \in c \cap (2\mathbb{Z} + 1)^{{b_1(\Gamma)}} } P_\Gamma(x_2,j)\bigg{)}, 
		\end{align*}
		where 
		\begin{enumerate}
			\item $\mathcal{BC}(\mathcal{A}_\Gamma(x_2))$ is the set of bounded chambers of $\mathcal{A}_\Gamma(x_2)$.
			\item $N=\rk(E)-1$ is the number of inner connected components of $\Gamma_c\setminus E$. 
			\item $a_0$ and $a_{N+1}$ ($a_1,\dots,a_N$) denote the number of vertices in the initial, respectively final, (inner) component(s) of $\Gamma_c\setminus E$.
		\end{enumerate}
	\end{theorem}
	\cref{theorem_wallcrossinggenusg>0explicit} is explicit. One can read off the wall crossing behavior of individual $x$-graphs from it since:
	\begin{itemize}
		\item $\binom{a}{a_0,a_1,\dots,a_N,a_{N+1}}$ is the number of orderings on $a:=|V(\Gamma)|$ vertices that refine $N+2$ given orderings on $a_0,a_1,\dots,a_N$ and $a_{N+1}$ vertices (see \cref{proposition_adjacentchamberstructure} in genus $0$ for reference).
		\item ${(-1)}^{rk(E)-1}$ plays a twofold role depending on whether cutting and regluing $\Gamma_c$ along reoriented edges of $E$ yields an adequate orientation on $\Gamma$ or not: If it does, then ${(-1)}^{rk(E)-1}$ provides the right prefactor, and ensures, if otherwise, that this orientation does not contribute (see \cite{MR2836109}, Lemma 6.12, for a precise statement).
	\end{itemize} 
	For the proof of \cref{theorem_wallcrossinggenusg>0explicit}, we essentially refer to \cite{zbMATH06806783}. However, our slightly different setup requires us to extend relevant machinery from integral to rational hyperplane arrangements (an arrangement is integral, if each hyperplane is defined by an equation with integral coordinates and the vertices of the arrangement have integral coordinates as well).
	Consequently, this proof is only a sketch. It is not self-contained, as it borrows notation from the original paper.
	\begin{proofsketch}[\cref{theorem_wallcrossinggenusg>0explicit}]
		Denote by $WC[\Gamma]$ the wall-crossing contribution of an $x$-graph $\Gamma$. The affine map 
		\begin{align*}
			\phi:  \mathbb{R}^{b_1(\Gamma)}\rightarrow \mathbb{R}^{b_1(\Gamma)}, j \mapsto \frac{1}{2}\text{Id}j - \frac{1}{2} \begin{psmallmatrix}
				1\\
				\vdots\\
				1
			\end{psmallmatrix}
		\end{align*}
		transforms the family of integer hyperplane arrangements $\mathbb{Z}^{n}  \rightarrow \mathbb{R}^{b_1(\Gamma)}, x \mapsto \mathcal{A}_\Gamma(x)$ into a family $\mathbb{Z}^{n}  \rightarrow \mathbb{R}^{b_1(\Gamma)}, x \mapsto \tilde{\mathcal{A}}_\Gamma(x):=\phi(\mathcal{A}_\Gamma (x))$ whose bounded chambers are polytopes with vertex coordinates being affine linear functions in $x$ with coefficients in $\frac{1}{2}\mathbb{Z}$ (see Proof of \cref{proposition_multGammaispolynomial}). Just like in \cite{zbMATH06806783}, consider the Gauss-Manin connection $\nabla_{\Gamma,12}$ of the vector bundle arising from $x \mapsto \tilde{\mathcal{A}}_\Gamma(x)$ and its adjoint $\nabla^*_{\Gamma,12}$ (\cite{zbMATH06806783}, Definition 5.1.).  
		By \cref{proposition_multGammaispolynomial} the contribution of a bounded chamber $\tilde{c}$ of $\tilde{\mathcal{A}}_\Gamma(x_1)$  (where $x_1\in C_1\cap \mathcal{H}$) is a polynomial $F$ on $C_1$. The value $F(x_2)$ at $x_2\in C_2$ can be obtained by summing $\tilde{P}_\Gamma(x_2,k)$ over a signed sum of bounded chambers of $\tilde{\mathcal{A}}_\Gamma(x_2)$. This signed sum is encoded in the Gauss-Manin connection (see \cite{zbMATH04017760}). Analogously to \cite{zbMATH06806783} (Equation (9)) we have:
		\begin{align*}
			&	WC[\Gamma]=\sum_{\hat{c}\in \tilde{\mathcal{A}}_\Gamma(x_2))} WC[\Gamma,\hat{c}]\underbrace{(\sum_{k \in \hat{c} \cap \mathbb{Z}^{{b_1(\Gamma)}} } \tilde{P}_\Gamma(x_2,k))}_{=:P^c_\Gamma(x_2)), \phi(c)=\hat{c}},\text{ where }\\
			&	WC[\Gamma,\hat{c}]:=\text{sgn}(\hat{c})(m(\hat{c})- \sum_{\tilde{c}\in \mathcal{BC}(\tilde{\mathcal{A}}_\Gamma(x_1))}m(\tilde{c}) \langle \tilde{c}, \nabla^*_{\Gamma,12}(\hat{c})),\\ & \text{sgn}(\hat{c}):=\text{sgn}(\tilde{P}_\Gamma(x_2,k)), \text{ and }
			m(\tilde{c}):=m(\phi^{-1}(\tilde{c})),
		\end{align*}
		and therefore:
		\begin{align*}
			WC[\Gamma]=\sum_{c\in \mathcal{A}_\Gamma(x_2))} WC[\Gamma,c] P^c_\Gamma(x_2)),\thinspace	WC[\Gamma,c]:=WC[\Gamma,\hat{c}]
		\end{align*}
		since $\phi$ defines a bijection between the set $\mathcal{BC}(\mathcal{A}_{\Gamma}(x))$ of bounded chambers in $\mathcal{A}_\Gamma(x))$ and the set $\mathcal{BC}(\tilde{\mathcal{A}}_\Gamma(x))$ of bounded chambers in $\tilde{\mathcal{A}}_\Gamma(x))$. 
		
		With this setup \cref{theorem_wallcrossinggenusg>0explicit} follows from the discussion in Section 6 and 7 of \cite{zbMATH06806783}.
	\end{proofsketch}

	\emph{Step 2: Clean Wall-Crossing formula.} Leveraging our combinatorial understanding of wall crossing, we have decomposed the algebraic expression $WC^r_{g,I}(x)$ according to building blocks of individual $x$-graphs and arrived at an explicit combinatorial formula that directly reflects their wall-crossing behavior (Step 1).

	To revert to an algebraic formulation, we consider particularly nice cuts (\cref{definition_thincuts}) that organize contributions from the inner components of $\Gamma_c\setminus E$.
	
	\begin{definition}[\cite{MR2836109}, Definition 8.1 ]\label{definition_thincuts}
		A cut $T\in Cuts_{\Gamma_c}(I)\subset E(\Gamma_c)$, where every edge is adjacent to either the connected component with ends $x_I$ or the one with ends $x_{I^c}$, is called a \emph{thin cut}. We say $T$ divides $\Gamma_c$ into three \emph{parts}, the initial and the final component, and a (possibly disconnected) middle part.
	\end{definition}
	Thin cuts induce a partition of the set of $I-$Cuts
	\begin{align*}
		Cuts_{\Gamma_c}(I)= \bigcup_T P(T), \text{ where } P(T):=\{E \in Cuts_{\Gamma_c}(I) \thinspace : \thinspace T\subset E\},
	\end{align*}
	according to which we split the sum in \cref{theorem_wallcrossinggenusg>0explicit} (see Lemma 8.2 \cite{MR2836109} as written in \cite{zbMATH07215113}, Remark 4.14)
	\begin{align*}
		\sum_{E\in Cuts_{\Gamma_c}(I)}{(-1)}^{\rk(E)-1} \binom{a}{a_0,a_1,\dots,a_N,a_{N+1}}= \sum_{T}{(-1)}^{t} \binom{a}{a_0,t,a_{N+1}}, 
	\end{align*}
	\text{ where } $t:=a_1+\dots+a_N$.
	Next, consider the three parts of $\Gamma_c\setminus T$ as $x-$graphs (for smaller counting problems) with adequate orientation inherited from $\Gamma_c$. Denote their underlying $x$-graphs by $\Gamma_1,\Gamma_2,\Gamma_3$ and factor $\frac{1}{|\text{Aut}(\Gamma)|}P_\Gamma$ accordingly:
	\begin{align*}
		\frac{1}{|\text{Aut}(\Gamma)|} P_\Gamma= \frac{\prod y_i \prod z_i}{y!z!} \prod^3_{i=1}\frac{1}{|\text{Aut}(\Gamma_i)|}P_{\Gamma_i}, \text{ where}
	\end{align*} 
	\begin{itemize}
		\item $y$ and $z$ are the partitions of edge weights connecting $\Gamma_1$ to $\Gamma_2$ and  $\Gamma_2$ to $\Gamma_3$.
		\item $\frac{1}{y!z!}$ comes from forgetting the labels on the edges corresponding to $y-$ and $z-$ends after gluing. 
	\end{itemize}
	Observing that $\Gamma_1,\Gamma_2$ and $\Gamma_3$ contribute to ${\mathbb{T}h_{g(\Gamma_1)}^{spin,r}}(x_I,-y),{\mathbb{T}h_{g(\Gamma_2)}^{spin,r,\bullet}}(y,-z)$ and ${\mathbb{T}h_{g(\Gamma_3)}^{spin,r}}(z,x_{I^c})$, finally paves the way for \cref{theorem_wallcrossinggenusg>0beautiful}.
	\begin{theorem}\label{theorem_wallcrossinggenusg>0beautiful}[Clean Wall Crossing formula]
		For the wall $W_I$ the wall-crossing formula is given by:
		\begin{align*}
			WC^r_{g,I}(x)=\sum_{(y,z,g_i,a_i)} \binom{a}{a_1,a_2,a_3}\thinspace\frac{\prod y_i \prod z_i}{y!z!} \thinspace
			{\mathbb{T}h_{g_1}^{spin,r}}(x_I,-y) {\mathbb{T}h_{g_2}^{spin, r,\bullet}}(y,-z){\mathbb{T}h_{g_3}^{spin,r}}(z,x_{I^c}), 
		\end{align*}
		where
		\begin{itemize}
			\item $a$ is the number of ramifications of $(r+1)-$completed cycle type of the original counting problem ${\mathbb{T}h_{g}^{spin,r}}(x)$.
			\item $y,z$ are partitions such that $\sum y_i=\sum z_i=\sum_{i\in I}x_i$ and define together with integers $g_i,a_i$ smaller counting problems.
		\end{itemize}
	\end{theorem}
	\begin{proof}
		See \cite{MR2836109}, Section 8.
	\end{proof}
	\begin{remark}
		Note that a summand for the discrete data $(y,z,g_1,g_2,g_3,a_1,a_2,a_3)$ contributes to $WC^r_{g,I}(x)$ if and only if the respective Riemann-Hurwitz conditions are satisfied.
	\end{remark}

\printbibliography
\end{document}